\documentclass{amsart}

\RequirePackage{fancyhdr}
\RequirePackage{lastpage}
\RequirePackage{bm}
\RequirePackage[table,dvipsnames]{xcolor}
\RequirePackage{pgfplots}
\RequirePackage{caption}
\RequirePackage[T1]{fontenc}
\RequirePackage{tabu}
\RequirePackage{graphicx}
\RequirePackage{multirow}
\RequirePackage{amsmath,amssymb,amsthm}
\RequirePackage{hyperref}
\hypersetup{colorlinks=true,allcolors=magenta}
\RequirePackage{courier}
\RequirePackage{tikz}
\RequirePackage{tikz-cd}
\usetikzlibrary{calc,matrix,arrows,decorations.markings}
\RequirePackage{array}
\RequirePackage{enumerate}
\RequirePackage{nicefrac}
\RequirePackage{listings}
\RequirePackage{centernot}
\RequirePackage{mathtools}
\RequirePackage{stmaryrd}
\RequirePackage{euscript}
\RequirePackage{manfnt}
\RequirePackage{marginnote}
\RequirePackage{mathrsfs}
\RequirePackage{calligra}
\RequirePackage[inline]{enumitem}
\RequirePackage[refpage]{nomencl}
\RequirePackage{accents}
\RequirePackage{ytableau}
\RequirePackage{tipa}
\RequirePackage{bbm}
\RequirePackage{dsfont}
\RequirePackage{setspace}
\RequirePackage{chngcntr}
\RequirePackage{afterpage}
\DeclareMathAlphabet\EuScript{U}{eus}{m}{n}
\SetMathAlphabet\EuScript{bold}{U}{eus}{b}{n}

\DeclareMathAlphabet{\mathpzc}{OT1}{pzc}{m}{it}

\newcommand{\F}{\mathbb{F}}

\newcommand{\Fbar}{\overline{\F}}

\newcommand{\Fq}{\F_q}
\newcommand{\Fqbar}{\Fbar_q}

\newcommand{\Q}{\mathbb{Q}}
\newcommand{\Qbar}{\overline{\Q}}

\newcommand{\Qlbar}{\Qbar_\ell}

\newcommand{\bO}{\mathbb{O}}

\newcommand{\bL}{\mathbb{L}}

\newcommand{\Z}{\mathbb{Z}}
\newcommand{\1}{\mathbf{1}}

\newcommand{\Aut}{\operatorname{Aut}}
\newcommand{\Out}{\operatorname{Out}}

\newcommand{\rmR}{{\rm R}}
\newcommand{\rmJ}{{\rm J}}

\newcommand{\SL}{{\rm SL}}
\newcommand{\GL}{\operatorname{GL}}
\newcommand{\PGL}{{\rm PGL}}

\newcommand{\Tr}{\operatorname{Tr}}

\renewcommand{\O}{\mathcal{O}}

\newcommand{\B}{{\rm B}}

\newcommand{\N}{{\rm N}}

\newcommand{\p}{\mathfrak{p}}

\renewcommand{\b}{{\mathfrak b}}
\newcommand{\z}{{\mathfrak z}}
\newcommand{\fA}{{\mathfrak A}}

\newcommand{\Hom}{\operatorname{Hom}}
\newcommand{\uHom}{\underline{\operatorname{Hom}}}

\newcommand{\Alg}{\operatorname{Alg}}

\newcommand{\Morita}{\operatorname{Morita}}

\newcommand{\id}{\operatorname{id}}

\newcommand{\A}{\mathbb{A}}

\newcommand{\ab}{{\rm ab}}

\newcommand{\angles}[1]{\langle #1\rangle}

\newcommand{\laurent}[1]{(\!(#1)\!)}

\newcommand{\Bun}{{\rm Bun}}
\newcommand{\Fl}{{\rm Fl}}
\newcommand{\Gr}{{\rm Gr}}

\newcommand{\Spec}{\operatorname{Spec}}

\renewcommand{\H}{{\rm H}}
\newcommand{\M}{{\rm M}}

\newcommand{\Vect}{\operatorname{Vect}}

\newcommand{\h}{\mathfrak{h}}

\newcommand{\vertsim}{\rotatebox{90}{$\sim$}}

\newcommand{\Ext}{\operatorname{Ext}}

\newcommand{\rmP}{{\rm P}}
\newcommand{\rmZ}{{\rm Z}}
\newcommand{\hBar}{{\rm Bar}}

\newcommand{\Kos}{{\rm Kos}}

\newcommand{\coh}{{\rm coh}}

\newcommand{\op}{{\rm op}}
\newcommand{\rev}{{\rm rev}}
\newcommand{\rgd}{{\rm rgd}}
\newcommand{\aff}{{\rm aff}}
\newcommand{\ext}{{\rm ext}}
\newcommand{\sch}{{\rm sch}}
\newcommand{\cov}{{\rm cov}}

\newcommand{\enr}{{\rm enr}}
\newcommand{\reg}{{\rm reg}}
\newcommand{\ren}{{\rm ren}}

\newcommand{\const}{{\rm const}}
\newcommand{\pre}{{\rm pre}}
\newcommand{\pr}{{\rm pr}}
\newcommand{\act}{{\rm act}}
\newcommand{\tin}{{\rm in}}
\newcommand{\mult}{{\rm mult}}

\newcommand{\unit}{{\rm unit}}
\newcommand{\End}{\operatorname{End}}

\newcommand{\sF}{\mathscr{F}}
\newcommand{\sG}{\mathscr{G}}

\newcommand{\sL}{\mathscr{L}}

\newcommand{\sS}{\mathscr{S}}

\newcommand{\wt}[1]{\widetilde{#1}}
\newcommand{\wh}[1]{\widehat{#1}}

\newcommand{\sHom}{\mathscr{H}\text{\kern -.35em{\calligra\Large om}\kern .08em}}
\newcommand{\sEnd}{\mathscr{E}\text{\kern -.35em{\calligra\Large nd}\kern .28em}}
\newcommand{\cHom}{\cH\hspace{-.04em}om}

\newcommand{\ucHom}{\underline{\cH\hspace{-.04em}om}}

\newcommand{\uSpec}{\operatorname{\underline{Spec}}}

\newcommand{\supp}{\operatorname{supp}}

\newcommand{\red}{{\rm red}}
\newcommand{\codim}{{\rm codim}}
\newcommand{\rank}{\operatorname{rank}}

\newcommand{\Rep}{\operatorname{Rep}}

\DeclareMathOperator*{\colim}{colim}
\newcommand{\cA}{\mathcal{A}}

\newcommand{\cC}{\mathcal{C}}

\newcommand{\cE}{\mathcal{E}}
\newcommand{\cF}{\mathcal{F}}
\newcommand{\cG}{\mathcal{G}}
\newcommand{\cK}{\mathcal{K}}
\newcommand{\cH}{\mathcal{H}}

\newcommand{\cL}{\mathcal{L}}

\newcommand{\cN}{\mathcal{N}}

\newcommand{\cR}{\mathcal{R}}
\newcommand{\cS}{\mathcal{S}}

\newcommand{\bfB}{\mathbf{B}}

\newcommand{\eA}{\EuScript{A}}
\newcommand{\eB}{\EuScript{B}}
\newcommand{\eC}{\EuScript{C}}

\newcommand{\eF}{\EuScript{F}}
\newcommand{\eH}{\EuScript{H}}

\newcommand{\eN}{\EuScript{N}}
\newcommand{\eM}{\EuScript{M}}

\newcommand{\eP}{\EuScript{P}}
\newcommand{\eQ}{\EuScript{Q}}

\newcommand{\eS}{\EuScript{S}}

\newcommand{\eX}{\EuScript{X}}
\newcommand{\eY}{\EuScript{Y}}
\newcommand{\eZ}{\EuScript{Z}}

\newcommand{\Sym}{\operatorname{Sym}}

\newcommand{\Pic}{\operatorname{Pic}}

\newcommand{\Lag}{\operatorname{Lag}}
\renewcommand{\M}{\operatorname{M}}

\newcommand{\bD}{\mathbb{D}}
\newcommand{\bG}{\mathbb{G}}

\newcommand{\doubles}[1]{\llbracket #1\rrbracket}

\newcommand{\SO}{\mathrm{SO}}
\newcommand{\g}{\mathfrak{g}}

\newcommand{\n}{\mathfrak{n}}

\renewcommand{\sl}{\mathfrak{sl}}

\newcommand{\Lie}{\operatorname{Lie}}

\newcommand{\Spin}{\operatorname{Spin}}

\newcommand{\Res}{\operatorname{Res}}

\newcommand{\Ind}{\operatorname{Ind}}

\newcommand{\HH}{\operatorname{HH}}

\newcommand{\ev}{\operatorname{ev}}

\newcommand{\Gp}{{\rm Gp}}

\newcommand{\ad}{{\rm ad}}

\newcommand{\Map}{\operatorname{Map}}
\newcommand{\bfMap}{\mathbf{Map}}
\newcommand{\bfEnd}{\mathbf{End}}

\newcommand{\cl}{{\rm cl}}
\newcommand{\Coc}{\operatorname{Coc}}
\newcommand{\Coh}{\operatorname{Coh}}
\newcommand{\Perf}{\operatorname{Perf}}

\newcommand{\QC}{\operatorname{QC}}

\newcommand{\Shv}{\operatorname{Shv}}

\newcommand{\dgCat}{\mathbf{dgCat}}
\newcommand{\rmmod}{{\rm mod}}
\newcommand{\rmcoh}{{\rm coh}}
\newcommand{\Par}{\mathrm{Par}}

\newcommand{\BG}{\operatorname{\mathsf{BG}}}
\newcommand{\preBG}{\prescript{\pre}{}{\BG}}

\newcommand{\dashmon}{\dash{\rm mon}}
\newcommand{\dashmod}{\operatorname{-mod}}

\newcommand{\dashperf}{\operatorname{-perf}}

\newcommand{\dashbmod}{\operatorname{-\mathbf{mod}}}

\newcommand{\dash}{\operatorname{-}}
\newcommand{\heart}{\ensuremath\heartsuit}

\newcommand{\rmss}{\mathrm{ss}}

\newcommand{\dq}{\mathrm{dq}}

\newcommand{\orb}{\mathrm{orb}}

\newcommand{\AS}{\mathrm{AS}}

\newcommand{\loccit}{\emph{loc.\ cit.}}
\newcommand{\etal}{\emph{et.\ al.}}
\newcommand{\git}{/\!/}
\newcommand{\Sing}{\operatorname{Sing}}

\newcommand{\hpi}{{}^{\h}\hspace{-0.5pt}\pi}
\newcommand{\hcE}{{}^{\h}\hspace{-1pt}\cE}
\newcommand{\hcA}{{}^{\h}\hspace{-2.5pt}\cA}
\tikzset{
    vert/.style={anchor=north, rotate=90, inner sep=.5mm}
}

\allowdisplaybreaks

\newtheorem{theorem}{Theorem}[subsection]
\theoremstyle{definition}
\newtheorem{lemma}[theorem]{Lemma}
\newtheorem{corollary}[theorem]{Corollary}
\newtheorem{proposition}[theorem]{Proposition}

\newtheorem{definition}[theorem]{Definition}
\newtheorem{notation}[theorem]{Notation}
\newtheorem{construction}[theorem]{Construction}
\newtheorem{remark}[theorem]{Remark}

\theoremstyle{remark}
\newtheorem{atom}[theorem]{}

\newcommand{\ps}[1]{{}^{#1}\hspace{-0.1em}}

\newcommand{\tens}[1]{%
  \mathbin{\mathop{\otimes}\limits_{#1}}%
}

\newcommand{\fibprod}[1]{%
  \mathbin{\mathop{\times}\limits_{#1}}%
}

\RequirePackage{scalerel}
\newcommand\reallywidehat[1]{\arraycolsep=0pt\relax%
\begin{array}{c}
\stretchto{
  \scaleto{
    \scalerel*[\widthof{\ensuremath{#1}}]{\kern-.5pt\bigwedge\kern-.5pt}
    {\rule[-\textheight/2]{1ex}{\textheight}}
  }{\textheight}%
}{0.5ex}\\
#1\\
\rule{-1ex}{0ex}
\end{array}
}

\makeatletter
\newcases{lrdcases}
  {\quad}
  {$\m@th\displaystyle{##}$\hfil}
  {$\m@th\displaystyle{##}$\hfil}
  {\lbrace}
  {\rbrace}
\makeatother

\makeatletter
\newcommand*{\da@rightarrow}{\mathchar"0\hexnumber@\symAMSa 4B }
\newcommand*{\da@leftarrow}{\mathchar"0\hexnumber@\symAMSa 4C }
\newcommand*{\xdashrightarrow}[2][]{%
  \mathrel{%
    \mathpalette{\da@xarrow{#1}{#2}{}\da@rightarrow{\,}{}}{}%
  }%
}
\newcommand{\xdashleftarrow}[2][]{%
  \mathrel{%
    \mathpalette{\da@xarrow{#1}{#2}\da@leftarrow{}{}{\,}}{}%
  }%
}
\newcommand*{\da@xarrow}[7]{%
  \sbox0{$\ifx#7\scriptstyle\scriptscriptstyle\else\scriptstyle\fi#5#1#6\m@th$}%
  \sbox2{$\ifx#7\scriptstyle\scriptscriptstyle\else\scriptstyle\fi#5#2#6\m@th$}%
  \sbox4{$#7\dabar@\m@th$}%
  \dimen@=\wd0 %
  \ifdim\wd2 >\dimen@
    \dimen@=\wd2 %
  \fi
  \count@=2 %
  \def\da@bars{\dabar@\dabar@}%
  \@whiledim\count@\wd4<\dimen@\do{%
    \advance\count@\@ne
    \expandafter\def\expandafter\da@bars\expandafter{%
      \da@bars
      \dabar@ 
    }%
  }%
  \mathrel{#3}%
  \mathrel{%
    \mathop{\da@bars}\limits
    \ifx\\#1\\%
    \else
      _{\copy0}%
    \fi
    \ifx\\#2\\%
    \else
      ^{\copy2}%
    \fi
  }%
  \mathrel{#4}%
}
\makeatother

\makeatletter
\DeclareRobustCommand{\tvdots}{%
  \vbox{\baselineskip4\p@\lineskiplimit\z@\kern0\p@\hbox{.}\hbox{.}\hbox{.}}}
\makeatother

\swapnumbers

\usepackage[top=1in,bottom=1in,left=1in,right=1in]{geometry}

\definecolor{zaffre}{rgb}{0.0, 0.08, 0.66}
\hypersetup{colorlinks=true,allcolors=zaffre}

\title{Cohomological boundedness of twisted coherent Springer sheaves}
\author{Oron Y. Propp}

\numberwithin{subsection}{section}
\numberwithin{theorem}{subsection}
\numberwithin{equation}{theorem}
\counterwithout{footnote}{subsection}

\begin{document}

\begin{abstract}
We prove that the coherent Springer sheaf and its parabolic analogues are concentrated in cohomological degree $0$, as predicted by Ben-Zvi--Chen--Helm--Nadler, Zhu, Emerton--Gee--Hellmann, Hansen, and others. More generally, we show that the universal trace functor for a mixed partial affine Hecke category is right t-exact with respect to the exotic t-structure given by Bezrukavnikov--Mirkovi\'c's noncommutative Springer resolution, and left t-exact with respect to the monoidally dual t-structure. To this end, we construct an explicit complex computing the universal trace functor for certain monoidal categories over quotient stacks.
\end{abstract}

\maketitle

\thispagestyle{plain}

\makeatletter
\renewcommand{\@makefnmark}{\hbox{\textsuperscript{\footnotesize{\@thefnmark}}}}
\renewcommand*{\thefootnote}{\fnsymbol{footnote}}
\makeatother

\makeatletter
\renewcommand{\@makefnmark}{\hbox{\textsuperscript{\tiny{\@thefnmark}}}}
\renewcommand*{\thefootnote}{\arabic{footnote}}
\setcounter{footnote}{0}
\makeatother

\tableofcontents

\section{Introduction}

\subsection{Origins of the problem}

\begin{atom}
Fix a split connected reductive group $\check{G}$ over a nonarchimedean local field $F$ with residue field $\F_{q}$, and let $G$ be its Langlands dual group over $k:=\Qlbar$ for a prime $\ell\nmid q$. The categorical local Langlands conjecture predicts the existence of a fully faithful embedding\footnote{This embedding may be extended to an equivalence from the category of D-modules on the stack of principal $\check{G}$-bundles on the Fargues--Fontaine curve. The latter interpolates between the categories of smooth $k$-linear representations of all inner forms of Levi subgroups of $\check{G}$ parameterized by the Kottwitz set.} of dg-categories
\begin{equation}
\label{eqn:cat-lang-conj}
\Rep_k\check{G}\hookrightarrow\QC^!(\Par_G),
\end{equation}
from smooth $k$-linear representations of $\check{G}$ to ind-coherent sheaves on a stack (over $\Spec k$) of $\ell$-adically continuous Langlands parameters \cite{hansen,zhu,hellmann,fargues-scholze,egh-intro}.

Ben-Zvi \etal\ have recently constructed such an embedding for the unramified principal series block of $\Rep_k\check{G}$ \cite{benzvi}. This is the subcategory generated by parabolic inductions of unramified characters of a maximal torus of $\check{G}$; equivalently, it is the subcategory of modules over the Iwahori--Hecke algebra of $\check{G}$ for the corresponding Borel subgroup. In this setting, the embedding \eqref{eqn:cat-lang-conj} admits a canonical quantum deformation, and holds for any algebraically closed field $k$ of characteristic $0$.

On the automorphic side, this deformation is given by the affine Hecke algebra $\cH$, a $k[v^{\pm 1}]$-algebra whose specialization at $v=q^{1/2}$ recovers the Iwahori--Hecke algebra, and whose specialization at $v=1$ recovers the group ring of the extended affine Weyl group of $\check{G}$. On the spectral side, it is given by a ``stack of Deligne--Lusztig parameters,'' which classically parameterize the simple $\cH$-modules \cite{kazhdanlusztig,reeder}.

More precisely, let $\g$ denote the Lie algebra of $G$, and $\eN\subset\g$ the nilpotent cone; let $G$ act on these spaces adjointly, let $\bG_m$ act by scaling in weight $-2$, and set $\wt{G}:=G\times\bG_m$. This stack is then the derived loop space (i.e., the derived self-intersection of the diagonal) of the quotient stack $\g/\wt{G}$, completed along $\eN$; we denote it by $\cL(\wh{\eN}/\wt{G})$. Its $k$-points are simply the classical inertia
\begin{equation*}
\cL(\wh{\eN}/\wt{G})(k)\cong\{(e,s,v)\in\eN\times\wt{G}:ses^{-1}=v^2e\}/\wt{G},
\end{equation*}
recovering the set of Deligne--Langlands parameters when $s$ is assumed semisimple. Specializing $\cL(\wh{\eN}/\wt{G})$ to $v=q^{1/2}$ gives the corresponding component of $\Par_G$, denoted $\cL_{q^{1/2}}(\wh{\eN}/G)$.

Ben-Zvi \etal\ then construct a remarkable ``coherent Springer sheaf'' $\cS\in\Coh(\cL(\wh{\eN}/\wt{G}))$ whose (derived) endomorphism algebra identifies with $\cH$, generating a fully faithful embedding
\begin{equation*}
\cH\dashmod\simeq\angles{\cS}\hookrightarrow\QC^!(\cL(\wh{\eN}/\wt{G})).
\end{equation*}
Specializing $\cS$ to $v=q^{1/2}$ yields a ``coherent $q$-Springer sheaf'' $\cS_{q^{1/2}}\in\Coh(\cL_{q^{1/2}}(\wh{\eN}/G))$, and the unramified principal series block of \eqref{eqn:cat-lang-conj}.
\end{atom}

\begin{atom}
\label{sec:intro-coh-spr-shf}
The existence of $\cS$ arises naturally from\footnote{Even more naturally, from its ``graded'' or ``mixed'' enhancement predicted by Ho--Li \cite{ho-li}.} the local geometric Langlands equivalence for tamely ramified local systems with unipotent monodromy, originally owed to Bezrukavnikov \cite{beztwo}, and generalized to the dg-categorical setting by Chen--Dhillon \cite{chen-dhillon}.

Let now $F:=\Fqbar\laurent{t}$ be the field of Laurent series, let $\check{G}_F$ denote the loop group of $\check{G}$, and let $I\subset\check{G}_F$ denote the Iwahori subgroup. Denote by $\pi\colon\wt{\eN}\to\g$ the Springer resolution of $\eN$, i.e., the cotangent bundle of the flag variety of $G$ equipped with its moment map and natural $\wt{G}$-action.

Then there is an equivalence of rigid monoidal dg-categories
\begin{equation}
\label{eqn:bez-chen-dhillon}
\Shv_{\ren}(I\backslash\check{G}/I)\simeq\QC^!(\wt{\eN}\fibprod{\g}\wt{\eN}/G),
\end{equation}
each equipped with its canonical convolution product. Here the former denotes a ``renormalization'' of the category of $I$-equivariant (ind-constructible, $\ell$-adic) \'etale sheaves on the affine flag variety of $\check{G}$, and is a natural categorification of the Iwahori--Hecke algebra; the latter denotes ind-coherent sheaves on the ``Steinberg stack'' of $G$. This equivalence intertwines the pullback by geometric Frobenius with the endofunctor $q^{-1/2,*}$ (i.e., pullback along the action of $q^{-1/2}\in\bG_m$, which scales by $q$), further allowing us to probe the structure of the ``mixed affine Hecke category''
\begin{equation*}
\eH^{\coh}:=\QC^!(\wt{\eN}\fibprod{\g}\wt{\eN}/\wt{G}).
\end{equation*}

The categorical trace of $\eH^{\coh}$ (in the sense of \cite{gkrv,campbell-ponto}) identifies with $\QC^!(\cL(\wh{\eN}/\wt{G}))$, and the coherent Springer sheaf is defined as the image of the monoidal unit $\Delta_*\O_{\wt{\eN}/\wt{G}}$ under the ``universal trace functor''
\begin{equation}
\label{eqn:univ-tr-funct-H-coh}
[-]\colon \eH^{\coh}\to\Tr(\eH^{\coh})\simeq\QC^!(\cL(\wh{\eN}/\wt{G})).
\end{equation}
Standard results on Hochschild homology then provide an equivalence
\begin{equation*}
\End_{\cL(\wh{\eN}/\wt{G})}(\cS)^\op\simeq\HH(\eH^{\coh}),
\end{equation*}
and Bezrukavnikov's equivalence yields a monoidal identification $\HH(\eH^{\coh})\simeq\cH$ (in particular, the Hochschild complex is concentrated in cohomological degree $0$). Ben-Zvi \etal\ furthermore show that $\cS\simeq\cL\pi_*\O_{\cL(\wt{\eN}/\wt{G})}$, i.e., the pushforward of the structure sheaf (or equivalently, the dualizing sheaf) under the loop space analog of the Springer resolution.

In all but the simplest cases, the stack $\cL(\wt{\eN}/\wt{G})$ has highly nontrivial derived structure. Thus, \emph{a priori}, the coherent Springer sheaf (and its specializations $\cS_v$) are neither connective nor coconnective, but merely objects of the bounded derived category. Several of the aforementioned citations\footnote{See \cite[Conj.~4.15]{benzvi}, \cite[Conj.~4.15]{zhu}, \cite[Conj.~2.17]{hellmann}, and \cite[Conj.~6.1.14(2)]{egh-intro}. A far more general conjecture appears in \cite[Conj.~3.2.1]{hansen}. The conjecture for $\cS$ has been confirmed in \cite[Cor.~4.4.6]{ginzburg-isospectral} for its Lie algebra analog at $v=1$, and in \cite[Prop.~4.19]{benzvi} when $G$ is of type $A_1$.} have conjectured that $\cS$, and each $\cS_v$, are in fact ``honest'' coherent sheaves, i.e., have cohomology concentrated in degree $0$. This would, for instance, yield an explicit description of the isomorphism $\End(\cS)^\op\simeq\cH$ via Demazure operators, as explained in \cite[Rmk.~4.16]{benzvi}.

In this paper, we prove these conjectures.
\end{atom}

\subsection{Statement of results}

\begin{atom}
More generally, we work with an arbitrary (standard) parabolic $\check{P}$ in place of the Borel subgroup of $\check{G}$, with a view towards a fuller understanding of $\Tr(\eH^{\coh})$.

On the automorphic side, $I$ is replaced by the parahoric $I_{\check{P}}\subset\check{G}_F$, defined as the preimage of $\check{P}$ under the evaluation map $\check{G}_O\to\check{G}$; here $O:=\Fqbar\doubles{t}$ is the ring of integers of $F$, and $\check{G}_O$ is the arc group of $\check{G}$. In our original setting, this corresponds to modules over the $I_{\check{P}}$-bi-invariant subalgebra of the Iwahori--Hecke algebra; such modules are generated by the $I_{\check{P}}$-invariants of unramified principal series representations.

On the spectral side, $\wt{\eN}$ is replaced by $\wt{\eN}_P$, the cotangent bundle to the partial flag variety $G/P$, and $\eN$ is replaced by $\eN_P$, the image of the partial Springer resolution $\pi_P\colon\wt{\eN}_P\to\g$. The equivalence \eqref{eqn:bez-chen-dhillon} and identification \eqref{eqn:univ-tr-funct-H-coh} continue to hold in this generality, with $\eH^{\coh}$ replaced by $\eH^{\coh}_P:=\QC^!(\wt{\eN}_P\times_{\g}\wt{\eN}_P/\wt{G})$, and so forth.

Rather than focus on $\cS$ (or its parabolic analog, $\cS_P$), we study the general t-exactness properties of the universal trace functor \eqref{eqn:univ-tr-funct-H-coh}. Specifically, we relate an ``exotic'' t-structure on $\eH^{\coh}_P$ to the standard t-structure on $\QC^!(\cL(\wh{\eN}_P/\wt{G}))$.

The former comes from Bezrukavnikov--Mirkovi\'c's ``noncommutative Springer resolution,'' which is a tilting vector bundle $\cE_P\in\QC(\wt{\eN}_P/\wt{G})$ \cite{bez-mirk,bez-losev}. Thus, the $\O(\g)$-algebra $\cA_P:=\End_{\wt{\eN}_P}(\cE_P)$ is concentrated in cohomological degree $0$, and there is an equivalence
\begin{equation}
\label{eqn:noncomm-spr-equiv}
\Hom_{\wt{\eN}_P}(\cE_P,-)\colon\QC(\wt{\eN}_P/\wt{G})\xrightarrow{\sim}\cA_P^\op\dashmod^{\wt{G}}
\end{equation}
from the category of quasi-coherent sheaves on $\wt{\eN}_P/\wt{G}$ to the category of $\wt{G}$-equivariant right $\cA_P$-modules. This yields a monoidal equivalence\footnote{Here we have ``renormalized'' the category of $\wt{G}$-equivariant $\cA_{P}$-bimodules so that its compact objects are cohomologically bounded complexes with finitely generated cohomology, rather than perfect complexes.}
\begin{equation}
\label{eqn:exotic-equiv-aff-hecke}
\Hom_{\wt{\eN}_P\fibprod{\g}\wt{\eN}_P}(\cE_P^{\vee}\boxtimes\cE_P,-)\colon\eH_P^\coh\xrightarrow{\sim}\cA_P\tens{\O(\g)}\cA_{P}^{\op}\dashmod^{\wt{G}}_{\ren}
\end{equation}
compatible with the right module structures on \eqref{eqn:noncomm-spr-equiv}. The exotic t-structure on $\eH_P^\coh$ is then transported from the standard t-structure on the right-hand side.

Our main result is as follows (Theorem~\ref{thm:Baff-t-str} in the sequel):
\end{atom}

\begin{theorem}
\label{thm:Baff-t-str-intro}
\begin{enumerate}[leftmargin=*]
\item The universal trace functor
\begin{equation}
\label{eqn:univ-tr-parab-intro}
[-]\colon\eH_P^\coh\to\QC^!(\cL(\wh{\eN}_P/\wt{G}))
\end{equation}
has cohomological amplitude in $[-\dim\wt{\eN}_P,0]$ with respect to the exotic t-structure on $\eH_P^\coh$ and the standard t-structure on $\QC^!(\cL(\wh{\eN}_P/\wt{G}))$. In particular, it is right t-exact.

\item Let $\cF$ be a compact object of $\eH_P^\coh$, and suppose that the right monoidal dual $\cF^{\vee,R}$ (or equivalently, by Remark~\ref{rem:conv-cat-pivotal}, the left dual $\cF^{\vee,L}$) is connective for the exotic t-structure. Then $[\cF]$ is coconnective for the standard t-structure. Equivalently, \eqref{eqn:univ-tr-parab-intro} is left t-exact with respect to the monoidal dual of the exotic t-structure on $\eH_P^\coh$ and the standard t-structure on $\QC^!(\cL(\wh{\eN}_P/\wt{G}))$.
\end{enumerate}
\end{theorem}

\begin{atom}
Analogous statements hold if we replace $\Tr(\eH_P^{\coh})$ by the trace of the ``unmixed'' category $\QC^!(\wt{\eN}_P\times_\g\wt{\eN}_P/G)$ with respect to the monoidal endofunctor $q^{1/2,*}$. The latter identifies with $\QC^!(\cL_{q^{1/2}}(\wh{\eN}/G))$, and the trace of the monoidal unit recovers the specialization $\cS_{P,q^{1/2}}$. We record these results in Theorem~\ref{thm:Baff-t-str-q} of the main text; in fact, they hold for any $v\in\bG_m$.

Before explaining the proof, we describe some simple consequences. Recall that Bezrukavnikov--Riche have endowed the category $\QC(\wt{\eN}/\wt{G})$ with a natural weak action of the (extended) affine braid group $B^{\ext}$ associated to the extended affine Weyl group $W^\ext$ of $\check{G}$ \cite{bez-riche}. Specifically, the action of any $a\in B^{\ext}$ is given by a sheaf $\cK_a\in\eH^{\coh}$, which acts on $\QC(\wt{\eN}/\wt{G})$ via left convolution.

The canonical projection $B^{\ext}\twoheadrightarrow W^{\ext}$ admits a section, which sends $w\in W^{\ext}$ to the product of generators of $B^\ext$ corresponding to any reduced decomposition of $w$. We denote by $B^{\ext}_+$ the submonoid of $B^{\ext}$ generated by the image of this section. Moreover, $B^{\ext}$ possesses a ``translation subgroup'' isomorphic to the weight lattice of $G$, whose intersection with $B^{\ext}_+$ comprises the dominant weights. The corresponding sheaves $\cK_a$ are simply given by $\Delta_{*}\O_{\wt{\eN}}(\lambda)$, where $\Delta\colon\wt{\eN}\hookrightarrow\wt{\eN}\times_{\g}\wt{\eN}$ denotes the diagonal map, and $\O_{\wt{\eN}}(\lambda)$ is the usual $\wt{G}$-equivariant line bundle on $\wt{\eN}$ obtained from the weight $\lambda$.

Now, an object of $\eH_P^{\coh}$ is connective for the exotic t-structure if and only if its action on $\QC(\wt{\eN}_P/\wt{G})$ is right t-exact. The ``braid-positivity'' property of the noncommutative Springer resolution (see \cite[\S1.4.1]{bez-mirk}) can thus be reformulated as stating that $\cK_a$ is connective for the exotic t-structure whenever $a\in B^{\ext}_+$. Theorem~\ref{thm:Baff-t-str-intro} and the construction of $\cE_P$ then immediately imply (Corollary~\ref{cor:twist-coh-Spr-t-str} in the sequel):
\end{atom}

\begin{corollary}
\label{cor:braid-co-conn-intro}
\begin{enumerate}[leftmargin=*]
\item For any $a\in B^\ext_+$, the sheaf $[\cK_a]$ is connective, and the sheaf $[\cK_{a^{-1}}]$ is coconnective.
\item For any dominant (resp.\ anti-dominant) weight $\lambda$ extending to a character\footnote{I.e., orthogonal to all coroots in the Levi factor of $P$.} of $P$, the sheaf $[\Delta_{*}\O_{\wt{\eN}_P}(\lambda)]$ is connective (resp.\ coconnective). In particular, the coherent Springer sheaf $\cS_P=[\Delta_{*}\O_{\wt{\eN}_P}]$ lies in cohomological degree $0$.
\end{enumerate}
\end{corollary}

\begin{atom}
The same statements hold for the specializations $[\cK_a]_v$ and $[\Delta_{*}\O_{\wt{\eN}_P}(\lambda)]_v$, for any $v\in\bG_m$. In particular, every $\cS_{P,v}$ lies in the heart of the standard t-structure.

We call the sheaf $\cS_P(\lambda):=[\Delta_{*}\O_{\wt{\eN}_P}(\lambda)]$ the ``$\lambda$-twisted (partial) coherent Springer sheaf,'' and its specialization $\cS_{P,v}(\lambda):=[\Delta_{*}\O_{\wt{\eN}_P}(\lambda)]_v$ the ``$\lambda$-twisted (partial) coherent $v$-Springer sheaf''; hence, the title of the present article. As for $\cS$, we have $\cS_P(\lambda)\simeq\cL\pi_{P,*}\ev^*\O_{\wt{\eN}_P}(\lambda)$, where $\ev\colon\cL(\wt{\eN}_P/\wt{G})\to\wt{\eN}_P/\wt{G}$ denotes the canonical ``loop evaluation'' map, and similarly for $\cS_{P,v}(\lambda)$. We hope to improve upon Corollary~\ref{cor:braid-co-conn-intro} in a future work.

Finally, we record a decomposition result (Proposition~\ref{prop:coh-spr-decomp}) for the twisted coherent Springer sheaves, analogous (and Koszul dual) to Proposition~3.38 of \cite{benzvi}. For instance, $\cS_{P,v}(\lambda)$ is a summand of $\cS_v(\lambda):=[\Delta_{*}\O_{\wt{\eN}}(\lambda)]_v$ whenever $v$ is not a root of unity.
\end{atom}

\subsection{Strategy of the proof}

\begin{atom}
We now comment on the proof of Theorem~\ref{thm:Baff-t-str-intro}. Our primary technical tool is an explicit complex computing the universal trace functor for $\eH_P^\coh$, which we term the \emph{Block--Getzler sheaf}. Indeed, it is a natural enhancement of the ``Block--Getzler complex'' introduced in \cite{block-getzler} (and fruitfully applied in \cite{chen-eq-loc,benzvi}) for computing Hochschild homology in equivariant settings. Here is a special case of our construction; for more precise statements, see Proposition~\ref{prop:bg-tr-res} and the discussion following Definition~\ref{def:bg}.
\end{atom}

\begin{proposition}
\label{prop:bg-intro}
Let $X$ be a derived scheme, let $G$ be a reductive group, and suppose the quotient stack $X/G$ is perfect\footnote{In the sense of Ben-Zvi--Francis--Nadler \cite{bfn}.}. Let $\eA$ be a compactly generated rigid monoidal category admitting a central functor $\Psi\colon\QC(X/G)\to\eA$, and let $\uHom_{X/G}$ and $\uHom_{\B G}$ denote the internal Hom spaces of $\eA$ in the categories $\QC(X/G)$ and $\QC(\B G)=\Rep(G)$, respectively.

For any compact object $a\in\eA^c$, consider the simplicial complex of sheaves on $(X\times G)/G$ whose $n$-simplices are given by
\begin{equation*}
\bigoplus_{a_0,\ldots,a_n\in\eA^c}\uHom_{\B G}(a_0,a_1)\tens{k}\cdots\tens{k}\uHom_{\B G}(a_{n-1},a_n)\tens{k}\uHom_{X/G}(a_n,a_0\otimes a)\boxtimes\O_{G},
\end{equation*}
and whose face maps $d_0,\ldots,d_n$ are the natural extensions of those for the Block--Getzler complex. The totalization of this complex can be lifted to an object $\BG_{X/G}(\eA,a)\in\QC(\cL(X/G))$ using an explicit homotopy. Moreover, we have a natural isomorphism
\begin{equation}
\label{eqn:BG-isom-intro}
\BG_{X/G}(\eA,a)\simeq\Tr(\Psi)^R([a]),
\end{equation}
where $\Tr(\Psi)^R$ is right-adjoint to the natural functor
\begin{equation}
\label{eqn:tr-psi}
\Tr(\Psi)\colon\QC(\cL(X/G))\simeq\Tr(\QC(X/G))\to\Tr(\eA).
\end{equation}
\end{proposition}

\begin{atom}
More generally, we allow for the categorical traces \eqref{eqn:tr-psi} to be taken with respect to certain monoidal endofunctors of $\QC(X/G)$ and $\eA$, and for $[a]$ to be replaced by the ``$2$-categorical class'' of an $\eA$-module category $\eM$ with a compatible endofunctor $F_{\eM}$, in the sense of \cite{gkrv}. For instance, the class of the pair $(\eA,-\otimes a)$ recovers $[a]$. This yields a complex $\BG_{X/G}(\eM,F_{\eM})$ and corresponding generalization of \eqref{eqn:BG-isom-intro}.

In particular, for any compact $\cF\in\eH_P^{\coh}$, we obtain $[\cF]\simeq\BG_{\g/\wt{G}}(\eH_P^\coh,\cF)$. General results on traces of convolution categories then yield
\begin{equation*}
[\cF]\simeq\BG_{\g/\wt{G}}(\QC(\wt{\eN}_P/\wt{G}),-\star\cF),
\end{equation*}
where $\star$ denotes the right convolution action. Using the compact generator $\cE_P$ of $\QC(\wt{\eN}_P/\wt{G})$, the latter is equivalent to an explicit complex
\begin{equation}
\label{eqn:bg-complex-iSe-res-S-intro}
\cdots\to\cA_P\tens{k}\cA_P\tens{k}\cF_{\rmmod}\boxtimes\O_{\wt{G}}\xrightarrow{d_0-d_1+d_2}\cA_P\tens{k}\cF_{\rmmod}\boxtimes\O_{\wt{G}}\xrightarrow{d_0-d_1}\cF_{\rmmod}\boxtimes\O_{\wt{G}},
\end{equation}
where $\cF_{\rmmod}:=\Hom_{\wt{\eN}_P\times_{\g}\wt{\eN}_P}(\cE_P^{\vee}\boxtimes\cE_P,\cF)$ denotes the $\cA_P$-bimodule of \eqref{eqn:exotic-equiv-aff-hecke}. The right t-exactness statement in Theorem~\ref{thm:Baff-t-str-intro} is now immediate.
\end{atom}

\begin{atom}
The coconnectivity statement is far more involved. We first reduce coconnectivity of $[\cF]$ to coconnectivity of its local cohomology at each nilpotent orbit in $\eN$, using the usual exact triangle for a complementary open and closed subscheme. To a representative $e$ of each nilpotent orbit, we may associate a Slodowy slice $\eS_e\subset\g$, which is transverse to the orbit at $e$ and stabilized by the reductive part of the centralizer of $e$ in $G$, denoted $Z_e$. Moreover, the Jacobson--Morozov theorem provides a cocharacter of $\wt{G}$ centralizing $Z_e$ and contracting $\eS_e$ to\footnote{In fact, our convention will be to use the opposite, repelling action.} the point $e$. This yields an action of $\wt{Z}_e:=Z_e\times\bG_m$ on $\eS_e$.

Let $i_{\eS_e}\colon\eS_e/\wt{Z}_e\to\g/\wt{G}$ denote the canonical ``inclusion.'' Using a standard cotangent complex argument, we reduce to showing that the local cohomology of $\cL i_{\eS_e}^*[\cF]$ along the closed substack $\cL(\{e\}/\wt{Z}_e)\subset\cL(\eS_e/\wt{Z}_e)$ is coconnective, for each $\eS_e$. Setting $\wt{\eN}_{P,\eS_e}:=\wt{\eN}_P\times_{\g}\eS_e$, we again obtain an equivalence
\begin{equation}
\label{eqn:LiSe-res-F-intro}
\cL i_{\eS_e}^*[\cF]\simeq\BG_{\eS_e/\wt{Z}_e}(\QC(\wt{\eN}_{P,\eS_e}/\wt{Z}_e),-\star i_{\eS_e}^*\cF),
\end{equation}
where $i_{\eS_e}^*\cF$ denotes the pullback of $\cF$ to $\wt{\eN}_{P,\eS_e}\times_{\eS_e}\wt{\eN}_{P,\eS_e}/\wt{Z}_e$. The restriction $\cE_{P,\eS_e}:=\cE_P|_{\wt{\eN}_{P,\eS_e}}$ remains a tilting generator, so one would hope that studying the complex analogous to \eqref{eqn:bg-complex-iSe-res-S-intro}, with $\cA_P$ replaced by $\cA_{P,\eS_e}:=\End_{\wt{\eN}_{P,\eS_e}}(\cE_{P,\eS_e})$, would be sufficient. Unfortunately, this complex extends into arbitrarily negative degrees, making coconnectivity difficult to deduce.

We are saved by the following technical result (Theorem~\ref{thm:Ze-cov} in the main text):
\end{atom}

\begin{proposition}
\label{prop:Zecov-intro}
There exists a finite cover $Z_e^\cov\twoheadrightarrow Z_e$ such that $\cA_{P,\eS_e}$ is $Z_e^{\cov}$-equivariantly Morita equivalent to a Koszul quadratic algebra $\cA_{P,\eS_e}^{\cov}$, with respect to the grading provided by the Jacobson--Morozov cocharacter.
\end{proposition}

\begin{atom}
That is, we have an equivalence
\begin{equation*}
\QC(\wt{\eN}_{P,\eS_e}/\wt{Z}_e^{\cov})\simeq\cA_{P,\eS_e}^{\cov,\op}\dashmod^{\wt{Z}_e^{\cov}}
\end{equation*}
as in \eqref{eqn:noncomm-spr-equiv}, where again $\wt{Z}_e^{\cov}:=Z_e^{\cov}\times\bG_m$. While the group $Z_e^{\cov}$ is new, Koszulity of $\cA_{P,\eS_e}^{\cov}$ is well-known in the Borel case \cite{bez-mirk,bez-losev}. The proof of Proposition~\ref{prop:Zecov-intro} combines the general existence of a Schur covering of $Z_e$ with a careful type-by-type analysis of the structure of $Z_e$. The group $Z_e^{\cov}$ has already found application in \cite{dawydiak-rigidity}.

Since the projection $Z_e^\cov\twoheadrightarrow Z_e$ is faithfully flat, the proof of Theorem~\ref{thm:Baff-t-str-intro} reduces to an analogous coconnectivity statement on $\cL(\eS_e/\wt{Z}_e^{\cov})$, where we may use $\cA_{P,\eS_e}^{\cov}$ to compute \eqref{eqn:LiSe-res-F-intro}. The Koszul bimodule resolution of $\cA_{P,\eS_e}^{\cov}$, which has length $\dim\wt{\eN}_{P,\eS_e}$ by Grothendieck--Serre duality, then yields a bounded model of this complex. Finally, Grothendieck local duality and the assumptions on $\cF$ allow us to bound its local cohomology.
\end{atom}

\subsection{Overview}

The remainder of the paper is organized as follows. Section~\ref{sec:part-aff-hecke} consists of background on partial affine Hecke categories and their categorical traces. Section~\ref{sec:noncomm-spr} reviews the noncommutative Springer resolution and its generalization to parabolic subgroups of a connected reductive group. In Section~\ref{sec:Ze-cov}, we construct the covering group $Z_e^{\cov}$ of the reductive centralizer of a nilpotent element. Section~\ref{sec:bg-sheaf} constructs the Block--Getzler sheaf and shows that it computes the $2$-categorical class map (under certain assumptions). In Section~\ref{sec:exotic-t-str}, we construct the exotic t-structure on $\eH_P^{\coh}$ (and its analog over a Slodowy slice) using Preygel's ``regularization'' formalism. Finally, in Section~\ref{sec:comp-coh-spr} we prove Theorem~\ref{thm:Baff-t-str-intro} and its specialization to any $v\in\bG_m$. We then deduce the (co)connectivity statements of Corollary~\ref{cor:braid-co-conn-intro} for the affine braid group action and twisted coherent Springer sheaves, and conclude that the coherent Springer sheaf lies in the heart. We also establish a splitting result for twisted coherent Springer sheaves (Proposition~\ref{prop:coh-spr-decomp}).

The paper concludes with several technical appendices. Appendix~\ref{sec:parab-ark-bez-equiv} constructs parabolic analogues of Gaitsgory's central functor and the Arkhipov--Bezrukavnikov equivalence (relating equivariant coherent sheaves on the Springer resolution to Iwahori--Whittaker sheaves on the affine flag variety) following \cite{chen-dhillon}. This is needed to establish the Koszul property of the noncommutative partial Springer resolution in Section~\ref{sec:noncomm-spr}. Appendix~\ref{sec:schur-mult} collects properties of cocycles and Schur multipliers, and generalizes them to algebraic groups for use in Sections~\ref{sec:noncomm-spr} and \ref{sec:Ze-cov}. Appendix~\ref{sec:tr-form-chap} reviews the formalism of $2$-categorical traces from \cite{gkrv}, and recalls its applications to symmetric monoidal categories of quasi-coherent sheaves and convolution categories such as $\eH_P^{\coh}$ (as developed by Ben-Zvi, Francis, Nadler, Preygel, and others). Finally, Appendix~\ref{sec:koszul-res} constructs a Koszul resolution for Koszul quadratic algebras with multiple simple modules, which we use to produce a bounded model for the universal trace functor in Section~\ref{sec:comp-coh-spr}.

\subsection{Assumptions and notation}
\label{sec:assumptions-notation}

Notational conventions in this paper are primarily drawn from \cite{gr,gr2} and \cite{benzvi}. While most conventions are introduced or indicated throughout the text, we collect the most salient ones here for the reader's convenience.

\begin{atom}
We work throughout over an algebraically closed field $k$ of characteristic $0$. All algebro-geometric objects (schemes, stacks, etc.) are implicitly defined over $k$, and we sometimes write $*:=\Spec k$ for this base scheme. We mostly work with $k$-algebras and $k$-modules, and denote $k$-linearization by $(-)_k:=-\otimes_{\Z}k$. Likewise, all dg-categories are $k$-linear.

Crucially, unless explicitly stated otherwise, all categories, functors, and $\Hom$-spaces in this paper are dg-derived, and all limits and colimits are homotopical. Thus, we write ``$(1,1)$-category'' for the classical notion of category, and take cohomology to recover non-derived functors from their derived counterparts. To this end, all (co)chain complexes in this paper are cohomologically indexed.

Given a category $\eC$ equipped with a t-structure, we let $\eC^{\le 0}$ and $\eC^{\ge 0}$ denote its full subcategories of connective and coconnective objects, respectively, and let $\eC^{\heart}$ denote its heart. We further denote by $\iota^{\le 0}\dashv\tau^{\le 0}$ and $\tau^{\ge 0}\dashv\iota^{\ge 0}$ the usual inclusion and truncation functors, and by $\H^*$ the functor of (co)chain cohomology. In particular, we use $\HH$ to denote the Hochschild \emph{complex}, rather than its cohomology groups $\H^*\HH$.

Our conventions for dg-categories follow those of \cite[Ch.~1]{gr}, which will be a sufficient reference for our purposes. Thus, we make frequent use of Lurie's language of $\infty$-categories and higher algebra, as in \cite{lurie-htt,lurie-ha}.

In particular, we let $\Vect_k$ denote the symmetric monoidal, cocomplete, stable $\infty$-category of chain complexes of vector spaces (obtained by applying the dg-nerve construction to the usual pre-triangulated dg-category, see \cite[Cons.~1.3.1.6]{lurie-ha}). We use the term \emph{dg-category} to mean a (presentable) cocomplete stable $\infty$-category equipped with a $\Vect_k$-module structure. All functors between dg-categories will be continuous (i.e., colimit preserving) unless explicitly stated otherwise. The $\infty$-category $\dgCat_k$ of dg-categories (and continuous functors) carries a symmetric monoidal structure given by the Lurie tensor product, with unit object $\Vect_k$.

For any dg-category $\eC$, we let $\eC^c$ denote the full subcategory of compact objects (i.e., those $X\in\eC$ for which $\Hom_{\eC}(X,-)$ preserves countable filtered colimits). If $\eC$ is compactly generated, then we may recover it as the Ind-completion $\Ind(\eC^c)$. Given an object $X\in\eC$, we write $\angles{X}\subset\eC$ for the full subcategory weakly generated by $X$.

Finally, given a symmetric monoidal dg-category $\eC$ and an algebra object $A\in\Alg(\eC)$, we let $A\dashmod_{\eC}$ denote the category of $A$-module objects in $\eC$, and $A\dashperf_{\eC}:=A\dashmod_{\eC}^c$ the full subcategory of $A$-perfect objects. We omit the subscript $\eC$ when $\eC=\Vect_k$. When $\eC=\dgCat_k$, we obtain the notion of monoidal dg-category $\eA$, and let $\eA\dashbmod$ denote the $(\infty,2)$-category of $\eA$-module categories as in \cite[\S3.6]{gkrv} (note that the $(\infty,2)$-structure on $\dgCat_k\simeq\Vect_k\dashbmod$ will be central to this article).
\end{atom}

\begin{atom}
We work exclusively with the language of derived algebraic geometry, primarily following \cite{gr} (though useful references abound). Here, $(1,1)$-functors-of-points from classical commutative rings to sets are replaced by \emph{prestacks}, which are $\infty$-functors from connective commutative dg-$k$-algebras to simplicial sets. Unless explicitly specified as ``classical,'' all schemes, stacks, fiber products, etc., are to be understood in the derived sense.\footnote{The sole exceptions are the notations $X^G,X^g$ for a group $G$ acting on a classical scheme $X$ and $g\in G$; these will denote the \emph{classical} fixed points. We instead use loop space notation as in Definition~\ref{def:loop-space} for the derived fixed points.} We denote the operation of classical truncation by $(-)^{\cl}$, reserving the notation $\pi_0$ for sets of connected components.

Given a prestack $\eX$, we denote its associated functor by $\eX(A)$ or $\eX(\Spec A)$, its structure sheaf by $\O_{\eX}$, and its dg-$k$-algebra of global functions by $\O(\eX)$. We sometimes write $x\in\eX$ to mean ``$x$ is a $k$-point of $\eX$'' (i.e., $x\in\eX(k)$), and denote by $k_x:=\O(\Spec(\{x\}))$ the associated $\O(\eX)$-algebra. We write $\bL_{\eX}$ for the cotangent complex of $\eX$, if it exists. Given a map $f\colon\eX\to\eY$, we write $df\colon f^*\bL_{\eY}\to\bL_{\eX}$ for the codifferential, and $\Delta_{\eX/\eY}\colon\eX\to\eX\times_{\eY}\eX$ for the relative diagonal.

We let $\QC(\eX)$ denote the symmetric monoidal dg-category of quasi-coherent sheaves on $\eX$, defined by right Kan extension from the assignment $\Spec(A)\mapsto A\dashmod$ on affine derived schemes. It carries a canonical t-structure induced from that on $A\dashmod$. We let $\Perf(\eX)\subset\QC(\eX)$ denote the full subcategory spanned by dualizable objects, i.e., those sheaves whose pullback to any affine scheme is quasi-isomorphic to a bounded complex of vector bundles. When $\eX$ is perfect (see Definition~\ref{def:passable-perfect}, as well as \cite{bfn}), we have $\QC(\eX)\simeq\Ind(\Perf(\eX))$.

When $X$ is a scheme which is almost of finite type, we let $\Coh(X)$ (resp.\ $\Coh^-(X),\Coh^+(X)$) denote the full subcategory of $\QC(X)$ comprising cohomologically bounded (resp.\ bounded above, bounded below) complexes with coherent cohomology. If $X$ is eventually coconnective, then $\Perf(X)\subset\Coh(X)$. When the prestack $\eX$ is locally almost of finite type, we let $\QC^!(\eX)$ denote the dg-category of ind-coherent sheaves, defined by a suitable right Kan extension from the assignment $X\mapsto\Ind(\Coh(X))$.

If $\eX$ is furthermore an Artin stack, then $\QC^!(\eX)$ carries a canonical t-structure (see \cite[Prop.~11.7.5]{gaitsgory-indcoh}). When $\eX$ is an algebraic stack, we may define a full subcategory $\Coh(\eX)\subset\QC^!(\eX)$ via $*$-pullback to a smooth atlas (and similarly for $\Coh^-(\eX),\Coh^+(\eX)$). When $\eX$ is QCA (see Definition~\ref{def:qca-stack}, as well as \cite{dg-fin}), we have $\QC^!(\eX)\simeq\Ind(\Coh(\eX))$. Finally, when $\eX$ is smooth, the notions of coherent and perfect, as well as quasi-coherent and ind-coherent, coincide.

We work freely with the six-functor formalism, letting $\cHom$ denote the internal sheaf-$\Hom$. We use the same notation across all categories of sheaves; for example, for a suitable eventually coconnective $f\colon\eX\to\eY$, we write $f^*\colon\QC(\eY)\to\QC(\eX)$ and $f^*\colon\QC^!(\eY)\to\QC^!(\eX)$, and similarly for $f_*,f^!$.
\end{atom}

\begin{atom}
Regarding representation theory, we allow reductive groups to be disconnected, instead specifying a group as ``connected reductive'' when necessary. Given a linear algebraic group $G$, we let $\B G:=*/G$ denote its classifying stack, and set $\Rep(G):=\QC(\B G)$. Thus, we write $\Rep(G)^c$ for the full subcategory of ``finite-dimensional'' representations. Given $A\in\Alg(\Rep(G))$, we set $A\dashmod^{G}:=A\dashmod_{\Rep(G)}$ and $A\dashperf^G:=A\dashperf_{\Rep(G)}$.

If $G$ acts on a set $S$, we let $G^s\subset G$ denote the stabilizer of an element $s\in S$. In particular, for $g\in G$ and $x\in\g:=\Lie(G)$, we write $G^g$ and $G^x$ for the corresponding centralizers, i.e., the stabilizers under the adjoint actions of $G$. Likewise, given a $\g$-representation $V$, we let $\g^v\subset\g$ denote the annihilator of $v\in V$. Thus, for $x$ as above, $\g^x$ denotes its centralizer.

We denote by $Z(G),Z(\g)$ the centers of $G,\g$, and by $X^*(G),X_*(G)$ the groups of characters (i.e., the Pontryagin dual) and cocharacters of $G$, respectively. We write $\pi_0(G),\pi_1(G)$ for the component group and fundamental group of $G$. Finally, we let $G^\circ,[G,G]$ denote the identity component and derived subgroup of $G$, as usual.

Throughout the text, we set $\wt{G}:=G\times\bG_m$, and occasionally write $\wt{\g}:=\g\oplus k$, where $k$ denotes the $1$-dimensional abelian Lie algebra.\footnote{Note that this conflicts with our notation for the Grothendieck simultaneous resolution; however, the notation $\wt{\g}:=\g\oplus k$ will be used only sparingly, so no confusion should arise.} Equivariance with respect to $\wt{G}$ yields an additional weight-grading, which is central to this work. Given any $\Z$-graded object $V$ and $w\in\Z$, we denote by $V_{w},V_{\ge w},V_{\le w}$ the sum of the components of $V$ lying in weight $w$, weights $\ge w$, and weights $\le w$ (and likewise for $V_{>w},V_{<w}$). We also let $\angles{-}$ denote the grading ``twist,'' or ``shift''; that is, $V\angles{w}$ is the $\Z$-graded object for which $(V\angles{w})_{w'}=V_{w+w'}$. The twist $\angles{n}$ therefore corresponds to the action of the weight $-n$ character of $\bG_m$.
\end{atom}

\subsection{Acknowledgments}

We gratefully acknowledge Roman Bezrukavnikov for his generous guidance and support throughout this work, which began as the author's Ph.D.\ thesis. We also thank David Ben-Zvi, Pablo Boixeda Alvarez, Justin Campbell, Harrison Chen, Stefan Dawydiak, Gurbir Dhillon, Matthew Emerton, Pavel Etingof, Dennis Gaitsgory, Benjamin Gammage, Victor Ginzburg, David Hansen, Sam Raskin, Zhiwei Yun, and Xinwen Zhu, among others, for many helpful discussions. This work was partially supported by the National Science Foundation Graduate Research Fellowship under Grant No.~2141064. We further acknowledge the support of the Departments of Mathematics at MIT and the University of Chicago, where this work was carried out.

\section{Partial affine Hecke categories}
\label{sec:part-aff-hecke}

\begin{atom}
In this section, we briefly review the constructions of partial Springer resolutions and Slodowy slices; for further details, see \cite{chrissginzburg}. We then define the (mixed) partial affine Hecke categories and their analogs over Slodowy slices, and recall the construction of Bezrukavnikov--Riche's categorical braid group action. Finally, we compute the categorical traces of these affine Hecke categories using the material in Appendix~\ref{sec:tr-form-chap}.
\end{atom}

\subsection{Springer theory}
\label{subsec:springer-theory}

\begin{atom}
\label{atm:springer-theory}

Fix a connected reductive group $G$, and let $\g$ denote its Lie algebra. We henceforth identify $\g\cong\g^\vee$ via a non-degenerate form $\angles{-,-}$. Fix a maximal torus and Borel subgroup $T\subset B$, let $P\supset B$ be a standard parabolic subgroup of $G$, let $\p\supset\b$ denote their Lie algebras, and let $\eB,\eP$ denote the corresponding (partial) flag varieties.

The cotangent bundle $\wt{\eN}_P:=T^*\eP$ (or rather, its total space) carries a canonical symplectic structure and Hamiltonian $G$-action. Note that as stacks, we have $\wt{\eN}_P/G\simeq\n_P/P$, where $\n_P$ denotes the nilpotent radical of $\p$. The moment map $\pi_P\colon\wt{\eN}_P\to\g$ is then $\wt{G}:=G\times\bG_m$ equivariant, where $G$ acts on $\g$ adjointly, and $\bG_m$ scales $\g$ and the fibers of $\wt{\eN}_P$ by weight $-2$.\footnote{That is, for any $z\in\bG_m$ and $x\in\g$, we have $z\cdot x=z^{-2}x$. Thus, $\bG_m$ acts on functions by weight $2$, i.e., $z\cdot f(-)=z^2f(-)$ for $f\in\g^\vee$. This convention differs from \cite{benzvi}, but the results of \loccit\ remain valid after straightforward modifications. Our convention ensures that the Jacobson--Morozov cocharacter of $\wt{G}$ projects to the tautological cocharacter of $\bG_m$ for any $e$. It also agrees with Lusztig's conventions for the affine Hecke algebra and its asymptotic counterpart (see \cite{lusztigcells1} and the subsequent series), as in the author's Ph.D.\ thesis.} We call $\pi_P$ a \emph{partial Springer resolution}, and write $\eN_P\subset\g$ for its image. The latter is the closure of a nilpotent orbit in $\g$, and $\pi_P$ factors through a symplectic resolution of its normalization, which identifies with the affinization $\Spec\O(\wt{\eN}_P)$ \cite[Prop.~7.4]{bk-nilp}.

In particular, $\pi_B$ is a resolution of singularities of the nilpotent cone $\eN:=\eN_B\subset\g$, called simply the \emph{Springer resolution}; we also write $\wt{\eN}:=\wt{\eN}_B$ and $\pi:=\pi_B$. Our convention throughout this paper will be to omit the word ``partial'' and any subscripts $P$ whenever $P=B$.

Recall moreover that $\eN$ is a union of finitely many $G$-orbits, each of which is a conical symplectic subvariety, hence even-dimensional. These $G$-orbits are equipped with a standard partial order via closures; the unique minimal orbit $\bO_0$ consists only of the $0$ nilpotent, and the unique maximal orbit $\bO_{\reg}$ consists of all regular nilpotent elements.

Fix a nilpotent element $e\in\eN$. We denote the fiber of $\pi_P$ over $e$ by $\eP_e:=\wt{\eN}_P\times_{\g}\{e\}$, and refer to it as the (derived) \emph{partial Springer fiber} above $e$; the usual (derived) Springer fiber is thus denoted $\eB_e$. In particular, it carries a $\wt{G}^e$-action, and its classical truncation is the classical partial Springer fiber $\eP_e^{\cl}\subset\eP$, consisting of all parabolic subgroups of $G$ of type $P$ whose Lie algebra contains $e$.
\end{atom}

\begin{atom}
\label{atm:jacobson-morozov}

Now, recall that by the Jacobson--Morozov theorem, we may extend $e$ (non-uniquely) to an $\sl_2$-triple $\{e,h,f\}$ in $\g$. Fixing such a choice, the adjoint action of $h$ yields a grading $\g=\bigoplus_{w\in\Z}\g_w$, i.e., a decomposition of $\g$ into weight-spaces. This grading is additive with respect to the Lie bracket $[-,-]$, and the actions of $e$ and $f$ raise and lower the grading by $2$, respectively. In particular, the centralizers $\g^e$ and $\g^f$ lie in non-negative and non-positive weights, respectively.

Set $\eS_e:=e+\g^f\subset\g$ (here and onward, we misuse notation slightly by only indicating the dependence on $e$). This is an affine subspace which intersects the orbit $G\cdot e$ transversally at $e$ (in $\g$); we refer to it as a \emph{Slodowy slice} at $e$.

We now wish to modify the $\bG_m$-action on $\eS_e$ coming from $\ad_h$ to be \emph{repelling}. To this end, let $\varphi_e\colon\SL_2\to G$ be the group homomorphism associated to our choice of $\sl_2$-triple, and define a cocharacter $\check{\lambda}_e\in X_*(\wt{G})$ via the formula
\begin{equation*}
\check{\lambda}_e(q):=(\varphi_e(\begin{bsmallmatrix}q&0\\0&q^{-1}\end{bsmallmatrix}), q)\in G\times\bG_m
\end{equation*}
for any $q\in\bG_m$. We refer to this as the \emph{Jacobson--Morozov cocharacter} of $\wt{G}$ for $e$ (though it of course depends on the choice of $\sl_2$-triple). In particular, the adjoint action of $\check{\lambda}_e$ on $\eS_e$ fixes $e$, and \emph{repels} $\eS_e$ from this point (i.e., the vector space underlying $\eS_e$ lies in \emph{strictly} negative weights; equivalently, the coordinate ring $\O(\eS_e)$ lies in strictly \emph{positive} weights).

We now extend this $\bG_m$-action to an action of a larger reductive subgroup of $\wt{G}$. Set $Z_e:=G^{e,h,f}\cong G^{e,\red}$, i.e., the common centralizer of the $\sl_2$-triple $\{e,h,f\}$, or equivalently, the reductive part of the centralizer of $e$. Then $Z_e$ commutes with the cocharacter $\check{\lambda}_e$, and we have an action of $\wt{Z}_e$ on $\eS_e$, with $\bG_m$ acting via $\check{\lambda}_e$. Equivalently, we have $\wt{Z}_e\cong \wt{G}^{\sl_2}\cong\wt{G}^{e,\red}$ (i.e., the centralizer of the $\sl_2$-subalgebra generated by $\{e,h,f\}$), and $\wt{G}^e\cong G^e\rtimes\bG_m$, with $\bG_m$ acting on the unipotent radical of $G^e$ via strictly positive weights. We set $\z_e$ to be the Lie algebra of $Z_e$, and continue to use the notation $\tilde{\z}_e\cong\z_e\oplus k$ for the Lie algebra of $\wt{Z}_e$.
\end{atom}

\begin{atom}
Finally, transversality yields (derived) pullback squares
\begin{equation}
\label{eqn:slodowy-pullback}
\begin{tikzcd}[column sep=large]
\eP_e/\wt{Z}_e\arrow{d}{\pi_e}\arrow[r,hook,"i_{\eP_e}"]&\wt{\eN}_{P,\eS_e}/\wt{Z}_e\arrow{r}{i_{\wt{\eN}_{P,\eS_e}}}\arrow{d}{\pi_{P,\eS_e}}&\wt{\eN}_P/\wt{G}\arrow{d}{\pi_P}\\
e/\wt{Z}_e\arrow[r,hook,"i_e"]&\eS_e/\wt{Z}_e\arrow{r}{i_{\eS_e}}&\g/\wt{G},
\end{tikzcd}
\end{equation}
where the variety $\wt{\eN}_{P,\eS_e}$ is the (classical, smooth, connected) \emph{partial resolution to the Slodowy slice}. Thus, as for $P$, our convention throughout this paper will be to omit any subscripts $\eS_e$ whenever $e=0$. Alternatively, since the map $i_{\eS_0}$ is the identity, we will sometimes write $\g$ for $\eS_0$.

The map $\pi_{P,\eS_e}$ again factors through a symplectic resolution of the affinization $\Spec\O(\wt{\eN}_{P,\eS_e})$, hence is semismall. In particular, $2\dim\eP_e\le\dim\wt{\eN}_{P,\eS_e}$, with equality precisely when the isotropic subvariety $\eP_e\subset\wt{\eN}_{P,\eS_e}$ is Lagrangian\footnote{This is known to be true in many cases, e.g., if $P=B$, $e=0$, or $G$ is of type A, see \cite{li-spalt}.} \cite[Lem.~3.3]{zivanovic}.
\end{atom}

\begin{atom}
\label{atm:Lag-functs}

Given another standard parabolic $Q\supset P$ with corresponding flag variety $\eQ$, we have the Lagrangian correspondence
\begin{equation}
\label{eqn:Lag-corresp}
\begin{tikzcd}
-\fibprod{\g/\wt{G}}\wt{\eN}_Q/\wt{G}&-\fibprod{\g/\wt{G}}(\wt{\eN}_Q\fibprod{\eQ}\eP)/\wt{G}\arrow[l,"p^P_Q"']\arrow{r}{i^P_Q}&-\fibprod{\g/\wt{G}}\wt{\eN}_P/\wt{G}
\end{tikzcd}
\end{equation}
for any stack over $\g/\wt{G}$ (e.g., $\eS_e/\wt{Z}_e$). Following \cite[\S6.0.3]{chen-dhillon}, we define adjoint pairs
\begin{equation}
\label{eqn:Lag-functs}
\begin{gathered}
\begin{tikzcd}[column sep=1in]
\QC^!(-\fibprod{\g/\wt{G}}\wt{\eN}_Q/\wt{G})\arrow[r,shift left,"\Lag^{P,*}_Q:=i^P_{Q,*}p^{P,*}_Q"]&\QC^!(-\fibprod{\g/\wt{G}}\wt{\eN}_P/\wt{G}),\arrow[l,shift left,"\Lag^{P}_{Q,*}:=p^P_{Q,*}i^{P,!}_{Q}"]
\end{tikzcd}\\
\begin{tikzcd}[column sep=1in]
\QC^!(-\fibprod{\g/\wt{G}}\wt{\eN}_Q/\wt{G})\arrow[r,shift right,"\Lag^{P,!}_Q:=i^P_{Q,*}p^{P,!}_Q"']&\QC^!(-\fibprod{\g/\wt{G}}\wt{\eN}_P/\wt{G}).\arrow[l,shift right,"\Lag^{P}_{Q,!}:=p^P_{Q,*}i^{P,*}_{Q}"']
\end{tikzcd}
\end{gathered}
\end{equation}

We adopt the following notational conventions: when $P=B$ or $Q=G$, we omit the relevant super- or sub-script, and when the fiber product is taken on the left-hand side, we write $\prescript{P}{Q}{\Lag}$ in place of $\Lag^P_Q$. We also declare the roots in $B$ to be \emph{negative} by convention, and use the superscript $(-)^-$ to denote the opposite of, e.g., a parabolic subgroup or Lie algebra. Thus, we let $\rho_P$ denote half the sum of positive roots in $\n_P^-$, and set $\rho_{Q/P}:=\rho_Q-\rho_P$ and $d_{Q/P}:=\dim Q/P$.

Then \cite[Prop.~7.3.8]{gaitsgory-indcoh} gives relations
\begin{equation*}
\label{eqn:pr-i-twist-shift}
p^{P,!}_Q\simeq p^{P,*}_Q(2\rho_{Q/P})[d_{Q/P}],\qquad i^{P,!}_Q\simeq i^{P,*}_Q(2\rho_{Q/P})[-d_{Q/P}],
\end{equation*}
where the twists denote the action of characters in $\QC(\eP/G)\simeq\Rep(P)$ (for the latter, note that the inclusion $\wt{\eN}_Q\times_{\eQ}\eP/G\hookrightarrow\wt{\eN}_P/G$ identifies with $\n_Q/P\hookrightarrow\n_P/P$). This yields corresponding relations between the functors in \eqref{eqn:Lag-functs}, and in particular, adjunctions
\begin{equation}
\label{eqn:Lag-extra-adjns}
\Lag^{P,!}_Q\dashv\Lag^P_{Q,!}[-2d_{Q/P}],\qquad \Lag^P_{Q,*}\dashv\Lag^{P,*}_Q[2d_{Q/P}].
\end{equation}
\end{atom}

\subsection{Convolution actions}

\begin{atom}
\label{atm:aff-hecke-cat}

Recall that the (\emph{mixed}) \emph{partial affine Hecke category}\footnote{This non-standard notation is intended to distinguish $\eH_P^{\coh}$ from its ``module'' incarnation, which will be introduced in \S\ref{sec:exotic-t-str}.} $\eH_P^\coh$ is the category of $\wt{G}$-equivariant ind-coherent sheaves on the (derived) partial Steinberg variety, i.e.,
\begin{equation}
\label{eqn:mix-aff-hecke-cat}
\eH_P^\coh:=\QC^!(\wt{\eN}_P\fibprod{\g}\wt{\eN}_P/\wt{G}).
\end{equation}
This is a compactly generated monoidal category acting on $\QC(\wt{\eN}_P/\wt{G})$, where the algebra and module structures are both given by (left) $*$-convolution\footnote{As opposed to the dual $!$-convolution.} (which we denote by ``$\star$''). Moreover, given any Slodowy slice $\eS_e$, we have a compactly generated monoidal category
\begin{equation}
\label{eqn:aff-hecke-category-Se}
\eH_{P,\eS_e}^\coh:=\QC^!(\wt{\eN}_{P,\eS_e}\fibprod{\eS_e}\wt{\eN}_{P,\eS_e}/\wt{Z}_e)
\end{equation}
acting on $\QC(\wt{\eN}_{P,\eS_e}/\wt{Z}_e)$, and a monoidal functor $i_{\eS_e}^*\colon\eH_P^\coh\to\eH_{P,\eS_e}^\coh$ as in \S\ref{atm:conv-cat-funct} (which is the identity when $e=0$).
\end{atom}

\begin{atom}
\label{atm:braid-action}
Bezrukavnikov--Riche have used these categorical Hecke actions to define braid group actions on $\QC(\wt{\eN}/\wt{G})$ and $\QC(\wt{\eN}_{\eS_e}/\wt{Z}_e)$, which we now recall \cite{bez-riche}. Let $\check{G}$ denote the Langlands dual reductive group of $G$, and recall that the \emph{extended affine Weyl group} of $\check{G}$ is given by the semidirect product
\begin{equation*}
W^\ext:=W\ltimes X_*(\check{T})=W\ltimes X^*(T),
\end{equation*}
where $\check{T}\subset\check{G}$ is the dual torus and $W$ denotes the finite Weyl group. Then the category $\QC(\wt{\eN}/\wt{G})$ carries a compact object-preserving weak action of the \emph{affine braid group} $B^\ext$ associated to\footnote{More specifically, the \emph{non-extended} affine Weyl group $W^{\aff}:=W\ltimes\angles{\Phi}\subset W^\ext$ admits a Coxeter presentation associated to the affine Dynkin diagram of $\check{G}$. The group $B^\ext$ is then the analogous extension of the Artin--Tits braid group $B^{\aff}$ associated to this Coxeter presentation.} $W^\ext$, owed to \cite{bez-riche}. This action is given by a ``homomorphism'' $B^\ext\to\eH^\rmcoh$, i.e., by coherent sheaves $\cK_a\in\eH^{\rmcoh,c}$ admitting isomorphisms $\cK_a*\cK_{a'}\simeq\cK_{aa'}$ for each $a,a'\in B^\ext$, followed by the convolution action of $\eH^\rmcoh$ on $\QC(\wt{\eN}/\wt{G})$ described in \S\ref{atm:aff-hecke-cat}. For any Slodowy slice $\eS_e$, composing the ``homomorphism'' $B^\ext\to\eH^\rmcoh$ with the monoidal functor $i_{\eS_e}^*\colon\eH^\rmcoh\to\eH_{\eS_e}^\rmcoh$ then yields an affine braid group action on $\QC(\wt{\eN}_{\eS_e}/\wt{Z}_e)$ as well.

Let us give a unique characterization of these sheaves. Given $w\in W^\ext$, we may consider a minimal decomposition of $w$ as a product of simple reflections, and take the product of the corresponding generators of $B^\ext$. This product is independent of the choice of decomposition of $w$, and hence yields a set-theoretic section of the canonical surjection $B^\ext\to W^\ext$, which we denote by $w\mapsto\wt{w}$. We further denote the sub-monoid of $B^\ext$ generated by the image of this section by $B^\ext_+$. It suffices to construct the sheaves $\cK_{\wt{w}}$ for generators $w\in W^\ext$.

First, given $\lambda\in X^*(T)$, there is an associated $\wt{G}$-equivariant line bundle $\O_{\wt{\eN}}(\lambda)$ on $\wt{\eN}$ obtained by pullback along the composition
\begin{equation*}
\wt{\eN}/\wt{G}\to\eB/\wt{G}\cong\B\wt{B}\to \B T.
\end{equation*}
Write $X^*(T)^+\subset X^*(T)$ for the set of dominant weights (with respect to $B^-$, as per our conventions). If $\lambda\in X^*(T)^+$, then $\cK_{\tilde{\lambda}}\simeq\Delta_{\wt{\eN}/\g,*}\O_{\wt{\eN}}(\lambda)\in\eH^\rmcoh$, where $\Delta_{\wt{\eN}/\g}\colon\wt{\eN}/\wt{G}\hookrightarrow\wt{\eN}\times_\g\wt{\eN}/\wt{G}$ denotes the diagonal map. More generally, $B^\ext$ has a ``translation subgroup'' isomorphic to $X^*(T)$, which acts by the sheaves $\Delta_{\wt{\eN}/\g,*}\O_{\wt{\eN}}(\lambda)$.

Next, for a simple reflection $s_\alpha\in W$, recall that the \emph{Grothendieck simultaneous resolution} $\hpi\colon\wt{\g}\to\g$ is the variety of pairs $(x,\b')$, where $x\in\g$ and $\b'\in\eB$ is a Borel subalgebra containing $x$; the Springer resolution is the (reduced) closed subvariety of $\wt{\g}$ given by requiring $x$ to be nilpotent. Let $\h$ denote the universal Cartan algebra of $\g$ (which carries a natural action of $W$), and consider the map $\wt{\g}\to\h$ given by $(x,\b')\mapsto x\bmod[\b',\b']$. The latter induces a resolution of singularities $\wt{\g}\to\g\times_{\h\git W}\h$, which is an isomorphism over the open subscheme $\g_{\reg,\rmss}\subset\g$ comprising all regular semi-simple elements. In particular, the map $\hpi$ is a principal $W$-torsor over $\g_{\reg,\rmss}$, so we may consider the closure of the graph of $s_\alpha$ on $(\wt{\g}\times_{\g}\wt{\g})\times_{\g}\g_{\reg,\rmss}$, which is a closed $\wt{G}$-stable (classical, smooth) subscheme $\prescript{\h}{}{\Gamma}_{s_\alpha}\subset\wt{\g}\times_{\g}\wt{\g}$ (note that the latter fiber product is in fact classical). The classical scheme-theoretic intersection
\begin{equation*}
\Gamma_{s_\alpha}:=\big(\prescript{\h}{}{\Gamma}_{s_\alpha}\fibprod{\wt{\g}\fibprod{\g}\wt{\g}}\wt{\eN}\fibprod\g\wt{\eN}\big)^\cl\subset\wt{\eN}\fibprod\g\wt{\eN}
\end{equation*}
is then a $\wt{G}$-stable closed subscheme of the Steinberg variety, and $\cK_{\tilde{s}_\alpha}\simeq\O_{\Gamma_{s_\alpha}}\in\eH^\rmcoh$ is its structure sheaf.
\end{atom}

\subsection{Categorical traces}

\begin{atom}
We now discuss the categorical traces of $\eH_P^\coh$ and $\eH_{P,\eS_e}^{\coh}$, as in Appendix~\ref{sec:tr-form-chap}. Set $\eS_{e,\eN_P}:=\eS_e\cap\eN_P=\pi_{P,\eS_e}(\wt{\eN}_{P,\eS_e})$, and let $\wh{\eS}_{e,\eN_P}$ denote the formal completion of $\eS_e$ at $\eS_{e,\eN_P}$. We write $\cL(-)$ for the loop space functor, and refer the reader to Definition~\ref{def:loop-space} and the ensuing discussion for our conventions. Set
\begin{equation*}
\QC^!(\cL(\wh{\eS}_{e,\eN_P}/\wt{Z}_e)):=\QC^!_{\eS_{e,\eN_P}}(\cL(\eS_e/\wt{Z}_e)),
\end{equation*}
i.e., the subcategory of sheaves set-theoretically supported on $\eS_{e,\eN_P}$ (via the loop evaluation). In particular, for $e=0$, we write $\wh{\eN}_P:=\g^{\wedge}_{\eN_P}$ following \cite{benzvi}. 
\end{atom}

\begin{corollary}
\label{cor:tr-He-coh-ident}
For any parabolic $P$ and Slodowy slice $\eS_e$, we have a natural identification
\begin{equation*}
\Tr(\eH_{P,\eS_e}^{\coh})\simeq\QC^!(\cL(\wh{\eS}_{e,\eN_P}/\wt{Z}_e)).
\end{equation*}
\end{corollary}

\begin{proof}
By Proposition~\ref{prop:tr-conv}, it suffices to show that the singular support condition $\Lambda_{\wt{\eN}_{P,\eS_e}/\eS_e}$ is vacuous; the proof follows \cite[Rmk.~4.14]{benzvi}. Using our chosen identification $\g\cong\g^\vee$, the cotangent space to any point in $\eS_e$ may be identified with $\g^{f,\vee}\cong\g/[f,\g]\cong\g^e$. The singular locus of $\cL(\wh{\eS}_{e,\eN_P}/\wt{Z}_e)$ at a $k$-point $(n,g,q)\in \eS_e\times\wt{Z}_e$ satisfying\footnote{Note that we no longer reserve $q$ for a prime power as in the introduction.} $gng^{-1}=q^2n$ is then the set
\begin{equation*}
\Sing(\cL(\wh{\eS}_e/\wt{Z}_e))_{(n,g,q)}=\{x\in\g^e:gxg^{-1}=q^{-2}x,[n,x]=0,\angles{n,x}=0\}.
\end{equation*}
Moreover, the singular locus of $\wt{\eN}_{P,\eS_e}\times_{\eS_e}\wt{\eN}_{P,\eS_e}$ at a $k$-point $(n,\p_1,\p_2)$ (where $n\in \eS_e$ and $\p_1,\p_2\in\eP$ are parabolic subalgebras containing $n$) is the set
\begin{equation*}
\Sing(\wt{\eN}_{P,\eS_e}\fibprod{\eS_e}\wt{\eN}_{P,\eS_e})_{(n,\p_1,\p_2)}=\g^e\cap\p_1\cap\p_2.
\end{equation*}
A calculation then shows that
\begin{equation*}
(\Lambda_{\wt{\eN}_{P,\eS_e}/\eS_e})_{(n,g,q)}=\{x\in\Sing(\cL(\wh{\eS}_{e,\eN_P}/\wt{Z}_e))_{(n,g,q)}:n,x\in\p'\text{ for some }\p'\in\eP\}.
\end{equation*}
Since $n$ and $x$ generate a two-dimensional solvable Lie algebra, they are contained in a Borel subalgebra, hence in a parabolic subalgebra classified by $\eP$. It follows that $\Lambda_{\wt{\eN}_{P,\eS_e}/\eS_e}=\Sing(\cL(\wh{\eS}_{e,\eN_P}/\wt{Z}_e))$, as desired.
\end{proof}

\begin{atom}
We now identify the trace of the monoidal functor $i_{\eS_e}^*\colon\eH_P^\coh\to\eH_{P,\eS_e}^\coh$. To apply Corollary~\ref{cor:tr-conv-funct}, we must show:
\end{atom}

\begin{lemma}
\label{lem:LiSe-event-coconn}
The morphism
\begin{equation*}
i_{\eS_e}\colon \eS_e/\wt{Z}_e\to\g/\wt{G}
\end{equation*}
is a smooth relative scheme.
\end{lemma}

\begin{proof}
The morphism $i_{\eS_e}$ is clearly schematic; we show that its relative cotangent complex is perfect of Tor-amplitude $[0,1]$. The latter is computed by the complex\footnote{Note that here $\wt{\g}$ denotes the Lie algebra of $\wt{G}$, rather than the Grothendieck simultaneous resolution.}
\begin{equation*}
\g^\vee\otimes\O_{\eS_e/\wt{Z}_e}\xrightarrow{(\act^*,\tin^*)}(\wt{\g}^\vee\oplus\g^{f,\vee})\otimes\O_{\eS_e/\wt{Z}_e}\xrightarrow{(\tin^*,\act^*)}\wt{\z}_e^\vee\otimes\O_{\eS_e/\wt{Z}_e}
\end{equation*}
of (locally) free sheaves, where we have let $\act,\tin$ denote the evident ``action'' and inclusion maps. Thus, we wish to show that the first map is injective, and its cokernel is locally free. Dualizing and applying Nakayama's lemma reduces us to showing that the dual map is surjective on all fibers.

Given $x\in\g^f$, the fiber of the dual map at $e+x\in \eS_e$ is computed by
\begin{equation*}
([-,e+x],\angles{-,e+x}^\vee,\tin)\colon\g\oplus k\oplus\g^f\to\g.
\end{equation*}
It therefore suffices to show that the composition
\begin{equation}
\label{eqn:comp-wt-sp-f-map}
[\g,f]\hookrightarrow\g\xrightarrow{[-,e+x]}\g\twoheadrightarrow\g/\g^f
\end{equation}
is an isomorphism. Consider the basis of $[\g,f]$ given by weight spaces for the $\sl_2$-action associated to $\eS_e$. We then have an associated basis of $\g/\g^f$ given by sending each basis vector $v$ to $[v,e]$ (which increases the weight by $2$). But for any basis vector $v$ of weight $w$, the vector $[v,x]$ lies in the span of weight spaces $\le w$ (as $x$ lies in the span of weight spaces $\le 0$). It follows that the matrix describing \eqref{eqn:comp-wt-sp-f-map} with respect to these bases is upper triangular, with $1$'s along the diagonal, hence an isomorphism.
\end{proof}

\begin{corollary}
\label{cor:tr-iSe-star-loop-pullback}
The adjoint pair
\begin{equation*}
\Tr(\Ind_{i_{\eS_e}^*})\colon\Tr(\eH_P^{\coh})\rightleftarrows\Tr(\eH_{P,\eS_e}^{\coh})\colon\Tr(\Res_{i_{\eS_e}^*})
\end{equation*}
identifies with
\begin{equation*}
\cL i_{\eS_e}^*\colon\QC^!(\cL(\wh{\eN}_P/\wt{G}))\rightleftarrows\QC^!(\cL(\wh{\eS}_{e,\eN_P}/\wt{Z}_e))\colon\cL i_{\eS_e,*}
\end{equation*}
via Corollary~\ref{cor:tr-He-coh-ident}.
\end{corollary}

\section{The noncommutative partial Springer resolution}
\label{sec:noncomm-spr}

\begin{atom}
This section introduces Bezrukavnikov--Mirkovi\'c's noncommutative Springer resolution and Lusztig's canonical basis of the $K$-theory of a Springer fiber, focusing on the key properties which will be needed in the sequel. We begin in the setting of a partial Springer resolution for a simply-connected semisimple group (where no non-trivial cocycles appear in the canonical basis). Here Appendix~\ref{sec:parab-ark-bez-equiv} is used to establish the Koszul grading property for a general parabolic. We then extend these properties to any connected reductive group, and construct the cocycles appearing in the canonical basis using Appendix~\ref{sec:schur-mult}.
\end{atom}

\subsection{Recollections}

\begin{atom}
\label{atm:nc-spr-res}

Assume \emph{for this subsection only} that $G$ is semisimple and simply-connected (we will remove these hypotheses in \S\ref{sec:gen-conn-red}). Then there exists a $\wt{G}$-equivariant vector bundle $\hcE$ on the Grothendieck simultaneous resolution $\wt{\g}$ known as the \emph{Bezrukavnikov--Mirkovi\'c tilting bundle} \cite{bez-mirk,bez-losev}. We write $\cE:=\hcE|_{\wt{\eN}}$ for the $*$-restriction to the Springer resolution, and refer to it by the same name (the distinction will always be clear from context).
\end{atom}

\begin{atom}
\label{atm:tilt-gen}

For any Slodowy slice $\eS_e\subset\g$, the pullback $\cE_{\eS_e}:=\cE_{B,\eS_e}:=i_{\wt{\eN}_{\eS_e}}^*\cE$ is a tilting generator. As in \S\ref{atm:Lag-functs}, we have adjunctions
\begin{equation}
\begin{tikzcd}[column sep=large]
\QC(\wt{\eN}_{P,\eS_e}/\wt{Z}_e)\arrow[r,shift right,"\Lag^{!}_P"']&\QC(\wt{\eN}_{\eS_e}/\wt{Z}_e).\arrow[l,shift right,"\Lag_{P,!}"']
\end{tikzcd}
\end{equation}
for any standard parabolic $P\ne B$. Then
\begin{equation*}
\cE_{P,\eS_e}:=\Lag_{P,!}\cE_{\eS_e}(-\rho).
\end{equation*}
is a tilting vector bundle\footnote{Here we have taken $\lambda=0$ and $\fA$ to be the identity alcove in \cite[Thm.~4.2(a)]{bez-mirk}.}, so that for any $P$,
\begin{equation*}
\cA_{P,\eS_e}:=\End_{\wt{\eN}_{P,\eS_e}}(\cE_{P,\eS_e})
\end{equation*}
is an $\O(\eS_e)$-algebra in cohomological degree $0$ with a compatible $\wt{Z}_e$-action. We call it the \emph{noncommutative partial Springer resolution}, regardless of whether $e=0$; our notational conventions continue to follow \S\ref{subsec:springer-theory}.

This gives $\Rep(\wt{Z}_e)$-linear equivalences
\begin{equation}
\label{eqn:bm-functor}
\Hom_{\wt{\eN}_{P,\eS_e}}(\cE_{P,\eS_e},-)\colon\QC(\wt{\eN}_{P,\eS_e}/\wt{Z}_e)\xrightarrow{\sim}\cA_{P,\eS_e}^\op\dashmod^{\wt{Z}_e},
\end{equation}
and we refer to the t-structure on $\Coh(\wt{\eN}_{P,\eS_e}/\wt{Z}_e)$ corresponding to the usual t-structure on the right-hand side as the \emph{exotic t-structure}. The induced homomorphisms
\begin{equation*}
\End(\Lag_{P,!}(-\rho))\colon\cA_{\eS_e}\to\cA_{P,\eS_e}
\end{equation*}
then yield commutative squares
\begin{equation}
\label{eqn:A-AP-res-ind-commdiag}
\begin{tikzcd}
\QC(\wt{\eN}_{\eS_e}/\wt{Z}_e)\arrow{r}{\sim}&\cA_{\eS_e}^{\op}\dashmod^{\wt{Z}_e}\\
\QC(\wt{\eN}_{P,\eS_e}/\wt{Z}_e)\arrow[u,"{\Lag_{P}^{!}(\rho)}"]\arrow{r}{\sim}&\cA_{P,\eS_e}^{\op}\dashmod^{\wt{Z}_e}\arrow[u,"\Res^{\cA_{P,\eS_e}^{\op}}_{\cA_{\eS_e}^{\op}}"']
\end{tikzcd}
\end{equation}
by the adjunction equivalences
\begin{equation}
\label{eqn:AP-Lag-adj}
\Hom_{\wt{\eN}_{P,\eS_e}}(\cE_{P,\eS_e},\cF)\simeq\Hom_{\wt{\eN}_{\eS_e}}(\cE_{\eS_e},(\Lag_{P}^{!}\cF)(\rho)).
\end{equation}
In particular, $\Lag_{P}^{!}(\rho)$ is t-exact with respect to the exotic t-structures.
\end{atom}

\begin{atom}
\label{atm:koszul-grading}

Write
\begin{equation}
\label{eqn:tilt-bundle-indec-summands}
\{\cE_{P,\eS_e}^b:b\in \bfB_{e,P}\}
\end{equation}
for the set of indecomposable summands of $\cE_{P,\eS_e}\in\Coh(\wt{\eN}_{P,\eS_e})$, and set\footnote{Note that $\cA_{\eS_e}^{\cov}$ is the algebra ``$\cA_e$'' of \cite[\S5.5]{bez-mirk}; our notation differs in order to emphasize its relation to the Slodowy slice, and its Morita equivalence to $\cA_{\eS_e}$, which can be made equivariant once all cocycles are trivialized as in Corollary~\ref{cor:bm-cov-equiv}.}
\begin{equation}
\label{eqn:ASe-cov}
\cE_{P,\eS_e}^{\cov}:=\bigoplus_{b\in\bfB_{e,P}}\cE_{P,\eS_e}^b,\qquad\qquad \cA_{P,\eS_e}^\cov:=\End_{\wt{\eN}_{P,\eS_e}}(\cE_{P,\eS_e}^{\cov}).
\end{equation}
\end{atom}

\begin{proposition}
\label{prop:parab-A-kosz-grad}
There exists a graded lift\footnote{For $P=B$, these graded lifts are unique up to simultaneous twists of \eqref{eqn:tilt-bundle-indec-summands}. We do not discuss uniqueness of these lifts for general $P$, which is not necessary for the present work. Thus, we do not attempt to generalize \cite[\S6.4.3]{bez-mirk} in the proof that follows.} of $\cE_{P,\eS_e}^{\cov}$ to $\Coh(\wt{\eN}_{P,\eS_e}/\bG_m)$ making $\cA_{P,\eS_e}^{\cov}$ into a Koszul quadratic algebra.
\end{proposition}

\begin{proof}
Proofs for all results in \cite[\S6.3]{bez-mirk} are identical, with the algebra $\H^0(\cA_{\eS_e}\otimes_{\O(\eS_e)}k_e)$ replaced by $\H^0(\cA_{P,\eS_e}\otimes_{\O(\eS_e)}k_e)$ and $\cE_{\eS_e}^b,\cL_{\eS_e}^b$ replaced by $\cE^{b}_{P,\eS_e},\cL^{b}_{P,\eS_e}$. Using Proposition~\ref{prop:arkh-bez-parab}, the choice of grading is now the same as in \S6.4.1 of \loccit, and the proofs of \S6.4.2 and \S6.4.4 in \loccit\ are identical. For Koszulity of $\cA_{P,\eS_e}^{\cov}$, it suffices as in Proposition~5.5 of \loccit\ to show that $\omega_{\wt{\eN}_{P,\eS_e}}\simeq\O_{\wt{\eN}_{P,\eS_e}}[\dim\wt{\eN}_{P,\eS_e}]\angles{\dim\wt{\eN}_{P,\eS_e}}$. Clearly we have
\begin{equation*}
\omega_{\wt{\eN}_{P}}\simeq\O_{\wt{\eN}_{P}}[\dim\wt{\eN}_{P}]\angles{\dim\wt{\eN}_{P}}\simeq\pi_P^*\omega_{\g}[\dim\wt{\eN}_{P}-\dim\g]\angles{\dim\wt{\eN}_{P}-2\dim\g}.
\end{equation*}
Decomposing $\g$ as an $\SL_2$-representation via $\varphi_e$ (see \ref{atm:jacobson-morozov}) shows that $\omega_{\eS_e}\simeq\O_{\eS_e}[\dim \eS_e]\angles{\dim\g+\dim \eS_e}$, and hence
\begin{equation}
\label{eqn:omega-N-P-Se}
\begin{split}
\omega_{\wt{\eN}_{P,\eS_e}}&\simeq\pi_{P,\eS_e}^{*}i_{\eS_e}^!\omega_{\g}[\dim\wt{\eN}_{P}-\dim\g]\angles{\dim\wt{\eN}_{P}-2\dim\g}\\
&\simeq\pi_{P,\eS_e}^{*}\omega_{\eS_e}[\dim\wt{\eN}_{P}-\dim\g]\angles{\dim\wt{\eN}_{P}-2\dim\g}\\
&\simeq\pi_{P,\eS_e}^{*}\O_{\eS_e}[\dim\wt{\eN}_{P}-\dim\g+\dim \eS_e]\angles{\dim\wt{\eN}_{P}-2\dim\g+\dim\g+\dim \eS_e}\\
&\simeq\O_{\wt{\eN}_{P,\eS_e}}[\dim\wt{\eN}_{P,\eS_e}]\angles{\dim\wt{\eN}_{P,\eS_e}},
\end{split}
\end{equation}
where in the final step we have used Lemma~\ref{lem:LiSe-event-coconn}.
\end{proof}

\begin{atom}
\label{atm:tilt-bund-simp-mods}

In particular, the induced grading on $\cA_{P,\eS_e}^{\cov}$ is Koszul (and compatible with the grading on $\O(\eS_e)$). Thus, for each $b,b'\in \bfB_{e,P}$, the vector space $\Hom_{\wt{\eN}_{P,\eS_e}}(\cE_{\eS_e}^b,\cE_{\eS_e}^{b'})$ is concentrated in non-negative $\bG_m$-weights; moreover, the weight-$0$ component is spanned by the identity map if $b=b'$, and is $0$ otherwise. We henceforth fix a \emph{choice} of graded lifts for each parabolic $P$.

Let $\cL_{P,\eS_e}^{b}\in\Coh(\wt{\eN}_{P,\eS_e}/\bG_m)$ denote the simple object in the heart of the exotic t-structure with projective cover $\cE_{P,\eS_e}^{b}$, so that
\begin{equation}
\label{eqn:exotic-simple-maps}
\Hom_{\wt{\eN}_{P,\eS_e}}(\cE_{P,\eS_e}^{b'},\cL_{P,\eS_e}^b)\simeq\begin{cases}
k\angles{0}&\text{ if }b=b',\\
0&\text{ otherwise}.
\end{cases}
\end{equation}
In particular, the $\O(\eS_e)$ action on \eqref{eqn:exotic-simple-maps} factors through the maximal ideal of $e$, so $\cL_{P,\eS_e}^b$ is supported on the partial Springer fiber $\eP_e$. The same argument as in \cite[\S5.4.4]{bez-mirk} shows that the set $\{[\cL_{P,\eS_e}^{b}]:b\in\bfB_{e,P}\}$ forms a $K_0(\B\bG_m)$-basis for the equivariant $K$-theory group $K_0(\eP_e/\bG_m)$; we refer to it as the \emph{canonical basis}, even though it is no longer fully canonical when $P\ne B$. We will often misuse terminology slightly and refer to the set $\bfB_{e,P}$ itself as ``the canonical basis.'' Note that as in the proof of Proposition~\ref{prop:parab-A-kosz-grad}, Grothendieck--Serre duality implies that $\Ext^n_{\wt{\eN}_{P,\eS_e}}(\cL_{P,\eS_e}^{b'},\cL_{P,\eS_e}^b)$ is concentrated in weight-$n$, and vanishes for $n\ge\dim\wt{\eN}_{P,\eS_e}$.

Finally, we let $E^b_{P,\eS_e}$ and $L^b_{P,\eS_e}$ denote the corresponding objects on the right-hand side of \eqref{eqn:bm-functor} (after forgetting the $Z_e$-equivariance). These are the indecomposable projective and simple right $\cA_{P,\eS_e}$-modules, respectively. Moreover, we set
\begin{equation}
\label{eqn:indec-proj-left-ASe-mod}
E_{P,\eS_e}^{b,\ell}:=\Hom_{\wt{\eN}_{P,\eS_e}}(\cE^b_{P,\eS_e},\cE_{P,\eS_e})\simeq\Hom_{\cA_{P,\eS_e}^\op}(E^b_{P,\eS_e},\cA_{P,\eS_e})
\end{equation}
to be the corresponding indecomposable projective \emph{left} $\cA_{P,\eS_e}$-modules.
\end{atom}

\begin{atom}
\label{atm:braid-positivity}
We now state the \emph{braid positivity} property of the noncommutative Springer resolution. Recall from \S\ref{atm:braid-action} that there is an affine braid group action on $\QC(\wt{\eN}_{\eS_e}/\wt{Z}_e)$ for any Slodowy slice $\eS_e$. The action of any $a\in B^\ext_+$ on $\QC(\wt{\eN}_{\eS_e}/\wt{Z}_e)$ is then right t-exact with respect to the \emph{exotic} t-structure constructed in \eqref{eqn:bm-functor}. We will in fact be interested in the action of $\eH_{\eS_e}^\coh$ on $\QC(\wt{\eN}_{\eS_e}/\wt{Z}_e)$ via \emph{right} convolution. Note that pullback along the ``swap'' map $\sigma\colon\wt{\eN}_{\eS_e}\times_{\eS_e}\wt{\eN}_{\eS_e}\to\wt{\eN}_{\eS_e}\times_{\eS_e}\wt{\eN}_{\eS_e}$ interchanging the two copies of $\wt{\eN}_{\eS_e}$ intertwines right and left convolution. Moreover, $\sigma^*$ preserves the sheaves $\Delta_{\wt{\eN}/\g,*}\O_{\wt{\eN}}(\lambda)$ and $\O_{\Gamma_{s_\alpha}}$, hence induces an anti-involution $\sigma^*\colon B^\ext\to B^{\ext,\op}$ (which sends any $\wt{w}$ to $\wt{w^{-1}}$). Thus, the braid positivity property also holds for right convolution.

For the partial resolutions, let $X^*(P)\subset X^*(T)$ denote the character lattice of $P$ (or equivalently, the sublattice of weights orthogonal to coroots in the Levi factor of $P$), and set $X^*(P)^+:=X^*(P)\cap X^*(T)^+$. Observe that $\Lag_P^!(\rho)$ is canonically $\Rep(P)$-linear for $e=0$, and hence for any $e$. Braid positivity and \eqref{eqn:AP-Lag-adj} then show that the endofunctor
\begin{equation*}
-\otimes\O_{\wt{\eN}_{P,\eS_e}}(\lambda)\colon\QC(\wt{\eN}_{P,\eS_e}/\wt{Z}_e)\to\QC(\wt{\eN}_{P,\eS_e}/\wt{Z}_e)
\end{equation*}
is right t-exact with respect to the exotic t-structure, for any $\lambda\in X^*(P)^+$. We hope to give a more general braid positivity statement for the partial resolutions in a future article.
\end{atom}

\subsection{Generalization to connected reductive groups}
\label{sec:gen-conn-red}

\begin{atom}
We now aim to remove the assumptions on $G$; so suppose once again that $G$ is connected and reductive. We shall need the following lemma:
\end{atom}

\begin{lemma}
\label{lem:der-gp-coc}
Set $\overline{G}:=G/Z(G)^\circ$. Then the functor
\begin{equation}
\label{eqn:res-funct-der-sub}
\Res^G_{[G,G]}\colon\Rep(G)\to\Rep([G,G])
\end{equation}
admits a (non-canonical) $\Rep(\overline{G})$-linear section. In particular, applying $-\otimes_{\Rep(\overline{G})}\QC(\wt{\eN}/\overline{G})$, we obtain a $\Rep(G)$-linear section of the restriction functor
\begin{equation*}
\QC(\wt{\eN}/G)\to\QC(\wt{\eN}/[G,G]).
\end{equation*}
\end{lemma}

\begin{proof}
Set $Z:=Z([G,G])\cap Z(G)^\circ$, and recall that we have a short exact sequence
\begin{equation}
\label{eqn:der-gp-ses}
1\to Z\xrightarrow{g\mapsto(g,g^{-1})}[G,G]\times Z(G)^\circ\to G\to 1.
\end{equation}
In particular, we have $\overline{G}\cong[G,G]/Z$. Since $Z(G)^\circ$ is a torus, we have $\Rep(\overline{G})$-linear decompositions
\begin{equation}
\label{eqn:rep-G-Gder-char-decomp}
\Rep(G)\simeq\bigoplus_{\chi\in X^*(Z(G)^\circ)}\Rep(\overline{G})_{\chi},\qquad\qquad\Rep([G,G])\simeq\bigoplus_{\chi\in X^*(Z)}\Rep(\overline{G})_{\chi},
\end{equation}
as in \eqref{eqn:coc-rep-G-decomp}. Moreover, the restriction functor \eqref{eqn:res-funct-der-sub} is determined by the canonical restriction map
\begin{equation}
\label{eqn:char-res-map}
X^*(Z(G)^\circ)\to X^*(Z),
\end{equation}
which is a surjection. Thus, choosing any set-theoretic splitting of \eqref{eqn:char-res-map} yields the result.
\end{proof}

\begin{atom}
Now, consider the homomorphisms
\begin{equation*}
[G,G]^\sch\twoheadrightarrow[G,G]\hookrightarrow G,
\end{equation*}
where $[G,G]^\sch$ denotes the Schur covering as in Proposition~\ref{prop:schur-cov}. By Lemma~\ref{lem:sch-min-rep} and \S\ref{atm:nc-spr-res}, we have a Bezrukavnikov--Mirkovi\'c tilting bundle $\cE$ on $\wt{\eN}/[G,G]^\sch$.

As in Notation~\ref{not:schur-mult}, we let $\M(H)$ denote the Schur multiplier of a linear algebraic group $H$; we refer to Appendix~\ref{sec:schur-mult} for additional background on Schur multipliers. For each nontrivial cocycle\footnote{Technically, we should write $[(\cC,\alpha)]\in\M([G,G])$ for the isomorphism class of the cocycle $(\cC,\alpha)$, but we will not worry too much about this distinction in our exposition.} $(\cC,\alpha)\in\M([G,G])$, choose a minuscule weight $\lambda_{(\cC,\alpha)}$ of $[G,G]^\sch$ as in Lemma~\ref{lem:sch-min-rep}, and set $\lambda_{(\O_{[G,G]},\id)}:=0$.\footnote{We do not actually need these weights to be minuscule (in fact, all we really need is to choose nonzero compact objects of $\Rep(G)^{(\cC,\alpha)}$). However, this is in a sense the ``simplest'' way to modify $\cE$ to be $[G,G]$-equivariant.} Thus, by \eqref{eqn:sch-decomp} and Lemma~\ref{lem:der-gp-coc}, we have functors
\begin{equation*}
\QC(\wt{\eN}/[G,G]^\sch)\simeq\bigoplus_{(\cC,\alpha)\in\M([G,G])}\QC(\wt{\eN}/[G,G])^{(\cC,\alpha)}\xrightarrow{\bigoplus V_{\lambda_{(\cC,\alpha)}}^\vee\otimes-}\QC(\wt{\eN}/[G,G])\to\QC(\wt{\eN}/G),
\end{equation*}
where $\QC(\wt{\eN}/[G,G])^{(\cC,\alpha)}$ denotes the $(\cC,\alpha)$-twist of $\QC(\wt{\eN}/[G,G])$, i.e., the category of $(\cC,\alpha)$-equivariant sheaves on $\wt{\eN}$ as in \S\ref{atm:coc-twist-cat}. Misusing notation, we also denote the image of $\cE$ under this composition by $\cE$; though it is not uniquely determined, it is evidently also a tilting bundle, and it is a compact generator of $\QC(\wt{\eN}/G)$ under the action of $\QC(\B G)$ by Lemma~\ref{lem:any-coc-rep-gens}.

Thus, the properties in \S\ref{atm:tilt-gen}, \S\ref{atm:koszul-grading}, and \S\ref{atm:tilt-bund-simp-mods} also hold for $\cE$ in this more general setting, and we carry over all notations from these paragraphs. Moreover, we may extend the braid group action of \S\ref{atm:braid-action} to this setting using the same constructions\footnote{Note that the de-equivariantized action on $\QC(\wt{\eN})$ factors through the affine braid group for $[G,G]^\sch$ as in \eqref{eqn:der-gp-ses}. We may thus regard the de-equivariantized affine braid group action for $G$ as obtained from the pushforward of that for $[G,G]^\sch$ under the evident map $\wt{\eN}\times_{[\g,\g]}\wt{\eN}\to\wt{\eN}\times_{\g}\wt{\eN}$ (which is an isomorphism on classical truncations). Since this gives a monoidal functor between the de-equivariantized affine Hecke categories, it is now straightforward to deduce the requisite relations between the sheaves $\cK_a$.}, and the braid positivity properties in \S\ref{atm:braid-positivity} evidently carry over.
\end{atom}

\begin{atom}
Our goal now is to lift the properties of $\cA_{P,\eS_e}^{\cov}$ in Proposition~\ref{prop:parab-A-kosz-grad} and \S\ref{atm:tilt-bund-simp-mods} to $\cA_{P,\eS_e}$. That is, we wish to construct $\wt{Z}_e$-equivariant analogs of the simple and indecomposable projective $\cA_{P,\eS_e}^{\cov,\op}$-modules satisfying analogous grading properties. Our first step will be to show that these modules admit equivariance structures with respect to certain canonical cocycles of their stabilizers in $Z_e$.
\end{atom}

\begin{notation}
Let $\overline{Z}_e:=Z_e/Z_e^\circ\cdot Z(G)$. The $Z_e$-action on $\cA_{P,\eS_e}$ induces a canonical action of $Z_e$ on the set $\bfB_{e,P}$ of simple modules, which evidently factors through $\overline{Z}_e$. For each $b\in \bfB_{e,P}$, let $Z_e^b\subset Z_e$ denote the stabilizer of $b$. We fix a set $\bfB_{e,P}^\orb$ of orbit representatives for the action of $\overline{Z}_e$ on $\bfB_{e,P}$.
\end{notation}

\begin{proposition}
\label{prop:coc-equiv-simple-indec-proj}
There exists a canonical cocycle $(\cC_b,\alpha_b)\in\M(Z_e^b)$ for which the modules $L_{P,\eS_e}^{b}$ and $E_{P,\eS_e}^{b}$ admit canonical $(\cC_b,\alpha_b)$-equivariant structures, i.e., canonical lifts to $\cA_{P,\eS_e}^\op\dashmod^{Z_e^b,(\cC_b,\alpha_b)}$.
\end{proposition}

\begin{atom}
Note that by \cite[Cor.~8.5.7]{gr}, we may regard $\cA_{P,\eS_e}^\op\dashmod^{Z_e^b,(\cC_b,\alpha_b)}$ as the category of modules for the monad $(-\otimes\cA_{P,\eS_e})\in\End(\Rep(Z_e^b))$ on the $\Rep(Z_e^b)$-module category $\Rep(Z_e^b)^{(\cC_b,\alpha_b)}$. Thus, it is equivalent to show that the vector spaces underlying these modules carry compatible $(\cC_b,\alpha_b)$-representation structures.
\end{atom}

\begin{proof}
To simplify the notation, we work in the case where $P=B$; the proof for general $P$ is identical. We begin by showing that $L^b_{\eS_e}$ carries a canonical projective representation of $Z_e^b$. Since $Z_e$ acts on $\cA_{\eS_e}$ by $k$-algebra automorphisms, it preserves the Jacobson radical $\rmJ(\cA_{\eS_e})$; since the quotient $\cA_{\eS_e}/\rmJ(\cA_{\eS_e})$ is finite-dimensional, it is Artinian, hence semisimple, and
\begin{equation}
\label{eqn:noncomm-spr-res-mod-jacobson-rad}
\cA_{\eS_e}/\rmJ(\cA_{\eS_e})\simeq\bigoplus_{b\in \bfB_e}\End_k(L^b_{\eS_e}).
\end{equation}
Moreover, the $Z_e$-action preserves the summands corresponding to each $\overline{Z}_e$-orbit in $\bfB_e$, and $Z_e^b$ preserves the summand corresponding to $b$. Thus, by the Skolem--Noether theorem, $Z_e^b$ acts by inner automorphisms of this summand, i.e., via a projective representation
\begin{equation}
\label{eqn:simple-proj-rep}
Z_e^b\to\Aut(\End_k(L^b_{\eS_e}))\simeq\PGL(\underline{L}^b_{\eS_e}),
\end{equation}
where $\underline{L}^b_{\eS_e}$ denotes the underlying vector space of the module $L^b_{\eS_e}$. Pulling back the canonical cocycle of $\PGL(\underline{L}^b_{\eS_e})$ (corresponding to the central extension given by $\GL(\underline{L}^b_{\eS_e})$) then gives a cocycle $(\cC_b,\alpha_b)\in\M(Z_e^b)$, and \eqref{eqn:rep-G-coc-twist-res} shows that $L^b_{\eS_e}$ admits a $(\cC_b,\alpha_b)$-equivariant structure.

Next, we claim that we have a decomposition
\begin{equation}
\label{eqn:noncomm-spr-decomp-indec-proj}
\cA_{\eS_e}\simeq\bigoplus_{b\in\bfB_e}\underline{L}^{b,\vee}_{\eS_e}\tens{k}E^b_{\eS_e}
\end{equation}
of right $\cA_{\eS_e}$-modules, such that $Z_e$ permutes the summands according to the $\overline{Z}_e$-action on $\bfB_e$. By the argument of \cite[Rem.~4.6]{bez-losev}, there exists a $\wt{Z}_e$-stable choice of Koszul grading on $\cA_{\eS_e}$ (which is not necessarily the same as that provided by the $\bG_m$-action on $\cA_{\eS_e}$); choosing such a grading, we obtain a $\wt{Z}_e$-stable subspace
\begin{equation}
\label{eqn:noncomm-spr-res-gr-0-comp}
\bigoplus_{b\in\bfB_e}\End_k(L^b_{\eS_e})\subseteq\cA_{\eS_e}
\end{equation}
given by the $0$th graded piece. In particular, we obtain a decomposition of the unit element into orthogonal idempotents $e_b:=\id_{L^b_{\eS_e}}$, on which $Z_e$ acts by permutation according to the $\overline{Z}_e$-action on $\bfB_e$. Thus, it suffices to show that $e_b\cA_{\eS_e}\simeq\underline{L}^{b,\vee}_{\eS_e}\otimes_kE^b_{\eS_e}$. Since the latter is a projective cover of $\underline{L}^{b,\vee}_{\eS_e}\otimes_kL^b_{\eS_e}$, the right $\cA_{\eS_e}$-module surjection $e_b\cA_{\eS_e}\twoheadrightarrow\underline{L}^{b,\vee}_{\eS_e}\otimes_kL^b_{\eS_e}$ furnished by \eqref{eqn:noncomm-spr-res-mod-jacobson-rad} shows that it is a direct summand of $e_b\cA_{\eS_e}$. Moreover, the complementary submodule is contained in $J(\cA_{\eS_e})$, hence is trivial.

It follows that the module $\underline{L}^{b,\vee}_{\eS_e}\otimes_kE^b_{\eS_e}$ carries a $Z_e^b$-equivariant structure. Moreover, choosing a basis of $\underline{L}^b_{\eS_e}$, we obtain a further decomposition of $e_b$ into primitive orthogonal idempotents, and these decompositions are permuted according to the projective action of $Z_e^b$ on $\underline{L}^b_{\eS_e}$. It follows that the $Z_e^b$-representation $\underline{L}^{b,\vee}_{\eS_e}\otimes_kE^b_{\eS_e}$ splits as a tensor product of projective representations, i.e., $E^b_{\eS_e}$ carries a projective representation of $Z_e^b$. Since $\underline{L}^{b,\vee}_{\eS_e}$ is a $(\cC_b^\vee,\alpha_b^\vee)$-representation of $Z_e^b$, this must be a $(\cC_b,\alpha_b)$-representation by \eqref{eqn:rep-G-coc-twist-mult}. Moreover, it is easy to see that this projective action is compatible with the right $\cA_{\eS_e}$-module structure, which yields the conclusion.
\end{proof}

\begin{atom}
\label{atm:kos-grad-ASe}

Note that by \S\ref{atm:coc-twist-cat}, the functor \eqref{eqn:bm-functor} also induces an equivalence
\begin{equation*}
\QC(\wt{\eN}_{P,\eS_e}/Z_e^b)^{(\cC_b,\alpha_b)}\xrightarrow{\sim}\cA_{P,\eS_e}^\op\dashmod^{Z_e^b,(\cC_b,\alpha_b)}.
\end{equation*}
Thus, the sheaves $\cL_{P,\eS_e}^{b}$ and $\cE_{P,\eS_e}^{b}$ also admit canonical $(\cC_b,\alpha_b)$-equivariant structures. It follows that the decomposition \eqref{eqn:noncomm-spr-decomp-indec-proj} also gives rise to a decomposition
\begin{equation}
\label{eqn:noncomm-spr-slod-res-coc-decomp}
\cE_{\eS_e}\simeq\bigoplus_{b\in\bfB_{e,P}}\underline{L}_{P,\eS_e}^{b,\vee}\tens{k}\cE_{P,\eS_e}^{b}
\end{equation}
of $Z_e$-equivariant vector bundles. Equipping the left-hand side with the $\bG_m$-equivariant structure arising from that on each $\cE_{P,\eS_e}^{b}$, we obtain a $\bG_m$-equivariant structure on $\cE_{P,\eS_e}$ inducing a Koszul grading on $\cA_{P,\eS_e}$ (moreover, the ``Koszul dual'' algebra $\End_{\cA_{P,\eS_e}^\op}(\cA_{P,\eS_e}^{\bG_m})$ has its degree-$n$ cohomology concentrated in weight $-n$). Henceforth, when we refer to $\cE_{P,\eS_e}$, we implicitly equip it with this particular $\wt{Z}_e$-equivariant structure, and likewise for $\cA_{P,\eS_e}$.
\end{atom}

\section{A covering group of the reductive centralizer}
\label{sec:Ze-cov}

\begin{atom}
In this section, we establish existence of a finite cover of $Z_e$ which trivializes all cocycles appearing in the canonical bases $\bfB_{e,P}$ (see Proposition~\ref{prop:Zecov-intro}). This will assist us in later homological computations by allowing us to replace $\cA_{\eS_e}$ with the Koszul quadratic algebra $\cA_{\eS_e}^\cov$.
\end{atom}

\begin{theorem}
\label{thm:Ze-cov}
There exists a finite cover $p_{\cov}\colon Z_e^\cov\twoheadrightarrow Z_e^\sch$ such that for any parabolic $P$ and $b\in\bfB_{e,P}$, the class of $(\cC_b,\alpha_b)$ lies in the kernel of the restriction map
\begin{equation}
\label{eqn:ze-cov-m-res}
\M(Z_e^b)\to\M(Z_e^{\cov,b}).
\end{equation}
\end{theorem}

\begin{remark}
In fact, over the course of the proof, we shall show that we may take $Z_e^\cov=Z_e^\sch$, unless the adjoint group of $G$ contains either
\begin{enumerate}
\item a simple factor of type $E_6$ for which the corresponding block of $e$ has weighted Dynkin diagram
\begin{equation}
\label{eqn:bad-nilp-weighted-dynkin-1}
\begin{tikzcd}[row sep=small,column sep=small]
0\arrow[r,no head]&0\arrow[r,no head]&2\arrow[r,no head]\arrow[d,no head]&0\arrow[r,no head]&0;\\
&&0&&
\end{tikzcd}
\end{equation}
\item a simple factor of type $E_8$ for which the corresponding block of $e$ has weighted Dynkin diagram
\begin{equation}
\label{eqn:bad-nilp-weighted-dynkin-2}
\begin{tikzcd}[row sep=small,column sep=small]
0\arrow[r,no head]&0\arrow[r,no head]&0\arrow[r,no head]&0\arrow[r,no head]&0\arrow[r,no head]\arrow[d,no head]&0\arrow[r,no head]&0.\\
&&&&2&&
\end{tikzcd}
\end{equation}
\end{enumerate}
If this occurs, we may take $Z_e^\cov$ to be a (non-central, non-split) extension of $Z_e^\sch$ by
\begin{equation*}
(\Z/2\Z\times\Z/2\Z)^{n_1}\times(\Z/3\Z)^{n_2},
\end{equation*}
where $n_1,n_2$ are the multiplicities of the nilpotents \eqref{eqn:bad-nilp-weighted-dynkin-1}, \eqref{eqn:bad-nilp-weighted-dynkin-2} in $e$, respectively.
\end{remark}

\begin{proof}
Let $b\in\bfB_{e,P}$.

\begin{atom}
It suffices to show that either $[(\cC_b,\alpha_b)]$ lies in the image of the restriction map
\begin{equation}
\label{eqn:stab-coc-res}
\M(Z_e)\to\M(Z_e^b),
\end{equation}
or there exists a finite cover $\ps{b}{Z}_e^{\cov}\twoheadrightarrow Z_e$ such that $[(\cC_b,\alpha_b)]$ lies in the kernel of
\begin{equation}
\label{eqn:coc-fin-cov-res}
\M(Z_e^b)\to\M(\ps{b}{Z}_e^{\cov}).
\end{equation}
Indeed, let $b_1,\ldots,b_r\in\bfB_{e,P}$ denote those canonical basis elements whose cocycles do not lie in the image of \eqref{eqn:stab-coc-res}. Consider the fiber square
\begin{equation}
\label{eqn:fib-sq-Ze-cov}
\begin{tikzcd}
Z_e^\cov\arrow[r,two heads]\arrow[d,two heads]&Z_e^\sch\arrow[d,two heads]\\
\ps{b_1}{Z}_e^\cov\fibprod{Z_e}\cdots\fibprod{Z_e}\ps{b_r}{Z}_e^\cov\arrow[r,two heads]&Z_e,
\end{tikzcd}
\end{equation}
and the corresponding diagram of stabilizers of $b$. If $[(\cC_b,\alpha_b)]$ lies in the image of \eqref{eqn:stab-coc-res}, then it lies in the kernel of $\M(Z_e^b)\to\M(Z_e^{\sch,b})$, hence in the kernel of \eqref{eqn:ze-cov-m-res}. Otherwise, since we have a factorization $Z_e^{\cov,b}\twoheadrightarrow\ps{b}{Z}_e^{\cov,b}\twoheadrightarrow Z_e^b$ for each $b\in\bfB_{e,P}$, the conclusion follows from \eqref{eqn:coc-fin-cov-res}.
\end{atom}

\begin{atom}
Now, any cocycle of $Z_e^b$ appearing in the canonical basis is pulled back from a cocycle of $Z_e^b/Z(G)$, which is the corresponding stabilizer for the adjoint group $G/Z(G)$. Thus, the commutative square
\begin{equation*}
\begin{tikzcd}
\M(Z_e/Z(G))\arrow{r}\arrow{d}&\M(Z_e^b/Z(G))\arrow{d}\\
\M(Z_e)\arrow{r}&\M(Z_e^b)
\end{tikzcd}
\end{equation*}
immediately reduces us to the case where $G$ is semisimple and adjoint (note that a finite cover of $Z_e^b/Z(G)$ pulls back to one of $Z_e^b$ with the corresponding property). In particular, $G$ splits as a product of adjoint simple groups; the cocycles appearing in $\bfB_{e,P}$ clearly all split as products over these simple factors, so we may further assume that $G$ is simple and adjoint.

We proceed type by type, using the classification of reductive centralizers of nilpotent elements, and show that \eqref{eqn:stab-coc-res} is surjective for any subgroup $Z_e^b$ which is a union of connected components, except in the situation of \eqref{eqn:bad-nilp-weighted-dynkin-2}, where we obtain a cover satisfying \eqref{eqn:coc-fin-cov-res}.
\end{atom}

\begin{atom}
In type $A$, the component group of $Z_e$ is always trivial, so there is nothing to check.

In types $B$, $C$, and $D$, recall that $Z_e$ can be written as a semidirect product $Z_e^\circ\rtimes\pi_0(Z_e)$, where $Z_e^\circ$ is the quotient of a product of symplectic groups and special orthogonal groups by $Z(G)\cong\Z/2\Z$, and $\pi_0(Z_e)\cong(\Z/2\Z)^n$, where $n$ is the number of orthogonal factors in $Z_e^\circ$. Lemma~\ref{lem:coc-rtimes} thus gives an exact sequence
\begin{equation}
\label{eqn:bcd-coc-exact-seq}
0\to\H^1(\pi_0(Z_e),X^*(Z_e^\circ))\to\ker\big(\M(Z_e)\twoheadrightarrow\M(\pi_0(Z_e))\big)\to\M(Z_e^\circ)^{\pi_0(Z_e)}\to\H^2(\pi_0(Z_e),X^*(Z_e^\circ)).
\end{equation}
Note that the action of $\pi_0(Z_e)$ on $\M(Z_e^\circ)$ is trivial; moreover, we claim that the final homomorphism is trivial. Indeed, by Lemma~\ref{lem:sch-prod} and functoriality of \eqref{eqn:bcd-coc-exact-seq}, the final homomorphism splits as a product over the factors of $Z_e^\circ$; thus, we may assume that $Z_e^\circ$ is either a symplectic or special orthogonal group, on which $\pi_0(Z_e)$ acts either trivially, or by projection onto a single factor of $\Z/2\Z$, respectively. In the former case, both $\M(Z_e^\circ)$ and $X^*(Z_e^\circ)$ are trivial as the symplectic groups are simply-connected and simple. In the latter case, if $Z_e^\circ\cong\SO_2\cong\bG_m$, then $\M(Z_e^\circ)$ is again trivial; otherwise, if $Z_e^\circ\cong\SO_m$ for $m>2$, then we have $X^*(Z_e^\circ)\simeq 0$, so the final term in \eqref{eqn:bcd-coc-exact-seq} is trivial.

Now, let $\Gamma\subset\pi_0(Z_e)$ be a finite subgroup. By functoriality of \eqref{eqn:bcd-coc-exact-seq}, to show that the restriction map $\M(Z_e)\to\M(Z_e^\circ\rtimes\Gamma)$ is a surjection, it suffices to show that
\begin{equation}
\label{eqn:coc-surj-conditions}
\begin{split}
\H^1(\pi_0(Z_e),X^*(Z_e^\circ))&\twoheadrightarrow\H^1(\Gamma,X^*(Z_e^\circ)),\\
\M(\pi_0(Z_e))&\twoheadrightarrow\M(\Gamma),
\end{split}
\end{equation}
i.e., both restriction maps are surjections. The latter is immediate, as any finite subgroup of $(\Z/2\Z)^n$ is a direct summand. For the former, we may assume as before that $X^*(Z_e^\circ)\cong\Z$ and that $\pi_0(Z_e)$ acts by projection onto a single factor of $\Z/2\Z$ (via negation). If $\Gamma$ acts trivially on $\Z$, then $\H^1(\Gamma,\Z)\cong\Hom(\Gamma,\Z)\simeq 0$, and we are done. Otherwise, let $K$ be the index-$2$ subgroup of $\pi_0(Z_e)$ fixing $\Z$; the inflation-restriction exact sequence then gives a commutative diagram
\begin{equation*}
\begin{tikzcd}
0\arrow{r}&\H^1(\Z/2\Z,\Z)\arrow{r}\arrow[d,"\sim" vert]&\H^1(\pi_0(Z_e),\Z)\arrow{r}\arrow{d}&\H^1(K,\Z)^{\Z/2\Z}\arrow{d}\\
0\arrow{r}&\H^1(\Z/2\Z,\Z)\arrow{r}&\H^1(\Gamma,\Z)\arrow{r}&\H^1(\Gamma\cap K,\Z)^{\Z/2\Z}
\end{tikzcd}
\end{equation*}
with exact rows. Since the right-most terms are trivial, the second vertical map is an isomorphism, and we are done.
\end{atom}

\begin{atom}
It remains to treat the exceptional types; we use the tables of centralizers of nilpotent elements for adjoint exceptional groups appearing in \cite{alexeevski}. When $Z_e$ is connected, the assertion is trivial. We may further disregard all cases in which $\pi_0(Z_e)\cong\Z/2\Z$, and either $\M(Z_e^\circ)$ is trivial or $Z_e\cong Z_e^\circ\times\pi_0(Z_e)$. In all other cases, we have $Z_e\cong Z_e^\circ\rtimes\pi_0(Z_e)$. Moreover, the final map in \eqref{eqn:bcd-coc-exact-seq} is again trivial: it is not hard to check that either $\M(Z_e^\circ)$ or $X^*(Z_e^\circ)$ is always trivial. As before, the action of $\pi_0(Z_e)$ on $\M(Z_e^\circ)$ is trivial, except in the case \eqref{eqn:coc-fin-cov-res}. Here, we have $Z_e\simeq\PGL_3\rtimes\Z/2\Z$, and \eqref{eqn:bcd-coc-exact-seq} gives $\M(Z_e)\cong 0$, whereas $\M(\PGL_3)\cong\Z/3\Z$. However, $Z_e$ admits a natural $3$-fold cover given by $\SL_3\rtimes\Z/2\Z$, which satisfies \eqref{eqn:coc-fin-cov-res} as $\M(\SL_3)\cong 0$. Thus, we must verify \eqref{eqn:coc-surj-conditions} in each of the remaining examples, for all subgroups $\Gamma\subset\pi_0(Z_e)$.

For the former assertion, we need only check the case in which the symmetric group $S_3$ acts on the weight lattice $\Lambda\cong\Z^3/\Z$ of $\SL_3$ by permutation, which appears in the situation of \eqref{eqn:bad-nilp-weighted-dynkin-1}. We may assume that either $\Gamma\cong\angles{(123)}$ or $\Gamma\cong\angles{(12)}$. In the former case,  we have
\begin{equation*}
\H^1(\Gamma,\Lambda)\cong\ker((1+(123)+(132))|_{\Lambda})/(1-(123))\Lambda\cong\Lambda/\angles{\Phi}\cong\Z/3\Z,
\end{equation*}
with the trivial action of $\Z/2\Z$ (here $\angles{\Phi}$ denotes the root lattice of $\SL_3$). Since $\Lambda^\Gamma\cong 0$, the inflation-restriction exact sequence implies that $\H^1(S_3,\Lambda)\cong\H^1(\Gamma,\Lambda)$, as desired. In the latter case, $\Lambda$ splits as a $\Gamma$-module into a sum of $\Z$ with the trivial $\Gamma$-action and $\Z$ with $\Gamma$ acting by negation; thus, $\H^1(\angles{(12)},\Lambda)\cong\Z/2\Z$, and the restriction map is not surjective. To obtain a cover of $Z_e\cong T\rtimes S_3$ killing this cocycle (here $T\subset\SL_3$ is a maximal torus), consider the $S_3$-equivariant short exact sequence
\begin{equation*}
0\to\Lambda\xrightarrow{2}\Lambda\to\Z/2\Z\times\Z/2\Z\to 0.
\end{equation*}
The endomorphism of $\H^1(\angles{(12)},\Lambda)$ induced by the first map is the zero map, so by the following paragraph, the corresponding surjection
\begin{equation*}
1\to\Z/2\Z\times\Z/2\Z\to T\rtimes S_3\xrightarrow{(-)^2\rtimes\id}T\rtimes S_3\to 1
\end{equation*}
suffices.

For the latter assertion, the only nontrivial component groups which appear are the symmetric groups $S_2,S_3,S_4,S_5$. The assertion is trivial in the first two cases. For the third case, we need only consider the restriction maps $\M(S_4)\to\M(A_4)$, $\M(S_4)\to\M(D_8)$ (the dihedral subgroup of order $8$), $\M(S_4)\to\M(\Z/2\Z\times\Z/2\Z)$ (the normal Klein four-subgroup) and $\M(S_4)\to\M(S_2\times S_2)$ (the non-normal Klein four-subgroup). Applying Lemma~\ref{lem:coc-rtimes} to $S_4\cong A_4\rtimes\Z/2\Z$, and using the well-known identity $\M(A_4)\cong\Z/2\Z$, we see that the first map is an isomorphism. For the second map, applying functoriality of Lemma~\ref{lem:coc-rtimes} reduces us to showing that $\M(A_4)\twoheadrightarrow\M(\Z/2\Z\times\Z/2\Z)$; since $A_4\cong(\Z/2\Z\times\Z/2\Z)\rtimes\Z/3\Z$, a further application of Lemma~\ref{lem:coc-rtimes} gives the result. This also shows that $\M(S_4)\twoheadrightarrow\M(\Z/2\Z\times\Z/2\Z)$. Finally, applying functoriality of Lemma~\ref{lem:coc-rtimes} to $S_2\times S_2\subset\Z/4\Z\rtimes S_2\cong D_8$ shows that the final map is an isomorphism.

For the fourth case, we need only consider the restriction maps $\M(S_5)\to\M(A_5)$, $\M(S_5)\to\M(S_4)$, and $\M(S_5)\to\M(S_3\times S_2)$. The first follows as for $A_4\subset S_4$. For the second, functoriality of Lemma~\ref{lem:coc-rtimes} reduces us to showing that $\M(A_5)\twoheadrightarrow\M(A_4)$, and it is well-known that the restriction of a Schur cover of $A_5$ to $A_4$ remains a Schur cover (indeed, the Schur cover of $A_n$ for $n=4,5$ and $n\ge 8$ is constructed by pulling back the double cover $\Spin_{n-1}\to\SO_{n-1}$ along the embedding $A_n\hookrightarrow\SO_{n-1}$). For the third, applying functoriality of Lemma~\ref{lem:coc-rtimes} to $S_3\times S_2\subset A_5\rtimes S_2$ shows that the restriction map is in fact trivial. However, consider the ``$+$'' and ``$-$'' type Schur double covers of $S_5$; it suffices to show that each of their restrictions to $S_2\times S_2\subset S_3\times S_2$ kills all elements of $M(S_2\times S_2)$. Indeed, these restrictions are given by the dihedral group and the quaternion group, respectively, both of which are Schur covers of $S_2\times S_2$.
\end{atom}

This completes the proof of Theorem~\ref{thm:Ze-cov}.
\end{proof}

\begin{atom}
In fact, the proof gives more information about the structure of $Z_e^{\cov}$, though this will not be needed in the present paper:
\end{atom}

\begin{corollary}
\label{cor:Ze-cov-simp-conn-der}
\begin{enumerate}[leftmargin=*]
\item The derived subgroup of $Z_e^{\cov,\circ}$ is simply-connected.
\item Suppose that all simple factors of $[G,G]$ are of classical types, and that $[Z_e^{\circ},Z_e^{\circ}]$ is simply-connected. Then $(\cC_b,\alpha_b)$ is trivial for all $b\in\bfB_{e,P}$ (and thus we may take $Z_e^\cov=Z_e$).
\end{enumerate}
\end{corollary}

\begin{proof}
We begin with the first assertion. Following the proof of Theorem~\ref{thm:Ze-cov}, we may check this type-by-type for the adjoint group $G/Z(G)$ by Lemma~\ref{lem:sch-cov-prod}. In type $A$, this follows from Lemma~\ref{lem:sch-cov-conn-red}. In types $B$, $C$, and $D$, the discussion following \eqref{eqn:bcd-coc-exact-seq} gives a short exact sequence
\begin{equation}
\label{eqn:ses-M-Ze}
0\to\M(Z_e^\circ)\to\M(Z_e)\to\M(\pi_0(Z_e))\to 0.
\end{equation}
The proof of Proposition~\ref{prop:schur-mult-bij-cent-ext} then shows that the identity component of $Z_e^\sch/X^*(\M(\pi_0(Z_e)))$ is a Schur cover of $Z_e^\circ$, so the result again follows by Lemma~\ref{lem:sch-cov-conn-red}. Finally, in the exceptional types, Lemma~\ref{lem:sch-cov-conn-red} covers the case where $Z_e$ is connected. Likewise, by Proposition~\ref{prop:kumar-neeb} and Lemma~\ref{lem:sch-cov-prod}, the statement is clear when either $\M(Z_e^\circ)$ is trivial or $Z_e\cong Z_e^\circ\times\pi_0(Z_e)$. Otherwise, we always have $X^*(Z_e^\circ)\cong 0$. When the action of $\pi_0(Z_e)$ on $\M(Z_e^\circ)$ is trivial, the statement follows as for types $B$, $C$, and $D$. In the only remaining case, we have shown that $Z_e^{\cov,\circ}\cong\SL_3$, which completes the proof.

For the second assertion, it suffices to show that the image of $[(\cC_b,\alpha_b)]$ under $\M(Z_e^b/Z(G))\to\M(Z_e^b)$ is trivial. We may again reduce to checking this type-by-type. As before, the conclusion is clear whenever $Z_e/Z(G)$ is connected, which proves the claim in type $A$. In types $B$, $C$, and $D$, note that the universal cover of $G/Z(G)$ is at most a four-fold cover. Thus, if $Z_e^\circ$ has simply-connected derived subgroup, then at most one special orthogonal factor appears in $(Z_e/Z(G))^\circ$. In particular, $\pi_0(Z_e/Z(G))$ is a subgroup of $\Z/2\Z$, and has trivial Schur multiplier. We may therefore assume that $Z_e^b=Z_e$, and the conclusion follows from functoriality of \eqref{eqn:ses-M-Ze}.
\end{proof}

\begin{remark}
In principle, one could verify the latter statement for exceptional types via a finite amount of computation. One must for instance show that the nontrivial cocycles of $Z_e$ in the case \eqref{eqn:bad-nilp-weighted-dynkin-1} do not appear in the canonical basis (as a torus has trivial derived subgroup).
\end{remark}

\begin{atom}
Finally, as for $\cA_{P,\eS_e}$, we obtain equivariant lifts of $\cA_{P,\eS_e}^\cov$ and its simple and indecomposable projective modules:
\end{atom}

\begin{corollary}
\label{cor:bm-cov-equiv}
The sheaves $\cL^b_{P,\eS_e}$ and $\cE^b_{P,\eS_e}$ admit lifts to $\Coh(\wt{\eN}_{P,\eS_e}/\wt{Z}_e^{\cov,b})$, and the objects $\bigoplus_{b'\in\overline{Z}_e\cdot b}\cL^{b'}_{P,\eS_e}$ and $\bigoplus_{b'\in\overline{Z}_e\cdot b}\cE^{b'}_{P,\eS_e}$ admit lifts to $\Coh(\wt{\eN}_{P,\eS_e}/\wt{Z}_e^\cov)$. In particular, we have an equivalence
\begin{equation*}
\Hom_{\wt{\eN}_{P,\eS_e}}(\cE_{P,\eS_e}^{\cov},-)\colon\QC(\wt{\eN}_{P,\eS_e}/\wt{Z}_e^\cov)\xrightarrow{\sim}\cA_{P,\eS_e}^{\cov,\op}\dashmod^{\wt{Z}_e^\cov}.
\end{equation*}
\end{corollary}

\begin{proof}
The first assertion follows from Proposition~\ref{prop:coc-equiv-simple-indec-proj}, \eqref{eqn:rep-G-coc-twist-res}, and Theorem~\ref{thm:Ze-cov}. For the second assertion, let $\Ind_{\wt{Z}_e^{\cov,b}}^{\wt{Z}_e^\cov}$ denote the pushforward along the projection $\wt{\eN}_{P,\eS_e}/\wt{Z}_e^{\cov,b}\to\wt{\eN}_{P,\eS_e}/\wt{Z}_e^\cov$, and simply note that
\begin{equation*}
(\Ind_{\wt{Z}_e^{\cov,b}}^{\wt{Z}_e^\cov}\cL^b_{P,\eS_e})^\dq\simeq\bigoplus_{b'\in\overline{Z}_e\cdot b}\cL^{b'}_{P,\eS_e},\qquad\qquad(\Ind_{\wt{Z}_e^{\cov,b}}^{\wt{Z}_e^\cov}\cE^b_{P,\eS_e})^\dq\simeq\bigoplus_{b'\in\overline{Z}_e\cdot b}\cE^{b'}_{P,\eS_e}.
\end{equation*}
The final assertion is now immediate from \S\ref{atm:koszul-grading}.
\end{proof}

\section{The Block--Getzler sheaf}
\label{sec:bg-sheaf}

\begin{atom}

Recall that given a monoidal category $\eA$ equipped with a monoidal endofunctor $F_{\eA}$, the $2$-categorical class map is a functor
\begin{equation*}
[-]\colon(\eA,F_{\eA})\dashbmod\to\Tr(\eA,F_{\eA})
\end{equation*}
from the category of $\eA$-module categories and $F_{\eA}$-semilinear endofunctors to the categorical trace of $\eA$ with respect to $F_{\eA}$. We refer to \S\ref{atm:mor-L-Morita}-\ref{atm:tr-A-F-cat-end} for all notations (and we will refer freely to Appendix~\ref{sec:tr-form-chap} throughout the exposition). In particular, sending any $a\in\eA$ to the $2$-categorical class $[\eA,F_{\eA}(-)\otimes a]$ of the regular $\eA$-module recovers the universal trace functor discussed in the introduction.

In this section, we construct our main technical tool: an explicit complex computing the $2$-categorical class map under certain assumptions on $\eA$ (which will suffice for all of our applications). More precisely, assume that $\eA$ is compactly generated and rigid, and that $\eA$ admits a central functor from the category of quasi-coherent sheaves on a suitable quotient stack $X/G$; we refer to \S\ref{sec:bg-assumptions} and \S\ref{atm:bg-tr-res} for the details. The ``Block--Getzler sheaf'' then computes the restriction of a $2$-categorical class to the categorical trace of $\QC(X/G)$ via \eqref{eqn:tr-adj-ind-res}. This trace identifies with the category of quasi-coherent sheaves on a certain (twisted) loop space by Proposition~\ref{prop:tr-qc-loop-ident}. Our construction is essentially a straightforward extension of the ``Block--Getzler complex'' of \cite{block-getzler} (see also \cite[\S2.1.2]{benzvi} and \cite[Def.~2.3.4]{chen-eq-loc}), which computes Hochschild homology in the equivariant setting.
\end{atom}

\subsection{Construction}

\begin{atom}
\label{sec:bg-assumptions}

Let $G$ be a reductive\footnote{But not necessarily connected; we need only that the Peter--Weyl theorem holds for $G$.} group acting on a scheme $X$, and suppose that the quotient stack $X/G$ is \emph{perfect} as in Definition~\ref{def:passable-perfect} (e.g., $X$ is finite-type and quasi-affine). Let $\phi\colon X\to X$ be a self-map commuting with the $G$-action, and denote by $\Gamma^G\colon\QC(X/G)\to\QC(\B G)=\Rep(G)$ the functor of equivariant global sections (i.e., the pushforward along the natural projection). Finally, let $G$ act diagonally on $X\times G$ as in \S\ref{atm:loop-space-quot-stack}.

We begin by constructing a precursor to the Block--Getzler sheaf, which we will soon equip with additional structure in Construction~\ref{cons:block-getzler-sheaf}.
\end{atom}

\begin{definition}
\label{def:bg}
Let $\eC$ be a compactly generated dg-category enriched in $\QC(X/G)$, and let $\ucHom_{\eC}$ denote the $\QC(X/G)$-internal $\Hom$. Suppose that $\eC$ is moreover equipped with a $\phi^*$-semilinear endofunctor $F_{\eC}$ preserving compact objects.\footnote{I.e., for any $c,c'\in\eC$, the map $F_{\eC}\colon\uHom_{\eC}(c,c')\to\uHom_{\eC}(F(c),F(c'))$ of $\O_{X/G}$-modules is $\phi^*$-semilinear, where $\phi^*\colon\O_{X/G}\to\O_{X/G}$ is the induced homomorphism. As in \S\ref{atm:mor-L-Morita}, when $F_{\eC}$ is the identity, we will often omit it from the notation.} The \emph{pre-Block--Getzler sheaf} $\preBG_{X/G,\phi}(\eC,F_{\eC})$ is defined to be the sum totalization of the simplicial quasicoherent sheaf on $(X\times G)/G$ given by
\begin{equation*}
\preBG_{X/G,\phi}^{-n}(\eC,F_{\eC}):=\bigoplus_{c_0,\ldots,c_n\in\eC^c}\big(\Gamma^G\ucHom_{\eC}(c_0,c_1)\tens{k}\cdots\tens{k}\Gamma^G\ucHom_{\eC}(c_{n-1},c_n)\tens{k}\ucHom_{\eC}(c_n,F_{\eC}(c_0))\big)\boxtimes\O_G,
\end{equation*}
where the ($\O_{X\times G}$-linear) face maps $d_i\colon\preBG_{X/G,\phi}^{-n}(\eC,F_{\eC})\to\preBG_{X/G,\phi}^{-(n-1)}(\eC,F_{\eC})$ (for $i=0,\ldots,n$) ``compose morphisms.'' More precisely\footnote{Note that for any $c,c',c''\in\eC^c$, we have a composition map $\Gamma^G\uHom_{\eC}(c,c')\otimes_k\uHom_{\eC}(c',c'')\to\uHom_{\eC}(c,c'')$ adjoint to $\Gamma^G$ of the usual composition map  $\uHom_{\eC}(c,c')\to\cHom_{X/G}(\uHom_{\eC}(c',c''),\uHom_{\eC}(c,c''))$. Taking $\Gamma^G$, we also obtain a composition map $\Gamma^G\uHom_{\eC}(c,c')\otimes_k\Gamma^G\uHom_{\eC}(c',c'')\to\Gamma^G\uHom_{\eC}(c,c'')$.}, we have
\begin{equation}
\label{eqn:prebg-face-maps}
\begin{split}
d_0(f_0\otimes\cdots\otimes f_n\boxtimes r)&=f_1\otimes\cdots\otimes f_{n-1}\otimes\varrho(\Gamma^GF_{\eC}(f_0))\circ f_n\boxtimes r,\\
d_i(f_0\otimes\cdots\otimes f_n\boxtimes r)&=f_0\otimes\cdots\otimes f_{i+1}\circ f_i\otimes\cdots\otimes f_n\boxtimes r,\text{ for }i=1,\ldots,n,
\end{split}
\end{equation}
where for any $V\in\Rep(G)$ we (misusing notation) let $\varrho\colon V\to\O(G)\otimes V\simeq V\otimes\O(G)$ denote the \emph{left} coaction map.\footnote{I.e., whose specialization at $g\in G$ is given by the map $g^{-1}$. Note that this convention is opposite to that of \cite[Def.~2.12]{benzvi}, and arises from our different definition of the face maps $d_i$. We will elaborate more on these conventions in footnote~\ref{fn:bg-face-maps}.}
\end{definition}

\begin{atom}
\label{atm:dq-functor}
We now recall the relationship between the pre-Block--Getzler sheaf and Hochschild homology, following \cite[\S2.1.2]{benzvi}. Let $\eC$ be a compactly generated $\QC(X/G)$-module category equipped with a $\phi^*$-semilinear endofunctor $F_{\eC}$ preserving compact objects. Then we may consider the \emph{de-equivariantization}
\begin{equation*}
\eC^\dq:=\Vect_k\tens{\Rep(G)}\eC,
\end{equation*}
which is a $\QC(X)$-module category, and admits a natural ``forgetful'' functor $(-)^\dq\colon\eC\to\eC^\dq$ preserving compact objects\footnote{E.g., by \cite[Cor.~9.3.3]{gr}. Note that the image of $(-)^\dq$ also generates $\eC^\dq$ under colimits.}. Since $\Vect_k^\dq\simeq\QC(G)$, the category $\eC^\dq$ carries a natural $\QC(G)$-module structure (where the monoidal structure on $\QC(G)$ is via convolution). In particular, for any $g\in G(k)$, the action of the skyscraper sheaf at $g$ yields an automorphism $g_*\colon\eC^\dq\to\eC^\dq$. We then have a diagram
\begin{equation*}
\begin{tikzcd}[column sep=large]
\eC\arrow{d}{F_{\eC}}\arrow{r}{(-)^\dq}&\eC^\dq\arrow{d}{F_{\eC}^\dq}\arrow{r}{g_*}&\eC^\dq\arrow{d}{F_{\eC}^\dq}\\
\eC\arrow{r}{(-)^\dq}&\eC^\dq\arrow{r}{g_*}&\eC^\dq,
\end{tikzcd}
\end{equation*}
where the left square is equipped with a canonical commuting structure as $F_{\eC}$ is canonically $\Rep(G)$-linear, and the right square is equipped with a canonical commuting structure as $F_{\eC}^\dq$ acquires a canonical $\QC(G)$-linear structure as before. Note that $g_*\circ(-)^\dq\simeq(-)^\dq$, as the same is true for
\begin{equation*}
(-)^\dq\colon\Rep(G)\to\Rep(G)^\dq\simeq\Vect_k.
\end{equation*}
Indeed, given $V,V'\in\Rep(G)$, the functor $g_*$ acts according to the canonical $G$-representation on $\Hom_k(V^\dq,V^{\prime,\dq})$. In particular, writing $F_{\eC,g}^\dq:=F_{\eC}^\dq\circ g_*\simeq g_*\circ F_{\eC}$ for the ``$g$-twisted'' endofunctor, we obtain a $2$-morphism of pairs
\begin{equation*}
\dq_g\colon(\eC,F_{\eC})\Rightarrow(\eC^\dq,F_{\eC,g}^\dq)
\end{equation*}
in $L(\Morita(\dgCat_k))_\rgd$, as in \S\ref{atm:hh-dg-cat}.
\end{atom}

\begin{lemma}
\label{lem:bg-hh}
The pair $(\eC,F_{\eC})$ admits a natural $\QC(X/G)$-enrichment $(\eC^\enr,F_{\eC}^\enr)$, such that $F_{\eC}^\enr$ is a $\phi^*$-semilinear endofunctor of $\eC^\enr$ in the sense of Definition~\ref{def:bg}. The global sections of the pre-Block--Getzler sheaf and its fibers over $G$ then compute Hochschild homology, i.e., we have a natural commutative diagram
\begin{equation*}
\begin{tikzcd}[column sep=large]
\Gamma(\preBG_{X/G,\phi}(\eC^\enr,F_{\eC}^\enr))\arrow{d}{\vertsim}\arrow{r}&\Gamma(i_g^*\preBG_{X/G,\phi}(\eC^\enr,F_{\eC}^\enr))\arrow{d}{\vertsim}\\
\HH(\eC,F_{\eC})\arrow{r}{\HH(\dq_{g^{-1}})}&\HH(\eC^\dq,F_{\eC,g^{-1}}^\dq),
\end{tikzcd}
\end{equation*}
where $i_g\colon X\times\{g\}\to(X\times G)/G$ denotes the natural map, and the top horizontal arrow is induced by the unit of the adjunction $i_g^*\dashv i_{g,*}$.
\end{lemma}

\begin{proof}
For the first assertion, note that by \S\ref{atm:pass-perf-qc-rgd}, the category $\QC(X/G)$ is compactly generated and rigid. Thus, for any $c\in\eC$, the functor $\act_c\colon\QC(X/G)\to\eC$ given by acting on $c$ has a $\QC(X/G)$-linear right adjoint $\act_c^R$ by \cite[Ch.~1, Lem.~9.3.2]{gr} (which is furthermore continuous when $c$ is compact). Given $c,c'\in\eC$, we define $\ucHom_{\eC^{\enr}}(c,c'):=\act_c^R(c')$, and let $F_{\eC}^\enr$ be adjoint to the composition
\begin{equation}
\label{eqn:F-enr-adj-def}
\act_c^R(c')\otimes F_{\eC}(c)\to\phi^*\act_c^R(c')\otimes F_{\eC}(c)\simeq F_{\eC}(\act_c^R(c')\otimes c)\to F_{\eC}(c'),
\end{equation}
where the first map is the natural $\phi^*$-semilinear map, and the final map is obtained by applying $F_{\eC}$ to the tautological map $\act_c^R(c')\otimes c\to c'$ obtained by adjunction. We leave it to the reader to verify the requisite axioms. Note that taking global sections is right-adjoint to the unit functor $\Vect_k\to\QC(X/G)$, and hence recovers the Hom-spaces in $\eC$ (and the original functor $F_{\eC}$). Similarly, taking $\Gamma^G$ and forgetting the $G$-equivariance recovers the Hom-spaces in $\eC^\dq$.

For the latter assertion, we have
\begin{align*}
&\Gamma(\preBG_{X/G,\phi}^{-n}(\eC^\enr,F_{\eC}^\enr))\\
&\simeq\bigoplus_{c_0,\ldots,c_n\in\eC^c}\big(\Gamma^G\ucHom_{\eC}(c_0,c_1)\tens{k}\cdots\tens{k}\Gamma^G\ucHom_{\eC}(c_{n-1},c_n)\tens{k}\Gamma^G\ucHom_{\eC}(c_n,F_{\eC}(c_0))\tens{k}\O(G)\big)^G,
\end{align*}
so we recover (a corrected version of\footnote{\label{fn:bg-face-maps}Note that as currently stated, the Block--Getzler complex of \cite[Def.~2.11]{benzvi} does \emph{not} in general constitute a simplicial object, as $d_ns_{n-1}$, for an appropriately defined degeneracy map $s_{n-1}$, is given by applying $F_{\eC}$, rather than the identity. To correct this error, note that the cyclic bar complex of a dg category $\eC$ is obtained by computing the (derived) tensor product
\begin{equation}
\label{eqn:hh-tensor-dg-bimod-delta}
\eC_{\Delta}\otimes_{\eC\otimes\eC^\op}{}_{F_{\eC}}\eC_{\Delta}
\end{equation}
of $(\eC,\eC)$-bimodules, i.e., functors $\eC\otimes\eC^{\op}\to\Vect_k$ (see for instance \cite[\S3]{keller} or \cite[\S2.4, \S5.1]{gorsky}). Here the ``diagonal bimodule'' $\eC_{\Delta}$ is given by the functor $c'\otimes c\mapsto\Hom_{\eC}(c,c')$, and ${}_{F_{\eC}}\eC_{\Delta}$ denotes its precomposition with $F_{\eC}\otimes\id_{\eC^\op}$. This tensor product may be computed using the usual bar resolution of $\eC_{\Delta}$: an element of the degree-$(-n)$ term of the bar resolution for $c'\otimes c\in\eC\otimes\eC^\op$ is given by
\begin{equation*}
f_{-1}\otimes\cdots\otimes f_n\in\Hom_{\eC}(c,c_0)\tens{k}\Hom_{\eC}(c_0,c_1)\tens{k}\cdots\tens{k}\Hom_{\eC}(c_n,c')
\end{equation*}
for some $c_0,\ldots,c_n\in\eC^c$, and therefore an element of the degree-$(-n)$ term after tensoring with ${}_{F_{\eC}}\eC_{\Delta}$ is given by
\begin{equation*}
f_0\otimes\cdots\otimes F_{\eC}(f_{-1})f_n\in\Hom_{\eC}(c_0,c_1)\tens{k}\cdots\tens{k}\Hom_{\eC}(c_n,F_{\eC}(c_0)).
\end{equation*}
It is then straightforward to verify that the usual face maps yield those in \eqref{eqn:prebg-face-maps}.}) the Block--Getzler complex of \cite[Def.~2.12]{benzvi}. Likewise, we have
\begin{align*}
&\Gamma(i_g^*\preBG_{X/G,\phi}^{-n}(\eC^\enr,F_{\eC}^\enr))\\
&\simeq\bigoplus_{c_0,\ldots,c_n\in\eC^c}\Hom_{\eC^\dq}(c_0^\dq,c_1^\dq)\tens{k}\cdots\tens{k}\Hom_{\eC^\dq}(c_{n-1}^\dq,c_n^\dq)\tens{k}\Hom_{\eC^\dq}(c_n^\dq,F_{\eC}^\dq(c_0^\dq)),
\end{align*}
and specializing the coaction map at $g$ shows that the face map $d_0$ is computed by $F_{\eC,g^{-1}}^\dq$ in place of $\varrho(\Gamma^GF_{\eC})$. Thus, we recover (a similarly modified version of) the specialized Block--Getzler complex of \loccit, and the result follows as in \cite[Prop.~2.13]{benzvi}.
\end{proof}

\begin{remark}
\label{rem:preBG-cpt-gens}
Note that, as for the cyclic bar complex, we usually need not consider all of $\eC^c$ when computing the pre-Block--Getzler sheaf. Rather, it suffices to restrict $c_0,\ldots,c_n$ to any set of compact objects of $\eC$ which generate under the action of $\Rep(G)$.

We henceforth omit the notation $(-)^\enr$ (and write simply $\preBG_{X/G,\phi}(\eC,F_{\eC})$ as in Definition~\ref{def:bg}) whenever the $\QC(X/G)$-module structure on $\eC$ is implicit.
\end{remark}

\begin{construction}
\label{cons:block-getzler-sheaf}
We now return to the setting of Definition~\ref{def:bg}, and lift $\preBG_{X/G,\phi}(\eC,F_{\eC})$ to a quasicoherent sheaf $\BG_{X/G,\phi}(\eC,F_{\eC})$ on $\cL_{\phi}(X/G)$ satisfying
\begin{equation}
\label{eqn:ev-push-BG-preBG}
\ev_{G,*}\BG_{X/G,\phi}(\eC,F_{\eC})\simeq\preBG_{X/G,\phi}(\eC,F_{\eC}).
\end{equation}
By \eqref{eqn:loop-fiber-prod}, this amounts to giving a homotopy between the two actions
\begin{equation*}
\O_X\tens{k}\preBG_{X/G,\phi}(\eC,F_{\eC})\rightrightarrows\preBG_{X/G,\phi}(\eC,F_{\eC})
\end{equation*}
coming from $\phi\circ\act$ and $\pr$. For each $i=0,\ldots,n$, we have ``degeneracy maps''
\begin{equation*}
s_i\colon\O_X\tens{k}\preBG_{X/G,\phi}^{-n}(\eC,F_{\eC})\to\preBG_{X/G,\phi}^{-(n+1)}(\eC,F_{\eC})
\end{equation*}
defined on each summand by
\begin{align*}
s_i(f\otimes f_0\otimes\cdots\otimes f_n\boxtimes r)&=f_0\otimes\cdots\otimes f_{i-1}\otimes f\cdot\id_{c_i}\otimes f_i\otimes\cdots\otimes f_n\boxtimes r.
\end{align*}
We claim that the collection of maps
\begin{equation}
\label{eqn:preBG-htpy-BG}
s^{-n}:=\sum_{i=0}^n(-1)^is_i\colon\O_X\tens{k}\preBG_{X/G,\phi}^{-n}(\eC,F_{\eC})\to\preBG_{X/G,\phi}^{-(n+1)}(\eC,F_{\eC})
\end{equation}
assemble into the desired homotopy; we leave this as an exercise.\footnote{The only terms of $d\circ s^{-n}+s^{-(n-1)}\circ d$ that survive are
\begin{equation*}
(d_0s_0-d_{n+1}s_n)(f\otimes f_0\otimes\cdots\otimes f_n\boxtimes r)=(\varrho(\phi^*(f))-f)\cdot(f_0\otimes\cdots\otimes f_n\boxtimes r),
\end{equation*}
using the $\phi^*$-semilinearity of $F_{\eC}$.} We refer to the resulting sheaf
\begin{equation*}
\BG_{X/G,\phi}(\eC,F_{\eC})\in\QC(\cL_{\phi}(X/G))
\end{equation*}
as the \emph{Block--Getzler sheaf} of the pair $(\eC,F_{\eC})$.
\end{construction}

\subsection{Computing the $2$-categorical class map}

\begin{atom}
\label{atm:bg-tr-res}
We now describe the setting in which the Block--Getzler sheaf computes the $2$-categorical class map. Let $\eA$ be a compactly generated rigid monoidal category, and let
\begin{equation}
\label{eqn:psi}
\Psi\colon\QC(X/G)\to\eA
\end{equation}
be a monoidal functor admitting a central structure, i.e., a factorization through the Drinfeld center of $\eA$ as in \S\ref{atm:drinf-cent-tr-adj}. Suppose that we are given a monoidal endofunctor $F_{\eA}$ of $\eA$ (which automatically preserves compact objects as in \S\ref{atm:hh-dg-cat}), equipped with an isomorphism $F_{\eA}\circ\Psi\simeq\Psi\circ\phi^*$. Then as in \eqref{eqn:tr-adj-ind-res}, we have adjoint functors
\begin{equation}
\label{eqn:tr-res-ind-bg}
\Tr(\Ind_{\Psi})\colon\QC(\cL_{\phi}(X/G))\simeq\Tr(\QC(X/G),\phi^*)\rightleftarrows\Tr(\eA,F_{\eA})\colon\Tr(\Res_{\Psi}),
\end{equation}
where the first identification follows from Proposition~\ref{prop:tr-qc-loop-ident}.

Now let $\eM$ be a compactly generated right-dualizable $\eA$-module category equipped with an $F_{\eA}$-semilinear endofunctor $F_{\eM}$ which preserves compact objects. As in \S\ref{atm:mor-L-Morita}, its $2$-categorical class is an object $[\eM,F_{\eM}]\in\Tr(\eA,F_{\eA})$. Our goal is to compute the quasi-coherent sheaf $\Tr(\Res_{\Psi})([\eM,F_{\eM}])$ on $\cL_{\phi}(X/G))$. Note that $\Psi$ gives $\QC(X/G)$-module structures on both $\eA$ and $\eM$, and that both $F_{\eA}$ and $F_{\eM}$ are canonically $\phi^*$-semilinear with respect to these structures. Thus, the pairs $(\eA,F_{\eA})$ and $(\eM,F_{\eM})$ both admit $\QC(X/G)$-enrichments by Lemma~\ref{lem:bg-hh}, and it is not hard to check that these are compatible with the $\eA$-module structure on $\eM$.
\end{atom}

\begin{remark}
\label{rem:alt-A-M-QC-enr}

Note that, in the terminology of \cite[Ch.~1, \S3.6]{gr}, the $\QC(X/G)$-enrichments are given by the relative inner $\Hom$ with respect to $\QC(X/G)$. In the above setup, both $\eA$ and $\eM$ also admit relative inner $\Hom$ spaces with respect to $\eA$, which we denote by $\uHom_{\eA}$. We may then define the $\QC(X/G)$-enrichments by $\ucHom_{\eA^\enr}(a,a'):=\Psi^R\uHom_{\eA}(a,a')$ for $a,a'\in\eA$, and similarly for $\eM$. Likewise, $F_{\eA}^\enr$ may be described by the composition
\begin{equation*}
\Psi^R\uHom_{\eA}(a,a')\to\phi^*\Psi^R\uHom_{\eA}(a,a')\to\Psi^RF(\uHom_{\eA}(a,a'))\simeq\Psi^R\uHom_{\eA}(F(a),F(a')),
\end{equation*}
where the first map is the natural $\phi^*$-semilinear map, and the second map is the usual adjunction base-change map. As usual, we omit the superscripts $(-)^{\enr}$ elsewhere in this document.
\end{remark}

\begin{atom}
We now state the main result of this section:
\end{atom}

\begin{proposition}
\label{prop:bg-tr-res}
In the setup of \S\ref{atm:bg-tr-res}, we have a canonical equivalence
\begin{equation}
\label{eqn:bg-tr-res}
\BG_{X/G,\phi}(\eM,F_{\eM})\simeq\Tr(\Res_{\Psi})([\eM,F_{\eM}]).
\end{equation}
\end{proposition}

\begin{proof}
As in \S\ref{atm:tr-M-subcat}, it suffices to give a functorial isomorphism
\begin{equation}
\label{eqn:bg-tr-res-funct-iso}
\Hom_{\cL_{\phi}(X/G)}([\cF],\BG_{X/G,\phi}(\eM,F_{\eM}))\simeq\Hom_{\cL_{\phi}(X/G)}([\cF],\Tr(\Res_{\Psi})([\eM,F_{\eM}]))
\end{equation}
for each $\cF\in\Perf(X/G)$. We begin by unwinding the left-hand side. By Proposition~\ref{prop:tr-qc-loop-ident}, the universal trace functor is given by
\begin{equation}
\label{eqn:univ-tr-qc}
[-]\simeq\ev^*\colon\QC(X/G)\to\QC(\cL_{\phi}(X/G)).
\end{equation}
Thus, adjunction and duality yield
\begin{align*}
\Hom_{\cL_{\phi}(X/G)}([\cF],\BG_{X/G,\phi}(\eM,F_{\eM}))&\simeq\Hom_{X/G}(\cF,\ev_*\BG_{X/G,\phi}(\eM,F_{\eM}))\\
&\simeq\Gamma(\cF^\vee\tens{\O_X}\ev_*\BG_{X/G,\phi}(\eM,F_{\eM}))\\
&\simeq\Gamma(\cF^\vee\tens{\O_X}\pr_*\preBG_{X/G,\phi}(\eM,F_{\eM}))
\end{align*}
by \eqref{eqn:loop-fiber-prod}. Moreover, as in the proof of Lemma~\ref{lem:bg-hh}, the functor $\uHom_{\eM}(m,-)$ is $\QC(X/G)$-linear for any $m\in\eM$. It follows by construction that
\begin{equation*}
\cF^\vee\otimes_{\O_X}\pr_*\preBG_{X/G,\phi}(\eM,F_{\eM})\simeq\pr_*\preBG_{X/G,\phi}(\eM,\Psi(\cF^\vee)\otimes F_{\eM}(-)),
\end{equation*}
where the latter functor is $F_{\eA}$-semilinear by our centrality assumption on $\Psi$, and preserves compact objects as $X/G$ is perfect. Finally, Lemma~\ref{lem:bg-hh} implies that
\begin{equation}
\label{eqn:ev-bg-hh}
\Gamma(\pr_*\preBG_{X/G,\phi}(\eM,\Psi(\cF^\vee)\otimes F_{\eM}(-)))\simeq\HH(\eM,\Psi(\cF^\vee)\otimes F_{\eM}(-)).
\end{equation}

Next, we unwind the right-hand side of \eqref{eqn:bg-tr-res-funct-iso}. By adjunction, we have
\begin{equation}
\label{eqn:bg-tr-class-ind-res-adj}
\begin{split}
\Hom_{\cL_{\phi}(X/G)}([\cF],\Tr(\Res_{\Psi})([\eM,F_{\eM}]))&\simeq\Hom_{\Tr(\eA,F_{\eA})}(\Tr(\Ind_{\Psi})([\cF]),[\eM,F_{\eM}])\\
&\simeq\Hom_{\Tr(\eA,F_{\eA})}([\Psi(\cF)],[\eM,F_{\eM}]).
\end{split}
\end{equation}
Moreover, the adjunction of \eqref{eqn:drinf-cent-tr-adj} and the identity \eqref{eqn:drinf-cent-tr-adj-2-cat-cl} yield
\begin{equation}
\label{eqn:bg-tr-class-drinf-cent-adj}
\Hom_{\Tr(\eA,F)}([\Psi(\cF)],[\eM,F_{\eM}])\simeq\Hom_{\Tr(\eA,F_{\eA})}([\eA,F_{\eA}],[\eM,\Psi(\cF^\vee)\otimes F_{\eM}(-)]).
\end{equation}
The conclusion now follows from Theorem~\ref{thm:gkrv-main} and \eqref{eqn:ev-bg-hh}.
\end{proof}

\begin{atom}
We may now state our primary application of the Block--Getzler sheaf. We begin in the general setup of \S\ref{atm:loop-space-funct}, so that
\begin{equation*}
\begin{tikzcd}
\eX\arrow{d}{p}\arrow{r}{\phi_{\eX}}&\eX\arrow{d}{p}\\
\eY\arrow{r}{\phi_{\eY}}&\eY
\end{tikzcd}
\end{equation*}
is a commutative diagram of stacks.
\end{atom}

\begin{corollary}
\label{cor:push-str-bg}
Suppose that
\begin{enumerate}[label=(\roman*)]
\item $\eX$ and $\eY$ are perfect;
\item $\eY\simeq Y/G$ for a scheme $Y$ and reductive group $G$; and
\item $\phi_{\eY}$ lifts to a $G$-equivariant self-map $\phi_Y\colon Y\to Y$.
\end{enumerate}
Then
\begin{enumerate}
\item\label{itm:push-str-dg-conv} For any $\cF\in\Perf(\eX)$, we have
\begin{equation*}
\cL p_*\ev^*\cF\simeq\BG_{Y/G,\phi_{Y}}(\QC(\eX),\phi_{\eX}^*(-)\otimes\cF).
\end{equation*}
In particular,
\begin{equation*}
\cL p_*\O_{\cL_{\phi_{\eX}}\eX}\simeq\BG_{Y/G,\phi_{Y}}(\QC(\eX),\phi_{\eX}^*).
\end{equation*}
\item\label{itm:class-str-dg-conv} Suppose moreover that the three assumptions in \S\ref{atm:tr-conv-cat-setup} hold. Then for any $\cF\in\Coh(\eX\times_{\eY}\eX)$, we have
\begin{equation*}
[\cF]^{\QC}\simeq\BG_{Y/G,\phi_Y}(\QC(\eX),\phi_{\eX}^*(-)\star\cF),
\end{equation*}
where the notation $(-)^{\QC}$ is as in \eqref{eqn:qc-class-not-cor}.
\end{enumerate}
\end{corollary}

\begin{proof}
Combine Proposition~\ref{prop:bg-tr-res} with \eqref{atm:tr-A-F-cat-end}, Corollary~\ref{cor:qc-tr-ind-res}, and Corollary~\ref{cor:reg-mod-rt-left-dual}. Note that throughout, we are taking \eqref{eqn:psi} to be the functor $\cL p^*\colon\QC(Y/G)\to\QC(\eX)$. For the central structure, note that the Drinfeld center of $\QC(\eX)$ is computed by the functor $\ev_*\colon\QC(\cL\eX)\to\QC(\eX)$ as in \cite[Cor.~5.2]{bfn}. Writing $i_{\eX}\colon\eX\to\cL\eX$ for the inclusion of the constant loops, we obtain the desired factorization $\cL p^*\simeq\ev_*\circ(i_{\eX,*}\cL p^*)$.
\end{proof}

\begin{atom}
Finally, we describe functoriality of the Block--Getzler sheaf in the pair $(\eM,F_{\eM})$. Suppose we are given another such pair $(\eN,F_{\eN})$, and a morphism $(\gamma,\theta)\colon(\eM,F_{\eM})\to(\eN,F_{\eN})$ in $(\eA,F_{\eA})\dashbmod$, i.e., an $\eA$-linear functor $\gamma\colon\eM\to\eN$ admitting an $\eA$-linear right adjoint, and a natural transformation $\theta\colon\gamma\circ F_{\eM}\Rightarrow F_{\eN}\circ\gamma$ of functors ${}_{F_{\eA}}\eM\to{}_{F_{\eA}}\eN$. As in \S\ref{atm:mor-L-Morita}, we obtain a morphism $[\gamma,\theta]\colon[\eM,F_{\eM}]\to[\eN,F_{\eN}]$ between $2$-categorical classes in $\Tr(\eA,F_{\eA})$.

On the other hand, we can construct a morphism
\begin{equation}
\label{eqn:BG-funct-map}
\BG_{X/G,\phi}(\gamma,\theta)\colon\BG_{X/G,\phi}(\eM,F_{\eM})\to\BG_{X/G,\phi}(\eN,F_{\eN})
\end{equation}
in $\QC(\cL_\phi(X/G))$ associated to the pair $(\gamma,\theta)$. Indeed, for any $m,m'\in\eM$, we have a natural map
\begin{equation*}
\gamma\colon\ucHom_{\eM}(m,m')\to\uHom_{\eN}(\gamma(m),\gamma(m'))
\end{equation*}
in $\QC(X/G)$ (constructed analogously to \eqref{eqn:F-enr-adj-def}), hence a composition
\begin{equation*}
\ucHom_{\eM}(m,F_{\eM}(m'))\xrightarrow{\gamma}\ucHom_{\eN}(\gamma(m),\gamma F_{\eM}(m'))\xrightarrow{\theta\circ-}\ucHom_{\eN}(\gamma(m),F_{\eN}\gamma(m')).
\end{equation*}
It is then not hard to check that the maps
\begin{align*}
\preBG_{X/G,\phi}^{-n}(\gamma,\theta)\colon\preBG_{X/G,\phi}^{-n}(\eM,F_{\eM})&\to\preBG_{X/G,\phi}^{-n}(\eN,F_{\eN})\\
f_0\otimes\cdots\otimes f_{n-1}\otimes f_n\boxtimes r&\mapsto\gamma(f_0)\otimes\cdots\otimes\gamma(f_{n-1})\otimes\theta\circ\gamma(f_n)\boxtimes r
\end{align*}
commute with all face maps \eqref{eqn:prebg-face-maps}, as well as the homotopy \eqref{eqn:preBG-htpy-BG}, and hence extend to a map as in \eqref{eqn:BG-funct-map}.

These two constructions are compatible under the identification of Proposition~\ref{prop:bg-tr-res}:
\end{atom}

\begin{proposition}
\label{prop:2-mor-bg-expl}
In the above setup, the diagram
\begin{equation*}
\begin{tikzcd}[column sep=1in]
\BG_{X/G,\phi}(\eM,F_{\eM})\arrow{r}{\BG_{X/G,\phi}(\gamma,\theta)}\arrow[d,"\vertsim"]&\BG_{X/G,\phi}(\eN,F_{\eN})\arrow[d,"\vertsim"]\\
\Tr(\Res_{\Psi})([\eM,F_{\eM}])\arrow{r}{\Tr(\Res_{\Psi})([\gamma,\theta])}&\Tr(\Res_{\Psi})([\eN,F_{\eN}])
\end{tikzcd}
\end{equation*}
commutes (up to a natural isomorphism).
\end{proposition}

\begin{proof}
Let $\cF\in\Perf(X/G)$. As in the proof of Proposition~\ref{prop:bg-tr-res}, it suffices to show that each of the squares
\begin{equation*}
\begin{tikzcd}[column sep=1in]
\Hom_{\cL_{\phi}(X/G)}([\cF],\BG_{X/G,\phi}(\eM,F_{\eM}))\arrow{r}{\BG_{X/G,\phi}(\gamma,\theta)\circ-}\arrow[d,"\vertsim"]&\Hom_{\cL_{\phi}(X/G)}([\cF],\BG_{X/G,\phi}(\eN,F_{\eN}))\arrow[d,"\vertsim"]\\
\HH(\eM,\Psi(\cF^{\vee})\otimes F_{\eM}(-))\arrow{r}{\HH(\gamma,\id_{\Psi(\cF^{\vee})}\otimes\theta)}&\HH(\eN,\Psi(\cF^{\vee})\otimes F_{\eN}(-))\\
\Hom_{\cL_{\phi}(X/G)}([\cF],\Tr(\Res_{\Psi})([\eM,F_{\eM}]))\arrow{r}{\Tr(\Res_{\Psi})([\gamma,\theta])\circ-}\arrow[u,"\vertsim"']&\Hom_{\cL_{\phi}(X/G)}([\cF],\Tr(\Res_{\Psi})([\eN,F_{\eN}]))\arrow[u,"\vertsim"']
\end{tikzcd}
\end{equation*}
commutes. Commutativity of the lower square is immediate from unwinding the adjunctions \eqref{eqn:bg-tr-class-ind-res-adj} and \eqref{eqn:bg-tr-class-drinf-cent-adj}. Similarly, as in \eqref{eqn:ev-bg-hh}, it suffices to show that the upper square commutes with the top row replaced by the morphism
\begin{equation*}
\Gamma(\preBG_{X/G,\phi}(\gamma,\id_{\Psi(\cF^{\vee})}\otimes\theta))\colon\Gamma(\preBG_{X/G,\phi}(\eM,\Psi(\cF^\vee)\otimes F_{\eM}(-)))\to\Gamma(\preBG_{X/G,\phi}(\eN,\Psi(\cF^\vee)\otimes F_{\eN}(-))).
\end{equation*}
As in the proofs of Lemma~\ref{lem:bg-hh} and \cite[Prop.~2.13]{benzvi}, this amounts to the analogous functoriality of the cyclic bar complex for a dg-category.\footnote{This is well-known, though we could not find a precise reference in the literature (see \cite[Thm.~5.2(a)]{keller} for the case where the endofunctors and natural transformation $\theta$ are trivial). Regardless, one can see this directly by examining the functoriality of \eqref{eqn:hh-tensor-dg-bimod-delta} and \eqref{eqn:hh-duality-data}, using the same general pattern as in Lemma~\ref{lem:tr-ind}.}
\end{proof}

\begin{atom}
Thus, the complex $\BG_{X/G,\phi}(-)$ essentially computes the functor
\begin{equation*}
\Tr(\Res_{\Psi})([-])\colon(\eA,F_{\eA})\dashbmod\to\QC(\cL_{\phi}(X/G))
\end{equation*}
composed from \eqref{eqn:2-cat-class-map} and \eqref{eqn:tr-res-ind-bg}.
\end{atom}

\begin{remark}
We expect that, when applicable, the $S^1$-equivariant structure on the Block--Getzler sheaf may be described by an explicit homotopy, analogous to Connes' boundary operator $B$ on the usual cyclic bar complex. This will not be needed in the sequel, so we leave the details to a future work.
\end{remark}

\section{The exotic t-structure}
\label{sec:exotic-t-str}

\begin{atom}
We now give an ``intrinsic'' reformulation of braid positivity using an exotic t-structure on $\eH_{P,\eS_e}^\rmcoh$ provided by work of Bezrukavnikov--Losev \cite[\S4.2]{bez-losev}. While we would like to identify $\eH_{P,\eS_e}^\rmcoh$ with the category of $\wt{Z}_e$-equivariant $\cA_{P,\eS_e}^\op$-bimodules as in \eqref{eqn:bm-functor}, this would not preserve compact objects (for instance, the regular bimodule is not perfect). Thus, we must imitate the passage from quasi-coherent to ind-coherent sheaves (which, e.g., allows the unit sheaf $\Delta_{\wt{\eN}_{P,\eS_e}/\g,*}\O_{\wt{\eN}_{P,\eS_e}}\in\eH_{P,\eS_e}^{\coh}$ to be compact), and ``renormalize'' this category of bimodules.
\end{atom}

\subsection{Generalities}

\begin{atom}
We begin by recalling the basic properties of t-structures that we shall need. Our presentation follows Preygel's treatment in \cite[\S4]{preygel-indcoh}.
\end{atom}

\begin{definition}
\label{def:t-str}
A t-structure $(\eC^{\le 0},\eC^{\ge 0})$ on a dg-category $\eC$ is
\begin{enumerate}
\item \emph{accessible} if the subcategory $\eC^{\ge 0}$ (or equivalently, $\eC^{\le 0}$) is compactly generated;
\item \emph{compatible with filtered colimits} if $\eC^{\ge 0}$ is closed under filtered colimits in $\eC$;
\item \emph{right-complete} if the inclusion functors induce an equivalence\footnote{Equivalently, by \cite[Ch.~1, Prop.~2.5.7]{gr}, the truncation functors induce an equivalence $\eC\to\lim_n\eC^{\le n}$.} $\colim_n\eC^{\le n}\to\eC$; and
\item \emph{coherent} if it is compatible with filtered colimits, right-complete, and the compact objects $\eC^{\heart,c}$ (in the classical sense) form a generating abelian subcategory of the heart $\eC^{\heart}$.
\end{enumerate}
Moreover, suppose that $(\eC^{\le 0},\eC^{\ge 0})$ is compatible with filtered colimits. We say that an object $X\in\eC$ is \emph{coherent} if
\begin{enumerate}
\item $X$ is bounded below, i.e., $X\in\eC^{\ge n}$ for some $n$; and
\item $\tau^{\ge n}X$ is a compact object of $\eC^{\ge n}$ for all $n$.
\end{enumerate}
We denote the full subcategory of coherent objects of $\eC$ by $\Coh(\eC)$, and refer to the compactly generated category $\eC_{\ren}:=\Ind(\Coh(\eC))$ as the \emph{renormalization}\footnote{In Preygel's terminology, this is the ``regularization'' of $\eC$; however, we shall not need the notion of regularity in this work.} of $\eC$. The latter carries a natural t-structure given by
\begin{equation*}
\eC_{\ren}^{\le 0}:=\Ind(\Coh(\eC)\cap\eC^{\le 0}),\qquad\qquad \eC_{\ren}^{\ge 0}:=\Ind(\Coh(\eC)\cap\eC^{\ge 0}).
\end{equation*}
\end{definition}

\begin{atom}
\label{atm:preygel-gen}
We now recall some of Preygel's general results regarding these constructions. Let $\eC$ be as in the latter part of Definition~\ref{def:t-str}. If $(\eC^{\le 0},\eC^{\ge 0})$ is coherent, then the subcategory $\Coh(\eC)$ consists of cohomologically bounded objects whose cohomologies all lie in $\Coh(\eC)\cap\eC^{\heart}=\eC^{\heart,c}$. Moreover, the natural continuous functor $\eC_{\ren}\to\eC$ obtained by Ind-extension is t-exact, and induces an equivalence on coconnective (or more generally, bounded below) objects, i.e., $\eC_{\ren}^{\ge 0}\to\eC^{\ge 0}$. In particular, the t-structure on $\eC_{\ren}$ is again coherent, and moreover accessible.\footnote{In fact, accessibility is automatic for compactly generated categories (see for instance \cite[Ch.~4, Lem.~1.2.4]{gr}).}

The prototypical example is as follows: if $\eX$ is a QCA stack, then the standard t-structure on $\QC(\eX)$ is coherent and accessible, and $\QC^!(\eX)$ identifies with the renormalization of $\QC(\eX)$ (moreover, $\QC(\eX)$ is the ``left-completion'' of $\QC^!(\eX)$, though we shall not need or define this notion).

The following lemma will allow us to compare renormalizations for different t-structures:
\end{atom}

\begin{lemma}
\label{lem:compare-t-str-coh}
Let $\eC$ be a dg-category equipped with t-structures $(\eC^{\le 0},\eC^{\ge 0})$ and $(\eC^{\le' 0},\eC^{\ge' 0})$ which are both compatible with filtered colimits. Suppose that there exist $x,y\in\Z$ such that $\eC^{\le x}\subset \eC^{\le' 0}\subset\eC^{\le y}$ (or equivalently, $\eC^{\ge y}\subset \eC^{\ge' 0}\subset\eC^{\ge x}$). Then the subcategories of coherent objects with respect to each of these t-structures are identical, i.e., $\Coh(\eC)=\Coh'(\eC)$. In particular, the renormalizations $\eC_{\ren}$ and $\eC_{\ren'}$ are canonically equivalent.
\end{lemma}

\begin{proof}
We show the inclusion $\Coh(\eC)\subset\Coh'(\eC)$; the opposite inclusion follows by symmetry. Suppose that $X\in\eC$ is coherent with respect to $(\eC^{\le 0},\eC^{\ge 0})$. Then it is clearly bounded below with respect to $(\eC^{\le' 0},\eC^{\ge' 0})$. Moreover, given $n\in\Z$, we have a diagram
\begin{equation}
\label{eqn:compart-t-str-diag}
\begin{tikzcd}[column sep=small]
\eC^{\ge'n}\arrow[dr,hook,"\iota^{\ge'n}"']\arrow[rr,hook,"i"]&&\eC^{\ge x+n}\arrow[dl,hook',"\iota^{\ge x+n}"]\\
&\eC&
\end{tikzcd}
\end{equation}
of fully faithful functors. It suffices to show that $i$ is both continuous and cocontinuous. Indeed, in this case we have $\tau^{\ge'n}X\simeq i^L\tau^{\ge x+n}X$, and the left-adjoint $i^L$ preserves compact objects. Note that the functors $\iota^{\ge'n},\iota^{\ge x+n}$ are both continuous and cocontinuous by assumption. It is now straightforward to check that $i$ admits left and right adjoints given by $\tau^{\ge'n}\iota^{\ge x+n}$ and $\iota^{\ge'n,R}\iota^{\ge x+n}$, respectively.
\end{proof}

\begin{atom}
We now note some properties of t-structures on module categories for connective algebras:
\end{atom}

\begin{lemma}
\label{lem:coh-t-str-mod}
Let $\eC$ be a symmetric monoidal dg-category equipped with an accessible t-structure $(\eC^{\le 0},\eC^{\ge 0})$. Let $A\in\Alg(\eC)$ be an algebra object, let $\Res^{A}_{\1_{\eC}}\colon A\dashmod_{\eC}\to\eC$ denote the forgetful functor, and suppose that the functor $A\otimes -\colon\eC\to\eC$ is right t-exact. Then the pair
\begin{equation}
\label{eqn:A-mod-t-str}
A\dashmod_{\eC}^{\le 0}:=(\Res^{A}_{\1_{\eC}})^{-1}(\eC^{\le 0}),\qquad\qquad A\dashmod_{\eC}^{\ge 0}:=(\Res^{A}_{\1_{\eC}})^{-1}(\eC^{\ge 0}),
\end{equation}
gives a t-structure on $A\dashmod_{\eC}$. Moreover, if the t-structure on $\eC$ is compatible with filtered colimits (resp.\ right-complete), then so is that on $A\dashmod_{\eC}$.

Finally, suppose that $A$ is a compact object of $\eC$, that the tensor product on $\eC$ preserves compact objects, and that the t-structure on $\eC$ is coherent. Then the t-structure on $A\dashmod_{\eC}$ is coherent.
\end{lemma}

\begin{proof}
The first assertion is a well-known construction (see for instance \cite[Thm.~2.1.2]{polishchuk}). Compatibility with filtered colimits is clear, as $\Res^{A}_{\1_{\eC}}$ is continuous. For right-completeness, consider the commutative square
\begin{equation*}
\begin{tikzcd}[column sep=huge]
\colim_nA\dashmod_{\eC}^{\le n}\arrow[d,"\colim_n\Res^{A}_{\1_{\eC}}"']\arrow{r}{\colim_n\iota^{\le n}}&A\dashmod_{\eC}\arrow{d}{\Res^{A}_{\1_{\eC}}}\\
\colim_n\eC^{\le n}\arrow[r,"\sim"]&\eC.
\end{tikzcd}
\end{equation*}
By our assumption on $A$, there is a functor $\colim_n\Ind_{\1_{\eC}}^{A}$ left-adjoint to $\colim_n\Res^{A}_{\1_{\eC}}$; moreover, the monad $\colim_n\Res^{A}_{\1_{\eC}}\Ind_{\1_{\eC}}^{A}$ acting on $\colim_n\eC^{\le n}$ evidently identifies with $\Res^{A}_{\1_{\eC}}\Ind_{\1_{\eC}}^{A}$ under the lower equivalence. Thus, it suffices to show that the functor $\colim_n\Res^{A}_{\1_{\eC}}$ is ``monadic'' (in the sense of \cite[Ch.~1,~Def.~3.7.5]{gr}). Indeed, the functors $\Res^{A}_{\1_{\eC}}$ and $\colim_n\iota^{\le n}$ are both continuous and conservative (as each $\iota^{\le n}$ is), so the same holds for $\colim_n\Res^{A}_{\1_{\eC}}$. The conclusion now follows from the Barr--Beck--Lurie theorem (see \cite[Thm.~4.7.0.3]{lurie-ha}).

For the final assertion, we first claim that $A\dashmod_{\eC}^{\heart}$ is compactly generated by objects of the form $\tau^{\ge 0}(A\otimes X)$ for $X\in\eC^{\heart,c}$. To see that these objects are compact, note that as in \eqref{eqn:compart-t-str-diag}, there are adjoint pairs
\begin{equation*}
\begin{tikzcd}
\eC^{\le 0}\arrow[r,shift left,"\Ind_{\1_{\eC}}^A"]&A\dashmod_{\eC}^{\le 0}\arrow[r,shift left,"\tau^{\ge 0}"]\arrow[l,shift left,"\Res_{\1_{\eC}}^A"]&A\dashmod_{\eC}^{\heart},\arrow[l,shift left,"\iota^{\ge 0}"]
\end{tikzcd}
\end{equation*}
where the lower composition is continuous and factors through $\eC^{\heart}$. To see that they generate, let $M\in A\dashmod_{\eC}^{\heart}$. The bar construction expresses $M$ as a simplicial colimit of modules of the form $A^{\otimes i}\otimes\Res_{\1_{\eC}}^AM$, all of which are connective. Applying $\tau^{\ge 0}$, we obtain an expression for $M$ as a simplicial colimit of modules of the form $\tau^{\ge 0}(A^{\otimes i}\otimes\Res_{\1_{\eC}}^AM)$, which may be computed in the abelian category $A\dashmod_{\eC}^{\heart}$. By the Dold--Kan correspondence, $M$ is quasi-isomorphic to the associated ``alternating face maps'' complex; in particular, we have an exact sequence
\begin{equation}
\label{eqn:bar-complex-coker-exact-seq}
\tau^{\ge 0}(A\otimes A\otimes\Res_{\1_{\eC}}^AM)\to\tau^{\ge 0}(A\otimes\Res_{\1_{\eC}}^AM)\to M\to 0
\end{equation}
in $A\dashmod_{\eC}^{\heart}$. Since $A\otimes-$ preserves colimits and compact objects, it suffices to exhibit each $\Res_{\1_{\eC}}^AM$ as a filtered colimit of objects of $\eC^{\heart,c}$. Such a presentation is immediate from our assumption that $\eC^\heart$ is compactly generated.

It remains to show that $A\dashmod_{\eC}^{\heart,c}$ is abelian; we need only establish closure under kernels. So let $f\colon M\to N$ be a morphism in $A\dashmod_{\eC}^{\heart,c}$. Since $A$ is compact, the functor $\Res_{\1_{\eC}}^A$ admits a continuous right adjoint which is left t-exact. Moreover, since the t-structure on $\eC$ is accessible and compatible with filtered colimits, the truncation $\tau^{\le 0}$ is continuous by \cite[Prop.~2.2.8]{chen}. Thus, the restricted functor $\Res_{\1_{\eC}}^A\colon A\dashmod_{\eC}^{\heart}\to\eC^{\heart}$ also admits a continuous right adjoint, hence preserves compact objects. In particular, $\Res_{\1_{\eC}}^A\ker(f)\simeq\ker(\Res_{\1_{\eC}}^A(f))$ is compact. Taking $\ker(f)$ in place of $M$ in the exact sequence \eqref{eqn:bar-complex-coker-exact-seq} now immediately exhibits $\ker(f)$ as compact.
\end{proof}

\subsection{Application to affine Hecke categories}

\begin{atom}
We now apply the results of the previous subsection to the affine Hecke categories $\eH_{P,\eS_e}^{\coh}$. So fix a standard parabolic $P$ and Slodowy slice $\eS_e$, and take $\eC:=\QC(\eS_e/\wt{Z}_e)$ with its standard t-structure. Then the algebra
\begin{equation*}
\cA_{P,\eS_e}\tens{\O(\eS_e)}\cA_{P,\eS_e}^\op\in\Alg(\QC(\eS_e/\wt{Z}_e))
\end{equation*}
is connective and compact, so by Lemma~\ref{lem:coh-t-str-mod} and \S\ref{atm:preygel-gen}, we have a renormalized category
\begin{equation}
\label{eqn:HSe-mod}
\eH_{P,\eS_e}^{\rmmod}:=\cA_{P,\eS_e}\tens{\O(\eS_e)}\cA_{P,\eS_e}^\op\dashmod^{\wt{Z}_e}_{\ren}.
\end{equation}
Note that the unrenormalized category $\cA_{P,\eS_e}\otimes_{\O(\eS_e)}\cA_{P,\eS_e}^\op\dashmod^{\wt{Z}_e}$ is monoidal under the tensor product of bimodules (but not rigid!). Since the algebra $\cA_{P,\eS_e}\otimes_{\O(\eS_e)}\cA_{P,\eS_e}^\op$ is eventually coconnective, any perfect module is coherent. More precisely, we have:
\end{atom}

\begin{lemma}
\begin{enumerate}[leftmargin=*]
\item\label{itm:coh-bimod} The category $\Coh(\cA_{P,\eS_e}\otimes_{\O(\eS_e)}\cA_{P,\eS_e}^\op\dashmod^{\wt{Z}_e})$ is given by cohomologically bounded complexes whose cohomology is finitely generated over $\H^0(\cA_{P,\eS_e}\otimes_{\O(\eS_e)}\cA_{P,\eS_e}^\op)$ (after forgetting the $\wt{Z}_e$-equivariance).
\item\label{itm:tens-coh-bimod} The tensor product on $\cA_{P,\eS_e}\otimes_{\O(\eS_e)}\cA_{P,\eS_e}^\op\dashmod^{\wt{Z}_e}$ preserves coherent objects. In particular, it extends to a monoidal structure on $\eH_{P,\eS_e}^{\rmmod}$.
\end{enumerate}
\end{lemma}

\begin{proof}
\eqref{itm:coh-bimod} By \S\ref{atm:preygel-gen}, it suffices to show that an object $M\in\cA_{P,\eS_e}\otimes_{\O(\eS_e)}\cA_{P,\eS_e}^\op\dashmod^{\wt{Z}_e,\heart}$ is compact if and only if it is finitely generated over $\H^0(\cA_{P,\eS_e}\otimes_{\O(\eS_e)}\cA_{P,\eS_e}^\op)$. First suppose the latter. Since $\H^0(\cA_{P,\eS_e}\otimes_{\O(\eS_e)}\cA_{P,\eS_e}^\op)$ is Noetherian (as it is a finite $\O(\eS_e)$-algebra), finite generation is equivalent to finite presentation. It follows that $\H^0\uHom(M,-)$ commutes with (classical) filtered colimits, where we have let $\uHom$ denote the $\Rep(\wt{Z}_e)$-internal $\Hom$. Thus, $\H^0\Hom(M,-)\cong\H^0\uHom(M,-)^{\wt{Z}_e}$ commutes with filtered colimits, as desired.

Conversely, suppose that $M$ is compact. As in the non-equivariant case, we may write $M$ as a direct limit in $\cA_{P,\eS_e}\otimes_{\O(\eS_e)}\cA_{P,\eS_e}^\op\dashmod^{\wt{Z}_e,\heart}$ of its finitely generated submodules (since its underlying vector space decomposes as a direct sum of $\wt{Z}_e$-isotypic components). By compactness, the identity map $\id_M$ factors through some such submodule, so $M$ is a direct summand of a finitely generated module, hence finitely generated.

\eqref{itm:tens-coh-bimod} Since the algebra $\cA_{P,\eS_e}$ has finite homological dimension, the tensor product $\otimes_{\cA_{P,\eS_e}}$ preserves cohomological boundedness. It therefore suffices to show that the tensor product of any finitely generated $\H^0(\cA_{P,\eS_e}\otimes_{\O(\eS_e)}\cA_{P,\eS_e}^\op)$-modules $M,N$ (in degree $0$) has finitely generated cohomology. Since $\cA_{P,\eS_e}$ is Noetherian, and $M,N$ are in particular finitely generated over $\cA_{P,\eS_e}$, we may resolve $N$ by a complex of finitely generated free $\cA_{P,\eS_e}$-modules. It follows that the cohomology of $M\otimes_{\cA_{P,\eS_e}}N$ is also finitely generated over $\cA_{P,\eS_e}$, hence over $\H^0(\cA_{P,\eS_e}\otimes_{\O(\eS_e)}\cA_{P,\eS_e}^\op)$, as desired.
\end{proof}

\begin{atom}
Altogether, we obtain:
\end{atom}

\begin{proposition}
\label{prop:mod-equiv-left-t-exact}
The functor
\begin{equation}
\label{eqn:coh-mod-functor}
(-)_{\rmmod}:=\Hom_{\wt{\eN}_{P,\eS_e}\fibprod{\eS_e}\wt{\eN}_{P,\eS_e}}(\cE^\vee_{P,\eS_e}\boxtimes\cE_{P,\eS_e},-)\colon\eH_{P,\eS_e}^\coh\to\eH_{P,\eS_e}^{\rmmod}
\end{equation}
is a left t-exact equivalence of rigid monoidal categories. Moreover, it is compatible with the equivalence \eqref{eqn:bm-functor} on right module categories, and the analogous equivalence on left module categories.
\end{proposition}

\begin{proof}
By \cite[Ch.~1, Prop.~8.5.4; Ch.~3, Prop.~3.5.3]{gr} and \eqref{eqn:bm-functor}, we have a commutative diagram
\begin{equation*}
\begin{tikzcd}[column sep=12em]
\QC(\wt{\eN}_{P,\eS_e}/\wt{Z}_e)\tens{\QC(\eS_e/\wt{Z}_e)}\QC(\wt{\eN}_{P,\eS_e}/\wt{Z}_e)\arrow[d,"\boxtimes","\vertsim"']\arrow[r,"{\Hom(\cE^\vee_{P,\eS_e},-)\otimes\Hom(\cE_{P,\eS_e},-)}","\sim"']&\cA_{P,\eS_e}\dashmod^{\wt{Z}_e}\tens{\QC(\eS_e/\wt{Z}_e)}\cA_{P,\eS_e}^\op\dashmod^{\wt{Z}_e}\arrow[d,"\boxtimes","\vertsim"']\\
\QC(\wt{\eN}_{P,\eS_e}\fibprod{\eS_e}\wt{\eN}_{P,\eS_e}/\wt{Z}_e)\arrow{r}{{\Hom(\cE^\vee_{P,\eS_e}\boxtimes\cE_{P,\eS_e},-)}}&\cA_{P,\eS_e}\tens{\O(\eS_e)}\cA_{P,\eS_e}^\op\dashmod^{\wt{Z}_e},
\end{tikzcd}
\end{equation*}
whose bottom row is therefore an equivalence. Renormalizing the left- and right-hand sides of this equivalence yields $\eH_{P,\eS_e}^\coh$ and $\eH_{P,\eS_e}^{\rmmod}$, respectively, so it suffices to show that the standard t-structures on either side satisfy the hypotheses of Lemma~\ref{lem:compare-t-str-coh}.

First let $\cF\in\eH_{P,\eS_e}^{\coh,\le 0}$. By \cite[Thm.~1.4.2]{dg-fin}, there exists a fixed $y\in\Z$ (depending only on $\wt{\eN}_{P,\eS_e}\times_{\eS_e}\wt{\eN}_{P,\eS_e}/\wt{Z}_e$, and using the fact that it is QCA) such that the global sections functor has cohomological amplitude $\le y$. Since the sheaf $\cHom(\cE^\vee_{P,\eS_e}\boxtimes\cE_{P,\eS_e},\cF)$ is connective (as $\cE^\vee_{P,\eS_e}\boxtimes\cE_{P,\eS_e}$ is a vector bundle), its global sections lie in cohomological degrees $\le y$, which gives one inclusion. Now let $\cF\in\eH_{P,\eS_e}^{\coh,\ge 0}$. Since the sheaf $\cHom(\cE^\vee_{P,\eS_e}\boxtimes\cE_{P,\eS_e},\cF)$ is coconnective, its global sections are as well, and we obtain the other inclusion (with $x=0$). In particular, the functor $(-)_{\rmmod}$ is left t-exact.

Finally, monoidality and compatibility with the module structures on \eqref{eqn:bm-functor} hold as in \cite[Lem.~4.3]{bez-losev} (and are straightforward exercises).
\end{proof}

\begin{atom}
\label{atm:exotic-t-str-aff-hecke}
We refer to the t-structure on $\eH_{P,\eS_e}^\coh$ transported from that on $\eH_{P,\eS_e}^{\rmmod}$ as the \emph{exotic t-structure}, as for \eqref{eqn:bm-functor}. We may at last state our ``intrinsic reformulation'' of braid positivity:
\end{atom}

\begin{lemma}
Right convolution by $\cF\in\eH_{P,\eS_e}^\coh$ on $\QC(\wt{\eN}_{P,\eS_e}/\wt{Z}_e)$ is right t-exact with respect to the exotic t-structure if and only if the image $\cF_{\rmmod}\in\eH_{P,\eS_e}^{\rmmod}$ under \eqref{eqn:coh-mod-functor} is connective, i.e., $\cF$ is connective with respect to the exotic t-structure on $\eH_{P,\eS_e}^\coh$.
\end{lemma}

\begin{proof}
It suffices to show that the convolution $\cE_{P,\eS_e}\star\cF$ is connective, and this object is exactly given by the $\cA_{P,\eS_e}^\op$-module structure on $\cF_{\rmmod}$.
\end{proof}

\begin{atom}
\label{atm:iSe-t-ex-pullback}
In particular, for $a\in B^\ext_+$, the sheaf $\cK_a$ of \S\ref{atm:braid-action} is connective for the exotic t-structure, and similarly for the sheaves $\Delta_{\wt{\eN}_{P}/\g,*}\O_{\wt{\eN}_{P}}(\lambda)$ with $\lambda\in X^*(P)^+$ (see \S\ref{atm:braid-positivity}).

Finally, we denote by $i_{\eS_e}^*\colon\eH^{\rmmod}_P\to\eH_{P,\eS_e}^{\rmmod}$ the extension of scalars $-\otimes_{\O(\g)}\O(\eS_e)$. This functor is evidently intertwined with $i_{\eS_e}^*\colon\eH^{\coh}_P\to\eH_{P,\eS_e}^\coh$ under Proposition~\ref{prop:mod-equiv-left-t-exact}, justifying our duplicate notation. Note that by Lemma~\ref{lem:LiSe-event-coconn}, these functors are both t-exact.
\end{atom}

\section{Bounding the universal trace functor}
\label{sec:comp-coh-spr}

\begin{atom}
In this section, we prove our main result, which uses the exotic t-structure of \S\ref{sec:exotic-t-str} to give conditions for the universal trace of a compact object of a partial affine Hecke category to be either connective or coconnective. We then discuss several applications, such as to the sheaves $\cK_a$ giving the affine braid group action on $\QC(\wt{\eN}/\wt{G})$ and the twisted partial coherent Springer sheaves. In particular, we show that the coherent Springer sheaf lies in cohomological degree $0$, and has a summand corresponding to each parabolic subgroup away from roots of unity.

In brief, our strategy is to use the Block--Getzler sheaf from \S\ref{sec:bg-sheaf} to compute these traces, and their restrictions to any Slodowy slice, in terms of the noncommutative partial Springer resolution. Base-changing to $\B Z_e^\cov$ for each nilpotent $e$ (see \S\ref{sec:Ze-cov}), we may then use the Koszul resolution of Appendix~\ref{sec:koszul-res} to obtain cohomological bounds.
\end{atom}

\subsection{The main result}

\begin{atom}
We begin in the ``mixed'' or ``quantum'' setting. The precise statement is as follows; its proof will occupy most of the remainder of this subsection.
\end{atom}

\begin{theorem}
\label{thm:Baff-t-str}
Let $P$ be a standard parabolic subgroup.
\begin{enumerate}
\item\label{itm:exot-rt-t-class-conn} For any Slodowy slice $\eS_e$, the universal trace functor
\begin{equation}
\label{eqn:univ-tr-Se-thm}
[-]\colon\eH_{P,\eS_e}^\coh\to\QC^!(\cL(\wh{\eS}_{e,\eN_P}/\wt{Z}_e))
\end{equation}
is right t-exact with respect to the exotic t-structure on $\eH_{P,\eS_e}^\coh$ and the standard t-structure on $\QC^!(\cL(\wh{\eS}_{e,\eN_P}/\wt{Z}_e))$.

\item\label{itm:exot-rt-t-class-coconn} Let $\cF$ be a compact object of $\eH_P^\coh$, and suppose that its right monoidal dual $\cF^{\vee,R}$ is connective for the exotic t-structure.\footnote{Note that by Remark~\ref{rem:conv-cat-pivotal}, we may equivalently consider the left monoidal dual $\cF^{\vee,L}$ in place of $\cF^{\vee,R}$.} Then $[\cF]$ is coconnective for the standard t-structure. In particular, for $e=0$, the functor \eqref{eqn:univ-tr-Se-thm} has cohomological amplitude in $[-\dim\wt{\eN}_P,0]$.
\end{enumerate}
\end{theorem}

\begin{proof}

\eqref{itm:exot-rt-t-class-conn} It suffices to establish right t-exactness for compact objects of $\eH_{P,\eS_e}^\coh$ (see for instance \cite[Ch.~4, Lem.~1.2.4(2)]{gr}). So let $\cF\in\eH_{P,\eS_e}^{\coh,c}$ be connective for the exotic t-structure. By Corollary~\ref{cor:reg-mod-rt-left-dual}, we reduce to showing that
\begin{equation}
\label{eqn:F-QC-class-funct}
[\cF]^{\QC}\simeq[\QC(\wt{\eN}_{P,\eS_e}/\wt{Z}_e),-\star\cF]\in\QC(\cL(\eS_e/\wt{Z}_e))
\end{equation}
is connective. By Corollary~\ref{cor:push-str-bg}, we have
\begin{equation*}
[\cF]^{\QC}\simeq\BG_{\eS_e/\wt{Z}_e}(\QC(\wt{\eN}_{P,\eS_e}/\wt{Z}_e),-\star\cF),
\end{equation*}
so by \eqref{eqn:ev-push-BG-preBG}, it suffices to show that $\preBG_{\g/\wt{G}}(\QC(\wt{\eN}_{P,\eS_e}/\wt{Z}_e),-\star\cF)$ is connective. As in Remark~\ref{rem:preBG-cpt-gens}, we may compute this pre-Block--Getzler sheaf using only the compact generator $\cE_{P,\eS_e}$ (i.e., the regular right $\cA_{P,\eS_e}$-module under the equivalence \eqref{eqn:bm-functor}), which yields a complex of the form
\begin{equation}
\label{eqn:coh-spr-shf-complex-g}
\cdots\to\cA_{P,\eS_e}\tens{k}\cA_{P,\eS_e}\tens{k}\cF_{\rmmod}\boxtimes\O_{\wt{Z}_e}\xrightarrow{d_0-d_1+d_2}\cA_{P,\eS_e}\tens{k}\cF_{\rmmod}\boxtimes\O_{\wt{Z}_e}\xrightarrow{d_0-d_1}\cF_{\rmmod}\boxtimes\O_{\wt{Z}_e}.
\end{equation}
Indeed, we have
\begin{equation*}
\cF_{\rmmod}\simeq\Hom_{\wt{\eN}_{P,\eS_e}}(\cE_{P,\eS_e},\cE_{P,\eS_e}\star\cF)
\end{equation*}
by definition. The conclusion is now immediate from connectivity of $\cF_{\rmmod}$.

\eqref{itm:exot-rt-t-class-coconn} Now let $\cF\in\eH_P^{\coh,c}$, and suppose that $\cF^{\vee,R}$ is connective for the exotic t-structure. As in \eqref{eqn:F-QC-class-funct}, it suffices to show that $[\cF]^{\QC}$ is coconnective. The proof procedes in several steps.

\begin{atom}
We begin by reducing to a local cohomology calculation on each nilpotent orbit. Define a stratification of $\g/G$ as follows: choose a total order
\begin{equation*}
\{0\}=\bO_0\le\bO_1\le\cdots\le\bO_m=\bO_{\reg}
\end{equation*}
on the set of nilpotent orbits of $\g$ refining the usual partial order of \S\ref{atm:springer-theory}. For each $r=0,\ldots,m+1$, set $\g_r:=\g\setminus\bigcup_{0\le r'<r}\bO_{r'}$, so that
\begin{equation*}
\g=\g_0\supseteq\g_1\supseteq\cdots\supseteq\g_{m+1}=\g\setminus\eN,
\end{equation*}
and let
\begin{equation*}
\begin{tikzcd}[column sep=large]
\bO_r/\wt{G}\arrow[r,hook,"i_r"]&\g_r/\wt{G}\arrow[d,hook,"j_r"]\arrow[r,hookleftarrow,"j^{r+1}"]&\g_{r+1}/\wt{G}\\
&\g/\wt{G}.&
\end{tikzcd}
\end{equation*}
be the closed and open inclusions, respectively. These remain closed and open after applying the loop space functor by Lemma~\ref{lem:open-closed-loops}; moreover, the top row remains complementary. We therefore have a distinguished triangle
\begin{equation*}
\Gamma_{\cL(\bO_r/\wt{G})}\to\id_{\QC(\cL(\g_r/\wt{G}))}\to\cL j^{r+1}_*\cL j^{r+1,*}
\end{equation*}
for each $r$, where the left-most functor is local cohomology with support in $\cL(\bO_r/\wt{G})$.\footnote{\label{fn:loc-cohom-equiv}We comment briefly on our conventions for the local cohomology functor. In \cite{dg-indschemes}, this functor was only defined for Zariski-closed subsets of derived schemes, rather than for arbitrary closed immersions of stacks. However, we will only be concerned with closed immersions of the form described in the hypotheses of Lemma~\ref{lem:open-closed-loops}; in this setup, we write $\Gamma_Z$ for $\Gamma_{Z^{\cl}}$ (in fact, $Z$ will always be classical in our applications), and the functor $\Gamma_Z\colon\QC(X)\to\QC(X)$ clearly upgrades to a functor $\Gamma_{Z/G}\colon\QC(X/G)\to\QC(X/G)$ using the same distinguished triangle. Regardless, it suffices to check any claim about t-structures after forgetting equivariance, so alternatively, we may simply take local cohomology with respect to the derived scheme underlying each loop space as in \eqref{eqn:loop-fiber-prod} (this does not alter any of the arguments).} Since $\cL j_{m+1}^*[\cF]^{\QC}\simeq 0$ and each $\cL j^{r+1}_*$ is left t-exact, it suffices to show that
\begin{equation}
\label{eqn:loc-cohom-res-pos-degs}
\Gamma_{\cL(\bO_r/\wt{G})}\cL j_r^*[\cF]^{\QC}\in\QC_{\cL(\bO_r/\wt{G})}(\cL(\g_r/\wt{G}))^{\ge 0}
\end{equation}
for each $r$.
\end{atom}

\begin{atom}
\label{atm:step-slod-cov-red}
Next, we reduce \eqref{eqn:loc-cohom-res-pos-degs} to a computation on a Slodowy slice. Let $e\in\bO_r$ and let $\eS_e$ be a Slodowy slice at $e$. Let $Z_e^\cov\twoheadrightarrow Z_e$ be as in Theorem~\ref{thm:Ze-cov} (with a view towards applying Corollary~\ref{cor:bm-cov-equiv}). Consider the diagram\footnote{Our notation conflicts slightly with that of \eqref{eqn:slodowy-pullback}, but this should not pose any confusion.}
\begin{equation*}
\begin{tikzcd}
e/\wt{Z}_e^\cov\arrow{r}{i_e}\arrow{d}{i_{\cov}}&\eS_e/\wt{Z}_e^\cov\arrow{d}{i_{\eS_e}}\\
\bO_r/\wt{G}\arrow{r}{i_r}&\g_r/\wt{G}.
\end{tikzcd}
\end{equation*}
Since $\eS_e$ and $\bO_r$ intersect transversally at $e$, it is (derived) cartesian, so applying the loop space functor gives a pullback square
\begin{equation*}
\begin{tikzcd}
\cL(e/\wt{Z}_e^\cov)\arrow{r}{\cL i_e}\arrow{d}{\cL i_{\cov}}&\cL(\eS_e/\wt{Z}_e^\cov)\arrow{d}{\cL i_{\eS_e}}\\
\cL(\bO_r/\wt{G})\arrow{r}{\cL i_r}&\cL(\g_r/\wt{G}).
\end{tikzcd}
\end{equation*}
In particular, the relative cotangent complexes for the horizontal maps satisfy
\begin{equation}
\label{eqn:rel-cotan-Se-pullback}
\cL i_{\cov}^*\bL_{\cL i_r}\simeq\bL_{\cL i_e}.
\end{equation}
Observe that $\bL_{\cL i_r}$ is perfect: indeed, $\cL(\bO_r/\wt{G})\simeq \wt{G}^e/\wt{G}^e$ is smooth, so it suffices to show that $\bL_{\cL i_r}$ is coherent; the latter follows from the exact triangle
\begin{equation*}
\cL i_r^*\bL_{\cL(\g_r/\wt{G})}\to\bL_{\cL(\bO_r/\wt{G})}\to\bL_{\cL i_r}
\end{equation*}
and Lemma~\ref{lem:loops-props}\eqref{itm:quasi-smooth-loops}. Thus, by \cite[Lem.~5.2]{halpern2}, the sheaf $\Gamma_{\cL(\bO_r/\wt{G})}\cL j_r^*[\cF]^{\QC}$ has a bounded-below increasing filtration whose associated graded is equivalent to
\begin{equation*}
\cL i_{r,*}\big(\Sym\bL_{\cL i_r}^\vee[1]\otimes\cL i_r^!\O_{\cL(\g_r/\wt{G})}\otimes\cL i_r^*\cL j_r^*[\cF]^{\QC}\big).
\end{equation*}
Since $\cL i_{r,*}$ is left t-exact, it suffices to show that
\begin{equation}
\label{eqn:loc-cohom-pullback-assoc-gr}
\Sym\bL_{\cL i_r}^\vee[1]\otimes\cL i_r^!\O_{\cL(\g_r/\wt{G})}\otimes\cL i_r^*\cL j_r^*[\cF]^{\QC}\in\QC(\cL(\bO_r/\wt{G}))^{\ge 0}.
\end{equation}

We claim that \eqref{eqn:loc-cohom-pullback-assoc-gr} is a direct sum of sheaves lying in $\Coh^-(\cL(\bO_r/\wt{G}))$. Indeed, it suffices to show that $\cL i_r^!\O_{\cL(\g_r/\wt{G})}$ and $\cL i_r^*\cL j_r^*[\cF]^{\QC}$ both lie in $\Coh^-(\cL(\bO_r/\wt{G}))$. For the former, we have
\begin{equation}
\label{eqn:!-pullback-det-L}
i_r^!\O_{\cL(\g_r/\wt{G})}\simeq\det\bL_{\cL i_r}[\rank\bL_{\cL i_r}]
\end{equation}
by \cite[Lem.~3.8]{halpern2} and Lemma~\ref{lem:loops-props}\eqref{itm:quasi-smooth-loops}. The latter is clear as $[\cF]^{\QC}$ is coherent. Thus, the following lemma reduces us to showing that
\begin{equation}
\label{eqn:i-red-pullback-assoc-gr}
\cL i_{\cov}^*\big(\Sym\bL_{\cL i_r}^\vee[1]\otimes\cL i_r^!\O_{\cL(\g_r/\wt{G})}\otimes\cL i_r^*\cL j_r^*[\cF]^{\QC}\big)\in\QC(\cL(e/\wt{Z}_e))^{\ge 0}.
\end{equation}
\end{atom}

\begin{lemma}
\label{lem:Licov-pullback-coconn}
Let $\cG\in\Coh^-(\wt{G}^e/\wt{G}^e)$, and suppose that $\cL i_{\cov}^*\cG\in\Coh^-(\wt{Z}_e^\cov/\wt{Z}_e^\cov)$ is coconnective. Then $\cG$ is coconnective.
\end{lemma}

\begin{proof}
We have a factorization
\begin{equation*}
\wt{Z}_e^\cov/\wt{Z}_e^\cov\to\wt{Z}_e/\wt{Z}_e\to\wt{G}^e/\wt{G}^e
\end{equation*}
of $\cL i_{\cov}$; since the first map is faithfully flat, it suffices to establish the claim with $\cL i_{\cov}$ replaced by the latter map.

Recall from \S\ref{atm:jacobson-morozov} that $\wt{G}^e\simeq\rmR_uG^e\rtimes\wt{Z}_e$, where $\rmR_uG^e$ denotes the unipotent radical of $G^e$, and the factor of $\bG_m$ in $\wt{Z}_e$ acts on $\rmR_uG^e$ with strictly positive weights. It clearly suffices to verify the claim after forgetting all but the $\bG_m$-equivariance; choosing an isomorphism $\rmR_uG^e/\bG_m\simeq\A^n/\bG_m$ for some $n\ge 0$, we reduce to the following claim:
\begin{enumerate}[leftmargin=*]
\item[($*$)] Let $n\ge 0$, and suppose we have an attracting action of $\bG_m$ on $\A^n$. Let $\cG\in\Coh^-(\wt{Z}_e\times\A^n/\bG_m)$, let $i_0\colon\{0\}\hookrightarrow\A^n$ denote the inclusion, and suppose that $(\id\times i_0)^*\cG\in\Coh^-(\wt{Z}_e\times\B\bG_m)$ is coconnective. Then $\cG$ is coconnective.
\end{enumerate}
We proceed by induction on $n$. The claim is trivial for $n=0$, so suppose $n\ge 1$ and the claim holds for $n-1$. Choose a factorization
\begin{equation*}
\{0\}\xhookrightarrow{i'_0}\A^{n-1}\xhookrightarrow{i}\A^n
\end{equation*}
of $i_0$ through some $\bG_m$-stable hyperplane. Since $(\id\times i_0)^*\cG\simeq(\id\times i'_0)^*((\id\times i)^*\cG)$, the inductive hypothesis gives
\begin{equation*}
(\id\times i)^*\cG\in\Coh(\wt{Z}_e\times\A^{n-1}/\bG_m)^{\ge 0}.
\end{equation*}
The morphism $\id\times i$ has Tor-dimension $\le 1$, so the convergent spectral sequence
\begin{equation*}
E_2^{s,t}=\H^s((\id\times i)^*\H^t(\cG))\Longrightarrow\H^{s+t}((\id\times i)^*\cG)
\end{equation*}
degenerates. Thus, it suffices to show that if $\H^t(\cG)$ is nonzero, then so is $\H^0((\id\times i)^*\H^t(\cG))$; this follows from Nakayama's lemma and $\bG_m$-equivariance of $\H^t(\cG)$.
\end{proof}

\begin{atom}
To compute \eqref{eqn:i-red-pullback-assoc-gr}, we first observe that
\begin{equation*}
\cL i_{\cov}^*\cL i_r^!\O_{\cL(\g_r/\wt{G})}\simeq\det\bL_{\cL i_e}[\rank\bL_{\cL i_e}]\simeq\cL i_e^!\O_{\cL(\eS_e/\wt{Z}_e^\cov)}
\end{equation*}
by \eqref{eqn:rel-cotan-Se-pullback}, \eqref{eqn:!-pullback-det-L}, and the corresponding statement for $\cL i_e$. It then follows from \eqref{eqn:rel-cotan-Se-pullback} and perfectness of the cotangent complexes that
\begin{equation*}
\cL i_{\cov}^*\big(\Sym\bL_{\cL i_r}^\vee[1]\otimes\cL i_r^!\O_{\cL(\g_r/\wt{G})}\otimes\cL i_r^*\cL j_r^*[\cF]^{\QC}\big)\simeq\Sym\bL_{\cL i_e}^\vee[1]\otimes\cL i_e^!\O_{\cL(\eS_e/\wt{Z}_e)}\otimes\cL i_e^*\cL i_{\eS_e}^*[\cF]^{\QC}.
\end{equation*}
Since $\cL i_e$ is affine, it is equivalent to show that
\begin{equation*}
\cL i_{e,*}\big(\Sym\bL_{\cL i_e}^\vee[1]\otimes\cL i_e^!\O_{\cL(\eS_e/\wt{Z}_e)}\otimes\cL i_e^*\cL i_{\eS_e}^*[\cF]^{\QC}\big)\in\QC(\cL(\eS_e/\wt{Z}_e)^{\ge 0}.
\end{equation*}
But by a further application of \cite[Lem.~5.2]{halpern2}, this is equivalent to the associated graded of a bounded below increasing filtration on $\Gamma_{\cL(e/\wt{Z}_e^\cov)}\cL i_{\eS_e}^*[\cF]^{\QC}$, and hence it suffices to show that the latter is coconnective.

We make one further reduction: consider the pullback square
\begin{equation*}
\begin{tikzcd}[column sep=large]
\cL((\eS_e\setminus e)/\wt{Z}_e^\cov)\arrow[r,hook,"\cL j_e"]\arrow{d}{\overset{\circ}{\ev}_{\wt{Z}_e^\cov}}&\cL(\eS_e/\wt{Z}_e^\cov)\arrow{d}{\ev_{\wt{Z}_e^\cov}}\\
((\eS_e\setminus e)\times\wt{Z}_e^\cov)/\wt{Z}_e^\cov\arrow[r,hook,"j_e\times\id"]&(\eS_e\times\wt{Z}_e^\cov)/\wt{Z}_e^\cov
\end{tikzcd}
\end{equation*}
as in \eqref{eqn:open-loops-fiber-prod}. It suffices to show that $\ev_{\wt{Z}_e^\cov,*}\Gamma_{\cL(e/\wt{Z}_e^\cov)}\cL i_{\eS_e}^*[\cF]^{\QC}$ is coconnective. The exact triangle
\begin{equation*}
\ev_{\wt{Z}_e^\cov,*}\Gamma_{\cL(e/\wt{Z}_e^\cov)}\to\ev_{\wt{Z}_e^\cov,*}\to\ev_{\wt{Z}_e^\cov,*}\cL j_{e,*}\cL j_e^*
\end{equation*}
and the equivalences
\begin{equation*}
\ev_{\wt{Z}_e^\cov,*}\cL j_{e,*}\cL j_e^*\simeq(j_e\times\id)_*\overset{\circ}{\ev}_{\wt{Z}_e^\cov,*}\cL j_e^*\simeq(j_e\times\id)_*(j_e\times\id)^*\ev_{\wt{Z}_e^\cov,*}
\end{equation*}
then imply that
\begin{equation}
\label{eqn:loc-cohom-ev-slodowy-S}
\ev_{\wt{Z}_e^\cov,*}\Gamma_{\cL(e/\wt{Z}_e^\cov)}\cL i_{\eS_e}^*[\cF]^{\QC}\simeq\Gamma_{\{e\}\times\wt{Z}_e^\cov/\wt{Z}_e^\cov}\ev_{\wt{Z}_e^\cov,*}\cL i_{\eS_e}^*[\cF]^{\QC},
\end{equation}
so we reduce to showing that the latter is coconnective.
\end{atom}

\begin{atom}
\label{atm:comp-iSe-res-S-BG-ASe}
We now use our Block--Getzler sheaf to compute \eqref{eqn:loc-cohom-ev-slodowy-S}. By Corollary~\ref{cor:tr-iSe-star-loop-pullback} (or rather, its obvious analog for $Z_e^\cov$) and Proposition~\ref{prop:bg-tr-res}, we have
\begin{equation}
\label{eqn:coh-spr-cov-bg}
\begin{split}
\ev_{\wt{Z}_e^\cov,*}\cL i_{\eS_e}^*[\cF]^{\QC}&\simeq\ev_{\wt{Z}_e^\cov,*}[\QC(\wt{\eN}_{P,\eS_e}/\wt{Z}_e^\cov),-\star i_{\eS_e}^*\cF]\\
&\simeq\ev_{\wt{Z}_e^\cov,*}\BG_{\eS_e/\wt{Z}_e^\cov}(\QC(\wt{\eN}_{P,\eS_e}/\wt{Z}_e^\cov),-\star i_{\eS_e}^*\cF)\\
&\simeq\preBG_{\eS_e/\wt{Z}_e^\cov}(\QC(\wt{\eN}_{P,\eS_e}/\wt{Z}_e^\cov),-\star i_{\eS_e}^*\cF).
\end{split}
\end{equation}
Note that by Corollary~\ref{cor:bm-cov-equiv} and Remark~\ref{rem:preBG-cpt-gens}, we may compute this pre-Block--Getzler sheaf using only the compact generator $\cE_{P,\eS_e}^{\cov}$ (i.e., the regular right $\cA_{P,\eS_e}^{\cov}$-module), which yields a complex $C^\bullet$ of the form
\begin{equation}
\label{eqn:coh-spr-shf-complex}
\cdots\to\cA_{P,\eS_e}^\cov\tens{k}\cA_{P,\eS_e}^\cov\tens{k}i_{\eS_e}^*\cF_{\rmmod}\boxtimes\O_{\wt{Z}_e^\cov}\xrightarrow{d_0-d_1+d_2}\cA_{P,\eS_e}^\cov\tens{k}i_{\eS_e}^*\cF_{\rmmod}\boxtimes\O_{\wt{Z}_e^\cov}\xrightarrow{d_0-d_1}i_{\eS_e}^*\cF_{\rmmod}\boxtimes\O_{\wt{Z}_e^\cov}
\end{equation}
as in \eqref{eqn:coh-spr-shf-complex-g}. Here we misuse notation slightly by writing $(-)_{\rmmod}$ for the equivalence given by $\cE_{P,\eS_e}^{\cov}$ in place of $\cE_{P,\eS_e}$; however, the notation $i_{\eS_e}^*\cF_{\rmmod}$ is then unambiguous, as $i_{\eS_e}^*$ and $(-)_{\rmmod}$ commute. Our goal now is to show that $\Gamma_{\{e\}\times\wt{Z}_e^\cov/\wt{Z}_e^\cov}C^\bullet$ is coconnective; we henceforth explicitly forget all $\wt{Z}_e^\cov$-equivariance (see footnote~\ref{fn:loc-cohom-equiv}).
\end{atom}

\begin{atom}
We begin by using the Koszul property of $\cA_{P,\eS_e}^\cov$ to replace \eqref{eqn:coh-spr-shf-complex} with a quasi-isomorphic bounded complex. Observe that we may write
\begin{equation*}
C^\bullet\simeq\cA_{P,\eS_e}^\cov\tens{{\cA_{P,\eS_e}^{\cov}\tens{k}\cA_{P,\eS_e}^{\cov,\op}}}i_{\eS_e}^*\cF_{\rmmod}\tens{k}\O(\wt{Z}_e^\cov),
\end{equation*}
where the $\cA_{P,\eS_e}^\cov$-action on $i_{\eS_e}^*\cF_{\rmmod}\otimes_k\O(\wt{Z}_e^\cov)$ is given by the left-multiplication on $i_{\eS_e}^*\cF_{\rmmod}$, and the $\cA_{P,\eS_e}^{\cov,\op}$-action is given by the algebra homomorphism $\varrho\colon\cA_{P,\eS_e}^\cov\to\cA_{P,\eS_e}^\cov\otimes_k\O(\wt{Z}_e^\cov)$ (i.e., the coaction map) and
 right-multiplication. Thus, Proposition~\ref{prop:kos-acyclic} and \eqref{eqn:ASe-cov} give
\begin{equation}
\label{eqn:ASe-cov-kos-tensor}
C^\bullet\simeq\Kos_{\cA_{P,\eS_e}^\cov}^{\le 0}\tens{\cA_{P,\eS_e}^\cov\tens{k}\cA_{P,\eS_e}^{\cov,\op}}i_{\eS_e}^*\cF_{\rmmod}\tens{k}\O(\wt{Z}_e^\cov).
\end{equation}
Moreover, Corollary~\ref{cor:kos-dual-ext} and \S\ref{atm:tilt-bund-simp-mods} show that the projective bimodule resolution $\Kos_{\cA_{P,\eS_e}^\cov}^{\le 0}$ has length exactly $\dim\wt{\eN}_{P,\eS_e}$.

Note that $\Gamma_{\{e\}\times\wt{Z}_e^\cov}$ has cohomological amplitude $[0,\dim \eS_e]$ (using the standard ``local Koszul complex'' on a choice of coordinate functions of the affine space $\eS_e$). Thus, the spectral sequence of the double complex obtained by applying $\Gamma_{\{e\}\times\wt{Z}_e^\cov}$ term-by-term to \eqref{eqn:ASe-cov-kos-tensor} reduces us to showing that
\begin{equation}
\label{eqn:loc-cohom-kos-n-tens}
\Gamma_{\{e\}\times\wt{Z}_e^\cov}\big(\Kos_{\cA_{P,\eS_e}^\cov}^{-n}\tens{\cA_{P,\eS_e}^\cov\tens{k}\cA_{P,\eS_e}^{\cov,\op}}i_{\eS_e}^*\cF_{\rmmod}\tens{k}\O(\wt{Z}_e^\cov)\big)
\end{equation}
is concentrated in degrees $\ge\dim\wt{\eN}_{P,\eS_e}$ for each $0\le n\le\dim\wt{\eN}_{P,\eS_e}$. Furthermore, by \eqref{eqn:kos-proj-left-right} and \eqref{eqn:indec-proj-left-ASe-mod}, it suffices to show that
\begin{equation}
\label{eqn:tens-proj-hom-ASe-cov}
E^{b'}_{P,\eS_e}\otimes_{\cA_{P,\eS_e}^\cov}i_{\eS_e}^*\cF_{\rmmod}\otimes_{\cA_{P,\eS_e}^\cov}E^{b,\ell}_{P,\eS_e}\simeq\uHom_{\wt{\eN}_{P,\eS_e}}(\cE^b_{P,\eS_e},\cE^{b'}_{P,\eS_e}\star i_{\eS_e}^*\cF)
\end{equation}
has this property after applying $\Gamma_e$ for each $b,b'\in\bfB_{e,P}$. By Grothendieck local duality (see for instance \cite[\href{https://stacks.math.columbia.edu/tag/0A84}{Thm.~0A84}]{stacks-project}), the $\O(\eS_e)$-module
\begin{equation*}
\Gamma_e\uHom_{\wt{\eN}_{P,\eS_e}}(\cE^b_{P,\eS_e},\cE^{b'}_{P,\eS_e}\star i_{\eS_e}^*\cF)
\end{equation*}
is Matlis dual to
\begin{equation}
\label{eqn:matlis-dual-hom}
\Hom_{\eS_e}(\uHom_{\wt{\eN}_{P,\eS_e}}(\cE^b_{P,\eS_e},\cE^{b'}_{P,\eS_e}\star i_{\eS_e}^*\cF),\omega_{\eS_e}),
\end{equation}
so it suffices to show that the latter is concentrated in cohomological degrees $\le-\dim\wt{\eN}_{P,\eS_e}$. Finally, since $\wt{\eN}_{P,\eS_e}$ is Calabi--Yau as in \eqref{eqn:res-Se-calabi-yau}, we have
\begin{align*}
\Hom_{\eS_e}(\uHom_{\wt{\eN}_{P,\eS_e}}(\cE^b_{P,\eS_e},\cE^{b'}_{P,\eS_e}\star i_{\eS_e}^*\cF),\omega_{\eS_e})&\simeq\pi_{P,\eS_e,*}\cHom_{\wt{\eN}_{P,\eS_e}}(\cHom_{\wt{\eN}_{P,\eS_e}}(\cE^b_{P,\eS_e},\cE^{b'}_{P,\eS_e}\star i_{\eS_e}^*\cF),\pi_{P,\eS_e}^!\omega_{\eS_e})\\
&\simeq\pi_{P,\eS_e,*}\cHom_{\wt{\eN}_{P,\eS_e}}(\cHom_{\wt{\eN}_{P,\eS_e}}(\cE^b_{P,\eS_e},\cE^{b'}_{P,\eS_e}\star i_{\eS_e}^*\cF),\omega_{\wt{\eN}_{P,\eS_e}})\\
&\simeq\uHom_{\wt{\eN}_{\eS_e}}(\cE^{b'}_{P,\eS_e}\star i_{\eS_e}^*\cF,\cE^b_{P,\eS_e})[\dim\wt{\eN}_{P,\eS_e}]\\
&\simeq\uHom_{\wt{\eN}_{P,\eS_e}}(\cE^{b'}_{P,\eS_e},\cE^b_{P,\eS_e}\star i_{\eS_e}^*\cF^{\vee,R})[\dim\wt{\eN}_{P,\eS_e}].
\end{align*}
The latter is a direct summand of $i_{\eS_e}^*\cF_{\rmmod}^{\vee,R}[\dim\wt{\eN}_{P,\eS_e}]$, which lies in cohomological degrees $\le-\dim\wt{\eN}_{P,\eS_e}$ by our assumption on $\cF^{\vee,R}$ and t-exactness of $i_{\eS_e}^*$ (see \S\ref{atm:iSe-t-ex-pullback}). This proves the first statement of \eqref{itm:exot-rt-t-class-coconn}. The second statement is immediate from \eqref{itm:exot-rt-t-class-conn} and the claim regarding \eqref{eqn:loc-cohom-kos-n-tens} after noting the inequality $\dim\wt{\eN}_{P,\eS_e}\le\dim\wt{\eN}_P$.
\end{atom}

This concludes the proof of Theorem~\ref{thm:Baff-t-str}.
\end{proof}

\begin{atom}
We now discuss specializations to $q\in\bG_m$. Consider the \emph{un-mixed} (or for us, $\bG_m$-deequivariantized) partial affine Hecke categories
\begin{equation*}
\eH_{P,\eS_e}^{\coh,\dq}:=\QC^!(\wt{\eN}_{P,\eS_e}\fibprod{\eS_e}\wt{\eN}_{P,\eS_e}/Z_e),
\end{equation*}
along with their monoidal endofunctors $\check{\lambda}_e(q)^*$ (for $e=0$, this is just the usual action of $q$, i.e., fiberwise scaling by $q^{-2}$). Then
\begin{equation*}
\Tr(\eH_{P,\eS_e}^{\coh,\dq},\check{\lambda}_e(q)^*)\simeq\QC^!(\cL_{\check{\lambda}_e(q)}(\wh{\eS}_{e,\eN_P}/Z_e))
\end{equation*}
by Proposition~\ref{prop:tr-conv} and an analogous argument to Corollary~\ref{cor:tr-He-coh-ident} (the only difference is that the condition $\angles{n,x}=0$ is omitted).

Write $\delta\colon *\to\B\bG_m$ for the projection and any maps base-changed from it, and likewise for the inclusion $\iota_q\colon\{q\}\hookrightarrow\bG_m$. Then $\delta\colon\eS_e/Z_e\to\eS_e/\wt{Z}_e$ intertwines the automorphisms $\check{\lambda}_e(q)$ and $\id_{\eS_e/\wt{Z}_e}$, so as in \S\ref{atm:conv-cat-funct}, we obtain a monoidal functor (i.e., $\bG_m$-deequivariantization)
\begin{equation*}
(-)^\dq=\delta^*\colon\eH_{P,\eS_e}^{\coh}\to\eH_{P,\eS_e}^{\coh,\dq}
\end{equation*}
intertwining the identity automorphism with $\check{\lambda}_e(q)^*$. Lemma~\ref{lem:tr-ind} and Corollary~\ref{cor:tr-conv-funct} then yield a commutative diagram
\begin{equation}
\label{eqn:q-hecke-tr-res}
\begin{tikzcd}
\eH_{P,\eS_e}^{\coh}\arrow{d}{\delta^*}\arrow{r}{[-]}&\Tr(\eH_{P,\eS_e}^{\coh})\arrow{d}{\Tr(\Ind_{\delta^*})}\arrow{r}{\sim}&\QC^!(\cL(\wh{\eS}_{e,\eN_P}/\wt{Z}_e))\arrow{d}{\cL_q\delta^*}\\
\eH_{P,\eS_e}^{\coh,\dq}\arrow{r}{[-]_q}&\Tr(\eH_{P,\eS_e}^{\coh,\dq},\check{\lambda}_e(q)^*)\arrow{r}{\sim}&\QC^!(\cL_{\check{\lambda}_e(q)}(\wh{\eS}_{e,\eN_P}/Z_e)),
\end{tikzcd}
\end{equation}
where we have written $[-]_q$ for the universal trace functor with respect to $\check{\lambda}_e(q)^*$ for notational clarity, and similarly for $\cL_q\delta$.

By \eqref{eqn:fact-loop-morphism}, the map $\cL_q\delta$ factors as the composition
\begin{equation*}
\cL_{\check{\lambda}_e(q)}(\wh{\eS}_{e,\eN_P}/Z_e)\xhookrightarrow{\iota_q}\cL(\wh{\eS}_{e,\eN_P}/\wt{Z}_e)\fibprod{\B\bG_m}*\xrightarrow{\delta}\cL(\wh{\eS}_{e,\eN_P}/\wt{Z}_e).
\end{equation*}
Thus, $\cL_q\delta_{*}$ is t-exact, and its left adjoint $\cL_q\delta^*$ is right t-exact. This already implies an analogous statement to Theorem~\ref{thm:Baff-t-str}\eqref{itm:exot-rt-t-class-conn} for universal traces of objects of $\eH_{P,\eS_e}^{\coh,\dq}$ admitting $\bG_m$-equivariant (or ``graded'') lifts. Moreover, by \cite[Lem.~3.12]{benzvi}, there is a canonical equivalence $\cL_q\delta^*\simeq\cL_q\delta^!$. Since $\delta^!\simeq\delta^*[1]$, the functor $\cL_q\delta^*$ has cohomological amplitude in $[-1,0]$.

Nonetheless, we have the following ``specialized version'' of Theorem~\ref{thm:Baff-t-str}:
\end{atom}

\begin{theorem}
\label{thm:Baff-t-str-q}
Let $P$ be a standard parabolic subgroup, and let $q\in\bG_m$.
\begin{enumerate}
\item\label{itm:exot-rt-t-class-conn-q} For any Slodowy slice $\eS_e$, the universal trace functor
\begin{equation*}
[-]_q\colon\eH_{P,\eS_e}^{\coh,\dq}\to\QC^!(\cL_{\check{\lambda}_e(q)}(\wh{\eS}_{e,\eN_P}/Z_e))
\end{equation*}
is right t-exact with respect to the exotic t-structure on $\eH_{P,\eS_e}^{\coh,\dq}$ and the standard t-structure on $\QC^!(\cL_{\check{\lambda}_e(q)}(\wh{\eS}_{e,\eN_P}/Z_e))$.

\item\label{itm:exot-rt-t-class-coconn-q} Let $\cF$ be a compact object of $\eH_P^{\coh}$, and suppose that its right monoidal dual $\cF^{\vee,R}$ is connective for the exotic t-structure.\footnote{Again, we may equivalently consider the left monoidal dual $\cF^{\vee,L}$, and we may also replace $\cF^{\vee,R}$ by $(\cF^{\dq})^{\vee,R}$.} Then $[\cF^{\dq}]_q$ is coconnective. In particular, for $e=0$, the functor $[(-)^{\dq}]_q$ has cohomological amplitude in $[-\dim\wt{\eN}_P,0]$.
\end{enumerate}
\end{theorem}

\begin{proof}
The proof runs identically to that of Theorem~\ref{thm:Baff-t-str}, so we only indicate the necessary modifications. For \eqref{itm:exot-rt-t-class-conn-q}, observe that for any $\cF\in\eH_{P,\eS_e}^{\coh,\dq,c}$, we have
\begin{equation*}
\Hom_{\wt{\eN}_{P,\eS_e}}(\cE_{P,\eS_e},\check{\lambda}_e(q)^*\cE_{P,\eS_e}\star\cF))\simeq\Hom_{\wt{\eN}_{P,\eS_e}}(\cE_{P,\eS_e},\cE_{P,\eS_e}\star\cF))\simeq\cF_{\rmmod}
\end{equation*}
by $\bG_m$-equivariance of $\cE_{P,\eS_e}$. Thus, replacing $\BG_{\eS_e/\wt{Z}_e}(\QC(\wt{\eN}_{P,\eS_e}/\wt{Z}_e),-\star\cF)$ by
\begin{equation*}
\BG_{\eS_e/Z_e,\check{\lambda}_e(q)}(\QC(\wt{\eN}_{P,\eS_e}/Z_e),\check{\lambda}_e(q)^*(-)\star\cF)
\end{equation*}
yields an essentially identical complex \eqref{eqn:coh-spr-shf-complex-g}.

For \eqref{itm:exot-rt-t-class-coconn-q}, we simply base-change all loop spaces and sheaves along $\cL_q\delta\colon\{q\}\to\bG_m/\bG_m$. Lemma~\ref{lem:Licov-pullback-coconn} is then replaced by:

\begin{lemma}
\label{lem:Licov-pullback-coconn-q}
Let $\cG\in\Coh^-(\wt{G}^e/\wt{G}^e)$, and suppose that $\cL_q\delta^*\cL i_{\cov}^*\cG\in\Coh^-(Z_e^\cov/Z_e^\cov)$ is coconnective. Then $\cL_q\delta^*\cG\in\Coh^-(\wt{G}^e/\wt{G}^e\times_{\bG_m/\bG_m}\{q\})$ is coconnective.
\end{lemma}

\begin{proof}
Equivalently, we may show that $\iota_q^*\cG\in\Coh^-(\wt{G}^e/\wt{G}^e\times_{\bG_m/\bG_m}\{q\}/\bG_m)$ is coconnective, assuming the same for $\iota_q^*\cL i_{\cov}^*\cG$. The same proof then applies after replacing $\wt{Z}_e/\wt{Z}_e\to\wt{G}^e/\wt{G}^e$ by its base-change along $\iota_q\colon\{q\}/\bG_m\to\bG_m/\bG_m$.
\end{proof}

The rest of the proof follows as above for \eqref{itm:exot-rt-t-class-conn-q}.
\end{proof}

\begin{remark}
When $q$ is not a root of unity, we believe that the requirement in Theorem~\ref{thm:Baff-t-str}\eqref{itm:exot-rt-t-class-coconn-q} that $\cF^{\dq}$ admits a graded lift can likely be removed as in \eqref{itm:exot-rt-t-class-conn-q}. However, we have not worked through all the details. Our statement seems to suffice for most applications (see for instance \cite[Prop.~2.19]{benzvi}).
\end{remark}

\subsection{Twisted coherent Springer sheaves}

\begin{atom}
We now note some elementary consequences of Theorems~\ref{thm:Baff-t-str} and \ref{thm:Baff-t-str-q}.
\end{atom}

\begin{definition}
Given a standard parabolic $P$ and $V\in\Rep P$, write $\O_{\wt{\eN}_P}(V)\in\QC(\wt{\eN}_P/\wt{G})$ for the $*$-pullback of $V$ along the projection $\wt{\eN}_P/\wt{G}\to\B P$. The \emph{$V$-twisted partial coherent Springer sheaf} is the class
\begin{equation*}
\label{eqn:twisted-coh-spr}
\cS_P(V):=[\Delta_{\wt{\eN}_P/\g,*}\O_{\wt{\eN}_P}(V)]\simeq\cL\pi_{P,*}\ev^*\O_{\wt{\eN}_P}(V)\in\Coh(\cL(\wh{\eN}_P/\wt{G})).
\end{equation*}
Similarly, for any $q\in\bG_m$, the \emph{$V$-twisted partial coherent $q$-Springer sheaf} is the class
\begin{equation*}
\label{eqn:twisted-coh-spr-q}
\cS_{P,q}(V):=[\Delta_{\wt{\eN}_P/\g,*}\O_{\wt{\eN}_P}(V)]_q\simeq\cL_q\pi_{P,*}\ev^*\O_{\wt{\eN}_P}(V)\in\Coh(\cL_q(\wh{\eN}_P/\wt{G})).
\end{equation*}
\end{definition}

\begin{atom}
For $P=B$, these are the classes $[\cK_a]$ and $[\cK_a^{\dq}]_q$ for $a$ contained in the translation subgroup of $B^\ext$ (see \S\ref{atm:braid-positivity}). By \eqref{eqn:q-hecke-tr-res}, the two sheaves are related by the restriction functor $\cL_q\delta^*$.

The following corollary contains the conjectures listed in \S\ref{sec:intro-coh-spr-shf} as a special case; we hope to improve this result in a future article.
\end{atom}

\begin{corollary}
\label{cor:twist-coh-Spr-t-str}
\begin{enumerate}[leftmargin=*]
\item\label{itm:braid-conn-coconn} For any $a\in B^\ext_+$, the class $[\cK_a]$ is connective, and the class $[\cK_{a^{-1}}]$ is coconnective. Similarly, for any $q\in\bG_m$, the class $[\cK_a^{\dq}]_q$ is connective, and the class $[\cK_{a^{-1}}^{\dq}]_q$ is coconnective.
\item\label{itm:weight-conn-coconn} Let $\lambda\in X^*(P)^+$ for any standard parabolic $P$, and let $q\in\bG_m$. Then $\cS_P(\lambda),\cS_{P,q}(\lambda)$ are connective and $\cS_P(-\lambda),\cS_{P,q}(-\lambda)$ are coconnective. In particular, the partial coherent Springer sheaf $\cS_P=\cS_P(0)$ and its $q$-specialization $\cS_{P,q}=\cS_{P,q}(0)$ lie in cohomological degree $0$.
\end{enumerate}
\end{corollary}

\begin{proof}
Assertion \eqref{itm:braid-conn-coconn} follows from \S\ref{atm:iSe-t-ex-pullback} and the fact that $\cK_{a}$ and $\cK_{a^{-1}}$ are monoidal inverses, hence mutually dual. Assertion \eqref{itm:weight-conn-coconn} follows similarly from \S\ref{atm:braid-positivity} and \S\ref{atm:iSe-t-ex-pullback}.
\end{proof}

\begin{remark}
Let $F$ be a local field with residue field $\Fq$. In Theorem~1.12 of \cite{benzvi}, the authors define a sheaf on the stack $\Par_{\GL_n}$ of Langlands parameters for $\GL_n$ inducing a categorical local Langlands correspondence for $\GL_n(F)$ as in \eqref{eqn:cat-lang-conj}. It follows immediately from the construction of this sheaf (in terms of the coherent $q$-Springer sheaf) that it also lies in cohomological degree $0$.
\end{remark}

\begin{atom}
We now establish a simple splitting result for the coherent partial $q$-Springer sheaves, analogous to Proposition~3.38 of \loccit:
\end{atom}

\begin{proposition}
\label{prop:coh-spr-decomp}
Let $Q\supset P$ be parabolics, and let $V\in\Rep Q$. Let $\rmP_{Q/P}$ denote the Poincar\'e polynomial of the partial flag variety $Q/P$, and let $q\in\bG_m$ with $\rmP_{Q/P}(q)\ne 0$. Then $\cS_{Q,q}(V)$ is a summand of $\cS_{P,q}(\Res^Q_PV)$.
\end{proposition}

\begin{proof}
The adjoint functors
\begin{equation*}
\begin{tikzcd}[column sep=large]
\QC(\wt{\eN}_Q/\wt{G})\arrow[r,shift right,"\Lag^{P,!}_Q"']&\QC(\wt{\eN}_P/\wt{G}),\arrow[l,shift right,"\Lag^{P}_{Q,!}"']
\end{tikzcd}
\end{equation*}
are $\QC(\g/\wt{G})$-linear, and hence induce morphisms
\begin{equation*}
[\QC(\wt{\eN}_Q/\wt{G}),-\otimes\O_{\wt{\eN}_Q}(V)]\xrightarrow{[\Lag^{P,!}_Q]}[\QC(\wt{\eN}_P/\wt{G}),-\otimes\O_{\wt{\eN}_P}(\Res_P^QV)]\xrightarrow{[\Lag^{P}_{Q,!}]}[\QC(\wt{\eN}_Q/\wt{G}),-\otimes\O_{\wt{\eN}_Q}(V)]
\end{equation*}
in $\Tr(\QC(\g/\wt{G}))\simeq\QC(\cL(\g/\wt{G}))$ by \S\ref{atm:mor-L-Morita} and \eqref{eqn:Lag-extra-adjns}. By Corollary~\ref{cor:reg-mod-rt-left-dual}, these identify with morphisms
\begin{equation}
\label{eqn:SQ-SP-comp}
\cS_Q(V)\xrightarrow{[\Lag^{P,!}_Q]}\cS_P(\Res^Q_PV)\xrightarrow{[\Lag^{P}_{Q,!}]}\cS_Q(V),
\end{equation}
whose composition we must show is invertible.

Let $\cF\in\QC(\wt{\eN}_Q/\wt{G})$, and write
\begin{equation*}
\H^*(Q/P,k):=\bigoplus_{n\ge 0}\H^n(Q/P,k)\angles{n}[-n]
\end{equation*}
for the weight-sheared singular cohomology of the partial flag variety $Q/P$. The base-change diagram
\begin{equation*}
\begin{tikzcd}
&&\n_Q\fibprod{\n_P}\n_Q/\wt{P}\arrow[dl,"\pr_1"']\arrow{dr}{\pr_2}&&\\
&\n_Q/\wt{P}\arrow[dl,"p^P_Q"']\arrow{dr}{i^P_Q}&&\n_Q/\wt{P}\arrow[dl,"i^P_Q"']\arrow{dr}{p^P_Q}&\\
\n_Q/\wt{Q}&&\n_P/\wt{P}&&\n_Q/\wt{Q}
\end{tikzcd}
\end{equation*}
then yields
\begin{align*}
\Lag_{Q,!}^{P}\Lag_Q^{P,!}\cF&\simeq p^P_{Q,*}i_{Q}^{P,*}i_{Q,*}^P p^{P,!}_Q\cF\\
&\simeq\Gamma(\Sym_{\O_{Q/P}}(\Omega_{Q/P}^\vee\angles{2}[1])\otimes\omega_{Q/P})\tens{k} p^P_{Q,*}p^{P,*}_Q\cF\\
&\simeq\Gamma(\Sym_{\O_{Q/P}}(\Omega_{Q/P}\angles{-2}[-1]))^\vee\tens{k}\cF\\
&\simeq\H^*(Q/P,k)^\vee\tens{k}\cF\\
&\simeq\H^*(Q/P,k)[2d_{Q/P}]\tens{k}\cF,
\end{align*}
where the third isomorphism follows by Serre duality, the fourth by standard results on the Hodge cohomology of partial flag varieties (see for instance \cite{virk}), and the fifth by Poincar\'e duality. Taking Hochschild classes, we obtain
\begin{equation*}
[\Lag^{P}_{Q,!}]\circ[\Lag^{P,!}_Q]\simeq\rmP_{Q/P}(v)\cdot\id_{\cS_Q(V)},
\end{equation*}
where $v$ denotes the coordinate function on $\bG_m$. Specializing \eqref{eqn:SQ-SP-comp} to $q$ then yields the desired splitting.
\end{proof}

\begin{remark}
\begin{enumerate}[wide, labelwidth=!, labelindent=0pt]
\item Note that by \cite[Thm.~1]{brion-peyre}, the polynomial $\rmP_{Q/P}(q)$ divides $\rmP_{G/B}(q)=\rmP_W(q^2)$, where $\rmP_W$ denotes the Poincar\'e polynomial of $W$. In particular, these splittings all hold away from roots of unity.

\item These splittings are, in a sense, Koszul-dual to those in \cite[Rmk.~4.13(4)]{benzvi}. For instance, when $\rmP_{G/B}(q)\ne 0$, Proposition~\ref{prop:coh-spr-decomp} shows that $\cS_{G,q}\simeq\O_{\cL_q(\{0\}/G)}\simeq\O_{G/G}$ is a summand of $\cS_q$, whereas the corresponding summand in \loccit\ is the anti-spherical module $\O_{\cL_q(\cN/G)}$.

\item If we only knew that $\cS_q$ was concentrated in cohomological degree $0$, the proposition would imply the same for each $\cS_{P,q}$ whenever $\rmP_{P/B}(q)\ne 0$ (obviating the need for Appendix~\ref{sec:parab-ark-bez-equiv}). However, Corollary~\ref{cor:twist-coh-Spr-t-str} applies over all of $\bG_m$.
\end{enumerate}
\end{remark}

\appendix

\section{Parabolic Arkhipov--Bezrukavnikov equivalences}
\label{sec:parab-ark-bez-equiv}

\begin{atom}
In this appendix, we generalize several important properties of Arkhipov--Bezrukavnikov's equivalence \cite{arkh-bez} to its parabolic analogues. This is used in Proposition~\ref{prop:parab-A-kosz-grad} to show that the noncommutative partial Springer resolution admits a Koszul grading. This material is likely well-known to experts, but we could not locate a reference.
\end{atom}

\begin{atom}
We begin by fixing notation for the automorphic side of Arkhipov--Bezrukavnikov's equivalence. Assume that $k=\Qlbar$ for some prime $\ell$, and that $q$ is a power of a prime distinct from $\ell$. Given a prestack $\eX$ over $\Fqbar$, we let $\Shv(\eX)$ denote the dg-category of ind-constructible $\Qlbar$-adic \'etale sheaves on $\eX$, defined by right Kan extension from the category of schemes along $!$-pullbacks. We refer the reader to \cite[\S1]{gaitsgory-ab} for further details.

We reserve the cohomological superscripts $\heart$, $\le n$, and $\ge n$ for the standard t-structure on $\Shv(\eX)$, and instead use $\heart_p$, $\le_p n$, and $\ge_p n$ for the perverse t-structure, when applicable (e.g., if $\eX$ is a ``placidly stratified'' stack as in \cite[Prop.~6.3.3]{bkv}, which will always be the case). Thus, $\Shv(\eX)^{\heart_p}$ denotes the abelian category of perverse sheaves on $\eX$.

In general (when $\eX$ is suitably stacky), the subcategory of compact objects of $\Shv(\eX)$ is properly contained in the subcategory of constructible sheaves, analogously to the subcategories of safe and coherent D-modules in characteristic $0$ (see \cite[Thm.~B]{chen-dhillon}). Following Definition~\ref{def:t-str} and \S1.1.11 of \loccit, we write $\Shv_{\ren}(\eX)$ for the renormalization of $\Shv(\eX)$ whose compact objects are the constructible sheaves.

Let $G,T,B,P,\eB,\eP$, and so forth be as in the main text (see \S\ref{atm:springer-theory}). Let $F:=\Fqbar\laurent{t}$ be the field of Laurent series over $\Fqbar$, let $O:=\Fqbar\doubles{t}$ be its ring of integers, and regard the loop group $\check{G}_F$ and arc group $\check{G}_O$ as $\Fqbar$-ind-schemes. We write $I_{\check{P}}\subset\check{G}_F$ for the Iwahori subgroup associated to the Langlands-dual parabolic, i.e., the preimage of $\check{P}$ under the evaluation map $\check{G}_O\to\check{G}$. We also set $I:=I_{\check{B}}$ and $I^-:=I_{\check{B}^-}$ for a choice of opposite Borel subgroup $B^-$. In general, we continue our convention of omitting all subscripts $\check{P}$ when $\check{P}=\check{B}$.

Let $\mathring{I},\mathring{I}^-$ denote the pro-unipotent radicals of $I,I^-$, respectively, i.e., the preimages of the unipotent radicals $\check{N},\check{N}^-$ under the evaluation maps, and let $\psi\colon\mathring{I}^-\to\bG_a$ denote the composition of the evaluation map $\mathring{I}^-\to\check{N}^-$ with a non-degenerate additive character $\check{N}^-\to\bG_a$. Fix an Artin--Schreier local system $\sL_{\AS}$ on $\bG_a$, and let $\mathring{I}^-,\psi\backslash-$ denote the Whittaker equivariance condition, i.e., the categorical $\mathring{I}^-$-invariants with respect to the induced character sheaf $\psi^*\sL_{\AS}$.

Finally, we define the usual $\Fqbar$-ind-schemes: the \emph{(partial) affine flag varieties} $\Fl_{\check{P}}:=\check{G}_F/I_{\check{P}}$, and the \emph{affine Grassmannian} $\Gr:=\Fl_{\check{G}}=\check{G}_F/\check{G}_O$. Given another parabolic subgroup $Q\supset P$, we write $\varpi^{\check{P}}_{\check{Q}}\colon\Fl_{\check{P}}\to\Fl_{\check{Q}}$ for the natural proper map. We adopt the same notational conventions as in \S\ref{atm:Lag-functs}: when $P=B$ or $Q=G$, we omit the relevant super- or sub-script, and when the quotient is taken on the left-hand side, we instead write $\prescript{\check{P}}{\check{Q}}{\varpi}$.
\end{atom}

\begin{atom}
We will require parabolic analogues of Gaitsgory's central sheaves on $\Fl$ \cite{central-sheaves}. Let $X/\Fqbar$ be a smooth curve with distinguished point $x\in X(\Fq)$, and fix an identification $O\simeq\wh{\O}_{X,x}$ with the completed local ring of $X$ at $x$. Let $\Bun_{\check{G}}$ denote the stack classifying principal $\check{G}$-bundles on $X$, and let $\eF^0$ denote the trivial $\check{G}$-bundle (on any scheme over $X$).

Recall that the \emph{Beilinson--Drinfeld Grassmannian} $\Gr_X$ is the $\Fqbar$-ind-scheme over $X$ representing the functor
\begin{equation*}
\Gr_X(S)=\{(y,\eF,\beta): y\in X(S),\eF\in\Bun_{\check{G}}(S),\beta\colon\eF|_{X\times S\setminus\Gamma_y}\xrightarrow{\sim}\eF^0|_{X\times S\setminus\Gamma_y}\}
\end{equation*}
for any affine scheme $S$ (here $\Gamma_y\subset X\times S$ denotes the graph of $y$). In particular, our chosen identification $O\simeq\wh{\O}_{X,x}$ yields an isomorphism $\Gr\simeq\Gr_X|_x$ with the fiber over $x$. Moreover, given any perverse sheaf $\sS\in\Shv(\check{G}_O\backslash\Gr)^{\heart_p}$, we may canonically attach a perverse sheaf $\sS_X\in\Shv(\Gr_X)^{\heart_p}$ as in \cite[\S2.1.2]{central-sheaves}. We write $\sS_{X\setminus x}:=\sS_X|_{X\setminus x}$ and $\sS_x:=\sS_X|_x$ for its $*$-restrictions to $\Gr_{X\setminus x}:=\Gr_X|_{X\setminus x}$ and $\Gr_x:=\Gr_X|_x$, respectively.

More generally, we define an $\Fqbar$-ind-scheme $\Fl_{\check{P},X}$ over $X$ by
\begin{equation*}
\Fl_{\check{P},X}(S)=\{(y,\eF,\beta,\epsilon): (y,\eF,\beta)\in\Gr_X(S)\text{, and }\epsilon\text{ is the data of a reduction of }\eF|_{\{x\}\times S}\text{ to }\check{P}\}
\end{equation*}
for any affine scheme $S$. As in \cite[Prop.~3]{central-sheaves}, we have canonical isomorphisms $\Fl_{\check{P},x}\simeq\Fl_{\check{P}}$ and $\Fl_{\check{P},X\setminus x}\simeq\Gr_{X\setminus x}\times\check{G}/\check{P}$ (using the analogous notations for $\Fl_{\check{P},X}$). Denote by
\begin{equation*}
\Psi_{\Fl_{\check{P},X}}\colon\Shv(\Fl_{\check{P},X}|_{X\setminus x})\to\Shv(\Fl_{\check{P},X}|_x)
\end{equation*}
the nearby cycles functor, and define
\begin{align*}
\rmZ_{\check{P}}\colon\Shv_{\ren}(\check{G}_O\backslash\Gr)^{\heart_p}&\to\Shv(\Fl_{\check{P}}),\\
\sS&\mapsto\Psi_{\Fl_{\check{P},X}}(\sS_{X\setminus x}\boxtimes\delta_{1_{\check{G}/\check{P}}})[2d_{G/P}],
\end{align*}
where $\delta_{1_{\check{G}/\check{P}}}$ denotes the skyscraper sheaf at the point $1_{\check{G}/\check{P}}\in\check{G}/\check{P}$ corresponding to the parabolic $\check{P}$. Thus, $\rmZ:=\rmZ_{\check{B}}$ is the usual central functor.\footnote{The shift by $2d_{G/P}$ is to account for our use of the $!$-convolution, rather than the $*$-convolution as in \cite{central-sheaves}. These are intertwined by $*$-convolution with the $!$-monoidal unit, i.e., the dualizing sheaf $\omega_{I_{\check{P}}\backslash I_{\check{P}}/I_{\check{P}}}\simeq\underline{\Qlbar}_{I_{\check{P}}\backslash I_{\check{P}}/I_{\check{P}}}[-2\dim\check{L}]$, and similarly for $\check{G}$ (here $\check{L}$ denotes the Levi factor of $\check{P}$).} By a basic property of the nearby cycles functor, the sheaves $\rmZ_{\check{P}}(\sS)$ all lie in perverse degree $-2d_{G/P}$.
\end{atom}

\begin{atom}
\label{atm:perv-coh-t-str}

We require one additional tool, owed to \cite[\S6.5]{bez-losev}. Recall from \S\ref{atm:springer-theory} that $\eN_P$ denotes the image of $\wt{\eN}_P$ in $\g$. Given any $G$-equivariant Gorenstein $\O(\eN_P)$-algebra $A$, there is a unique t-structure $(A\dashmod^{G,\le_p0},A\dashmod^{G,\ge_p0})$ on $A\dashmod^G$ with
\begin{equation}
\label{eqn:perv-coh-t-str}
\begin{split}
A\dashmod^{G,\le_p0}&:=\{M\in A\dashmod^G : \codim_{\eN_P}\supp(\H^n(M))\ge 2n\text{ for all }n\in\Z\},\\
A\dashmod^{G,\ge_p0}&:=\{\Hom_{A^{\op}}(M,A) : M\in A^{\op}\dashmod^{G,\le_p0}\}.
\end{split}
\end{equation}
In particular, for $A=\O(\eN_P)$, we obtain the perverse coherent t-structure (of middle perversity) on $\QC(\eN_P/G)$, and the forgetful functor $A\dashmod^G\to\QC(\eN_P/G)$ is t-exact. We therefore also refer to the t-structure \eqref{eqn:perv-coh-t-str} as the \emph{perverse coherent t-structure (of middle perversity)}.

Recall from \S\ref{atm:nc-spr-res} that $\hcE$ denotes the Bezrukavnikov--Mirkovi\'c tilting bundle on $\wt{\g}$. We set $\hcA:=\End_{\wt{\g}}(\hcE)$, and term it the \emph{noncommutative Grothendieck simultaneous resolution}. As for $\cA$, it is a $\wt{G}$-equivariant $\O(\g)$-algebra in cohomological degree $0$, and yields an equivalence as in \eqref{eqn:bm-functor}. The perverse coherent t-structure \eqref{eqn:perv-coh-t-str} then applies for $A=\cA_P^\op,\hcA\otimes_{\O(\g)}\cA_P^\op$, and so forth.
\end{atom}

\begin{atom}
Finally, equip the automorphic category $\Shv_{\ren}(I_{\check{P}}\backslash\Fl_{\check{P}})$ with its $!$-convolution product; we equip the spectral category $\QC^!(\wt{\eN}_{P}\times_{\g}\wt{\eN}_{P}/G)$ with its $*$-convolution product as usual. We may now state the parabolic version of Gaitsgory's results:
\end{atom}

\begin{proposition}
\label{prop:Z-parab}
We have a commutative diagram\footnote{Our use of $\Shv_{\ren}$ rather than $\Shv$ is to align with the convention appearing throughout Bezrukavnikov's work, see e.g.\ \cite{beztwo}.} of monoidal functors
\begin{equation}
\label{eqn:ZP-commdiag}
\begin{tikzcd}
\Shv_{\ren}(\check{G}_O\backslash\Gr)^{\heart_p}\arrow{d}{\rmZ_{\check{P}}}\arrow{r}{\sim}&\Rep(G)^\heart\arrow{d}{-\otimes\Delta_{\wt{\eN}_P/\g,*}\O_{\wt{\eN}_P/G}}\\
\Shv_{\ren}(I_{\check{P}}\backslash\Fl_{\check{P}})\arrow{r}{\sim}&\QC^!(\wt{\eN}_{P}\fibprod{\g}\wt{\eN}_{P}/G).
\end{tikzcd}
\end{equation}
In particular, the functor $\rmZ_{\check{P}}$ carries canonical tensor and central structures, and canonically lifts to the $I_{\check{P}}$-equivariant category. Moreover,
\begin{enumerate}
\item\label{itm:Z-push} we have $\varpi^{\check{P}}_!\rmZ_{\check{P}}[-2d_{G/P}]\simeq\id$;
\item\label{itm:Z-conv-t-ex} convolution with the image of $\rmZ_{\check{P}}$ is t-exact with respect to the perverse t-structure; and
\item\label{itm:Z-mon-tens} the monodromy endomorphism of $\rmZ_{\check{P}}$ is nilpotent and compatible with its tensor structure.
\end{enumerate}
\end{proposition}

\begin{proof}
Rather than repeat Gaitsgory's proofs, we deduce the result from the usual Borel case \cite{central-sheaves,beztwo}. Set
\begin{equation*}
\eta_{\wt{\eN}\fibprod{\g}\wt{\eN}_P/G}:=(i_P\times p_P)_*\omega_{\wt{\eN}_P\fibprod{\eP}\eB/G}(-\rho-2\rho_{P/B})\in\QC^!(\wt{\eN}\fibprod{\g}\wt{\eN}_P/G),
\end{equation*}
and consider the diagram
\begin{equation}
\label{eqn:big-z-comm-diag}
\begin{tikzcd}[column sep=huge]
\Shv_{\ren}(I\backslash\Fl)\arrow[dd,"\varpi_{\check{P},*}"']\arrow{rrr}{\sim}&&&\QC^!(\wt{\eN}\fibprod{\g}\wt{\eN}/G)\arrow{dd}{\Lag_{P,*}(-\rho-2\rho_{P/B})}\\
&\Shv_{\ren}(\check{G}_O\backslash\Gr)^{\heart_p}\arrow[ul,"\rmZ"']\arrow[dl,dashed,near start,"{\rmZ_{\check{P}}[2d_{P/B}]}"']\arrow{r}{\sim}&\Rep(G)^\heart\arrow{ur}{-\otimes\Delta_{\wt{\eN}/\g,*}\omega_{\wt{\eN}/G}}\arrow[dr,near start,"-\otimes\eta_{\wt{\eN}\fibprod{\g}\wt{\eN}_P/G}"]&\\[-1.5em]
\Shv_{\ren}(I\backslash\Fl_{\check{P}})\arrow{rrr}{\sim}&&&\QC^!(\wt{\eN}\fibprod{\g}\wt{\eN}_P/G)\\
&&&\\
\Shv_{\ren}(I_{\check{P}}\backslash\Fl_{\check{P}})\arrow{uu}{\prescript{}{\check{P}}{\varpi}^!}\arrow[from=uuur,dotted,crossing over,"{\rmZ_{\check{P}}}"]\arrow{rrr}{\sim}&&&\QC^!(\wt{\eN}_P\fibprod{\g}\wt{\eN}_P/G)\arrow[from=uuul,crossing over,"-\otimes\Delta_{\wt{\eN}_P/\g,*}\omega_{\wt{\eN}_P/G}"']\arrow[uu,"\prescript{}{P}{\Lag}^!(-\rho-2\rho_{P/B})"'].
\end{tikzcd}
\end{equation}
Here the dashed and dotted arrows have not yet been defined, the short horizontal equivalence is the (underived) geometric Satake equivalence, and the three long horizontal equivalences are as in the ``strict-strict'' case of \cite[Thm.~A]{chen-dhillon}. Note that the topmost trapezoid commutes by the Borel case.

We first argue that the top and bottom rectangles commute. By Theorem~6.0.4 of \loccit, we have commutative squares
\begin{equation}
\label{lang-equiv-com-square}
\begin{tikzcd}
\Shv(\mathring{I}\dashmon\backslash\Fl)\arrow[r,"\sim"]\arrow[d,shift left,"\varpi_{P,!}"]&\QC^!(\wt{\g}\fibprod{\g}\wt{\eN}/G)\arrow[d,shift left,"{\Lag_{P,!}(-\rho)[-d_{P/B}]}"]\\
\Shv(\mathring{I}\dashmon\backslash\Fl_{\check{P}})\arrow[r,"\sim"]\arrow[u,shift left,"\varpi_{P}^!"]&\QC^!(\wt{\g}\fibprod{\g}\wt{\eN}_P/G),\arrow[u,shift left,"{\Lag_{P}^!(\rho)[d_{P/B}]}"]
\end{tikzcd}
\quad
\begin{tikzcd}
\Shv(\mathring{I}\dashmon\backslash\Fl)\arrow[r,"\sim"]\arrow[d,shift right,"\varpi_{P,*}"']&\QC^!(\wt{\g}\fibprod{\g}\wt{\eN}/G)\arrow[d,shift right,"\Lag_{P,*}(-\rho-2\rho_{P/B})"']\\
\Shv(\mathring{I}\dashmon\backslash\Fl_{\check{P}})\arrow[r,"\sim"]\arrow[u,shift right,"\varpi_{P}^*"']&\QC^!(\wt{\g}\fibprod{\g}\wt{\eN}_P/G),\arrow[u,shift right,"\Lag_{P}^*(\rho+2\rho_{P/B})"']
\end{tikzcd}
\end{equation}
which are compatible with the left actions of the monoidally equivalent\footnote{Again, the former category is equipped with its $!$-convolution product, and the latter with its $*$-convolution product.}
\begin{equation}
\label{eqn:hecke-mon-cat-lang-equiv}
\Shv(\mathring{I}\dashmon\backslash\check{G}_F/\mathring{I}\dashmon)\simeq\QC^!_{\eN}(\wt{\g}\fibprod{\g}\wt{\g}/G).
\end{equation}
Here the notation $\mathring{I}\dashmon\backslash-$ refers to $\mathring{I}$-monodromicity, rather than equivariance; i.e.,  we take the full subcategory generated by the strictly equivariant sheaves. Likewise (dually), we have an equivalence
\begin{equation*}
\Shv(I\backslash\check{G}_F/\mathring{I}\dashmon)\simeq \QC^!(\wt{\eN}\fibprod{\g}\wt{\g}/G),
\end{equation*}
which is compatible with the right actions of \eqref{eqn:hecke-mon-cat-lang-equiv}. Under this duality, we have identifications\footnote{On the automorphic side, the duality data (as $\Shv_{\ren}(I_{\check{P}}\backslash\Fl_{\check{P}})$-modules) are given by the $!$-convolutions
\begin{align*}
\Shv_{\ren}(I_{\check{P}}\backslash\Fl)\otimes\Shv_{\ren}(I\backslash\Fl_{\check{P}})&\xrightarrow{\star}\Shv_{\ren}(I_{\check{P}}\backslash\Fl_{\check{P}}),\\
\Shv_{\ren}(I_{\check{P}}\backslash\Fl_{\check{P}})\otimes\Shv_{\ren}(I_{\check{P}}\backslash\Fl_{\check{P}})&\xrightarrow{\star}\Shv_{\ren}(I_{\check{P}}\backslash\Fl_{\check{P}}).
\end{align*}
One then verifies that $\varpi_{\check{P},*}\sF\star\sG\simeq\sF\star\prescript{}{\check{P}}{\varpi}^!\sG$. An analogous statement holds for the duality data on the spectral side.} $(\Lag_P^!)^\vee\simeq\prescript{}{P}{\Lag}_{*}$ and $(\varpi_{\check{P},*})^\vee\simeq\prescript{}{\check{P}}{\varpi}^!$. Thus, by Corollary~5.5.4 and Proposition~3.3.7 in \loccit, applying 
\begin{equation*}
\Shv(I\backslash\check{G}_F/\mathring{I}\dashmon)\tens{\Shv(\mathring{I}\dashmon\backslash\check{G}_F/\mathring{I}\dashmon)}-\,\simeq\,\QC^!(\wt{\eN}\fibprod{\g}\wt{\g}/G)\tens{\QC^!_{\eN}(\wt{\g}\fibprod{\g}\wt{\g}/G)}-
\end{equation*}
to the right-hand square of \eqref{lang-equiv-com-square} and renormalizing gives commutativity of the top rectangle in \eqref{eqn:big-z-comm-diag}. An analogous argument gives commutativity of the bottom rectangle.

Next, we claim that the two triangles involving $\Rep(G)^\heart$ commute. We have a commutative diagram of stacks:
\begin{equation}
\label{eqn:Lag-bc-diag}
\begin{tikzcd}[column sep=small]
&&\wt{\eN}/G\arrow{dr}{\Delta_{\wt{\eN}/\g}}&\\
&\wt{\eN}_P\fibprod{\eP}\eB/G\arrow{ur}{i_P}\arrow{dr}{i_P\times\id_{\wt{\eN}_P\fibprod{\eP}\eB}}\arrow[dddd,equal]&&\wt{\eN}\fibprod{\g}\wt{\eN}/G\\
&&\wt{\eN}\fibprod{\g}(\wt{\eN}_P\fibprod{\eP}\eB)/G\arrow[dr,near start,"\id_{\wt{\eN}}\times p_P"]\arrow[ur,"\id_{\wt{\eN}}\times i_P"']&\\
\B G\arrow[from=uur]\arrow[from=ddr]\arrow[from=uuurr,bend right=60]\arrow[from=dddrr,bend left=60]&&&\wt{\eN}\fibprod{\g}\wt{\eN}_P/G\\
&&(\wt{\eN}_P\fibprod{\eP}\eB)\fibprod{\g}\wt{\eN}_P/G\arrow{ur}{i_P\times\id_{\wt{\eN}_P}}\arrow[dr,"p_P\times\id_{\wt{\eN}_P}"]&\\
&\wt{\eN}_P\fibprod{\eP}\eB/G\arrow[ur,"\id_{\wt{\eN}_P\fibprod{\eP}\eB}\times p_P"']\arrow{dr}{\pr_1}&&\wt{\eN}_P\fibprod{\g}\wt{\eN}_P/G\\
&&\wt{\eN}_P/G.\arrow{ur}{\Delta_{\wt{\eN}_P/\g}}&
\end{tikzcd}
\end{equation}
Moreover, a standard lemma on pullback diagrams implies that both diamonds are cartesian. The claim now follows by a diagram chase.

Finally, since nearby cycles commute with proper pushforwards, we have
\begin{align*}
\varpi_{\check{P},!}\rmZ(\sS)&\simeq\varpi_{\check{P},!}\Psi_{\Fl_{X}}(\sS_{X\setminus x}\boxtimes\delta_{1_{\check{G}/\check{B}}})[2d_{G/B}]\\
&\simeq\Psi_{\Fl_{\check{P},X}}(\varpi_{\check{P},!}(\sS_{X\setminus x}\boxtimes\delta_{1_{\check{G}/\check{B}}}))[2d_{G/B}]\\
&\simeq\Psi_{\Fl_{\check{P},X}}(\sS_{X\setminus x}\boxtimes\delta_{1_{\check{G}/\check{P}}})[2d_{G/B}]\\
&\simeq \rmZ_{\check{P}}(\sS)[-2d_{G/P}+2d_{G/B}]\\
&\simeq \rmZ_{\check{P}}(\sS)[2d_{P/B}]
\end{align*}
using the evident commutative diagram
\begin{equation*}
\begin{tikzcd}
\Fl_{X\setminus x}\arrow{r}{\sim}\arrow{d}{\varpi_{\check{P}}}&\Gr_{X\setminus x}\times\check{G}/\check{B}\arrow{d}{\id_{\Gr_{X\setminus x}}\times p_{\check{P}}}\\
\Fl_{\check{P},X\setminus x}\arrow{r}{\sim}&\Gr_{X\setminus x}\times\check{G}/\check{P}.
\end{tikzcd}
\end{equation*}
This endows $\rmZ_{\check{P}}(\sS)$ with a canonical $I$-equivariant structure, providing the dashed arrow in \eqref{eqn:big-z-comm-diag} and commutativity of the top triangle involving $\Shv(\check{G}_O\backslash\Gr)^{\heart_p}$ (note that $\varpi_{\check{P},*}=\varpi_{\check{P},!}$). Commutativity of the trapezoid involving $\rmZ_{\check{P}}[2d_{P/B}]$ and $-\otimes\eta_{\wt{\eN}\times_{\g}\wt{\eN}_P/G}$ then follows. Thus, to define the dotted arrow (i.e., a canonical $I_{\check{P}}$-equivariant structure on $\rmZ_{\check{P}}(\sS)$, noting that $\prescript{}{\check{P}}{\varpi}^!\simeq\prescript{}{\check{P}}{\varpi}^*[2d_{P/B}]$), it suffices to provide a canonical factorization of $-\otimes\eta_{\wt{\eN}\times_{\g}\wt{\eN}_P/G}$ through $\QC^!(\wt{\eN}_P\times_{\g}\wt{\eN}_P/G)$, but this is simply $-\otimes\Delta_{\wt{\eN}_P/\g,*}\omega_{\wt{\eN}_P/G}$. The entirety of \eqref{eqn:big-z-comm-diag} then commutes.

Finally, we must account for the fact that \cite{chen-dhillon} uses the $!$-convolution, whereas our convention is to use the $*$-convolution. As in \cite[Rmk.~1.2.2]{bznp}, the endofunctor
\begin{equation}
\label{eqn:endofunct-conv-inter}
-\otimes\pr_2^*\omega_{\wt{\eN}_P/G}^\vee\colon\QC^!(\wt{\eN}_P\fibprod{\g}\wt{\eN}_P/G)\xrightarrow{\sim}\QC^!(\wt{\eN}_P\fibprod{\g}\wt{\eN}_P/G)
\end{equation}
intertwines these two monoidal structures, hence also the central functors $-\otimes\Delta_{\wt{\eN}_P/\g,*}\omega_{\wt{\eN}_P/G}$ and $-\otimes\Delta_{\wt{\eN}_P/\g,*}\O_{\wt{\eN}_P/G}$ (note that \eqref{eqn:endofunct-conv-inter} is isomorphic to a shift).

For \eqref{itm:Z-push}, it is now immediate that
\begin{equation*}
\varpi^{\check{P}}_!\rmZ_{\check{P}}[-2d_{G/P}]\simeq\varpi^{\check{P}}_!\varpi_{\check{P},!}\rmZ[-2d_{G/P}-2d_{P/B}]\simeq\varpi_!\rmZ[-2d_{G/B}]\simeq\id.
\end{equation*}
For \eqref{itm:Z-conv-t-ex}, consider the diagram
\begin{equation}
\label{eqn:N-g-tilde-compact}
\begin{tikzcd}
\Shv_{\ren}(I\backslash\Fl_{\check{P}})\arrow{r}{\sim}\arrow{d}{\prescript{\check{N}}{\check{B}}{\varpi}^*}&\QC^!(\wt{\eN}\fibprod{\g}\wt{\eN}_P/G)\arrow{d}{(i_{\wt{\eN}}\times\id_{\wt{\eN}_P})_*}\\
\Shv_{\ren}(\mathring{I}\dashmon\backslash\Fl_{\check{P}})\arrow{r}{\sim}&\QC^!(\wt{\g}\fibprod{\g}\wt{\eN}_P/G)
\end{tikzcd}
\end{equation}
which commutes by Theorem~6.0.2 of \cite{chen-dhillon}; here $\prescript{\check{N}}{\check{B}}{\varpi}^*$ denotes the evident restriction functor, in keeping with our earlier conventions. It suffices to show that right-convolution with the image of $\rmZ_{\check{P}}$ is t-exact in $\Shv_{\ren}(\mathring{I}\dashmon\backslash\Fl_{\check{P}})$. By \cite[Prop.~6.11(2)]{bez-losev}, the composite equivalence
\begin{equation}
\label{eqn:bimod-equiv-B-P}
\Shv_{\ren}(\mathring{I}\dashmon\backslash\Fl_{\check{P}})\simeq\QC^!(\wt{\g}\fibprod{\g}\wt{\eN}_P/G)\simeq\hcA\tens{\O(\g)}\cA_P^{\op}\dashmod^G
\end{equation}
matches $\Shv_{\ren}(\mathring{I}\dashmon\backslash\Fl_{\check{P}})^{\heart_p}$ and the shifted heart $\hcA\otimes_{\O(\g)}\cA_P^{\op}\dashmod^{G,\heart_p}[-d_{P/B}]$ of the perverse coherent t-structure of \S\ref{atm:perv-coh-t-str}.\footnote{\label{fn:cd-bl-conventions}We comment on the relationship between our conventions and those of \cite{chen-dhillon,bez-losev}. The relationship between the two is explained in the proof of \cite[Thm.~6.0.4]{chen-dhillon}, though note that there $\rho_P$ is written in place of $-\rho_{P/B}$. Moreover, it is impossible to twist the right-hand factor of Chen--Dhillon's equivalence $\Phi_{B,P}$ by $\pm\rho$ (as it is not a character of $P$), and the equivalence $\Phi'$ is already predetermined. Thus, we compose Bezrukavnikov--Losev's equivalence for $P$ with $-\otimes\O_{\wt{\g}\times_{\g}\wt{\eN}_P}(-\rho,0)[-d_{P/B}]$ to ensure equivariance as left-modules over \eqref{eqn:hecke-mon-cat-lang-equiv} (rather than $\rho$-twisted equivariance), and we take $\Lag_P^*(\rho+2\rho_{P/B})$ in place of $\Lag_P^*$ as in \eqref{lang-equiv-com-square} (and similarly for $\Lag_{P,*}$). Note that this preserves the equivariance as right modules over \eqref{eqn:two-geom-real-for-P}, since the latter equivalence is defined by taking categorical endomorphisms of the former equivalence (see the proof of \cite[Thm.~6.0.13]{chen-dhillon}). Finally, the right-most equivalence in \eqref{eqn:bimod-equiv-B-P} is given in \cite{bez-losev} by the functor $\Hom_{\wt{\g}\times_{\g}\wt{\eN}_P}(\hcE^\vee(\rho)\boxtimes\cE_P,-)$ (note that $\hcE^\vee(\rho)\boxtimes\cE_P(-2\rho)$ is written in \loccit, but this cannot make sense as written; tracing the proof shows that $\cE_P(-2\rho)$ arises as $(\Lag_P^!(\rho))^L\cE\simeq\Lag_{P,!}\cE(-\rho)=\cE_P$). However, after applying $-\otimes\O_{\wt{\g}\times_{\g}\wt{\eN}_P}(-\rho,0)[-d_{P/B}]$, this agrees (up to shift) with our convention of using $\Hom_{\wt{\g}\times_{\g}\wt{\eN}_P}(\hcE^\vee\boxtimes\cE_P,-)$.} Moreover, it is compatible with the right actions of
\begin{equation}
\label{eqn:two-geom-real-for-P}
\Shv_{\ren}(I_{\check{P}}\backslash\Fl_{\check{P}})\simeq\QC^!(\wt{\eN}_P\fibprod{\g}\wt{\eN}_P/G)\simeq\cA_P\tens{\O(\g)}\cA_P^{\op}\dashmod^G_{\ren}
\end{equation}
as in Proposition~\ref{prop:mod-equiv-left-t-exact}. By \eqref{eqn:ZP-commdiag}, the image of $\rmZ_{\check{P}}$ corresponds to $\cA_P$-bimodules of the form $V\otimes\cA_P$, for $V\in\Rep(G)^{\heart}$. The claim is now immediate from the definition of the perverse coherent t-structure.

Finally, for \eqref{itm:Z-mon-tens}, it is a straightforward exercise from the definition of the nearby cycles functor to show that the monodromy endomorphism is compatible with proper pushforwards. Thus, the monodromy endomorphism of the dashed $\rmZ_{\check{P}}[2d_{P/B}]$ in \eqref{eqn:big-z-comm-diag} is given by applying $\varpi_{\check{P},!}$ to the monodromy endomorphism of $\rmZ$. On the spectral side, the monodromy endomorphism of $\rmZ$ corresponds to the ``tautological'' endomorphism of $-\otimes\Delta_{\wt{\eN}/\g,*}\omega_{\wt{\eN}/G}$. In more detail, there is a tautological endomorphism of $\id_{\QC^!(\g/G)}$ which ``acts by $x$'' on the fiber over each $x\in\g$ (see \cite[\S4.1.4]{beztwo}). In particular, for any $V\in\Rep(G)^{\heart}$, this gives an endomorphism of $V\otimes\omega_{\g/G}$, and applying $\Delta_{\wt{\eN}/\g,*}\pi^!$ then yields an endomorphism of $V\otimes\Delta_{\wt{\eN}/\g,*}\omega_{\wt{\eN}/G}$. By \eqref{eqn:Lag-bc-diag}, applying $\Lag_{P,*}(-\rho-2\rho_{P/B})$ to the latter endomorphism then gives the analogously defined endomorphism of $V\otimes\eta_{\wt{\eN}\times_{\g}\wt{\eN}_P/G}$.

We wish to identify the monodromy endomorphism of the dotted $\rmZ_{\check{P}}$ in \eqref{eqn:big-z-comm-diag}. This is evidently sent to the monodromy endomorphism of the dashed $\rmZ_{\check{P}}[2d_{P/B}]$ via $\prescript{}{\check{P}}{\varpi}^!$. But for any $V\in\Rep(G)^{\heart}$, the $\Qlbar$-algebra map
\begin{equation*}
\H^0\!{\prescript{}{P}{\Lag}}^!(-\rho-2\rho_{P/B})\colon\H^0\End_{\wt{\eN}_P\fibprod{\g}\wt{\eN}_P/G}(V\otimes\Delta_{\wt{\eN}_P/\g,*}\omega_{\wt{\eN}_P/G})\to\H^0\End_{\wt{\eN}\fibprod{\g}\wt{\eN}_P/G}(V\otimes\eta_{\wt{\eN}\fibprod{\g}\wt{\eN}_P/G})
\end{equation*}
is an isomorphism. Indeed, by \eqref{eqn:Lag-bc-diag}, we may identify it with the $\Qlbar$-algebra map
\begin{equation*}
p^{!}_P\colon\H^0\End_{\wt{\eN}_P/G}(V\otimes\omega_{\wt{\eN}_P/G})\to\H^0\End_{\wt{\eN}_P\fibprod{\eP}\eB/G}(V\otimes\omega_{\wt{\eN}_P\fibprod{\eP}\eB/G}),
\end{equation*}
where the claim is clear. Examining \eqref{eqn:Lag-bc-diag} as before, this shows that the monodromy endomorphism of $\rmZ_{\check{P}}$ \emph{must} correspond to the tautological endomorphism of $-\otimes\Delta_{\wt{\eN}_P/\g,*}\omega_{\wt{\eN}_P/G}$. Since the latter is nilpotent (as $\wt{\eN}_P\times_{\g}\wt{\eN}_P$ is supported over $\eN$) and compatible with the tensor structure by construction, the same must be true for the former, completing the proof.
\end{proof}

\begin{atom}
Finally, we give the desired generalization of Arkhipov--Bezrukavnikov's equivalence to parabolic subgroups:
\end{atom}

\begin{proposition}
\label{prop:arkh-bez-parab}
There is an equivalence of categories
\begin{equation}
\label{eqn:lang-parab-equiv}
\Shv(\mathring{I}^-,\psi\backslash\Fl_{\check{P}})\simeq\QC^!(\wt{\eN}_P/G)
\end{equation}
such that
\begin{enumerate}
\item\label{itm:whit-mod-str} the module structures over $\Shv_{\ren}(I_{\check{P}}\backslash\Fl_{\check{P}})\simeq\QC^!(\wt{\eN}_{P}\times_{\g}\wt{\eN}_{P}/G)$ are intertwined;
\item\label{itm:small-subcat} the small subcategories of constructible and coherent sheaves are exchanged; 
\item\label{itm:frob} pullback along geometric Frobenius is exchanged with pullback along scaling by $q$;
\item\label{itm:satake-action} the action of $\rmZ_{\check{P}}$ is exchanged with the action of $\Rep(G)^{\heart}$; and
\item\label{itm:perv-t-strs} the perverse t-structure is exchanged with the perverse coherent (exotic) t-structure arising from the equivalence $\QC^!(\wt{\eN}_P/G)\simeq\cA_P^{\op}\dashmod^{G}$.
\end{enumerate}
\end{proposition}

\begin{proof}
The equivalence \eqref{eqn:lang-parab-equiv} and properties \eqref{itm:whit-mod-str} and \eqref{itm:frob} are included in the ``Whittaker-strict'' case of \cite[Thm.~A]{chen-dhillon}. Property \eqref{itm:small-subcat} is (essentially) included in Theorem~B of \loccit\, and property \eqref{itm:satake-action} holds by Proposition~\ref{prop:Z-parab}.

For the final property, we use an argument analogous to one in the proof of Proposition~\ref{prop:Z-parab}. Applying\footnote{See Theorem~6.0.8 of \loccit}
\begin{equation*}
\Shv(\mathring{I}^-,\psi\backslash\check{G}_F/\mathring{I}\dashmon)\tens{\Shv(\mathring{I}\dashmon\backslash\check{G}_F/\mathring{I}\dashmon)}-\,\simeq\,\QC^!_{\eN}(\wt{\g}/G)\tens{\QC^!_{\eN}(\wt{\g}\fibprod{\g}\wt{\g}/G)}-
\end{equation*}
to the left-hand square of \eqref{lang-equiv-com-square} and taking shifts yields a commutative square
\begin{equation}
\label{eqn:anti-sph-LagP-comm-diag}
\begin{tikzcd}
\Shv(\mathring{I}^-,\psi\backslash\Fl)\arrow[r,"\sim"]&\QC(\wt{\eN}/G)\\
\Shv(\mathring{I}^-,\psi\backslash\Fl_{\check{P}})\arrow[r,"\sim"]\arrow{u}{\varpi_{\check{P}}^![-d_{P/B}]}&\QC(\wt{\eN}_P/G).\arrow[u,"{\Lag_{P}^!(\rho)}"']
\end{tikzcd}
\end{equation}
Combining this with \eqref{eqn:A-AP-res-ind-commdiag} yields the commutative square
\begin{equation*}
\begin{tikzcd}
\Shv(\mathring{I}^-,\psi\backslash\Fl)\arrow[r,"\Phi","\sim"']&\cA^{\op}\dashmod^G\\
\Shv(\mathring{I}^-,\psi\backslash\Fl_{\check{P}})\arrow[r,"{\Phi_P}","\sim"']\arrow{u}{\varpi_{\check{P}}^![-d_{P/B}]}&\cA_P^{\op}\dashmod^G.\arrow[u,"\Res^{\cA_P^{\op}}_{\cA^{\op}}"']
\end{tikzcd}
\end{equation*}
By \cite[Thm.~6.2.1]{bez-mirk}, the top horizontal equivalence exchanges the perverse and perverse coherent t-structures on either side (i.e., \eqref{itm:perv-t-strs} holds in the Borel case). Moreover, the functor $\varpi_{\check{P}}^![-d_{P/B}]$ is t-exact, and both vertical functors are conservative.

We claim that $\Res^{\cA_P^{\op}}_{\cA^{\op}}[-d_{P/B}]$ is t-exact with respect to the perverse coherent t-structures. Since $\eN_P$ is the closure of a nilpotent orbit, and the morphism $\pi_P\colon\wt{\eN}_P\to\eN_P$ is generically finite, we have $\dim\eN_P=\dim\wt{\eN}_P=2\dim\eP$. Thus, $\codim_{\eN}\eN_P=2d_{P/B}$, and $\Res^{\cA_P^{\op}}_{\cA^{\op}}[-d_{P/B}]$ is right t-exact by \eqref{eqn:perv-coh-t-str}. For left t-exactness, let $M\in\cA_P\dashmod^{G,\le_p 0}$. We must show that
\begin{equation}
\label{eqn:Ark-Bez-left-t-ex-goal}
\Res^{\cA_P^{\op}}_{\cA^{\op}}\Hom_{\cA_P}(M,\cA_P)[-d_{P/B}]\in\cA^{\op}\dashmod^{G,\ge_p0}.
\end{equation}
Observe that by \eqref{eqn:A-AP-res-ind-commdiag} and \eqref{eqn:Lag-extra-adjns}, we have
\begin{align*}
\Hom_{\cA^{\op}}(\cA_P,\cA)&\simeq\Hom_{\wt{\eN}}((\Lag_P^!\cE_P)(\rho),\cE)\\
&\simeq\Hom_{\wt{\eN}_P}(\cE_P,\Lag_{P,!}\cE(-\rho)[-2d_{P/B}])\\
&\simeq\Hom_{\wt{\eN}_P}(\cE_P,\cE_P)[-2d_{P/B}]\\
&\simeq\cA_P[-2d_{P/B}].
\end{align*}
Thus,
\begin{align*}
\Res^{\cA_P^{\op}}_{\cA^{\op}}\Hom_{\cA_P}(M,\cA_P)[-d_{P/B}]&\simeq\Res^{\cA_P^{\op}}_{\cA^{\op}}\Hom_{\cA_P}(M,\Hom_{\cA^{\op}}(\cA_P,\cA)[2d_{P/B}])[-d_{P/B}]\\
&\simeq\Hom_{\cA}(\Res^{\cA_P}_{\cA}M[-d_{P/B}],\cA).
\end{align*}
As above, the functor $\Res^{\cA_P}_{\cA}[-d_{P/B}]$ is right t-exact with respect to the perverse coherent t-structures, so \eqref{eqn:Ark-Bez-left-t-ex-goal} is coconnective by definition.

Finally, to show \eqref{itm:perv-t-strs}, suppose that $\sF\in\Shv(\mathring{I},\psi\backslash\Fl_{\check{P}})^{\le_p 0}$. Then
\begin{align*}
\Phi(\varpi_{\check{P}}^!\sF[-d_{P/B}])&\simeq\Res^{\cA_P^{\op}}_{\cA^{\op}}\Phi_P(\sF)\\
&\simeq\tau^{\le_p 0}\Res^{\cA_P^{\op}}_{\cA^{\op}}\Phi_P(\sF)\\
&\simeq\Res^{\cA_P^{\op}}_{\cA^{\op}}(\tau^{\le_p d_{P/B}}\Phi_P(\sF))
\end{align*}
by t-exactness of $\Res^{\cA_P^{\op}}_{\cA^{\op}}[-d_{P/B}]$. It follows that $\Res^{\cA_P^{\op}}_{\cA^{\op}}(\tau^{\ge_p d_{P/B}+1}\Phi_P(\sF))\simeq 0$, hence $\tau^{\ge d_{P/B}+1}\Phi_P(\sF)\simeq 0$ by conservativeness, and so $\Phi_P(\sF)$ lies in degrees $\le_p d_{P/B}$. An analogous argument shows that $\Phi_P[d_{P/B}]$ is left t-exact, completing the proof.
\end{proof}

\section{Schur multipliers}
\label{sec:schur-mult}

\begin{atom}
In this appendix, we collect various results on the Schur multiplier of a linear algebraic group $G$, for use in \S\ref{sec:noncomm-spr} and \S\ref{sec:Ze-cov}. Other discussions of this topic may be found in \cite{elagin}, and in \cite{rosengarten} for the case of connected groups. Most statements in this appendix are well-known for finite groups; our task is only to show that they carry over to linear algebraic groups. In particular, we show that $G$ admits a Schur covering, and give various criteria for computing its Schur multiplier.
\end{atom}

\subsection{Cocycles on linear algebraic groups}

\begin{atom}
We first review some general properties of cocycles\footnote{These are sometimes also referred to as ``multiplicative'' or ``translation-invariant'' line bundles.} on $G$ that will be used in the sequel. We begin by recalling the definition, owed to Elagin:
\end{atom}

\begin{definition}[{\cite[Def.~1.4]{elagin}}]
\label{def:coc-elagin}
Let $m\colon G\times G\to G$ denote the multiplication map. A \emph{cocycle on $G$} is the data of a line bundle $\cC$ on $G$ and an isomorphism $\alpha\colon\cC\boxtimes\cC\simeq m^*\cC$ satisfying the following associativity condition: the isomorphisms
\begin{equation}
\label{eqn:coc-assoc}
(\id\times m)^*\alpha\circ(\id\boxtimes\alpha),(m\times\id)^*\alpha\circ(\alpha\boxtimes\id)\colon\cC\boxtimes\cC\boxtimes\cC\simeq(m\circ(\id\times m))^*\cC
\end{equation}
of line bundles on $G\times G\times G$ are equal. A \emph{morphism of cocycles} $(\cC,\alpha)\to(\cC',\alpha')$ is a morphism of line bundles $\cC\to\cC'$ commuting with $\alpha,\alpha'$.
\end{definition}

\begin{notation}
\label{not:schur-mult}
We denote the resulting $(1,1)$-category of cocycles on $G$ by $\Coc(G)$; it carries a natural rigid symmetric monoidal structure under the tensor product of line bundles. We let $\M(G)$ denote the abelian group of isomorphism classes of $\Coc(G)$, and refer to it as the \emph{Schur multiplier} of $G$. Moreover, given a group homomorphism $\varphi\colon G'\to G$, we have a natural monoidal functor $\varphi^*\colon\Coc(G)\to\Coc(G')$ and corresponding restriction homomorphism $\varphi^*\colon\M(G)\to\M(G')$.
\end{notation}

\begin{atom}
We now recall the relationship between Definition~\ref{def:coc-elagin} and the classical notion of Schur multiplier:
\end{atom}

\begin{proposition}
\label{prop:schur-mult-bij-cent-ext}
The $(1,1)$-category $\Coc(G)$ is naturally monoidally equivalent to the $2$-group of central extensions\footnote{Here, morphisms are commutative diagrams of group homomorphisms
\begin{equation*}
\begin{tikzcd}[ampersand replacement=\&]
1\arrow{r}\&\bG_m\arrow{r}\arrow[d,equals]\&G'\arrow{r}\arrow{d}\&G\arrow{r}\arrow[d,equals]\&1\\
1\arrow{r}\&\bG_m\arrow{r}\&G''\arrow{r}\&G\arrow{r}\&1,
\end{tikzcd}
\end{equation*}
which are automatically isomorphisms. Moreover, the monoidal structure is provided by the ``Baer sum'' of extensions, and monoidal inverses are given by inversion in $\bG_m$.} of $G$ by $\bG_m$ (in the category of linear algebraic groups). Moreover, let $(\cC,\alpha)\in\Coc(G)$, and let
\begin{equation}
\label{eqn:coc-cent-ext}
1\to\bG_m\to G_{(\cC,\alpha)}\xrightarrow{p_{(\cC,\alpha)}}G\to 1
\end{equation}
be the associated central extension. Then the isomorphism class $[(\cC,\alpha)]$ lies in the kernel of the restriction map
\begin{equation*}
p_{(\cC,\alpha)}^*\colon\M(G)\to\M(G_{(\cC,\alpha)}).
\end{equation*}
Finally, if $[(\cC,\alpha)]$ has finite order, then \eqref{eqn:coc-cent-ext} canonically descends to a central extension
\begin{equation}
\label{eqn:coc-cent-ext-fin}
1\to X^*(\angles{[(\cC,\alpha)]})\to\overline{G}_{(\cC,\alpha)}\xrightarrow{\overline{p}_{(\cC,\alpha)}} G\to 1.
\end{equation}
\end{proposition}

\begin{proof}
For the first assertion, we recall the construction, and leave the remaining details (some of which are carried out in \cite[\S1]{elagin}) to the reader. Given a cocycle $(\cC,\alpha)\in\M(G)$, we may form the graded $\O_G$-algebra
\begin{equation}
\label{eqn:R-coc-alg}
\cR_{(\cC,\alpha)}:=\bigoplus_{n\in\Z}\cC^{\otimes n}.
\end{equation}
The relative spectrum
\begin{equation*}
G_{(\cC,\alpha)}:=\uSpec_G(\cR_{(\cC,\alpha)})\xrightarrow{p_{(\cC,\alpha)}}G
\end{equation*}
is then a principal $\bG_m$-bundle. Moreover, we may equip $G_{(\cC,\alpha)}$ with a multiplication via the map of $\O_G$-algebras
\begin{equation*}
m^*\cR_{(\cC,\alpha)}\to\cR_{(\cC,\alpha)}\boxtimes\cR_{(\cC,\alpha)}
\end{equation*}
generated by (the inverse of) $\alpha$. The associativity condition \eqref{eqn:coc-assoc} now guarantees associativity of this multiplication, and the remaining properties can be verified similarly. Conversely, given a central extension
\begin{equation*}
1\to\bG_m\to G'\to G\to 1,
\end{equation*}
the coordinate ring $\O_{G'}$ carries a grading induced by the right-regular representation of $\bG_m$, and it is not difficult to show that the weight-$1$ component\footnote{Alternatively, we may take the weight-$(-1)$ component with respect to the left-regular representation.} yields a cocycle on $G$.

For the second assertion, note that we have natural $\cR_{(\cC,\alpha)}$-module isomorphisms
\begin{equation*}
p_{(\cC,\alpha)}^*\cC\simeq\cC\otimes\cR_{(\cC,\alpha)}\simeq\cR_{(\cC,\alpha)},
\end{equation*}
which clearly trivialize $p_{(\cC,\alpha)}^*\alpha$.

Finally, for the third assertion, note that when $[(\cC,\alpha)]$ has finite order $d$, we also have a $\Z/d\Z$-graded $\O_G$-algebra
\begin{equation*}
\overline{\cR}_{(\cC,\alpha)}:=\bigoplus_{0\le n<d}\cC^{\otimes n}
\end{equation*}
via the isomorphism $\cC^{\otimes d}\simeq\O_G$. It is now straightforward to verify that this yields a central extension \eqref{eqn:coc-cent-ext-fin} with the desired property.
\end{proof}

\begin{atom}
In fact, we shall soon see that \emph{every} cocycle on $G$ has finite order. Regardless, in the situation of \eqref{eqn:coc-cent-ext}, we obtain a canonical decomposition\footnote{See for instance \cite[Rem.~2.36]{benzvi}.}
\begin{equation}
\label{eqn:coc-rep-G-decomp}
\Rep(G_{(\cC,\alpha)})\simeq\bigoplus_{n\in X^*(\bG_m)}\Rep(G_{(\cC,\alpha)})_n
\end{equation}
as $\Rep(G)$-module categories (and similarly for $\overline{G}_{(\cC,\alpha)}$), where we have let $n$ denote the $n$th power of the tautological character of $\bG_m$. This allows us to ``twist'' any $\Rep(G)$-module category by a cocycle:
\end{atom}

\begin{definition}
\label{def:coc-twist-cat}
Let $\eC$ be a $\Rep(G)$-module category. The \emph{$(\cC,\alpha)$-twist} of $\eC$ is the $\Rep(G)$-module category
\begin{equation*}
\eC^{(\cC,\alpha)}:=\eC\tens{\Rep(G)}\Rep(G_{(\cC,\alpha)})_1.
\end{equation*}
Likewise, we define the $(\cC,\alpha)$-twist of a small category via the corresponding decomposition of $\Rep(G_{(\cC,\alpha)})^c$.
\end{definition}

\begin{atom}
\label{atm:coc-twist-cat}
Equivalently, we may write $\eC^{(\cC,\alpha)}\simeq\eC\otimes_{\Rep(G)}\Rep(G)^{(\cC,\alpha)}$. Note that by Proposition~\ref{prop:schur-mult-bij-cent-ext}, we have canonical $\Rep(G)$-linear equivalences
\begin{equation}
\label{eqn:rep-G-coc-twist-mult}
\Rep(G)^{(\cC,\alpha)}\tens{\Rep(G)}\Rep(G)^{(\cC',\alpha')}\simeq\Rep(G)^{(\cC,\alpha)\cdot(\cC',\alpha')}
\end{equation}
for any cocycles $(\cC,\alpha),(\cC',\alpha')\in\Coc(G)$ (see also \cite[Prop.~1.5]{elagin}). In particular,
\begin{equation}
\label{eqn:rep-G-coc-twist-comp}
\Rep(G_{(\cC,\alpha)})_n\simeq\Rep(G)^{(\cC,\alpha)^n}
\end{equation}
for each character $n$, and similarly for $\overline{G}_{(\cC,\alpha)}$. Moreover, given a homomorphism $\varphi\colon G'\to G$, the naturality statement yields a $\Rep(G)$-linear functor
\begin{equation}
\label{eqn:rep-G-coc-twist-res}
\Res^{G}_{G'}\colon\Rep(G)^{(\cC,\alpha)}\to\Rep(G')^{\varphi^*(\cC,\alpha)}.
\end{equation}

Two cases of Definition~\ref{def:coc-twist-cat} will be of especial interest to us. Letting $X$ be a scheme with an action of $G$, we refer to the objects of $\QC(X/G)^{(\cC,\alpha)}$ as \emph{$(\cC,\alpha)$-equivariant sheaves} on $X$, and to the objects of $\Rep(G)^{(\cC,\alpha)}$ as \emph{$(\cC,\alpha)$-representations} of $G$.
\end{atom}

\begin{remark}
An alternative definition of these objects which is less amenable to $\infty$-categorical methods, though more explicit, was given in \cite{elagin}; we record it for the reader's intuition. There, a $(\cC,\alpha)$-equivariant sheaf on $X$ was defined as a sheaf $\cF$ on $X$ and an isomorphism $\theta\colon\cC\boxtimes\cF\simeq\act^*\cF$, satisfying a cocycle condition given by commutativity of the following diagram:
\begin{equation*}
\begin{tikzcd}[row sep=tiny,column sep=huge]
&\cC\boxtimes\act^*\cF\arrow{r}{(\id\times\act)^*\theta}&(\act\circ(\id\times\act))^*\cF\arrow[dd,equals]\\
\cC\boxtimes\cC\boxtimes\cF\arrow{ur}{\id\boxtimes\theta}\arrow[dr,"\alpha\boxtimes\id"']&&\\
&m^*\cC\boxtimes\cF\arrow{r}{(m\times\id)^*\theta}&(\act\circ(m\times\id))^*\cF.
\end{tikzcd}
\end{equation*}
In particular, an $(\O_G,\id)$-equivariant sheaf is just a usual $G$-equivariant sheaf. Compatibility of these two notions is shown in \cite[Prop.~1.6]{elagin}.
\end{remark}

\begin{atom}
Finally, we record the following general properties of cocycles on $G$ (which are not needed in the present work).
\end{atom}

\begin{lemma}
\label{lem:coc-ad-eqvt}
Let $(\cC,\alpha)\in\Coc(G)$. Then $\cC$ is canonically equivariant with respect to the adjoint action of $G$, and $\alpha$ descends to an isomorphism of $G$-equivariant sheaves.
\end{lemma}

\begin{proof}
The structure sheaf $\O_{G_{(\cC,\alpha)}}$ is canonically equivariant with respect to the adjoint action of $G_{(\cC,\alpha)}$. This action factors through $G$, and commutes with the action of $\bG_m$ on $G_{(\cC,\alpha)}$ by translations. Thus, each summand of \eqref{eqn:R-coc-alg} is canonically $G$-equivariant, and the conclusion follows.
\end{proof}

\begin{remark}
One can construct this equivariance structure more explicitly via $\alpha$. Specifically, pulling back \eqref{eqn:coc-assoc} along the map
\begin{align*}
G\times G&\to G\times G\times G\\
(g,g')&\mapsto(g,g',g^{-1})
\end{align*}
gives an isomorphism
\begin{equation}
\label{eqn:coc-ad-iso}
(\cC\otimes i^*\cC)\boxtimes\cC\simeq\ad^*\cC,
\end{equation}
where $i\colon G\to G$ denotes the inversion map, and $\ad\colon G\times G\to G$ denotes the adjoint action. Next, pulling back $\alpha$ along the map $G\to G\times G$ given by $g\mapsto(g,g^{-1})$ gives an isomorphism
\begin{equation}
\label{eqn:coc-i-iso}
\cC\otimes i^*\cC\simeq\O_G\tens{k}\cC_e,
\end{equation}
where $e\in G(k)$ denotes the identity element. Finally, pulling back $\alpha$ to $(e,e)$ gives an isomorphism $\cC_e\otimes\cC_e\simeq\cC_e$, so tensoring with $\cC_e^\vee$ yields a canonical isomorphism $\cC_e\simeq k$. Thus, using \eqref{eqn:coc-ad-iso} and \eqref{eqn:coc-i-iso}, we obtain an isomorphism $\O_G\boxtimes\cC\simeq\ad^*\cC$, and one can check using the associativity condition \eqref{eqn:coc-assoc} that this gives a conjugation-equivariant structure on $\cC$ commuting with $\alpha$.
\end{remark}

\begin{lemma}
\label{lem:any-coc-rep-gens}
Let $(\cC,\alpha)\in\Coc(G)$, and let $V$ be a compact $(\cC,\alpha)$-representation. Then $V$ generates $\Rep(G)^{(\cC,\alpha)}$ as a $\Rep(G)$-module. Moreover, if $G$ is reductive\footnote{Recall that we use ``reductive'' to mean ``linearly reductive''; in particular, $G$ is not assumed to be connected.}, then we have canonical isomorphisms
\begin{equation}
\label{eqn:glob-sec-coc-K}
\Gamma(G/G,\cC)\simeq K_0(\Rep(G)^{(\cC,\alpha)})_k\cong R(G)_k\cdot[V]\subset R(G_{(\cC,\alpha)})_k,
\end{equation}
where the inclusion into the $k$-linearized representation ring of $G_{(\cC,\alpha)}$ is via \eqref{eqn:coc-rep-G-decomp}.
\end{lemma}

\begin{proof}
Let $V'\in\Rep(G)^{(\cC,\alpha)}$. Then $V'\otimes V^\vee$ is a $G$-representation by \eqref{eqn:rep-G-coc-twist-mult}, and $V'$ is a summand of $(V'\otimes V^\vee)\otimes V$ as the characteristic of $k$ is $0$. The latter assertion is now immediate from the Peter--Weyl theorem and the proof of Proposition~\ref{prop:schur-mult-bij-cent-ext}.
\end{proof}

\subsection{Computing the Schur multiplier}

\begin{atom}
We now give a series of criteria for computing Schur multipliers under various assumptions on $G$. We begin by recalling how to compute the Schur multiplier of a connected group (in characteristic $0$). This is a very specific case of the main result of \cite{kumar-neeb}, or alternatively, an (unpublished) result of Gabber.
\end{atom}

\begin{proposition}[{\cite[Thm.~0.1]{kumar-neeb}, \cite[Thm.~1.3]{rosengarten}}]
\label{prop:kumar-neeb}
Suppose $G$ is connected. Then
\begin{equation*}
\M(G)\cong X^*(\pi_1([G,G]))\cong\Pic(G),
\end{equation*}
functorially in $G$. In particular, it is finite.
\end{proposition}

\begin{atom}
From this, we deduce the following finiteness result, which will allow us to construct Schur coverings in Proposition~\ref{prop:schur-cov}:
\end{atom}

\begin{corollary}
\label{cor:M-fin}
For any linear algebraic group $G$, the group $\M(G)$ is finite.
\end{corollary}

\begin{proof}
Consider the restriction homomorphism
\begin{equation}
\label{eqn:M-res}
\M(G)\to\M(G^\circ).
\end{equation}
By Proposition~\ref{prop:kumar-neeb}, the latter is finite, so it suffices to show that the kernel of \eqref{eqn:M-res} is finite.

Central extensions in the kernel of \eqref{eqn:M-res} are in bijection with short exact sequences
\begin{equation}
\label{eqn:ses-ker-M}
1\to\bG_m\times G^\circ\to G'\to\pi_0(G)\to 1
\end{equation}
such that the outer action
\begin{equation*}
\pi_0(G)\to\Out(\bG_m\times G^\circ)
\end{equation*}
is of the form
\begin{equation}
\label{eqn:out-matrix}
\begin{pmatrix}
\const_{\id_{\bG_m}}&c\\
0&\psi_G
\end{pmatrix},
\end{equation}
where $\psi_G$ is the original outer action of the extension
\begin{equation*}
1\to G^\circ\to G\to\pi_0(G)\to 1,
\end{equation*}
and $c\colon\pi_0(G)\to X^*(G^\circ)$ is a $1$-cocycle (the action of $\pi_0(G)$ on $X^*(G^\circ)$ is given by pullback along $\psi_G$). Moreover, the short exact sequences \eqref{eqn:ses-ker-M} are considered up to splittings of the extension
\begin{equation*}
1\to\bG_m\to G'^\circ\to G^\circ\to 1,
\end{equation*}
i.e., up to conjugation of \eqref{eqn:out-matrix} by
\begin{equation}
\label{eqn:conj-out-matrix}
\begin{pmatrix}
\id_{\bG_m}&a\\
0&\id_{G^\circ}
\end{pmatrix}
\end{equation}
for $a\in X^*(G^\circ)$. This operation corresponds to subtracting the coboundary corresponding to $a$ from $c$; thus, such outer actions are classified by the cohomology group
\begin{equation}
\label{eqn:H1-ses}
\H^1(\pi_0(G),X^*(G^\circ)).
\end{equation}

Since $\pi_0(G)$ is finite, \eqref{eqn:H1-ses} is torsion; moreover, $X^*(G^\circ)$ is a finite-rank lattice, and therefore \eqref{eqn:H1-ses} is finitely generated, hence finite. Finally, since $\bG_m\times G^\circ$ is linear and $k$ is of characteristic $0$, the set of short exact sequences \eqref{eqn:ses-ker-M} corresponding to any element of \eqref{eqn:H1-ses} is finite by \cite[Thm.~4.1]{arteche}, and so the kernel of \eqref{eqn:M-res} is finite, as desired.
\end{proof}

\begin{atom}
We now recall several useful criteria for computing the Schur multipliers of products, semidirect products, and central extensions, respectively. All of these statements are well-known in the case of finite groups (see for instance \cite[\S2]{hatui-kakkar-yadav}).
\end{atom}

\begin{notation}
Given linear algebraic groups $G,H$, we let $X^*(G,H)$ denote the group of \emph{bimultiplicative} morphisms $G\times H\to\bG_m$ (which in particular, necessarily factor through the abelianizations $G^\ab\times H^\ab$). Equivalently, we have
\begin{equation*}
\Hom_{\Gp}(H,X^*(G))\cong X^*(G,H)\cong\Hom_{\Gp}(G,X^*(H)),
\end{equation*}
i.e., the group of homomorphisms from one group to the Pontryagin dual of the other.
\end{notation}

\begin{lemma}
\label{lem:sch-prod}
Let $G,H$ be linear algebraic groups. There is a functorial isomorphism
\begin{equation}
\label{eqn:sch-prod}
\M(G\times H)\simeq\M(G)\times\M(H)\times X^*(G,H).
\end{equation}
In particular, if $|\pi_0(G^\ab)|$ and $|\pi_0(H^{\ab})|$ are coprime (e.g., if either $G$ or $H$ is connected), then
\begin{equation*}
\M(G\times H)\simeq\M(G)\times\M(H).
\end{equation*}
\end{lemma}

\begin{proof}
Consider the restriction map
\begin{equation*}
\M(G\times H)\to\M(G)\times\M(H).
\end{equation*}
Taking the external tensor product of cocycles immediately impies that it is split surjective. Moreover, a central extension in the kernel is clearly given by a semidirect product
\begin{equation}
\label{eqn:semidirect}
(\bG_m\times G)\rtimes H
\end{equation}
such that the action
\begin{equation*}
H\to\Aut(\bG_m\times G)
\end{equation*}
is of the form
\begin{equation}
\label{eqn:aut-matrix}
\begin{pmatrix}
\const_{\id_{\bG_m}}&c\\
0&\const_{\id_G}
\end{pmatrix},
\end{equation}
where $c\colon H\to X^*(G)$ is a homomorphism. Since the group of homomorphisms of the form \eqref{eqn:aut-matrix} is abelian, quotienting by splittings of the restriction of \eqref{eqn:semidirect} to $G$ (as in \eqref{eqn:conj-out-matrix}) does nothing. Thus, we obtain \eqref{eqn:sch-prod}.

To show that the latter condition implies that $X^*(G,H)$ is trivial, recall that a commutative linear algebraic group is the product of its semisimple and unipotent parts. Thus, we may assume that $G^\ab$ is the product of a torus and a finite abelian group. Since any homomorphism from $H$ to a free group is trivial, we may assume that $G^\ab$ is finite. Thus, it suffices to classify homomorphisms $\pi_0(H^\ab)\to X^*(\pi_0(G^\ab))$, which are all trivial if $|\pi_0(G^\ab)|$ and $|\pi_0(H^\ab)|$ are coprime.
\end{proof}

\begin{lemma}
\label{lem:coc-rtimes}
Let $G\rtimes\Gamma$ be a semidirect product, with either $\Gamma$ finite or $X^*(G)$ trivial.\footnote{In the latter case, we interpret the group cohomology on either end as the trivial group.} Then we have an exact sequence
\begin{equation*}
0\to\H^1(\Gamma,X^*(G))\to\ker\big(\M(G\rtimes\Gamma)\twoheadrightarrow\M(\Gamma)\big)\to\M(G)^\Gamma\to\H^2(\Gamma,X^*(G)),
\end{equation*}
which is suitably functorial in both $G$ and $\Gamma$.
\end{lemma}

\begin{proof}
We begin by showing that the natural restriction map
\begin{equation}
\label{eqn:rtimes-res}
\ker\big(\M(G\rtimes\Gamma)\twoheadrightarrow\M(\Gamma)\big)\to\M(G)
\end{equation}
lands in $\M(G)^\Gamma$. Suppose we have a central extension
\begin{equation}
\label{eqn:cent-ext-rtimes}
1\to\bG_m\to G'\to G\rtimes\Gamma\to 1
\end{equation}
whose pullback to $\Gamma$ is split. Then the projection $G'\to G\rtimes\Gamma\to\Gamma$ admits a section, and we obtain an adjoint action of $\Gamma$ on $G'$ lifting that on $G\rtimes\Gamma$. In particular, this action fixes $G'\times_{G\rtimes\Gamma}G$, and so the restriction of \eqref{eqn:cent-ext-rtimes} to $\M(G)$ is fixed by $\Gamma$.

Moreover, as in \eqref{eqn:semidirect}, the kernel of \eqref{eqn:rtimes-res} is given by semidirect products
\begin{equation*}
(\bG_m\times G)\rtimes\Gamma
\end{equation*}
with outer action as in \eqref{eqn:out-matrix}; as in \eqref{eqn:H1-ses}, these are classified by the group cohomology $\H^1(\Gamma,X^*(G))$.

To construct the final map, let
\begin{equation}
\label{eqn:cent-ext-Gamma-fix}
1\to\bG_m\to G'\to G\to 1
\end{equation}
be a central extension fixed under the action of $\Gamma$, and consider the group $A$ of automorphisms of this central extensions, i.e., commutative diagrams
\begin{equation*}
\begin{tikzcd}
1\arrow{r}&\bG_m\arrow{r}\arrow[d,equals]&G'\arrow{r}\arrow[d,"\sim" vert]&G\arrow{r}\arrow[d,"\sim" vert]&1\\
1\arrow{r}&\bG_m\arrow{r}&G'\arrow{r}&G\arrow{r}&1.
\end{tikzcd}
\end{equation*}
There is an evident exact sequence
\begin{equation}
\label{eqn:aut-cent-ext-ses}
1\to X^*(G)\to A\to\Aut(G),
\end{equation}
and since $\Gamma$ fixes \eqref{eqn:cent-ext-Gamma-fix}, pulling back along $\Gamma\to\Aut(G)$ gives a short exact sequence
\begin{equation}
\label{eqn:Gamma-ses}
1\to X^*(G)\to A\fibprod{\Aut(G)}\Gamma\to\Gamma\to 1
\end{equation}
with the obvious outer action of $\Gamma$ on $X^*(G)$. Thus, \eqref{eqn:Gamma-ses} represents a class in $\H^2(\Gamma,X^*(G))$, and it is straightforward to verify that it depends only on the isomorphism class of \eqref{eqn:cent-ext-Gamma-fix}.

Suppose now that \eqref{eqn:Gamma-ses} is a semi-direct product, i.e., we have a splitting homomorphism $\Gamma\to A$. Then \eqref{eqn:cent-ext-Gamma-fix} lifts to a central extension
\begin{equation*}
1\to\bG_m\to G'\rtimes\Gamma\to G\rtimes\Gamma\to 1,
\end{equation*}
whose restriction to $\M(\Gamma)$ is evidently trivial.

We leave the verification of functoriality for maps $G'\to G$ commuting with the $\Gamma$-actions, and for maps  $\Gamma'\to\Gamma$, as an exercise.
\end{proof}

\begin{lemma}
\label{lem:M-central}
Let $Z\subset G$ be central. Then we have a functorial exact sequence
\begin{equation*}
0\to X^*([G,G]\cap Z)\to\M(G/Z)\to\M(G)\to X^*(G,Z).
\end{equation*}
\end{lemma}

\begin{proof}
Suppose we are given a central extension
\begin{equation}
\label{eqn:Z-cent-ext}
1\to\bG_m\to G'\to G/Z\to 1
\end{equation}
whose pullback to $G$ is trivial. Then we have a short exact sequence
\begin{equation*}
1\to Z\to\bG_m\times G\to G'\to 1,
\end{equation*}
and we see that \eqref{eqn:Z-cent-ext} is determined by an element of $X^*(Z)$. Moreover, two such elements determine isomorphic central extensions if and only if they differ by an element of the image of $X^*(G)\to X^*(Z)$. The exact sequence
\begin{equation*}
1\to[G,G]\cap Z\to Z\to G^\ab
\end{equation*}
of abelian groups then shows that this quotient is isomorphic to $X^*([G,G]\cap Z)$.

To construct the final map, let
\begin{equation}
\label{eqn:G-cent-ext}
1\to\bG_m\to G'\to G\to 1
\end{equation}
be a central extension, let
\begin{equation}
\label{eqn:pullback-Z-cent-ext}
1\to\bG_m\to Z'\to Z\to 1
\end{equation}
be its pullback to $Z$, and let $A$ be as in \eqref{eqn:aut-cent-ext-ses}. The action of $Z'$ on $G'$ via inner automorphisms gives a homomorphism $Z'\to A$ that is trivial on $\bG_m$; moreover, the induced homomorphism $Z\to\Aut(G)$ is trivial by assumption, so we obtain a homomorphism $Z\to X^*(G)$, i.e., an element of $X^*(G,Z)$. It is easy to see that this element depends only on the isomorphism class of \eqref{eqn:G-cent-ext}.

Finally, suppose that this element of $X^*(G,Z)$ is trivial. Then $Z'\subset G'$ is central, and in particular, commutative. Since any injective homomorphism from a torus to a commutative linear algebraic group admits a retract, the extension \eqref{eqn:pullback-Z-cent-ext} is split, and so $Z$ lifts to a central subgroup of $G'$. In particular, \eqref{eqn:G-cent-ext} is pulled back from the central extension
\begin{equation*}
1\to\bG_m\to G'/Z\to G/Z\to 1,
\end{equation*}
as desired.

We leave the verification of functoriality as an exercise.
\end{proof}

\subsection{Schur coverings}

\begin{atom}
We now turn to the construction of Schur coverings of linear algebraic groups; these are (non-canonical) central extensions of $G$ by $\M(G)$ which trivialize all cocycles of $G$.
\end{atom}

\begin{proposition}
\label{prop:schur-cov}
The group $G$ admits a \emph{Schur covering}, i.e., there exists a central extension\footnote{Note that, unlike some authors, we do not require $X^*(\M(G))$ to be contained in the derived subgroup of $G^\sch$.}
\begin{equation}
\label{eqn:G-sch}
1\to X^*(\M(G))\to G^\sch\xrightarrow{p_{\sch}}G\to 1
\end{equation}
such that the map
\begin{equation}
\label{eqn:sch-cov-pullback}
\M(G)\to\M(G^\sch)
\end{equation}
is trivial. Moreover, we have a canonical $\Rep(G)$-linear decomposition
\begin{equation}
\label{eqn:sch-decomp}
\Rep(G^\sch)\simeq\bigoplus_{[(\cC,\alpha)]\in\M(G)}\Rep(G)^{(\cC,\alpha)},
\end{equation}
where the summand $\Rep(G)^{(\cC,\alpha)}$ is equivalent to the full subcategory of representations on which $X^*(\M(G))$ acts through the character given by evaluation at $[(\cC,\alpha)]$.
\end{proposition}

\begin{proof}
Since $\M(G)$ is finite abelian by Corollary~\ref{cor:M-fin}, we may choose a finite set of generators $[(\cC_1,\alpha_1)],\ldots,[(\cC_r,\alpha_r)]$, each of finite order. By Proposition~\ref{prop:schur-mult-bij-cent-ext}, we have central extensions
\begin{equation*}
1\to X^*(\angles{[(\cC_i,\alpha_i)]})\to\overline{G}_{(\cC_i,\alpha_i)}\to G\to 1.
\end{equation*}
for each $i$. Thus, taking fiber products over $G$ gives a central extension
\begin{equation*}
1\to X^*(\angles{[(\cC_1,\alpha_1)]})\times\cdots\times X^*(\angles{[(\cC_r,\alpha_r)]})\to\overline{G}_{(\cC_1,\alpha_1)}\fibprod{G}\cdots\fibprod{G}\overline{G}_{(\cC_r,\alpha_r)}\to G\to 1,
\end{equation*}
which is precisely \eqref{eqn:G-sch}. Applying Proposition~\ref{prop:schur-mult-bij-cent-ext} to the factorizations
\begin{equation*}
\begin{tikzcd}
G^\sch\arrow[r,two heads]&[-1em]\overline{G}_{(\cC_i,\alpha_i)}\arrow[r,two heads,"p_{(\cC_i,\alpha_i)}"]&[1em]G
\end{tikzcd}
\end{equation*}
of $p_\sch$ then gives \eqref{eqn:sch-cov-pullback}. Finally, applying \eqref{eqn:rep-G-coc-twist-mult} and \eqref{eqn:rep-G-coc-twist-comp} to the decomposition
\begin{equation*}
\Rep(G^\sch)\simeq\Rep(\overline{G}_{(\cC_1,\alpha_1)})\tens{\Rep(G)}\cdots\tens{\Rep(G)}\Rep(\overline{G}_{(\cC_r,\alpha_r)})
\end{equation*}
yields \eqref{eqn:sch-decomp}.
\end{proof}

\begin{atom}
The following lemma allows us to construct ``small'' $(\cC,\alpha)$-representations of connected semisimple groups, and will be used in \S\ref{sec:noncomm-spr} to remove the simply-connectedness assumption for the noncommutative Springer resolution.
\end{atom}

\begin{lemma}
\label{lem:sch-min-rep}
If $G$ is connected and semisimple, then any Schur covering $G^\sch$ is isomorphic to the universal cover of $G$. Moreover, for each $(\cC,\alpha)\in\Coc(G)$, there exists a minuscule weight $\lambda$ of $G^\sch$ such that the irreducible $G^\sch$-representation $V_{\lambda}$ descends to a $(\cC,\alpha)$-representation of $G$.
\end{lemma}

\begin{proof}
By Proposition~\ref{prop:kumar-neeb}, we have $X^*(\M(G))\simeq\pi_1(G)$, from which the first assertion is immediate. For the latter, let $T\subset G$ be a maximal torus, and recall that
\begin{equation*}
\pi_1(G)\cong X^*(T)^\vee/\angles{\Phi^\vee},
\end{equation*}
where $\angles{\Phi^\vee}$ denotes the coroot lattice. Thus,
\begin{equation}
\label{eqn:fund-gp-dual}
\M(G)\cong X^*(\pi_1(G))\cong\Lambda/X^*(T),
\end{equation}
where $\Lambda$ denotes the abstract weight lattice (of the root system of $G$), and for any dominant $\lambda\in\Lambda^+$, the character of $\pi_1(G)$ on the $G^\sch$-representation $V_{\lambda}$ is given by the image of $\lambda$ in \eqref{eqn:fund-gp-dual}. Recall that $\Lambda/\angles{\Phi}$ is canonically isomorphic to the set of minuscule weights, where $\angles{\Phi}$ denotes the root lattice (namely, take the unique minimal dominant weight lifting any element of the former set). Since $\Lambda/\angles{\Phi}\twoheadrightarrow\Lambda/X^*(T)$, it suffices by \eqref{eqn:sch-decomp} to choose any lift of $[(\cC,\alpha)]$ to $\Lambda/\angles{\Phi}$.
\end{proof}

\begin{atom}
Finally, the following two lemmas will be used in \S\ref{sec:Ze-cov} to show that a certain covering group has simply-connected derived subgroup.
\end{atom}

\begin{lemma}
\label{lem:sch-cov-conn-red}
If $G$ is reductive (resp., connected), then so is any Schur covering $G^\sch$. When $G$ is both connected and reductive, $G^\sch$ has simply-connected derived subgroup.
\end{lemma}

\begin{proof}
The first assertion holds as any extension of reductive groups is reductive. For the second assertion, we first reduce to the case where $G$ is reductive. By \cite[Cor.~14.11]{borel-lag}, the surjection $p_\sch\colon G^{\sch,\circ}\twoheadrightarrow G$ induces a surjection $p_\sch\colon \rmR_uG^{\sch,\circ}\twoheadrightarrow\rmR_uG$ of unipotent radicals. Since the kernel of this map is finite, it is an isomorphism. Note that both are normal subgroups; we claim that
\begin{equation*}
G^\sch/\rmR_uG^{\sch,\circ}\to G/\rmR_uG
\end{equation*}
is a Schur covering. Denote these groups by $G^{\sch,\red}$ and $G^\red$, respectively, and consider the commutative diagram
\begin{equation}
\label{eqn:M-G-sch-red-diag}
\begin{tikzcd}
\M(G^\red)\arrow{r}\arrow{d}&\M(G^{\sch,\red})\arrow{d}\\
\M(G)\arrow{r}{0}&\M(G^\sch).
\end{tikzcd}
\end{equation}
It suffices to show that the upper horizontal map is trivial. By \cite[Prop.~5.4.1]{conrad}, we have
\begin{equation*}
[G,G]\cong\rmR_u([G,G])\rtimes[G^\red,G^\red],
\end{equation*}
hence $\pi_1([G,G])\cong\pi_1([G^\red,G^\red])$ (as the two spaces are related via an affine fibration). Proposition~\ref{prop:kumar-neeb} then implies that the left-most vertical map in \eqref{eqn:M-G-sch-red-diag} is an isomorphism. Moreover, Proposition~\ref{lem:coc-rtimes} and a further application of \cite[Prop.~5.4.1]{conrad} imply that the right-most vertical map is also an isomorphism (note that $\M(\rmR_uG)$ is trivial by Proposition~\ref{prop:kumar-neeb}). This proves the claim.

So suppose that $G$ is reductive. By the previous assertion, $G^{\sch,\circ}$ is reductive; thus, $[G^{\sch,\circ},G^{\sch,\circ}]$ is connected and semisimple, and its projection to $[G,G]$ is an isogeny. Lemma~\ref{lem:sch-min-rep} therefore gives a unique factorization
\begin{equation}
\label{eqn:sch-der-diag}
\begin{tikzcd}[column sep=tiny]
{[G,G]^\sch}\arrow[rr,dashed,two heads]\arrow[dr,two heads]&&{[G^{\sch,\circ},G^{\sch,\circ}]}\arrow[dl,two heads]\\
&{[G,G]}.&
\end{tikzcd}
\end{equation}
Moreover, by Proposition~\ref{prop:kumar-neeb}, we have a commutative diagram
\begin{equation*}
\begin{tikzcd}
\M([G,G])\arrow{d}{\vertsim}\arrow{r}&\M([G^{\sch,\circ},G^{\sch,\circ}])\arrow{d}{\vertsim}\\
\M(G)\arrow{r}{0}&\M(G^\sch).
\end{tikzcd}
\end{equation*}
In particular, the map $\pi_1([G^{\sch,\circ},G^{\sch,\circ}])\to\pi_1([G,G])$ is trivial, so by \eqref{eqn:sch-der-diag}, we have $[G,G]^\sch\cong[G^{\sch,\circ},G^{\sch,\circ}]$. The kernel of the isogeny \eqref{eqn:sch-der-diag} then has cardinality at least $|\pi_1([G,G])|=|X^*(\M(G))|$, so the same is true of the kernel of $p_\sch|_{G^{\sch,\circ}}$, and therefore $G^\sch\cong G^{\sch,\circ}$, as desired. The final assertion now follows from \eqref{eqn:sch-der-diag}.
\end{proof}

\begin{lemma}
\label{lem:sch-cov-prod}
Let $G,H$ be linear algebraic groups. Then for any Schur covering $(G\times H)^\sch$, there exist Schur coverings $G^\sch,H^\sch$ and a commutative diagram
\begin{equation*}
\begin{tikzcd}[column sep=tiny]
(G\times H)^\sch\arrow[rr,dashed]\arrow[dr,"p_\sch"']&&G^\sch\times H^\sch\arrow{dl}{p_{\sch}\times p_{\sch}}\\
&G\times H.&
\end{tikzcd}
\end{equation*}
\end{lemma}

\begin{proof}
By Lemma~\ref{lem:sch-prod}, we have a central extension
\begin{equation*}
1\to X^*(\M(G))\times X^*(\M(H))\to (G\times H)^\sch/X^*(X^*(G,H))\to G\times H\to 1.
\end{equation*}
The proof of Proposition~\ref{prop:schur-mult-bij-cent-ext} now easily implies that this splits as a product of central extensions \eqref{eqn:G-sch} for $G$ and $H$.
\end{proof}

\section{The trace formalism}
\label{sec:tr-form-chap}

\begin{atom}
In this mostly expository appendix, we review material on $2$-categorical traces, or ``categorical Hochschild homology,'' which will be needed throughout the main article (particularly \S\ref{sec:part-aff-hecke}, \S\ref{sec:bg-sheaf}, and \S\ref{sec:comp-coh-spr}). Section~\ref{subsec:tr-cat} focuses on the general trace formalism of \cite[\S3]{gkrv}. The only ``new'' material is explication of some adjunctions between categorical traces after \S\ref{atm:tr-adj}. Section~\ref{subsec:tr-qc} collects facts about loop spaces of stacks and traces of (symmetric monoidal) categories of quasicoherent sheaves. Finally, Section~\ref{subsec:tr-conv} recalls results of Ben-Zvi \etal\ on traces of convolution categories \cite[\S3]{benzvi}. The only ``new'' material in the latter two sections is explicit identification of the induction-restriction adjunction in \S\ref{subsec:tr-cat}, and a useful criterion for computing the universal trace functor (Corollary~\ref{cor:reg-mod-rt-left-dual}).
\end{atom}

\subsection{Categorical traces}
\label{subsec:tr-cat}

\begin{atom}
\label{atm:dg-cat-assume}
Let $\dgCat_k$ denote the $(\infty,2)$-category with objects given by (presentable) cocomplete stable $\infty$-categories equipped with a $\Vect_k$-module structure (with respect to the Lurie tensor product), $1$-morphisms given by continuous (i.e., colimit-preserving) functors, and $2$-morphisms given by natural transformations. The Lurie tensor product (over $\Vect_k$) endows $\dgCat_k$ with a symmetric monoidal structure with unit object $\Vect_k$. As in \S\ref{sec:assumptions-notation}, We refer to objects of $\dgCat_k$ as ``$k$-linear dg-categories,'' or simply ``dg-categories.''
\end{atom}

\begin{atom}
We denote by $\Morita(\dgCat_k)$ the $(\infty,2)$-category whose objects are symbols $\underline{\eA\dashbmod}$, where $\eA$ is a monoidal dg-category, and whose $1$-morphisms are given by the $(\infty,1)$-category
\begin{equation*}
\bfMap_{\Morita(\dgCat_k)}(\underline{\eA\dashbmod},\underline{\eB\dashbmod}):=\eB\otimes\eA^\rev\dashbmod,
\end{equation*}
where $\eA^\rev$ denotes $\eA$ with the opposite monoidal structure. The composition law is given by the tensor product of bimodules, and the unit $1$-morphism of an object $\underline{\eA\dashbmod}$ is given by the regular bimodule $\eA$.

Moreover, $\Morita(\dgCat_k)$ carries a symmetric monoidal structure via
\begin{equation*}
\underline{\eA\dashbmod}\otimes\underline{\eB\dashbmod}:=\underline{\eA\otimes\eB\dashbmod},
\end{equation*}
with unit object $\underline{\Vect_k\dashbmod}$. Note that we have a canonical identification
\begin{equation}
\label{eqn:end-id-morita-dgcat}
\bfEnd_{\Morita(\dgCat_k)}(\underline{\Vect_k\dashbmod})\simeq\dgCat_k.
\end{equation}
Moreover, any object $\underline{\eA\dashbmod}$ of $\Morita(\dgCat_k)$ is dualizable, with dual given by $\underline{\eA^\rev\dashbmod}$. Indeed, the unit and counit of this duality are given by
\begin{equation*}
\underline{\Vect_k\dashbmod}\xrightarrow{\eA}\underline{\eA\otimes\eA^\rev\dashbmod}\xrightarrow{\eA}\underline{\Vect_k\dashbmod},
\end{equation*}
respectively. The dual $1$-morphism to any $\eM\colon\underline{\eA\dashbmod}\to\underline{\eB\dashbmod}$ is then given by the same bimodule, i.e., by $\eM\colon\underline{\eB^\rev\dashbmod}\to\underline{\eA^\rev\dashbmod}$.
\end{atom}

\begin{atom}
\label{atm:l-morita-rgd}

We denote\footnote{But do not explain the general meaning of these notations; we refer the reader to \loccit\ for the details.} by $L(\Morita(\dgCat_k))_\rgd$ the following $(\infty,2)$-category:
\begin{enumerate}
\item An object of $L(\Morita(\dgCat_k))_\rgd$ is a pair $(\underline{\eA\dashbmod},\eP)$, where
\begin{enumerate}
\item $\eA$ is a monoidal dg-category; and
\item $\eP$ is an $\eA$-bimodule.
\end{enumerate}
\item Given objects $(\underline{\eA\dashbmod},\eP)$ and $(\underline{\eB\dashbmod},\eQ)$ of $L(\Morita(\dgCat_k))_\rgd$, a $1$-morphism
\begin{equation*}
(\underline{\eA\dashbmod},\eP)\to(\underline{\eB\dashbmod},\eQ)
\end{equation*}
in $L(\Morita(\dgCat_k))_\rgd$ is a pair $(\eM,\alpha)$, where
\begin{enumerate}
\item $\eM$ is a right-dualizable $(\eB,\eA)$-bimodule $\eM$; and
\item $\alpha\colon\eM\otimes_{\eA}\eP\to\eQ\otimes_{\eB}\eM$ is a $(\eB,\eA)$-linear functor.
\end{enumerate}
The composition law is the evident extension of that for $\Morita(\dgCat_k)$, and the unit $1$-morphism is given by the regular bimodule and canonical equivalence.
\item Given $1$-morphisms $(\eM,\alpha),(\eN,\beta)\colon(\underline{\eA\dashbmod},\eP)\to(\underline{\eB\dashbmod},\eQ)$, a $2$-morphism
\begin{equation*}
(\eM,\alpha)\Rightarrow(\eN,\beta)
\end{equation*}
is a pair $(\gamma,\theta)$, where
\begin{enumerate}
\item $\gamma\colon\eM\to\eN$ is a $(\eB,\eA)$-linear functor admitting a continuous $(\eB,\eA)$-linear right adjoint; and
\item $\theta\colon(\id_{\eQ}\otimes\gamma)\circ\alpha\Rightarrow\beta\circ(\gamma\otimes\id_{\eP})$ is a natural transformation of functors $\eM\otimes_{\eA}\eP\to\eQ\otimes_{\eB}\eN$, which we represent diagramatically as
\begin{equation*}
\begin{tikzcd}
\eM\tens{\eA}\eP\arrow{r}{\alpha}\arrow[d,"\gamma\otimes\id_{\eP}"']&\eQ\tens{\eB}\eM\arrow{d}{\id_{\eQ}\otimes\gamma}\arrow[dl,Rightarrow,"\theta"']\\
\eN\tens{\eA}\eP\arrow{r}{\beta}&\eQ\tens{\eB}\eN.
\end{tikzcd}
\end{equation*}
\end{enumerate}
The composition law is the evident extension of that for $\Morita(\dgCat_k)$, and the unit $2$-morphism is given by the identity functor and natural transformation.
\end{enumerate}
Moreover, $L(\Morita(\dgCat_k))_\rgd$ carries a symmetric monoidal structure via
\begin{equation*}
(\underline{\eA\dashbmod},\eP)\otimes(\underline{\eB\dashbmod},\eQ):=(\underline{\eA\otimes\eB\dashbmod},\eP\otimes\eQ),
\end{equation*}
with unit object $(\underline{\Vect_k\dashbmod},\Vect_k)$. There is an evident forgetful symmetric monoidal functor
\begin{equation*}
L(\Morita(\dgCat_k))_\rgd\to\Morita(\dgCat_k).
\end{equation*}
\end{atom}

\begin{atom}
The $(\infty,2)$-category $L(\Morita(\dgCat_k))_\rgd$ is equipped with a canonical symmetric monoidal functor of $(\infty,2)$-categories
\begin{equation}
\label{eqn:L-morita-tr-funct}
\Tr\colon L(\Morita(\dgCat_k))_{\rgd}\to\dgCat_k,
\end{equation}
known as the \emph{$2$-categorical trace}. We recall (some of) its construction. Given an object $(\underline{\eA\dashbmod},\eP)$ of $L(\Morita(\dgCat_k))_{\rgd}$, we set $\Tr(\underline{\eA\dashbmod},\eP)$ to be the composition
\begin{equation*}
\underline{\Vect_k\dashbmod}\xrightarrow{\eA}\underline{\eA\otimes\eA^\rev\dashbmod}\xrightarrow{\eP\otimes\eA^\rev}\underline{\eA\otimes\eA^\rev\dashbmod}\xrightarrow{\eA}\underline{\Vect_k\dashbmod}
\end{equation*}
in $\Morita(\dgCat_k)$, that is,
\begin{equation*}
\Tr(\underline{\eA\dashbmod},\eP):=\eA\tens{\eA\otimes\eA^\rev}\eP\in\dgCat_k.
\end{equation*}
via \eqref{eqn:end-id-morita-dgcat}. Thus, we are simply imitating the usual construction of ``trace'' in a symmetric monoidal category, one categorical level higher. As in the case of classical algebras, we adopt the shorthand notation $\Tr(\eA,\eP):=\Tr(\underline{\eA\dashbmod},\eP)$, and refer to it as the \emph{categorical trace} (or \emph{categorical Hochschild homology}) of $\eP$.

Moreover, given a $1$-morphism
\begin{equation*}
(\eM,\alpha)\colon(\underline{\eA\dashbmod},\eP)\to(\underline{\eB\dashbmod},\eQ),
\end{equation*}
the $1$-morphism
\begin{equation*}
\Tr(\eM,\alpha)\colon\Tr(\eA,\eP)\to\Tr(\eB,\eQ)
\end{equation*}
is given by the composition
\begin{equation}
\label{eqn:2-mor-comp-L-morita}
\begin{tikzcd}[column sep=large,row sep=large]
\underline{\Vect_k\dashbmod}\arrow{r}{\eA}\arrow{d}{\Vect_k}&\underline{\eA\otimes\eA^\rev\dashbmod}\arrow{r}{\eP\otimes\eA^\rev}\arrow{d}{\eM\otimes\eM^R}\arrow[dl,Rightarrow,"\epsilon_{\eM}"']&\underline{\eA\otimes\eA^\rev\dashbmod}\arrow{r}{\eA}\arrow{d}{\eM\otimes\eM^R}\arrow[dl,Rightarrow,"\alpha"']&\underline{\Vect_k\dashbmod}\arrow{d}{\Vect_k}\arrow[dl,Rightarrow,"\eta_{\eM}"']\\
\underline{\Vect_k\dashbmod}\arrow{r}{\eB}&\underline{\eB\otimes\eB^{\rev}\dashbmod}\arrow{r}{\eQ\otimes\eB^\rev}&\underline{\eB\otimes\eB^{\rev}\dashbmod}\arrow{r}{\eB}&\underline{\Vect_k\dashbmod},
\end{tikzcd}
\end{equation}
of $2$-morphisms in $\Morita(\dgCat_k)$ (i.e., we obtain a $2$-morphism from the outer clockwise circuit to the outer counterclockwise circuit). Here $\eM^R$ denotes the right-dual to $\eM$, and
\begin{equation}
\label{eqn:unit-counit-rt-dual-bimod}
\eta_{\eM}\colon\eA\to\eM^R\tens{\eB}\eM,\qquad\qquad\epsilon_{\eM}\colon\eM\tens{\eA}\eM^R\to\eB
\end{equation}
are the ($\eA$-bilinear) unit and ($\eB$-bilinear) counit of this duality, respectively. We refer the reader to \cite[\S3.9]{gkrv} for the description of the $2$-categorical trace of a $2$-morphism. Finally, we note that \eqref{eqn:L-morita-tr-funct} satisfies the usual cyclicity properties for $1$- and $2$-morphisms; we refer the reader to \cite[\S3.1.4,~\S3.2.4]{gkrv} for the details.
\end{atom}

\begin{atom}
\label{atm:mor-L-Morita}

Given $(\underline{\eA\dashbmod},\eP)\in L(\Morita(\dgCat_k))_\rgd$, we define the $(\infty,1)$-category
\begin{equation*}
(\eA,\eP)\dashbmod:=\bfMap_{L(\Morita(\dgCat_k))_\rgd}((\underline{\Vect_k\dashbmod},\Vect_k),(\underline{\eA\dashbmod},\eP)).
\end{equation*}
The functor $\Tr$ of \eqref{eqn:L-morita-tr-funct} then restricts to an $(\infty,1)$-functor
\begin{equation}
\label{eqn:2-cat-class-map}
[-]\colon(\eA,\eP)\dashbmod\to\Map_{\dgCat_k}(\Vect_k,\Tr(\eA,\eP))\simeq\Tr(\eA,\eP),
\end{equation}
which we term the \emph{$2$-categorical class map}. Let us spell this construction out in more detail.

As in \S\ref{atm:l-morita-rgd}, an object of $(\eA,\eP)\dashbmod$ is the same as a right-dualizable (left) $\eA$-module category $\eM$ equipped with an $\eA$-module homomorphism $\alpha\colon\eM\to\eP\otimes_{\eA}\eM$. We refer to the resulting object $[\eM,\alpha]\in\Tr(\eA,\eP)$ as the \emph{$2$-categorical class} of $(\eM,\alpha)$. Moreover, given another such pair $(\eN,\beta)$, a $1$-morphism $(\eM,\alpha)\to(\eN,\beta)$ in $(\eA,\eP)\dashbmod$ is the same as an $\eA$-linear functor $\gamma\colon\eM\to\eN$ admitting an $\eA$-linear right adjoint, and a natural transformation $\theta\colon(\id_{\eP}\otimes\gamma)\circ\alpha\Rightarrow\beta\circ\gamma$ of functors $\eM\to\eP\otimes_{\eA}\eN$. We denote the resulting morphism between $2$-categorical classes by $[\gamma,\theta]\colon[\eM,\alpha]\to[\eN,\beta]$.
\end{atom}

\begin{atom}
\label{atm:tr-A-F-cat-end}
In particular, for a monoidal category $\eA$ equipped with a monoidal endofunctor $F_{\eA}$, we set
\begin{equation*}
\Tr(\eA,F_{\eA}):=\Tr(\eA,{}_{F_{\eA}}\eA):=\eA\tens{\eA\otimes\eA^{\rev}}{}_{F_{\eA}}\eA,
\end{equation*}
and refer to it simply as the \emph{categorical trace} of $\eA$ (with respect to the endofunctor $F_{\eA}$). We will sometimes also write $(\eA\dashbmod,F_{\eA}):=(\eA\dashbmod,{}_{F_{\eA}}\eA)$. Objects of $(\eA,F_{\eA})\dashbmod$ are now pairs $(\eM,F_{\eM})$ as before; we refer to the $\eA$-module homomorphism $F_\eM\colon\eM\to{}_{F_{\eA}}\eM$ as an \emph{$F_{\eA}$-semilinear} endofunctor. When $F_{\eA}$ or $F_{\eM}$ is the identity, we will often omit it from the notation.

Note that we also have a canonical ``universal trace'' functor
\begin{equation}
\label{eqn:tr-char}
[-]\colon\eA\to\eA\tens{\eA\otimes\eA^{\rev}}{}_{F_{\eA}}\eA
\end{equation}
given by sending $a\in\eA$ to the image of $a\otimes\1_{\eA}$, where the latter denotes the monoidal unit of $\eA$. This functor factors through \eqref{eqn:2-cat-class-map} as follows: given any $a\in\eA$, we may define an $F_{\eA}$-semilinear endofunctor $F_{\eA,a}(-):=F_{\eA}(-)\otimes a$ of the regular $\eA$-module $\eA$. It is then not hard to see that $[a]=[\eA,F_{\eA,a}]$ via \eqref{eqn:tr-char}. In particular, the trace of the monoidal unit agrees with that of the regular representation, that is, $[\1_\eA]=[\eA,F_{\eA}]$.
\end{atom}

\begin{atom}
\label{atm:mon-cat-rgd-def}
Recall that a monoidal functor $\Psi\colon\eA\to\eB$ is \emph{rigid} if
\begin{enumerate}
\item the right-adjoint $\Psi^R$ is continuous and $\eA$-bilinear; and
\item the multiplication map $\mult_{\eB}\colon\eB\otimes_{\eA}\eB\to\eB$ admits a continuous, $\eB$-bilinear right adjoint.
\end{enumerate}
In particular, a monoidal category $\eA$ is rigid if the unit functor $\unit_{\eA}\colon\Vect_k\to\eA$ is rigid. If $\eA$ is compactly generated, then by \cite[Lem.~9.1.5]{gr}, this is equivalent to requiring
\begin{enumerate}
\item the unit $\1_{\eA}$ is compact;
\item the multiplication $\mult_{\eA}\colon\eA\otimes\eA\to\eA$ preserves compact objects; and
\item every compact object of $\eA$ admits a left and right monoidal dual.
\end{enumerate}
Thus, in this case, we recover a more traditional notion of rigidity. Finally, note that \emph{any} monoidal functor between rigid monoidal categories is itself rigid.\footnote{See \cite[\S3.10.3]{gkrv}. In more detail, this follows from \cite[Ch.~1, Lem.~9.2.6(b)]{gr}, and by imitating the proof of \cite[Ch.~1, Prop.~8.7.2]{gr}.}.
\end{atom}

\begin{atom}
\label{atm:mon-funct-dual}
The notion of rigidity also has an interpretation in terms of dualizability of \emph{categories}. Given any monoidal functor $\Psi\colon\eA\to\eB$, the $(\eB,\eA)$-bimodule $\Ind_{\Psi}:=\eB_\Psi$ is right-dualizable. Specifically, its right-dual is given by the analogous $(\eA,\eB)$-bimodule $\Res_{\Psi}:={}_{\Psi}\eB$, with the unit given by
\begin{equation}
\label{eqn:mon-coev}
\eA\xrightarrow{\Psi}{}_\Psi\eB_\Psi\simeq \Res_{\Psi}\tens{\eB}\Ind_{\Psi}
\end{equation}
and the counit given by the multiplication map
\begin{equation}
\label{eqn:mon-ev}
\mult_{\eB}\colon\Ind_{\Psi}\tens{\eA}\Res_{\Psi}\to\eB.
\end{equation}
If $\Psi$ is rigid, then $\Res_{\Psi}$ is moreover left-dual to $\Ind_{\Psi}$, via the right-adjoints to \eqref{eqn:mon-coev} and \eqref{eqn:mon-ev}.
\end{atom}

\begin{atom}
\label{atm:hh-dg-cat}
Given a dualizable dg-category $\eC$ with an endofunctor $F_{\eC}$, the \emph{Hochschild homology} $\HH(\eC,F_{\eC})$ is the $2$-categorical class $[\eC,F_{\eC}]\in\Tr(\Vect_k)\in\Vect_k$. Equivalently, by \eqref{eqn:2-mor-comp-L-morita}, it is the composition
\begin{equation}
\label{eqn:hh-duality-data}
\Vect_k\xrightarrow{\eta_{\eC}}\eC\otimes\eC^\vee\xrightarrow{F_{\eC}\otimes\id_{\eC^\vee}}\eC\otimes\eC^\vee\xrightarrow{\epsilon_{\eC}}\Vect_k,
\end{equation}
where $\eta_{\eC},\epsilon_{\eC}$ are as in \eqref{eqn:unit-counit-rt-dual-bimod}. Moreover, its functoriality in the pair $(\eC,F_{\eC})$ is as in \S\ref{atm:mor-L-Morita}. This has the following consequences:
\begin{enumerate}
\item For any compact object $c\in\eC^c$ equipped with a morphism $\theta\colon c\to F_{\eC}(c)$, we obtain a map
\begin{equation*}
[c,\theta]\colon k\simeq[\Vect_k,\id_{\Vect_k}]\to[\eC,F_{\eC}]
\end{equation*}
upon interpreting $c$ as a functor $\Vect_k\to\eC$ and $\theta$ as a natural transformation. We likewise denote the image of the unit under this map by $[c,\theta]\in\HH(\eC,F_{\eC})$, and refer to this element as the \emph{(Hochschild) class} of $(c,\theta)$ as in \eqref{eqn:2-cat-class-map}. In particular, when $F_{\eC}$ is the identity, we write $[c]:=[c,\id_c]$.
\item Let $\eA$ be a monoidal category equipped with a monoidal endofunctor $F_{\eA}$, and suppose that $\unit_{\eA}$, $\mult_{\eA}$, and $F_{\eA}$ all admit continuous right adjoints (for instance, this holds if $\eA$ is compactly generated and all preserve $\eA^c$). Then the Hochschild homology $\HH(\eA,F_{\eA})$ inherits an algebra structure with unit $[\1_{\eA}]$ (omitting the unit isomorphism of $F_{\eA}$). In particular, this holds when $\eA$ is rigid, in which case $F_{\eA}$ is itself automatically rigid as in \S\ref{atm:mon-cat-rgd-def}.
\end{enumerate}

We now state the main result of \cite[\S3]{gkrv}, and one of our primary technical tools:
\end{atom}

\begin{theorem}[{\cite[Thm.~3.8.5]{gkrv}}]
\label{thm:gkrv-main}
Let $\eA$ be a rigid monoidal category equipped with a monoidal endofunctor $F_{\eA}$. Then there is an equivalence of algebras\footnote{We comment on the presence of the opposite algebra, which does not appear in the original formulation of \cite{gkrv}. The algebra isomorphism \eqref{eqn:gkrv-alg} arises from the adjunction \eqref{eqn:tr-adj-ind-res} as follows: on the one hand, the monad $\Tr(\Res_{\unit_{\eA}}\circ\Ind_{\unit_{\eA}})\in\Alg(\Vect_k)$ identifies with $\HH(\eA,F_{\eA})$ as in \S\ref{atm:hh-dg-cat}. On the other hand, by functoriality of \eqref{eqn:L-morita-tr-funct}, this monad may equivalently be expressed as $\Tr(\Res_{\unit_{\eA}})\circ\Tr(\Ind_{\unit_{\eA}})$. Note that by definition, the functor $\Tr(\Ind_{\unit_{\eA}})$ is given by the $\Vect_k$-action on $[\eA,F_{\eA}]$, and hence its right-adjoint $\Tr(\Res_{\unit_{\eA}})$ is given by $\Hom_{\Tr(\eA,F_{\eA})}([\eA,F_{\eA}],-)$. It is now straightforward to verify that the algebra structure on
\begin{equation*}
\Tr(\Res_{\unit_{\eA}})\circ\Tr(\Ind_{\unit_{\eA}})\simeq\Hom_{\Tr(\eA,F_{\eA})}([\eA,F_{\eA}],[\eA,F_{\eA}])
\end{equation*}
is the opposite one.}
\begin{equation}
\label{eqn:gkrv-alg}
\HH(\eA,F_{\eA})\simeq\End_{\Tr(\eA,F_{\eA})}([\eA,F_{\eA}])^\op,
\end{equation}
which extends to an equivalence of functors
\begin{equation}
\label{eqn:gkrv-funct}
\HH(-)\simeq\Hom_{\Tr(\eA,F_{\eA})}([\eA,F_{\eA}],[-]):(\eA,F_{\eA})\dashbmod\to\HH(\eA,F_{\eA})\dashmod.
\end{equation}
In particular, if $[\eA,F_{\eA}]$ is compact, then the left adjoint to $\Hom_{\Tr(\eA)}([\eA,F_{\eA}],-)$ defines a fully-faithful embedding which preserves compact objects, and whose essential image is the category generated by $[\eA,F_{\eA}]$:
\begin{equation}
\label{eqn:2-cat-class-HH-adj}
\begin{tikzcd}[row sep=huge]
\HH(\eA,F_{\eA})\dashmod\arrow{dr}{\simeq}\arrow[rr,hook,"{[\eA,F_{\eA}]\otimes_{\End([\eA,F_{\eA}])}-}",shift left]&&\Tr(\eA,F_{\eA})\arrow[ll,"{\Hom([\eA,F_{\eA}],-)}",shift left]\arrow[dl,hook,shift left,"\pr_{[\eA,F_{\eA}]}"]\\
&\angles{[\eA,F_{\eA}]}\arrow[ur,hook,shift left,"\tin_{[\eA,F_{\eA}]}",near start].&
\end{tikzcd}
\end{equation}
\end{theorem}

\begin{atom}
\label{atm:tr-M-subcat}
Note that if $\eA$ is compactly generated, then by \cite[Ch.~1, Cor.~8.7.4]{gr}, the category $\Tr(\eA,F_{\eA})$ is compactly generated by objects of the form $[a]$, where $a\in\eA^c$. In particular, the object $[\eA,F_{\eA}]$ is compact, and the right adjoint $\pr_{[\eA,F_{\eA}]}$ is continuous.
\end{atom}

\begin{atom}
\label{atm:tr-adj}
Finally, we give two methods for constructing adjunctions between monoidal traces, following the approach of \cite[\S3.10.4]{gkrv}.
\end{atom}

\begin{atom}
We begin by discussing the induction-restriction adjunction. Let $(\eA,F_{\eA})$ and $(\eB,F_{\eB})$ be as in \S\ref{atm:tr-A-F-cat-end}, and suppose that we are given a rigid monoidal functor $\Psi\colon\eA\to\eB$, equipped with an isomorphism $\Psi\circ F_{\eA}\simeq F_{\eB}\circ\Psi$. Then by \S\ref{atm:mon-funct-dual}, we have $1$-morphisms
\begin{equation}
\label{eqn:Tr-Ind-Res-adj}
(\Ind_{\Psi},F_{\eB}\otimes\id_{\eA})\colon(\underline{\eA\dashbmod},{}_{F_{\eA}}\eA)\rightleftarrows(\underline{\eB\dashbmod},{}_{F_{\eB}}\eB)\colon(\Res_{\Psi},\id_{\eB})
\end{equation}
in $L(\Morita(\dgCat_k))_\rgd$. More precisely, these functors are given by the compositions
\begin{gather*}
\eB_{\Psi}\tens{\eA}{}_{F_{\eA}}\eA\xrightarrow{F_{\eB}\otimes\id_{\eA}}{}_{F_{\eB}}\eB_{F_{\eB}\circ\Psi}\tens{\eA}{}_{F_{\eA}}\eA\simeq{}_{F_{\eB}}\eB_{\Psi\circ F_{\eA}}\tens{\eA}{}_{F_{\eA}}\eA\to{}_{F_{\eB}}\eB_{\Psi}\tens{\eA}\eA\simeq{}_{F_{\eB}}\eB_{\Psi},\\
{}_{\Psi}\eB\tens{\eB}{}_{F_{\eB}}\eB\simeq{}_{F_{\eB}\circ\Psi}\eB\simeq{}_{\Psi\circ F_{\eA}}\eB,
\end{gather*}
respectively. We now aim to endow \eqref{eqn:Tr-Ind-Res-adj} with the structure of an adjunction in $L(\Morita(\dgCat_k))_\rgd$, i.e., to construct unit and counit $2$-morphisms
\begin{equation}
\label{eqn:tr-ind-res-adj-2-mor}
\begin{split}
\id_{(\underline{\eA\dashbmod},{}_{F_{\eA}}\eA)}&\Rightarrow(\Res_{\Psi},\id_{\eB})\circ(\Ind_{\Psi},F_{\eB}\otimes\id_{\eA}),\\
(\Ind_{\Psi},F_{\eB}\otimes\id_{\eA})\circ(\Res_{\Psi},\id_{\eB})&\Rightarrow\id_{(\underline{\eB\dashbmod},{}_{F_{\eB}}\eB)},
\end{split}
\end{equation}
satisfying the usual identities. By rigidity, the functors \eqref{eqn:mon-coev} and \eqref{eqn:mon-ev} admit continuous bilinear right adjoints. Moreover, we take the natural isomorphisms exhibiting commutativity of the evident diagrams
\begin{equation*}
\begin{tikzcd}[column sep=large]
\eA\tens{\eA}{}_{F_\eA}\eA\arrow{rr}{\sim}\arrow{d}{\Psi\otimes\id_{\eA}}&&{}_{F_\eA}\eA\tens{\eA}\eA\arrow{d}{\id_{\eA}\otimes\Psi}\\
{}_{\Psi}\eB_{\Psi}\tens{\eA}{}_{F_\eA}\eA\arrow{r}{F_{\eB}\otimes\id_{\eA}}&{}_{F_{\eB}\circ\Psi}\eB_{\Psi}\arrow{r}{\sim}&{}_{F_\eA}\eA\tens{\eA}{}_{\Psi}\eB_{\Psi},
\end{tikzcd}
\end{equation*}
\begin{equation*}
\begin{tikzcd}[column sep=large]
(\eB_{\Psi}\tens{\eA}{}_{\Psi}\eB)\tens{\eB}{}_{F_\eB}\eB\arrow{r}{\sim}\arrow{d}{\mult_{\eB}\otimes\id_{\eB}}&\eB_{\Psi}\tens{\eA}{}_{\Psi\circ F_{\eA}}\eB\arrow{r}{F_{\eB}\otimes\id_{\eB}}&{}_{F_\eB}\eB\tens{\eB}(\eB_{\Psi}\tens{\eA}{}_{\Psi}\eB)\arrow{d}{\id_{\eB}\otimes\mult_{\eB}}\\
\eB\tens{\eB}{}_{F_\eB}\eB\arrow{rr}{\sim}&&{}_{F_\eB}\eB\tens{\eB}\eB,
\end{tikzcd}
\end{equation*}
or equivalently, the identities
\begin{align*}
F_{\eB}(\Psi(a))\Psi(a')&\simeq\Psi(F_{\eA}(a)a'),\\
F_{\eB}(bb')b''&\simeq F_{\eB}(b)(F_{\eB}(b')b'')
\end{align*}
for $a,a'\in\eA$ and $b,b',b''\in\eB$, respectively. It is now straightforward to verify that the resulting $2$-morphisms \eqref{eqn:tr-ind-res-adj-2-mor} yield an adjunction. Applying the functor \eqref{eqn:L-morita-tr-funct} then yields an adjunction
\begin{equation}
\label{eqn:tr-adj-ind-res}
\Tr(\Ind_{\Psi},F_{\eB}\otimes\id_{\eA})\colon\Tr(\eA,F_{\eA})\rightleftarrows\Tr(\eB,F_{\eB})\colon\Tr(\Res_{\Psi},\id_{\eB}),
\end{equation}
as desired. We henceforth misuse notation slightly and denote the $1$-morphisms of \eqref{eqn:Tr-Ind-Res-adj} simply by $\Ind_{\Psi}$ and $\Res_{\Psi}$, respectively; the functors of \eqref{eqn:tr-adj-ind-res} are then denoted by $\Tr(\Ind_{\Psi})$ and $\Tr(\Res_{\Psi})$ as in \cite[(3.30)]{gkrv}.

For future reference, we note the following explicit description of $\Tr(\Ind_{\Psi})$ (which holds even if $\Psi$ is not necessarily rigid).
\end{atom}

\begin{lemma}
\label{lem:tr-ind}
The functor
\begin{equation*}
\Tr(\Ind_{\Psi})\colon\Tr(\eA,F_{\eA})\to\Tr(\eB,F_{\eB})
\end{equation*}
is given by the composition
\begin{equation*}
\eA\tens{\eA\otimes\eA^\rev}{}_{F_{\eA}}\eA\xrightarrow{\Psi\otimes\Psi}{}_{\Psi}\eB_{\Psi}\tens{\eA\otimes\eA^\rev}{}_{F_{\eB}\circ\Psi}\eB_{\Psi}\to\eB_{\Psi}\tens{\eB\otimes\eB^\rev}{}_{F_{\eB}}\eB.
\end{equation*}
In particular, the diagram
\begin{equation*}
\begin{tikzcd}[column sep=large]
\eA\arrow{r}{\Psi}\arrow{d}{[-]}&\eB\arrow{d}{[-]}\\
\Tr(\eA,F_{\eA})\arrow{r}{\Tr(\Ind_{\Psi})}&\Tr(\eB,F_{\eB})
\end{tikzcd}
\end{equation*}
commutes.
\end{lemma}

\begin{proof}
It follows directly from the definitions (using the duality of \S\ref{atm:mon-funct-dual}) that $\Tr(\Ind_{\Psi})$ is given by the composition
\begin{align*}
\eA\tens{\eA\otimes\eA^{\rev}}{}_{F_{\eA}}\eA&\xrightarrow{\Psi\otimes\id_{\eA}}{}_{\Psi}\eB_{\Psi}\tens{\eA\otimes\eA^{\rev}}{}_{F_{\eA}}\eA\\
&\xrightarrow{\sim}\eB\tens{\eB\otimes\eB^{\rev}}({}_{\Psi}\eB\otimes\eB_{\Psi})\tens{\eA\otimes\eA^{\rev}}{}_{F_{\eA}}\eA\\
&\xrightarrow{\sim}\eB\tens{\eB\otimes\eB^{\rev}}(\eB_{\Psi}\tens{\eA}{}_{\Psi\circ F_{\eA}}\eB)\\
&\xrightarrow{\id_{\eB}\otimes(F_{\eB}\otimes\id_{\eB})}\eB\tens{\eB\otimes\eB^{\rev}}({}_{F_{\eB}}\eB_{F_{\eB}\circ\Psi}\tens{\eA}{}_{F_{\eB}\circ\Psi}\eB)\\
&\xrightarrow{\id_{\eB}\otimes\mult_{\eB}}\eB\tens{\eB\otimes\eB^{\rev}}{}_{F_{\eB}}\eB.
\end{align*}
Tracing the constructions immediately yields the result.
\end{proof}

\begin{atom}
\label{atm:drinf-cent-tr-adj}
Next, we discuss adjunctions arising from dualizable objects of the Drinfeld center. Given a monoidal category $\eB$, recall that the \emph{Drinfeld center} of $\eB$ is the dg-category
\begin{equation*}
Z(\eB):=\Hom_{\eB\otimes\eB^\rev}(\eB,\eB),
\end{equation*}
which carries a natural $E_2$-monoidal structure via composition, and is equipped with a universal central functor
\begin{equation}
\label{eqn:univ-cent-funct}
Z(\eB)\to\eB
\end{equation}
given by evaluation at $\1_{\eB}$. In particular, for any $\eB$-bimodule $\eQ$, there is a natural $Z(\eB)$-module structure on $\Tr(\eB,\eQ)$. It follows that for any right-dualizable (equivalently, left-dualizable) object $b\in Z(\eB)$, we have a natural adjunction
\begin{equation}
\label{eqn:tr-adj-drinf-cent}
b\otimes-\colon\Tr(\eB,\eQ)\rightleftarrows\Tr(\eB,\eQ)\colon b^{\vee,R}\otimes-.
\end{equation}

In particular, suppose we are given a monoidal functor $\Psi\colon\eA\to\eB$ which admits a central structure, i.e., a factorization of $\Psi$ through \eqref{eqn:univ-cent-funct}. Then for any $a\in\eA$, we have a $\eB$-bilinear functor $-\otimes\Psi(a)\colon\eQ\to\eQ$, and hence a $1$-morphism
\begin{equation*}
(\eB,-\otimes\Psi(a))\colon(\underline{\eB\dashbmod},\eQ)\to(\underline{\eB\dashbmod},\eQ)
\end{equation*}
in $L(\Morita(\dgCat_k))$. Moreover, if $a$ is left-dualizable\footnote{A similar construction applies using the left-action of $\Psi(a)$ on $\eQ$, in which case it is more natural to take $a$ to be right-dualizable. However, the right-action on $\eQ$ agrees more naturally with the left-action on $\eB$ in the tensor product describing $\Tr(\eB,\eQ)$. In either case, the two constructions are related by the $E_2$-monoidal structure on $Z(\eB)$.}, applying $\Psi$ to the unit and counit maps in $\eA$ yields $2$-morphisms 
\begin{gather*}
(\eB,-\otimes\Psi(a))\circ(\eB,-\otimes\Psi(a^{\vee,L}))\simeq(\eB,-\otimes\Psi(a^{\vee,L}\otimes a))\Rightarrow(\eB,-\otimes\Psi(\1_{\eA}))\simeq(\eB,\id_{\eQ}),\\
(\eB,\id_{\eQ})\simeq(\eB,-\otimes\Psi(\1_{\eA}))\Rightarrow(\eB,-\otimes\Psi(a\otimes a^{\vee,L}))\simeq(\eB,-\otimes\Psi(a^{\vee,L}))\circ(\eB,-\otimes\Psi(a))
\end{gather*}
in $L(\Morita(\dgCat_k))$ satisfying the same identities, hence an adjunction
\begin{equation}
\label{eqn:drinf-cent-tr-adj}
(\eB,-\otimes\Psi(a))\colon(\underline{\eB\dashbmod},\eQ)\rightleftarrows(\underline{\eB\dashbmod},\eQ)\colon(\eB,-\otimes\Psi(a^{\vee,L}))
\end{equation}
in $L(\Morita(\dgCat_k))$. Applying the functor \eqref{eqn:L-morita-tr-funct} then recovers the adjunction \eqref{eqn:tr-adj-drinf-cent} (using the $E_2$-monoidal structure on $Z(\eB)$). Finally, note that given any $(\eN,\beta)\in(\eB,\eQ)\dashbmod$, we have
\begin{equation*}
(\eB,-\otimes\Psi(a))\circ(\eN,\beta)\simeq(\eN,\Psi(a)\otimes\beta(-))\in(\eB,\eQ)\dashbmod,
\end{equation*}
so we obtain
\begin{equation}
\label{eqn:drinf-cent-tr-adj-2-cat-cl}
\Tr(\eB,-\otimes\Psi(a))([\eN,\beta])\simeq[\eN,\Psi(a)\otimes\beta(-)]
\end{equation}
on $2$-categorical classes.
\end{atom}

\subsection{Traces of categories of quasicoherent sheaves}
\label{subsec:tr-qc}

\begin{definition}
\label{def:loop-space}
Let $\eX$ be a (derived) stack equipped with a self-map $\phi_{\eX}\colon\eX\to\eX$. The \emph{$\phi_{\eX}$-twisted loop space} (or \emph{derived $\phi_{\eX}$-fixed points}) of $\eX$ is given by the fiber product
\begin{equation*}
\begin{tikzcd}
\cL_{\phi}\eX\arrow{r}\arrow{d}{\ev}&\eX\arrow{d}{\Gamma_{\phi_{\eX}}}\\
\eX\arrow{r}{\Delta_{\eX}}&\eX\times\eX,
\end{tikzcd}
\end{equation*}
where $\Delta_{\eX}$ denotes the diagonal morphism, and $\Gamma_{\phi_{\eX}}=(\phi_{\eX},\id_{\eX})$ denotes the graph of $\phi_{\eX}$. We will sometimes notate this fiber product by $\eX\times_{\eX\times\eX}^{\Gamma_{\phi_{\eX}}}\eX$. We refer to the morphism $\ev$ as the \emph{loop evaluation}. When $\phi_{\eX}=\id_{\eX}$, we write $\cL\eX:=\cL_{\phi_{\eX}}\eX$, and refer to it simply as the \emph{loop space} of $\eX$.
\end{definition}

\begin{remark}
Alternatively, we have $\cL\eX\simeq\Map(S^1,\eX)$, the derived mapping stack from the circle, which thus carries a natural $S^1$-action. We shall not need this in the present work, however. Note that we reserve the notation $\eX^{\phi_{\eX}}$ for schemes, in which case it denotes the \emph{classical} fixed points $(\cL_{\phi_{\eX}}\eX)^{\cl}$.
\end{remark}

\begin{atom}
\label{atm:loop-space-funct}
The formation of twisted loop spaces is functorial in the pair $(\eX,\phi_{\eX})$, and commutes with fiber products. More precisely, if $\eY$ is another stack equipped with a self-map $\phi_{\eY}\colon\eY\to\eY$, and $p\colon\eX\to\eY$ is a morphism intertwining these self-maps, that is, $p\circ\phi_{\eX}\simeq\phi_{\eY}\circ p$, then we have an induced morphism $\cL p\colon\cL_{\phi_{\eX}}\eX\to\cL_{\phi_{\eY}}\eY$.
\end{atom}

\begin{atom}
\label{atm:loop-space-quot-stack}

When $\eX=X/G$ is a quotient stack, for $X$ a (derived) scheme and $G$ a linear algebraic group, its loop space admits a more explicit description. Suppose that $\phi_{\eX}$ commutes with the natural projection $\eX\to\B G$, i.e., that $\phi_{\eX}$ lifts to an endomorphism $\phi_X\colon X\to X$ commuting with the $G$-action. Then by \cite[Prop.~3.1.6]{chen}, we have a Cartesian square
\begin{equation}
\label{eqn:loop-fiber-prod}
\begin{tikzcd}
\cL_{\phi}(X/G)\arrow{d}{\ev}\arrow{r}{\ev_G}&(X\times G)/G\arrow{d}{(\phi_X\circ\act,\pr)}\\
X/G\arrow{r}{\Delta_X}&(X\times X)/G,
\end{tikzcd}
\end{equation}
where $\act,\pr\colon X\times G\to X$ denote the action and projection maps, respectively, and $G$ acts diagonally on $X\times X$ and $X\times G$ (via the adjoint action on $G$). In particular, at the level of $k$-points, we have
\begin{equation*}
\cL_{\phi}(X/G)(k)\cong\{(x,g)\in X(k)\times G(k):\phi_X(g\cdot x)=x\}/G(k).
\end{equation*}

We now record two lemmas translating properties of stacks into properties of their twisted loop spaces:
\end{atom}

\begin{lemma}
\label{lem:loops-props}
Suppose we are in the setup of \S\ref{atm:loop-space-funct}.
\begin{enumerate}
\item\label{itm:closed-proper-loops} If $p$ is a closed immersion (resp. proper), then so is $\cL p$.
\item\label{itm:quasi-smooth-loops} If $p$ is smooth, then $\cL p$ is quasi-smooth.
\end{enumerate}
\end{lemma}

\begin{proof}
\eqref{itm:closed-proper-loops} Both statements are immediate from the factorization
\begin{equation}
\label{eqn:fact-loop-morphism}
\cL_{\phi_{\eX}}\eX\simeq\eX\fibprod{\eX\times\eX}^{\hspace{-5pt}\Gamma_{\phi_{\eX}}}\eX\xrightarrow{\Delta_{\eX/\eY}\times\id_{\eX}}(\eX\times_{\eY}\eX)\fibprod{\eX\times\eX}^{\hspace{-5pt}\Gamma_{\phi_{\eX}}}\eX\simeq\eX\fibprod{\eY\times\eY}^{\hspace{-4pt}\Gamma_{\phi_{\eY}}}\eY\xrightarrow{p\times\id_{\eY}}\eY\fibprod{\eY\times\eY}^{\hspace{-4pt}\Gamma_{\phi_{\eY}}}\eY\simeq\cL_{\phi_{\eY}}\eY
\end{equation}
of $\cL p$, as noted in \cite[Rem.~4.6]{nonlinear-traces} (see also \cite[Lem.~3.10]{benzvi}). Here $\Delta_{\eX/\eY}\colon\eX\to\eX\times_{\eY}\eX$ denotes the relative diagonal, and the latter maps $\eX\to\eX\times\eX$ and $\eY\to\eY\times\eY$ are given by $\Gamma_{\phi_{\eX}}$ and $\Gamma_{\phi_{\eY}}$, respectively.

\eqref{itm:quasi-smooth-loops} Suppose that $p$ is smooth, i.e., the cotangent complex $\bL_{p}$ is perfect of Tor-amplitude $[0,1]$. The factorization \eqref{eqn:fact-loop-morphism} yields an exact triangle
\begin{equation*}
(\Delta_{\eX/\eY}\times\id_{\eX})^*\bL_{p\times\id_{\eY}}\to\bL_{\cL p}\to\bL_{\Delta_{\eX/\eY}\times\id_{\eX}}.
\end{equation*}
Since $\bL_{p\times\id_{\eY}}\simeq\pr_1^*\bL_p$ and $\bL_{\Delta_{\eX/\eY}\times\id_{\eX}}\simeq\pr_1^*\bL_{\Delta_{\eX/\eY}}$ by base-change (where we have let $\pr_1,\pr_2$ denote the respective projections), we reduce to showing that $\bL_{\Delta_{\eX/\eY}}$ is perfect of Tor-amplitude $[-1,1]$. The exact triangle associated to the composition
\begin{equation*}
\eX\xrightarrow{\Delta_{\eX/\eY}}\eX\fibprod{\eY}\eX\xrightarrow{\pr_2}\eX
\end{equation*}
now yields isomorphisms
\begin{equation*}
\bL_{\Delta_{\eX/\eY}}\simeq\Delta_{\eX/\eY}^*\bL_{\pr_2}[1]\simeq\Delta_{\eX/\eY}^*\pr_1^*\bL_{p}[1]\simeq\bL_p[1],
\end{equation*}
and the conclusion follows.
\end{proof}

\begin{lemma}
\label{lem:open-closed-loops}
Let $X$ be a derived scheme equipped with an action of a linear algebraic group $G$ and with a self-map $\phi\colon X\to X$ commuting with the $G$-action. Let $i\colon Z\hookrightarrow X\hookleftarrow U\colon j$ be a complementary closed and open immersion, respectively, and suppose that $Z$ and $U$ are both $G$-stable and $\phi$-stable. Then
\begin{equation}
\label{eqn:open-closed-loops}
\cL i\colon\cL_\phi(Z/G)\to\cL_\phi(X/G)\leftarrow\cL_\phi(U/G)\colon\cL j
\end{equation}
are a complementary closed and open immersion, respectively.
\end{lemma}

\begin{proof}
The map $\cL i$ is a closed immersion by Lemma~\ref{lem:loops-props}\eqref{itm:closed-proper-loops}. Alternatively, recall that the property of being a closed immersion depends only on the underlying classical stacks. The diagram \eqref{eqn:loop-fiber-prod} implies that $\cL_\phi(Z/G)^\cl$ is computed by the classical fiber product
\begin{equation}
\label{eqn:closed-loops-fiber-prod}
Z^\cl/G\fibprod{X^\cl/G}\cL_\phi(X/G)^\cl,
\end{equation}
from which the conclusion is immediate.

Next, consider the commutative cube
\begin{equation*}
\begin{tikzcd}
\cL_\phi(U/G)\arrow{rr}\arrow{dd}\arrow{dr}{\cL j}&&(U\times G)/G\arrow[dd,pos=0.7,"{(\phi\circ\act,\pr)}"]\arrow{dr}{j\times\id_G}&\\
&\cL_\phi(X/G)\arrow{rr}\arrow{dd}&&(X\times G)/G\arrow{dd}{(\phi\circ\act,\pr)}\\
U/G\arrow[rr,pos=0.4,"\Delta_U"]\arrow{dr}{j}&&(U\times U)/G\arrow{dr}{j\times j}&\\
&X/G\arrow{rr}{\Delta_X}&&(X\times X)/G.
\end{tikzcd}
\end{equation*}
Its front and back faces are cartesian; moreover, its bottom face is cartesian as $U\times_XU\simeq U$. Thus, by a standard lemma on cartesian diagrams, the top face is cartesian as well, i.e.,
\begin{equation}
\label{eqn:open-loops-fiber-prod}
\cL_\phi(U/G)\simeq(U\times G)/G\fibprod{(X\times G)/G}\cL_\phi(X/G)\simeq U/G\fibprod{X/G}\cL_\phi(X/G),
\end{equation}
which implies the result.

Finally, the expressions \eqref{eqn:closed-loops-fiber-prod} and \eqref{eqn:open-loops-fiber-prod} imply that the immersions of \eqref{eqn:open-closed-loops} are complementary, as this condition depends only on the underlying topological spaces.
\end{proof}

\begin{atom}
We now turn to studying the category $\QC(\eX)$ of quasi-coherent sheaves on $\eX$ and its trace. Recall that $\QC(\eX)$ carries a symmetric monoidal structure via tensor product. For any stack (even prestack) $\eX$, the dualizable objects of $\QC(\eX)$ are given by the perfect complexes $\Perf(\eX)\subset\QC(\eX)$ (i.e., by sheaves whose pullback to any affine scheme mapping to $\eX$ is perfect). However, the category $\QC(\eX)$ is not in general compactly generated, or even rigid. To rectify these problems, we introduce the following conditions, following \cite[Ch.~3, \S3.5]{gr} and \cite{bfn}:
\end{atom}

\begin{definition}
\label{def:passable-perfect}
A stack $\eX$ is \emph{passable} if
\begin{enumerate}
\item its diagonal morphism is quasi-affine;
\item the structure sheaf $\O_{\eX}\in\QC(\eX)$ is compact; and
\item the category $\QC(\eX)$ is dualizable.
\end{enumerate}
It is \emph{perfect} if its diagonal morphism is furthermore affine, and $\QC(\eX)$ is furthermore compactly generated.
\end{definition}

\begin{atom}
\label{atm:pass-perf-qc-rgd}
By \cite[Ch.~3, Prop.~3.4.2]{gr}, the category $\QC(\eX)$ is rigid if $\eX$ is passable, and compactly generated by $\Perf(\eX)$ if it is moreover perfect. In \cite{bfn}, it is shown that the class of perfect stacks includes
\begin{enumerate}
\item quasi-compact schemes with affine diagonal;
\item quotient stacks $X/G$, where $G$ is a linear algebraic group and $X$ is a finite-type scheme endowed with a $G$-equivariant ample line bundle (for instance, when $X$ is quasi-affine, we may take the structure sheaf); and
\item fiber products of perfect stacks.
\end{enumerate}
Finally, using \cite[Ch.~3, Prop.~3.5.3]{gr} (though see also \cite[Thm.~4.7]{bfn}), we can compute the categorical trace of $\QC(\eX)$ with respect to the endofunctor $\phi_{\eX}^*$:
\end{atom}

\begin{proposition}
\label{prop:tr-qc-loop-ident}
Suppose that $\eX$ is passable. Then we have a natural identification
\begin{equation}
\label{eqn:tr-qc-loop-ident}
\Tr(\QC(\eX),\phi_{\eX}^*)\simeq\QC(\cL_{\phi_{\eX}}\eX),
\end{equation}
with the universal trace functor \eqref{eqn:tr-char} given by pullback along the loop evaluation:
\begin{equation*}
[-]\simeq\ev^*\colon\QC(\eX)\to\QC(\cL_{\phi_{\eX}}\eX).
\end{equation*}
\end{proposition}

\begin{atom}
Note that here, Theorem~\ref{thm:gkrv-main} yields
\begin{equation}
\label{eqn:hh-qc-glob-sec-loops}
\HH(\QC(\eX),\phi_{\eX}^*)\simeq\Gamma(\O_{\cL_{\phi_{\eX}}\eX}),
\end{equation}
which may be checked directly using the duality data for $\QC(\eX)$. Similarly, though $\QC^!(\eX)$ is not generally rigid, it is dualizable whenever $\eX$ is QCA (see \S\ref{atm:qca-discussion}), and hence \eqref{eqn:hh-duality-data} yields
\begin{equation}
\label{eqn:hh-indcoh-glob-sec-loops}
\HH(\QC(\eX),\phi_{\eX}^!)\simeq\Gamma(\omega_{\cL_{\phi_{\eX}}\eX}).
\end{equation}
Finally, we describe the functoriality of \eqref{eqn:tr-qc-loop-ident}:
\end{atom}

\begin{corollary}
\label{cor:qc-tr-ind-res}
Suppose we are in the setup of \S\ref{atm:loop-space-funct}, and that $\eX,\eY$ are both passable, so that
\begin{equation*}
p^*\colon\QC(\eY)\to\QC(\eX)
\end{equation*}
is a monoidal functor of rigid monoidal categories intertwining the monoidal endofunctors $\phi_{\eY}^*$ and $\phi_{\eX}^*$. Then under Proposition~\ref{prop:tr-qc-loop-ident}, the adjoint functors
\begin{equation*}
\Tr(\Ind_{p^*})\colon\Tr(\QC(\eY),\phi_{\eY}^*)\rightleftarrows\Tr(\QC(\eX),\phi_{\eX}^*)\colon\Tr(\Res_{p^*})
\end{equation*}
of \eqref{eqn:tr-adj-ind-res} identify with
\begin{equation*}
\cL p^*\colon\QC(\cL_{\phi_{\eY}}\eY)\rightleftarrows\QC(\cL_{\phi_{\eX}}\eX)\colon\cL p_*.
\end{equation*}
\end{corollary}

\begin{proof}
The identification $\Tr(\Ind_{p^*})\simeq\cL p^*$ is immediate from Lemma~\ref{lem:tr-ind} and naturality of \cite[Ch.~3, Prop.~3.5.3]{gr}. The identification $\Tr(\Res_{p^*})\simeq\cL p_*$ then follows by adjunction.
\end{proof}

\subsection{Traces of convolution categories}
\label{subsec:tr-conv}

\begin{atom}
We begin with some technical recollections:
\end{atom}

\begin{definition}
\label{def:qca-stack}
An algebraic stack is \emph{QCA}\footnote{In the sense of \cite{benzvi}, rather than in the more general sense of \cite{dg-fin}.} if it is quasi-compact, of finite presentation, and has affine finitely presented diagonal.
\end{definition}

\begin{atom}
\label{atm:qca-discussion}
For instance, the quotient stack of a finitely presented affine scheme by an affine algebraic group is QCA. Moreover, it is not hard to see that fiber products (in particular, loop spaces) of QCA stacks are QCA. As in \cite[Thm.~4.3.1]{dg-fin}, any QCA stack is passable. Most saliently for our purposes, the category of ind-coherent sheaves on a QCA stack $\eX$ is compactly generated by its coherent subcategory, i.e., $\QC^!(\eX)\simeq\Ind(\Coh(\eX))$.
\end{atom}

\begin{atom}
\label{atm:sing-supp}
Next, given any quasi-smooth Artin stack $\eX$, we may consider its (classical) stack of singularities
\begin{equation*}
\Sing(\eX):=\uSpec_{\eX}\Sym_{\O_{\eX}}\H^1(\bL_{\eX}^\vee)
\end{equation*}
as in \cite{arinkin-gaitsgory}, which carries a canonical fiberwise $\bG_m$-action. Given any singular support condition $\Lambda$, i.e., a conical closed subset $\Lambda\subset\Sing(\eX)$, we may define a full subcategory $\QC^!_{\Lambda}(\eX)\subset\QC^!(\eX)$ spanning sheaves whose singular support is contained in $\Lambda$. This inclusion then admits a continuous colocalization; we denote this adjoint pair by
\begin{equation*}
\iota_{\Lambda}\colon\QC^!_{\Lambda}(\eX)\rightleftarrows\QC^!(\eX)\colon\Gamma_{\Lambda}.
\end{equation*}
In particular, letting $\{0\}_{\eX}\subset\Sing(\eX)$ denote the $0$-section, we have $\QC^!_{\{0\}_{\eX}}(\eX)\simeq\QC(\eX)$. On the other hand, for the vacuous singular support condition, we have $\QC^!_{\Sing(\eX)}(\eX)\simeq\QC^!(\eX)$. Moreover, for any closed substack $\eZ\subset\eX$, we have $\QC^!_{\eZ\times_{\eX}\{0\}_{\eX}}(\eX)\simeq\QC_{\eZ}(\eX)$ and $\QC^!_{\eZ\times_{\eX}\Sing(\eX)}(\eX)\simeq\QC^!_{\eZ}(\eX)$, i.e., the full subcategories of sheaves set-theoretically supported on $\eZ$.

Finally, given a map of such stacks $p\colon\eX\to\eY$, there is a correspondence
\begin{equation*}
\Sing(\eX)\xleftarrow{\Sing(p)}\Sing(\eY)\fibprod{\eY}\eX\xrightarrow{p}\Sing(\eY).
\end{equation*}
Note that by Lemmas~2.4.3 and 2.4.4 in \loccit, the ``singular codifferential'' $\Sing(p)$ is closed if $p$ is quasi-smooth, and an isomorphism if $p$ is smooth. Thus, given singular support conditions $\Lambda_{\eX}$ and $\Lambda_{\eY}$ for $\eX$ and $\eY$, we may define singular support conditions
\begin{equation*}
p_*\Lambda_{\eX}:=\overline{p(\Sing(p)^{-1}(\Lambda_{\eX}))},\qquad\qquad p^!\Lambda_{\eY}:=\overline{\Sing(p)(p^{-1}(\Lambda_{\eY}))}
\end{equation*}
for $\eY$ and $\eX$, respectively. We then have functors
\begin{equation*}
p_*\colon\QC^!_{\Lambda_{\eX}}(\eX)\to\QC^!_{p_*\Lambda_{\eX}}(\eY),\qquad\qquad p^!\colon\QC^!_{\Lambda_{\eY}}(\eY)\to\QC^!_{p^!\Lambda_{\eY}}(\eX).
\end{equation*}
If $p$ is moreover quasi-smooth (or more generally, Gorenstein), then the same holds for $p^*$ in place of $p^!$ (as \cite[Prop.~7.3.8]{gaitsgory-indcoh} implies that the two differ by a shifted line bundle).
\end{atom}

\begin{atom}
\label{atm:tr-conv-cat-setup}

We may now state the description of the categorical trace of a convolution category given in \cite[\S3]{benzvi}. Suppose we are in the setup of \S\ref{atm:loop-space-funct}, and moreover, that
\begin{enumerate}
\item the stacks $\eX,\eY$ are smooth and QCA;
\item the map $p\colon\eX\to\eY$ is proper; and
\item the self-maps $\phi_{\eX},\phi_{\eY}$ are automorphisms.
\end{enumerate}
Then the fiber product $\eX\times_{\eY}\eX$ is QCA and quasi-smooth (as in Lemma~\ref{lem:loops-props}\eqref{itm:quasi-smooth-loops}), and the category $\QC^!(\eX\times_{\eY}\eX)$ is monoidal under the $*$-convolution.\footnote{Note that our convention differs from that of \cite{benzvi}, where the $!$-convolution is instead used.} Moreover, the automorphism $\phi:=\phi_{\eX}\times_{\phi_{\eY}}\phi_{\eX}$ of $\eX\times_{\eY}\eX$ yields a monoidal endofunctor $\phi^*$ of $\QC^!(\eX\times_{\eY}\eX)$. We then have:
\end{atom}

\begin{proposition}
\label{prop:tr-conv}
The monoidal category $\QC^!(\eX\times_{\eY}\eX)$ is rigid. Moreover, we have a natural identification
\begin{equation*}
\Tr(\QC^!(\eX\fibprod{\eY}\eX),\phi^*)\simeq\QC_{\Lambda_{\eX/\eY,\phi}}(\cL_{\phi_{\eY}}\eY),
\end{equation*}
where
\begin{equation*}
\Lambda_{\eX/\eY,\phi}:=(p\times\id_{\eY})_*\pr_1^!\Sing(\eX\fibprod{\eY}\eX)\subset\Sing(\cL_{\phi_{\eY}}\eY)
\end{equation*}
is the singular support condition obtained via the ``trace correspondence''\footnote{Here the middle term is as in \eqref{eqn:fact-loop-morphism}.}
\begin{equation}
\label{eqn:univ-trace-corresp}
\begin{tikzcd}[column sep=large]
\eX\fibprod{\eY}\eX&\arrow[l,"\pr_1"'](\eX\fibprod{\eY}\eX)\fibprod{\eX\times\eX}^{\hspace{-5pt}\Gamma_{\phi_{\eX}}}\eX\simeq\eX\fibprod{\eY\times\eY}^{\hspace{-4pt}\Gamma_{\phi_{\eY}}}\eY\arrow{r}{p\times\id_{\eY}}&\cL_{\phi_{\eY}}\eY.
\end{tikzcd}
\end{equation}
Moreover, the universal trace functor \eqref{eqn:tr-char} is given by $(p\times\id_{\eY})_*\pr_1^*$.
\end{proposition}

\begin{remark}
\label{rem:conv-cat-pivotal}
Let us describe the monoidal duality in $\QC^!(\eX\times_{\eY}\eX)$ explicitly. Set $\eZ:=\eX\times_{\eY}\eX$, and let $\cF\in\Coh(\eX\times_{\eY}\eX)$. The proof of \cite[Thm.~3.25]{benzvi} shows that the right-dual of $\cF$ is given by
\begin{equation*}
\cF^{\vee,R}\simeq\sigma^*(\omega_{\eZ/\eX}\otimes\bD_{\eZ}(\cF)\otimes\omega_{\eZ}^\vee),
\end{equation*}
where $\sigma\colon\eZ\to\eZ$ is the ``swap'' map as in \S\ref{atm:braid-positivity}, $\omega_{\eZ/\eX}:=\pr_1^!\O_{\eX}$ is the relative dualizing sheaf with respect to the first projection, $\bD_{\eZ}$ denotes Grothendieck--Serre duality on $\eZ$, and $\omega_{\eZ}$ denotes the dualizing sheaf of $\eZ$ (which is a shifted line bundle as $\eZ$ is quasi-smooth).

Now suppose that $\eX$ is Calabi--Yau, i.e., has trivial canonical bundle. Then we claim that $\cF$ is \emph{pivotal}, that is, its left and right monoidal duals coincide. Indeed, in this case, we have $\omega_{\eZ/\eX}\simeq\omega_{\eZ}[-\dim\eX]$, hence $\eF^{\vee,R}\simeq\sigma^*\bD_{\eZ}(\cF)[-\dim\eX]$. Note that the functors $\sigma^*$ and $\bD_{\eZ}$ commute. Moreover, $\sigma^*$ is clearly involutive, and $\bD_{\eZ}$ is involutive by \cite[\S4.4.3]{dg-fin}. It follows that
\begin{equation*}
(\cF^{\vee,R})^{\vee,R}\simeq\sigma^*\bD_{\eZ}(\sigma^*\bD_{\eZ}(\cF)[-\dim\eX])[-\dim\eX]\simeq\sigma^*\bD_{\eZ}\sigma^*\bD_{\eZ}(\cF)\simeq\cF,
\end{equation*}
as desired. In particular, this property holds for the ``partial affine Hecke categories'' $\eH_{P,\eS_e}^\coh$ of \eqref{eqn:aff-hecke-category-Se}, as
\begin{equation}
\label{eqn:res-Se-calabi-yau}
\omega_{\wt{\eN}_{P,\eS_e}/\wt{Z}_e}\simeq\O_{\wt{\eN}_{P,\eS_e}/\wt{Z}_e}[\dim\wt{\eN}_{P,\eS_e}-\dim\wt{Z}_e]\angles{\dim\wt{\eN}_{P,\eS_e}}
\end{equation}
up to a character of $Z_e$ by \eqref{eqn:omega-N-P-Se} (these twists clearly do not disrupt the above argument).
\end{remark}

\begin{atom}
Note that by base-change, we have
\begin{equation}
\label{eqn:class-unit-conv-cat}
[\Delta_{\eX/\eY,*}\O_{\eX}]\simeq\cL p_*\O_{\eX}\in\QC_{\Lambda_{X/Y,\phi}}(\cL_{\phi_{\eY}}\eY).
\end{equation}
Thus, Theorem~\ref{thm:gkrv-main} yields an algebra isomorphism
\begin{equation}
\label{eqn:gkrv-main-conv-cat}
\HH(\QC^!(\eX\fibprod{\eY}\eX),\phi^*)\simeq\End_{\cL_{\phi_{\eY}}\eY}(\cL p_*\O_{\eX})^\op.
\end{equation}

In particular, when $p$ is the identity, we recover Proposition~\ref{prop:tr-qc-loop-ident} (since the categories of ind-coherent and quasi-coherent sheaves are equivalent for smooth stacks). We will now relate the two situations more generally. Observe the category $\QC(\eX)$ is equipped with both $(\QC^!(\eX\times_{\eY}\eX),\QC(Y))$- and $(\QC(Y),\QC^!(\eX\times_{\eY}\eX))$-bimodule structures via left and right convolution, respectively; we refer to it in either case as the ``regular bimodule.'' Note that the action of the subcategory $\Coh(\eX\times_{\eY}\eX)$ preserves compact objects of $\QC(\eX)$ by \cite[Thm.~1.1.3]{ben-zvi-nadler-preygel}. Assuming right-dualizability, we obtain a diagram
\begin{equation*}
\begin{tikzcd}[column sep=7em]
(\QC^!(\eY)\dashbmod,\phi_{\eY}^*)\arrow[r,shift left,"{(\QC(\eX),\phi_{\eX}^*)}"]&(\QC^!(\eX\times_{\eY}\eX)\dashbmod,\phi^*)\arrow[l,shift left,"{(\QC(\eX),\phi_{\eX}^*)}"]
\end{tikzcd}
\end{equation*}
in $L(\Morita(\dgCat_k))_\rgd$. The following lemma characterizes the induced functors on categorical traces:
\end{atom}

\begin{lemma}
\label{lem:reg-mod-rt-left-dual}
The regular $(\QC^!(\eX\times_{\eY}\eX),\QC(\eY))$-bimodule $\QC(\eX)$ is both left- and right-dual to the regular $(\QC(\eY),\QC^!(\eX\times_{\eY}\eX))$-bimodule $\QC(\eX)$. Moreover, we have commutative squares
\begin{equation}
\label{eqn:comm-sq-reg-bimod-tr-conv}
\begin{tikzcd}[column sep=large,row sep=large]
\Tr(\QC(\eY),\phi_{\eY}^*)\arrow[rr,shift left,"{\Tr(\QC(\eX),\phi_{\eX}^*)}"]\arrow{d}{\vertsim}&&\Tr(\QC^!(\eX\times_{\eY}\eX),\phi^*)\arrow{d}{\vertsim}\arrow[ll,shift left,"{\Tr(\QC(\eX),\phi_{\eX}^*)}"]\\
\QC(\cL_{\phi_{\eY}}\eY)\arrow[r,shift left,"\Gamma_{\{0\}_{p(\eX)}}"]&\QC^!_{\{0\}_{p(\eX)}}(\cL_{\phi_{\eY}}\eY)\arrow[r,shift left,"\iota_{\{0\}_{p(\eX)}}"]\arrow[l,shift left,"\iota_{\{0\}_{p(\eX)}}"]&\QC^!_{\Lambda_{\eX/\eY,\phi}}(\cL_{\phi_{\eY}}\eY)\arrow[l,shift left,"\Gamma_{\{0\}_{p(\eX)}}"],
\end{tikzcd}
\end{equation}
where the vertical identifications are those of Proposition~\ref{prop:tr-conv}, and $\{0\}_{p(\eX)}:=\ev^!p(\eX)\subset\{0\}_{\cL_{\phi_{\eY}}\eY}$ denotes the pullback of the classical support condition $p(\eX)\subset\Sing(\eY)\simeq\eY$.
\end{lemma}

\begin{proof}
Right-duality was established in the proof of \cite[Prop.~3.32]{benzvi} (it is not difficult to verify that all results apply in the setting of the $*$-convolution, in addition to the $!$-convolution). Left-duality now follows immediately using \cite[Cor.~6.4.1]{gaitsgory-dgcat} and rigidity of the monoidal categories $\QC^!(\eX\times_{\eY}\eX)$ and $\QC(\eY)$. More explicitly, the functors
\begin{equation}
\label{eqn:reg-bimod-left-dual-funct}
\begin{gathered}
\QC(\eX)\tens{\QC^!(\eX\fibprod{\eY}\eX)}\QC(\eX)\simeq\QC^!_{p(\eX)}(\eY)\xhookrightarrow{\iota_{p(\eX)}}\QC^!(\eY)\\
\QC^!(\eX\fibprod{\eY}\eX)\xrightarrow{\Gamma_{\{0\}_{\eX\times_{\eY}\eX}}}\QC(\eX\fibprod{\eY}\eX)\simeq\QC(\eX)\tens{\QC(\eY)}\QC(\eX)
\end{gathered}
\end{equation}
exhibit this left-duality. Note that the counit is evidently $\QC^!(\eY)$-bilinear and continuous. Moreover, the unit is the continuous right adjoint to the evident $\QC^!(\eX\times_{\eY}\eX)$-bilinear inclusion, hence $\QC^!(\eX\times_{\eY}\eX)$-bilinear by \cite[Cor.~6.2.4]{gaitsgory-dgcat}. The duality identities now follow immediately from the proof of \cite[Prop.~3.32]{benzvi} by adjunction.

For the latter assertion, commutativity of the clockwise square was established in \cite[Prop.~3.32]{benzvi}. The counter-clockwise square follows by an identical argument, using the functors \eqref{eqn:reg-bimod-left-dual-funct}.
\end{proof}

\begin{atom}
We now use this result to give an alternative description of the universal trace functor for the convolution category $\QC^!(\eX\times_{\eY}\eX)$. Write
\begin{equation}
\label{eqn:qc-class-not-cor}
(-)^{\QC}:=\iota_{\{0\}_{p(\eX)}}\Gamma_{\{0\}_{p(\eX)}}\colon\QC^!_{\Lambda_{\eX/\eY,\phi}}(\cL_{\phi_{\eY}}\eY)\to\QC(\cL_{\phi_{\eY}}\eY).
\end{equation}
\end{atom}

\begin{corollary}
\label{cor:reg-mod-rt-left-dual}
Let $\cF\in\QC^!(\eX\times_{\eY}\eX)$.
\begin{enumerate}
\item We have
\begin{equation*}
[\cF]^{\QC}\simeq[\QC(\eX),\phi_{\eX}^*(-)\star\cF]\in\QC(\cL_{\phi_{\eY}}\eY).
\end{equation*}
\item If $\cF$ is coherent, then so is $[\cF]$. Moreover, $[\cF]$ is connective (resp.\ coconnective) if and only if $[\QC(\eX),\phi_{\eX}^*(-)\star\cF]$ is connective (resp.\ coconnective).
\end{enumerate}
\end{corollary}

\begin{proof}
The first two assertions are immediate from \eqref{eqn:comm-sq-reg-bimod-tr-conv} and \S\ref{atm:tr-M-subcat}, respectively. The third assertion then follows from \cite[Prop.~11.7.5]{gaitsgory-indcoh}.
\end{proof}

\begin{atom}
\label{atm:conv-cat-funct}
Finally, we explain the functoriality of Proposition~\ref{prop:tr-conv}, as in Corollary~\ref{cor:qc-tr-ind-res}. Let $\eY'$ be a smooth QCA stack equipped with an automorphism $\phi_{\eY'}$, and let $f\colon\eY'\to\eY$ be a morphism intertwining $\phi_{\eY'}$ and $\phi_{\eY}$. Define $p'\colon\eX':=\eY'\times_{\eY}\eX\to\eY'$ and $\phi_{\eX'}:=\phi_{\eY'}\times_{\phi_{\eY}}\phi_{\eX}$ by pullback along $f$, and set $\phi':=\phi_{\eX'}\times_{\phi_{\eY'}}\phi_{\eX'}$ as before. Since $f$ is quasi-smooth, it is locally eventually coconnective by \cite[Cor.~2.2.4]{arinkin-gaitsgory}, and hence (misusing notation slightly) we have an adjoint pair
\begin{equation}
\label{eqn:f-conv-adj}
f^*\colon\QC^!(\eX\fibprod{\eY}\eX)\rightleftarrows\QC^!(\eX'\fibprod{\eY'}\eX')\colon f_*.
\end{equation}
In particular, $f^*$ is compact object-preserving and monoidal, and it intertwines the monoidal endofunctors $\phi^*$ and $\phi^{\prime,*}$.
\end{atom}

\begin{corollary}
\label{cor:tr-conv-funct}
Suppose that $f$ is a smooth relative scheme. Then we have
\begin{equation*}
\Lambda_{\eX'/\eY',\phi'}=\cL f^!\Lambda_{\eX/\eY,\phi},\qquad\qquad\cL f_*\Lambda_{\eX'/\eY',\phi'}\subseteq\Lambda_{\eX/\eY,\phi},
\end{equation*}
and the adjoint functors
\begin{equation*}
\Tr(\Ind_{f^*})\colon\Tr(\QC^!(\eX\fibprod{\eY}\eX),\phi^*)\rightleftarrows\Tr(\QC^!(\eX'\fibprod{\eY'}\eX'),\phi^{\prime,*})\colon\Tr(\Res_{f^*})
\end{equation*}
of \eqref{eqn:tr-adj-ind-res} identify with
\begin{equation}
\label{eqn:tr-conv-loop-adj}
\cL f^*\colon\QC^!_{\Lambda_{\eX/\eY,\phi}}(\cL_{\phi_{\eY}}\eY)\rightleftarrows\QC^!_{\Lambda_{\eX'/\eY',\phi'}}(\cL_{\phi_{\eY'}}\eY')\colon\cL f_*
\end{equation}
under Proposition~\ref{prop:tr-conv}.
\end{corollary}

\begin{proof}
The claims about singular supports are easily checked from the definitions and properties listed in \S\ref{atm:sing-supp} (in particular, the second claim follows from the first). Since $\cL f$ is quasi-smooth by Lemma~\ref{lem:loops-props}\eqref{itm:quasi-smooth-loops}, we indeed have an adjoint pair \eqref{eqn:tr-conv-loop-adj} as in \eqref{eqn:f-conv-adj}. It therefore suffices to identify $\Tr(\Ind_{f^*})$ with $\cL f^*$ under Proposition~\ref{prop:tr-conv}. By \S\ref{atm:tr-M-subcat}, we reduce to showing this for images of compact objects under the universal trace functor. Given $\cF\in\Coh(\eX\times_{\eY}\eX)$, we have $\Tr(\Ind_{f^*})([\cF])\simeq[f^*\cF]$ by Lemma~\ref{lem:tr-ind}. The result now follows from \eqref{eqn:univ-trace-corresp} by noting that the rightmost square in the commutative diagram
\begin{equation*}
\begin{tikzcd}[column sep=large]
\eX\fibprod{\eY}\eX&\arrow[l,"\pr_1"'](\eX\fibprod{\eY}\eX)\fibprod{\eX\times\eX}\eX\arrow{r}{\sim}&[-2em]\eX\fibprod{\eY\times\eY}\eY\arrow{r}{p\times\id_{\eY}}&\cL_{\phi_{\eY}}\eY\\
\eX'\fibprod{\eY'}\eX'\arrow{u}{f}&\arrow[l,"\pr_1"'](\eX'\fibprod{\eY'}\eX')\fibprod{\eX'\times\eX'}\eX'\arrow{u}\arrow{r}{\sim}&\eX'\fibprod{\eY'\times\eY'}\eY'\arrow{r}{p'\times\id_{\eY'}}\arrow{u}&\cL_{\phi_{\eY'}}\eY'\arrow{u}{\cL f}
\end{tikzcd}
\end{equation*}
is Cartesian; here all vertical arrows are induced by $f$ in the evident manner. 
\end{proof}

\begin{atom}
In particular, we obtain $[\Delta_{\eX'/\eY',*}\O_{\eX'}]\simeq\cL f^*[\Delta_{\eX/\eY,*}\O_{\eX}]$, which agrees with the base-change isomorphism coming from \eqref{eqn:class-unit-conv-cat}.
\end{atom}

\section{Resolutions of Koszul algebras with multiple simple modules}
\label{sec:koszul-res}

\begin{atom}
In this short appendix, we construct a Koszul resolution of the regular bimodule for a Koszul algebra with finitely many simple modules (rather than just one). This material is likely well-known, but we could find no reference in the literature. We apply it in \S\ref{sec:comp-coh-spr} to the noncommutative Springer resolution, to obtain a bounded complex computing the universal trace functor.
\end{atom}

\begin{atom}
Let $A$ be a non-negatively graded classical $k$-algebra, and write $A\simeq\bigoplus_{i\in I}E_i$ for the decomposition of the regular right $A$-module into indecomposable projectives, where $I$ is some finite indexing set. Assume that the $\{E_i\}$ are pairwise nonisomorphic, and let $\{L_i\}$ denote the corresponding simple modules (concentrated in weight $0$). As for any $k$-algebra, the \emph{bar complex} $\hBar_A^\bullet$ is the acyclic complex of $A$-bimodules
\begin{equation*}
\cdots\to A\otimes A\otimes A\xrightarrow{a\otimes b\otimes c\mapsto ab\otimes c-a\otimes bc}A\otimes A\xrightarrow{a\otimes b\mapsto ab}A\to 0\to 0\to\cdots
\end{equation*}
concentrated in degrees $\le 1$ (and free in degrees $\le 0$), with differentials given by the usual alternating sum of face maps. Consider the subcomplex $\hBar_{\{E_i\}}^\bullet\subseteq\hBar_A^\bullet$ defined by
\begin{equation}
\label{eqn:bar-proj}
\hBar_{\{E_i\}}^{-n}:=\bigoplus_{i_0,\ldots,i_{n+2}\in I}\Hom_{A^\op}(E_{i_0},E_{i_1})\tens{k}\Hom_{A^\op}(E_{i_1},E_{i_2})\tens{k}\cdots\tens{k}\Hom_{A^\op}(E_{i_{n+1}},E_{i_{n+2}})
\end{equation}
for each $n\ge-1$, which is evidently preserved under each face map $d_1,\ldots,d_{n+1}$. It is a complex of $A$-bimodules (each of which is projective in degrees $\le 0$) via the algebra isomorphism $A\cong\End_{A^\op}(\bigoplus_{i\in I}E_i)$, and the restriction of the extra degeneracy from $\hBar_A^\bullet$ exhibits it as acyclic.

Now, for each $n\ge-1$ and $i,j\in I$, consider the summand ${}_i(\hBar_{\{E_i\}}^{-n})_j\subseteq\hBar_{\{E_i\}}^{-n}$ given by setting $i_0=i$, $i_{n+2}=j$, and letting $i_1,\ldots,i_{n+1}\in I$ be arbitrary. Furthermore, let ${}_i(\hBar_{\{E_i\}}^{-n})_j^{1,\ldots,1}\subseteq{}_i(\hBar_{\{E_i\}}^{-n})_j$ be the subspace spanned by terms whose $n+2$ tensor factors all lie in weight $1$. Set
\begin{equation*}
{}_i(A^{!,\vee}_n)_j:=\bigcap_{m=1}^n\ker\big(d_m\colon{}_i(\hBar_{\{E_i\}}^{-n+2})_j^{1,\ldots,1}\to{}_i(\hBar_{\{E_i\}}^{-n+3})_j\big).
\end{equation*}
for each $n\ge 1$; it is clearly a finite-dimensional vector space concentrated in weight $n$ (for instance, we have ${}_i(A^{!,\vee}_1)_j=\Hom_{A^\op}(E_i,E_j)_1\subseteq A_1$). Consider the subcomplex $\Kos_A^\bullet\subseteq\hBar_{\{E_i\}}^\bullet$ given by
\begin{equation*}
\Kos_A^{-n}:=\bigoplus_{i_0,i_1,i_{n+1},i_{n+2}\in I}\Hom_{A^\op}(E_{i_0},E_{i_1})\tens{k}{}_{i_1}(A^{!,\vee}_n)_{i_{n+1}}\tens{k}\Hom_{A^\op}(E_{i_{n+1}},E_{i_{n+2}})
\end{equation*}
for each $n\ge 1$, and $\Kos_A^{-n}:=\hBar_{\{E_i\}}^{-n}$ otherwise. Equivalently, let $E_i^\ell$ be the indecomposable projective \emph{left} $A$-module corresponding to $E_i$ (i.e., generated by the same primitive idempotent), and $L_i^\ell$ be the corresponding simple module. Setting ${}_i(A^{!,\vee}_0)_j:=\Hom_{A^\op}(E_i,E_j)_0$, we may write
\begin{equation}
\label{eqn:kos-proj-left-right}
\Kos_A^{-n}=\bigoplus_{i,j\in I}E_i^\ell\tens{k}{}_i(A^{!,\vee}_n)_j\tens{k}E_j
\end{equation}
for each $n\ge 0$, which is a projective $A$-bimodule. Note that the differential on $\Kos_A^{-n}$ is given by the restriction of $d_1+(-1)^nd_{n+1}$, whose image is clearly contained in $\Kos_A^{-n+1}$.
\end{atom}

\begin{atom}
Recall that $A$ is \emph{Koszul} if for any $i,j\in I$ and $n\ge 0$, the graded vector space $\Ext_{A^\op}^n(L_i,L_j)$ is concentrated in weight $-n$. We may now state our main result:
\end{atom}

\begin{proposition}
\label{prop:kos-acyclic}
If $A$ is Koszul, then the complex $\Kos_A^\bullet$ is acyclic.
\end{proposition}

\begin{proof}
It suffices to show that $L_i\otimes_A\Kos_A^\bullet\simeq 0$ for each $i\in I$. Indeed, suppose the complex $\Kos_A^\bullet$ of projective left $A$-modules has its highest nonzero cohomology in degree $d$. Then the convergent spectral sequence
\begin{equation*}
E_2^{s,t}=\H^s(L_i\tens{A}\H^t(\Kos_A^\bullet))\Longrightarrow\H^{s+t}(L_i\tens{A}\Kos_A^\bullet)=0
\end{equation*}
shows that $\H^0(L_i\otimes_A\H^d(\Kos_A^\bullet))=0$ for each $i\in I$, and hence $\H^d(\Kos_A^\bullet)=0$ by the graded version of Nakayama's lemma, a contradiction.

Consider the normalized version $\N\hBar_{E_i}^\bullet$ of the complex \eqref{eqn:bar-proj} obtained by quotienting by the subspaces generated by the degeneracies of the associated simplicial object. Explicitly, we have
\begin{equation}
\label{eqn:nbar-complex}
\N\hBar_{\{E_i\}}^{-n}\simeq\bigoplus_{i_1,\ldots,i_{n+1}\in I}E_{i_1}^\ell\tens{k}\Hom_{A^\op}(E_{i_1},E_{i_2})_{\ge 1}\tens{k}\cdots\tens{k}\Hom_{A^\op}(E_{i_{n}},E_{i_{n+1}})_{\ge 1}\tens{k}E_{i_{n+1}}
\end{equation}
for each $n\ge 0$ (where the subscripts denote the strictly positive components), and the natural inclusion $\N\hBar_{\{E_i\}}^\bullet\to\hBar_{\{E_i\}}^\bullet\simeq 0$ is a quasi-isomorphism. Thus, it suffices to show that the natural map
\begin{equation}
\label{eqn:kos-to-nbar}
L_i\tens{A}\Kos_A^\bullet\to L_i\tens{A}\N\hBar_{\{E_i\}}^\bullet
\end{equation}
is a quasi-isomorphism. In fact, it suffices to show that it is a quasi-isomorphism after applying $-\otimes_AL_j^\ell$, for any $j\in I$: indeed, the cone of \eqref{eqn:kos-to-nbar} is a complex of projective right $A$-modules (concentrated in degrees $\le 0$), so we may apply the argument used in our first paragraph.

Let $j\in I$, and observe that
\begin{equation*}
L_i\tens{A}\Kos_A^{\le 0}\tens{A}L_j^\ell\simeq\bigoplus_{n\ge 0}{}_i(A^{!,\vee}_n)_j^\vee[n],
\end{equation*}
where $\Kos_A^{\le 0}$ denotes the ``na\"ive'' truncation of $\Kos_A^\bullet$. Since each graded component of $\N\hBar_{\{E_i\}}^\bullet$ is perfect over $k$ (as $\N\hBar_{\{E_i\}}^{-n}$ lies in weights $\ge n$), and since $\Hom_{A^\op}(E_{j'},L_j)\simeq E_{j'}\otimes_AL_j^\ell$ for any $j'\in I$, we have
\begin{equation*}
\Hom_{A^\op}(L_i,L_j)\simeq\Hom_{A^\op}(L_i\tens{A}\N\hBar_{\{E_i\}}^{\le 0},L_j)\simeq(L_i\tens{A}\N\hBar_{\{E_i\}}^{\le 0}\tens{A}L_j^\ell)^\vee.
\end{equation*}
Thus, Koszulity of $A$ implies that the cohomology of $L_i\otimes_A\N\hBar_{\{E_i\}}^{\le 0}\otimes_AL_j^\ell$ lies only in degree $-n$ and weight $n$ for $n\ge 0$. It follows that $\H^{-n}(L_i\otimes_A\N\hBar_{\{E_i\}}^{\le 0}\otimes_AL_j^\ell)$ is given by the kernel of the usual differential on
\begin{equation*}
\bigoplus_{i_1,\ldots,i_{n+1}\in I}\Hom_{A^\op}(E_{i_1},E_{i_2})_{1}\tens{k}\cdots\tens{k}\Hom_{A^\op}(E_{i_{n}},E_{i_{n+1}})_{1}
\end{equation*}
for each $n\ge 0$. Now, $L_i\otimes_A\N\hBar_{\{E_i\}}^{-n+1}\otimes_AL_j^\ell$ is graded by $(n-1)$-tuples of weights; since each face map lands in a distinct such tuple, this kernel coincides with ${}_i(A^{!,\vee}_n)_j^\vee$, as desired.
\end{proof}

\begin{atom}
In particular, we have shown:
\end{atom}

\begin{corollary}
\label{cor:kos-dual-ext}
Suppose $A$ is Koszul. Then for each $n\ge 0$ and $i,j\in I$, we have
\begin{equation*}
{}_i(A^{!,\vee}_n)_j\simeq\Ext_{A^\op}^n(L_i,L_j)^\vee.
\end{equation*}
\end{corollary}

\clearpage

\bibliographystyle{amsalpha}
\bibliography{references}

\end{document}